\documentclass[11pt]{amsart}
\usepackage{amssymb,amsmath,geometry,latexsym,graphics,graphicx,tabularx,shapepar,enumerate,
hyperref,stmaryrd,glossaries}
\numberwithin{equation}{section}

\newcommand{\ZZ}{\mathbb{Z}}
\newcommand{\FF}{\mathbb{F}}
\newcommand{\CC}{\mathbb{C}}
\newcommand{\QQ}{\mathbb{Q}}
\newcommand{\RR}{\mathbb{R}}
\newcommand{\EE}{\mathbb{E}}

\newcommand{\KK}{\mathbb{K}}
\newcommand{\LL}{\mathbb{L}}

\newcommand{\NN}{\mathbb{N}}

\newcommand{\TT}{\mathbb{T}}
\newcommand{\ev}{\operatorname{ev}}

\newcommand{\tildegamma}{\widetilde{\Gamma}}
\newcommand{\GL}{\operatorname{GL}}

\newcommand{\Angle}[1]{\langle\underline{#1}\rangle}

\newcommand{\Ring}[2]{#1\langle\underline{#2}\rangle}
\newcommand{\Rring}[2]{#1\langle\!\langle\underline{#2}\rangle\!\rangle}
\newcommand{\Rringcirc}[2]{#1^{\circ}\langle\!\langle\underline{#2}\rangle\!\rangle}
\newcommand{\Rringbullet}[2]{#1^{\bullet}\langle\!\langle\underline{#2}\rangle\!\rangle}
\newcommand{\Tame}[1]{#1\langle\!\langle\underline{e}\rangle\!\rangle^b}
\newcommand{\Tamecirc}[1]{#1^\circ\langle\!\langle\underline{e}\rangle\!\rangle^b}

\newcommand{\Exp}[2]{\mathcal{T}_{\underline{#1},#2}}

\newtheorem{Theorem}{Theorem}[section]
\newtheorem{Lemma}[Theorem]{Lemma}
\newtheorem{MainData}[Theorem]{Data}
\newtheorem{Proposition}[Theorem]{Proposition}
\newtheorem{Question}[Theorem]{Question}
\newtheorem{Corollary}[Theorem]{Corollary}
\newtheorem{Definition}[Theorem]{Definition}
\newtheorem{Remark}[Theorem]{Remark}
\newtheorem{Conjecture}[Theorem]{Conjecture}

\title{The analytic theory of vectorial Drinfeld modular forms}
\date{\today}
\author{F. Pellarin}

\address{Federico Pellarin\\
Institut Camille Jordan, UMR 5208\\
Site de Saint-Etienne \\
23 rue du Dr. P. Michelon \\
42023 Saint-Etienne, France}

\begin{document}

\maketitle

\begin{abstract}
In this text we generalize the notion of Drinfeld modular form for the group $\Gamma:=\operatorname{GL}_2(\FF_q[\theta])$ to a vector-valued setting, where the target spaces are certain modules over positive characteristic Banach algebras over which are defined what we call the 'representations of the first kind'. Under quite reasonable restrictions, we show that the spaces of such modular forms are finite-dimensional,  endowed with certain generalizations of Hecke operators, with differential operators \`a la Serre etc. The crucial point of this work is the introduction of a 'field of uniformizers', a field extension of the valued field of formal Laurent series $\CC_\infty((u))$ where $u$ is the usual uniformizer for Drinfeld modular forms, in which we can study the expansions at the cusp infinity of our modular forms and which is wildly ramified and not discretely valued. Examples of such modular forms are given through the construction of Poincar\'e and Eisenstein series. 

After the discussion of these fundamental properties, the paper continues with a more detailed analysis of the special case of modular forms associated to a restricted class of representations $\rho_\Sigma^*$ of $\Gamma$ which has more importance in arithmetical applications. More structure results are given in this case, and a harmonic product formula is obtained which allows, with the help of conjectures on the structure of an $\FF_p$-algebra of $A$-periodic multiple sums, multiple Eisenstein series etc., to produce conjectural formulas for Eisenstein series. Other properties such as integrality of coefficients of Eisenstein series, specialization at roots of unity etc. are included as well.
\end{abstract}

\tableofcontents

%

\section{Introduction}\label{introduction}

The aim of this volume is to revisit the analytic theory of Drinfeld modular forms for the {\em Drinfeld modular group} $\GL_2(\FF_q[\theta])$, initiated by Goss in his Ph. D. Thesis (see \cite{GOS2}) and continued in the work of Gekeler \cite{GEK}, and then in the work of several other authors. Presently, modular forms with values in positive characteristic fields (\footnote{With $\FF_q((\theta^{-1}))^{sep}$ a separable closure of the local field $\FF_q((\theta^{-1}))$ and $\widehat{\cdot}$ denoting the completion.}) such as
$$\CC_\infty:=\widehat{\FF_q((\theta^{-1}))^{sep}},$$
is an active domain of research with deep
developments for more general groups $\GL_n(A)$ ($n\geq 2$), with $A$ ring of functions over a smooth projective geometrically irreducible curve regular away from an infinity point, its congruence subgroups,  leading to an algebraic and analytic theory of modular forms and to compactification problems
as in the works of Pink and
Basson, Breuer and Pink \cite{PIN,BBP1,BBP2,BBP3}, Gekeler \cite{GEK-DMF1,GEK-DMF2,GEK-DMF3,GEK-DMF4}, H\"aberli \cite{HAB}, Hartl and Yu \cite{HAR&YU}. The arithmetic theory of Drinfeld modular forms, if compared with that
 of classical modular forms, also has a different flavor. We mention the investigations related to Galoisian representations and the cohomological theory of crystals by B\"ockle \cite{BOE1,BOE2} and aspects of $P$-adic continuous families of Drinfeld modular forms by Hattori \cite{HAT1} and 
Nicole and Rosso \cite{NIC&ROS1}. These are few illustrations of how the theory ramifies deeply in a multitude of directions and the reader is warned that the above list of references is not completely representative of the great and valuable effort of a large community. More references can be found in the above mentioned works.
 
In the present volume, we voluntarily restrict our attention to the simplest case of the group $\GL_2(\FF_q[\theta])$ and we follow yet another direction of research which, as far as we can see, has not been deeply investigated yet. We want to begin the study of analytic properties of modular forms associated with an {\em extended notion of type}, as introduced and considered by Gekeler in \cite{GEK}.
The {\em type} of a Drinfeld modular form for the group 
$$\Gamma:=\GL_2\Big(\FF_q[\theta]\Big)$$ can be viewed as a one-dimensional representation of $\Gamma.$ We are interested in certain higher dimensional representations of this group, naturally behaving in rigid analytic families at the infinity place. One of the principal initial reasons for our endeavor comes from remarks on analytic families of modular forms first raised in the paper \cite{PEL0}, and later, in \cite{PEL&PER}.

There are $p$-adic families of classical ($\CC$-valued) modular forms containing sequences of modular forms $(f_n)_n$ with $f_n\in\Gamma(p^{i_n})$
where $\Gamma(N)$ denotes the principal congruence subgroup of $\operatorname{SL}_2(\ZZ)$ of level $N\in\NN^*=\ZZ_{>1}$, $p$ is a prime number, and $(i_n)_n$ is a sequence of positive integers tending to infinity. This is a well known feature, and analogous families indeed occur in the theory of Drinfeld modular forms, which take values in complete, algebraically closed fields of positive characteristic. Less known is the existence of certain non-trivial $\infty$-adic families, where $\infty$ is the infinity place of the field $\FF_q(\theta)$, of Drinfeld modular forms containing modular forms for $\Gamma(P)$ and 
$P$ varying in an infinite subset of irreducible polynomials of $\FF_q[\theta]$. The most frequently, these families are isobaric, that is, weights do not vary. They have been first observed in the paper \cite{PEL3} and they do not seem to have analogues in the classical theory (over $\QQ$). The reader can find a study in our \S \ref{structureofvmf}. 

This work grew up with the purpose of defining appropriate analytic tools to ease the study of these structures. The modular forms studied here are vector functions with values in certain positive characteristic ultrametric Banach algebras, close in their functional behavior, to $\CC$-valued vector modular forms, but with the important difference that the values can themselves be non-constant rigid analytic functions. 
 
To study congruences or $p$-adic analytic families of modular forms it is very useful to `tame' the behavior of modular forms at a cusp. For instance, Gekeler's seminal paper \cite{GEK} uses `Fourier series' of modular forms (we could say `$u$-expansions') in an essential way.  Our case does not make exception to this principle. However, the problem of `taming' the behavior of modular forms at a cusp is here more difficult. For this reason we introduce a {\em field of uniformizers} (Definition \ref{definition-field-uniformizers}) described by $$\mathfrak{K}=\Tamecirc{\KK}((u))=\Big\{\sum_{i\geq i_0}f_iu^i:f_i\in \Tamecirc{\KK}\Big\}.$$ The adopted notation, and the field $\mathfrak{K}$, will be described in \S \ref{tameseriestheory}. Here note that for all $i$, $f_i$ is a {\em tame series}, a class of $\KK$-entire functions of the variable $z$ over $\CC_\infty$ with $\KK$ a certain Banach $\CC_\infty$-algebra with a specified growth condition in the neighborhood of $u=u(z)=0$ (the $\KK$-vector space of such functions is denoted by the apparently awkward notation $\Tamecirc{\KK}$).

The field $\mathfrak{K}$ is a wildly ramified, non-discretely valued field extension of $\CC_\infty((u))$ (the latter corresponds to the case of $f_i$ a constant $\CC_\infty$-valued function for all $i$) and can be viewed as the main new tool of the present work. Hence, it is not possible to choose a single uniformizer, the powers of which generate a field of series expansions of modular forms. 
On another side the Drinfeld modular forms considered in the present volume have a series expansion in $\mathfrak{K}$ which is uniquely determined and is, in general, very difficult to describe. 

We present a spectrum of aspects which is limited to the construction of the analytic uniformization at the cusp infinity and some consequences going from the proof that vector spaces of modular forms are of finite dimension to the construction of explicit examples of Eisenstein and Poincar\'e series and the analysis of their expansions in $\mathfrak{K}$. This already offered challenges and rich pictures and we decided to confine our attention only to those aspects which are tangible by an appropriate generalization of the viewpoint of Gekeler's seminal paper \cite{GEK}. 

The volume is essentially self-contained. It presents the foundations to study new aspects of Drinfeld modular forms and to do this, it presents new tools with an elementary approach. It is enriched with several questions, problems and conjectures. Other crucial aspects such as the interpretation of modular forms of our settings as sections of algebraic vector bundles, link with the theory of harmonic cocycles \`a la Teitelbaum etc. will be developed in other texts.

\subsection{Description of the basic objects}
Let $q=p^e$ be a power of a
prime number $p$ with $e>0$ an integer, let $\FF_q$ be the finite field with $q$ elements and characteristic $p$, and $\theta$ an indeterminate over $\FF_q$. All along this text, we denote by $A$ the $\FF_q$-algebra $\FF_q[\theta]$. We set $K=\FF_q(\theta)$.
 On $K$, we consider the multiplicative valuation
$|\cdot|$ defined by $|a|=q^{\deg_\theta(a)}$, $a$ being in $K$, so that
$|\theta| = q$.  Let $K_\infty :=\FF_q((1/\theta))$ be the local field
which is the completion of $K$ for this absolute value, let $K_\infty^{\text{sep}}$ be a
separable algebraic closure of $K_\infty$, let $\CC_\infty$ be the completion of
$K_\infty^{\text{sep}}$ for the unique extension of $|\cdot|$ to $K_\infty^{\text{sep}}$.
Then, the field $\CC_\infty$ is at once algebraically closed and complete for $|\cdot|$ with valuation group $q^\QQ$ and residual field $\FF_q^{sep}$, an algebraic closure of $\FF_q$. 

The 'Drinfeld half-plane' $\Omega=\CC_\infty\setminus K_\infty$, with the usual rigid analytic structure in the sense of \cite[Definition 4.3.1]{FP}, carries an action of $\Gamma=\GL_2(A)$ and $\widetilde{\Gamma}=\operatorname{PGL}_2(A)$ by homographies: if $\gamma=(\begin{smallmatrix}a & b \\ c & d\end{smallmatrix})\in\tildegamma$, and $z\in\Omega$, $$\gamma(z):=\frac{az+b}{cz+d}.$$ Denote by 
$$J_{(\begin{smallmatrix} * & * \\ c & d\end{smallmatrix})}(z)=cz+d$$ the usual factor of automorphy 
$\Gamma\times\Omega\rightarrow\CC_\infty^\times$. Let us consider $w,m\in\ZZ$; then, if $w\equiv 2m\pmod{q-1}$, the map $(\gamma,z)\mapsto J_\gamma(z)^w\det(\gamma)^{-m}$ defines a factor of automorphy for $\tildegamma$. There is a bijection between these factors of automorphy and the couples $(w,m)\in\ZZ\times \ZZ/(q-1)\ZZ$ submitted to the above congruence.

We thus suppose that $w\in\ZZ$ and $m\in\ZZ/(q-1)\ZZ$ are such that $w\equiv 2m\pmod{q-1}$. 
We recall the definition of Drinfeld modular forms (as considered by Gekeler and Goss, see \cite[Definition (5.7)]{GEK}). 

\begin{Definition}\label{classical-Drinfeld-modular} {\em A {\em Drinfeld modular form of weight $w\in\ZZ$ and type $m\in\ZZ/(q-1)\ZZ$} for the group $\Gamma$ is a rigid analytic function $\Omega\xrightarrow{f}\CC_\infty$
such that $$f(\gamma(z))=J_\gamma(z)^w\det(\gamma)^{-m}f(z)\quad \forall z\in\Omega,\quad \forall \gamma\in\tildegamma$$ and such that additionally, 
there exists $0<c<1$ with the property that if $z\in\Omega$ is such that $|u(z)|\leq c$, where 
\begin{equation}\label{Goss-uniformiser}
u(z)=\frac{1}{\widetilde{\pi}}\sum_{a\in A}\frac{1}{z-a},\end{equation}
$\widetilde{\pi}\in\CC_\infty\setminus K_\infty$ being a fundamental period of Carlitz's module (\footnote{'Our' analogue of $2\pi i$, see (\ref{pitilde}).}),
 then there is a uniformly convergent series expansion
\begin{equation}\label{regularity}
f(z)=\sum_{n\geq 0}f_nu(z)^n,\quad f_n\in\CC_\infty.\end{equation}
We say that a function $f$ in (\ref{regularity}) is {\em regular at the infinity cusp}.}\end{Definition} 

Note that (\ref{regularity}) is not the only formulation of the regularity at the infinity cusp (\footnote{Note also that classically,
a {\em modular form for $\operatorname{SL}_2(\ZZ)$} (or for a subgroup of $\operatorname{SL}_2(\RR)$ which is commensurable with it) can be also defined as a holomorphic function $f:\mathcal{H}=\{z=x+\sqrt{-1}y\in\CC:x,y\in\RR,y>0\}\rightarrow\CC$ satisfying a well known family of functional 
relations and such that, if $z=x+\sqrt{-1}y$ with $x,y\in\RR$, there exists $c\in\RR$ such that $f(x+iy)=\mathcal{O}(y^c+y^{-c})$ (compare with Miyake's \cite[Theorem 2.1.4]{MIY}).}). We can restate
(\ref{regularity}) equivalently by asking that  the set of real numbers $|f(z)|$ is bounded if we choose $z\in\Omega$ such that 
$|u(z)|$ is small.

The {\em type} in Definition \ref{classical-Drinfeld-modular} corresponds to a representation 
\begin{equation}\label{scalarrepresentations}
\Gamma\xrightarrow{\det^{-m}}\operatorname{GL}_1(\FF_q),\quad m\in\ZZ/(q-1)\ZZ.
\end{equation} 
In dimension $>1$ it happens that certain representations of $\Gamma$ naturally have non-trivial analytic deformations, and this makes it natural to consider functions with values in positive-dimensional Tate algebras or in similar ultrametric Banach algebras. 
We consider $\Sigma\subset\NN^*$ a finite subset. Let $\FF_q(\underline{t}_\Sigma)$ be the field of rational fractions with coefficients in $\FF_q$ in the set of independent variables $\underline{t}_\Sigma:=(t_i:i\in\Sigma)$. We choose a representation
\begin{equation}\label{prototype-representation}
\Gamma\xrightarrow{\rho}\operatorname{GL}_N\Big(\FF_q(\underline{t}_\Sigma)\Big).
\end{equation}
Let $w\in\ZZ$ be such that the map $(\gamma,z)\mapsto J_\gamma(z)^w\rho(\gamma)$ defines a factor of automorphy $$\tildegamma\times \Omega\rightarrow\operatorname{GL}_N\Big(\FF_q(\underline{t}_\Sigma)\Big).$$ The necessary and sufficient condition for this is that
\begin{equation}\label{condition-degree}
\rho(\mu I_2)=\mu^{-w} I_N,\quad \mu\in\FF_q^\times,
\end{equation}
as it comes out after a simple computation.

We consider the field $$\KK_\Sigma=\CC_\infty(\underline{t}_\Sigma)^{\wedge}=\widehat{\CC_\infty(\underline{t}_\Sigma)}$$ (the completion for the Gauss norm) (\footnote{Observe the notation $(\cdot)^\wedge$ that will be used when the other notation will lead to a too large hat.}) so that $\GL_N(\FF_q(\underline{t}_\Sigma))\subset\GL_N(\KK_\Sigma)$. We denote by $\|\cdot\|$ the multiplicative valuation of $\KK_\Sigma$, extending $|\cdot|$ of $\CC_\infty$. We further extend this to a norm on matrices with entries in $\KK_\Sigma$ in the usual way by taking the supremum of the multiplicative valuations of the entries. In 
\S \ref{holomorphicfunc} we discuss the notion of rigid analytic functions with values in $\KK_\Sigma$. Taking this notion into account:

\begin{Definition}\label{modularform}
{\em A rigid analytic function
$$\Omega\xrightarrow{f}\KK_\Sigma^{N\times 1}$$
such that 
\begin{equation}\label{functional}
f(\gamma(z))=J_\gamma(z)^w\rho(\gamma)f(z)\quad \forall z\in\Omega,\quad \forall \gamma\in\tildegamma,\end{equation} is called {\em modular-like of weight $w$ for $\rho$}. Additionally, we say that such a function $f={}^t(f_1,\ldots,f_N)$ is:
\begin{enumerate}
\item A {\em weak modular form of weight $w$ for $\rho$} if there exists $M\in\ZZ$ such that $\|u(z)^Mf(z)\|$ is bounded as $0<|u(z)|<c$ for some $c<1$.
\item A {\em modular form of weight $w$ (for $\rho$)} if $\|f(z)\|$ is bounded as $0<|u(z)|<c$ for some $c<1$.
\item A {\em cusp form of weight $w$} if $\|f(z)\|\rightarrow0$ as $u(z)\rightarrow 0$.
\end{enumerate}}
\end{Definition}

Let $B$ be a $\CC_\infty$-sub-algebra of $\KK_\Sigma$. We suppose that $\rho$ as in (\ref{prototype-representation}) has image in $\GL_N(B)$.
We denote by $M^!_w(\rho;B)$ (resp. $M_w(\rho;B)$, $S_w(\rho;B)$) the $B$-modules 
of weak modular forms (resp. modular forms, cusp forms) of weight $w$ for $\rho$ such that their images are contained in $B^{N\times 1}$. We have that $$S_w(\rho;B)\subset M_w(\rho;B)\subset M^!_w(\rho;B).$$ If $B=\CC_\infty$, $N=1$ and $\rho=\det^{-m}$, these $\CC_\infty$-vector spaces coincide with the correspondent spaces of 'classical' Drinfeld modular forms of weight $w$, type $m$ in the framework of Definition \ref{classical-Drinfeld-modular}.

To be relevant, Definition \ref{modularform} must deliver certain primordial properties such as
the finite dimensionality of the modules $M_w(\rho;B)$, or their invariance under the action of variants of Hecke operators. We are far from being able to return satisfactory answers in such a level of generality. However, there is a class of representations (called {\em representations of the first kind}, introduced and discussed in \S \ref {representations-first-kind}) which looks suitable for our investigation because they contain a variety of arithmetically interesting examples. An explicit example of such representations is, with $t$ a variable, the one which  
associates to a matrix $\gamma=(\begin{smallmatrix} a & b \\ c & d\end{smallmatrix})\in\Gamma$, the matrix 
\begin{equation}\label{rhot}
\rho_t(\gamma)=\begin{pmatrix} \chi_t(a) & \chi_t(b) \\ \chi_t(c) & \chi_t(d)\end{pmatrix}\in\operatorname{GL}_2(\FF_q[t]),
\end{equation} where $\chi_t$
is the unique $\FF_q$-algebra morphism $\FF_q[\theta]\rightarrow\FF_q[t]$ sending $\theta$ to $t$. Another interesting example is $\rho_t^*:={}^t\rho^{-1}$, investigated in \cite{PEL1,PEL&PER3}; in the latter case, we have explicitly described the module structure of $M_w(\rho_t^*;\widehat{\CC_\infty[t]})$ (the values are in the Tate algebra $\widehat{\CC_\infty[t]}$ completion of $\CC_\infty$ for the Gauss valuation above the valuation of $\CC_\infty$) and proved that these $\widehat{\CC_\infty[t]}$-modules are endowed with endomorphisms given by a natural generalization of Hecke operators.

\subsection{Suitability of the use of the field of uniformizers}\label{necessity}
The strong point of Definition \ref{modularform} is its simplicity but in practice it does not allow to manipulate Drinfeld modular forms. If compared with Definition \ref{classical-Drinfeld-modular}, we clearly miss here a valuation at the infinity cusp, available at least in the case of classical Drinfeld modular forms by considering the order in $u$ in (\ref{regularity}). This problem is already mentioned in \cite{PEL&PER3}. The entries of the Eisenstein series that we describe in the present volume (see the defining expression (\ref{Eisenstein-series-definition}) and \S \ref{eisensteinseries}) in association with the representations $\rho^*_\Sigma$ as in (\ref{varphidefi}) can be expanded in certain convergent series (see Proposition \ref{seriesofEisensteinseries}) involving {\em Perkins' series} (as in \S \ref{theweightofperkins}). It is important to mention that 
Perkins' investigations borrow heavily from the notion of {\em quasi-periodic functions} of Gekeler (as in \cite{GEK01}). The necessity of introducing tame series and the field of uniformizers $\mathfrak{K}$ already appears, implicitly, in the works \cite{PER,PEL&PER}. Perkins proved formulas such as
\begin{equation}\label{perkins-formulas}
\sum_{a\in A}\frac{a(t_1)\cdots a(t_s)}{z-a}=\frac{\widetilde{\pi}\prod_{i=1}^s\sum_{j_i\geq 0}\exp_C\Big(\frac{\widetilde{\pi}z}{\theta^{j_i+1}}\Big)t_i^{j_i}}{\omega(t_1)\cdots\omega(t_s)\exp_C\Big(\widetilde{\pi}z\Big)}, \quad s<q,\quad z\in\CC_\infty\setminus A.
\end{equation}
On the left-hand side we have, with $\Sigma=\{1,\ldots,s\}$, a Perkins' series, a meromorphic function $\CC_\infty\setminus A\rightarrow\TT_\Sigma:=\widehat{\CC_\infty[\underline{t}_\Sigma]}$ (the target space is the Tate algebra in the variables
$\underline{t}_\Sigma$, the completion is for the Gauss valuation) which is represented, in the right-hand side as the ratio of two entire functions $\CC_\infty\rightarrow\TT_\Sigma$ expressed in terms of the Carlitz exponential $\exp_C$, the Anderson-Thakur function $\omega(t)$ (all these items will be reviewed in \S \ref{preliminaries-etc}). This ratio typically expresses an element of $\mathfrak{K}$. First of all, it is well known that
$\exp_C(\widetilde{\pi}z)^{-1}=u(z)$. As for the numerator in the fraction, it is a very basic example of 
tame series (Definition \ref{deftameseries}). So we can write that the left-hand side belongs to $\mathfrak{K}$.
More explicitly, it is a series $\sum_if_iu^i$ with a unique non-zero monomial, of the form 
$f_1u^1$ where
$$f_1=\frac{\widetilde{\pi}}{\omega(t_1)\cdots\omega(t_s)}\prod_{i=1}^s\sum_{j_i\geq 0}\exp_C\Big(\frac{\widetilde{\pi}z}{\theta^{j_i+1}}\Big)t_i^{j_i},$$
and the additive valuation $v$ extending that of $\CC_\infty((u))$, evaluated on the left-hand side of (\ref{perkins-formulas}), equals 
$$v(u)+v(f_1)=1-sv(\exp_C(\widetilde{\pi}z/\theta))=1-\frac{s}{q},$$ where the leading term of the tame series $f_1$ is proportional to $\exp_C(\widetilde{\pi}z/\theta)^s$. This allows to compute $v$-valuations of the entries of our Eisenstein series. Unfortunately, the formulas (\ref{perkins-formulas}) do not hold if $s\geq q$. Perkins in \cite{PER} succeeded in obtaining explicit formulas up to $s\leq 2(q-1)$ and sparse explicit formulas can be given for even higher values of $s$ up to non-rewarding computational efforts. In the work \cite{PEL&PER3}
the authors were convinced that somehow, Perkins' series would have been appropriate analogues of 
the uniformizer at infinity $u$ (which also is a Perkins' series) but this is also a non-rewarding approach. 

We hope that at this point the reader can appreciate the reason of introducing tame series and the field of uniformizers. It is extremely difficult to compute products of Perkins series and obtain a generalization of (\ref{perkins-formulas}). At once, the construction of the field $\mathfrak{K}$, a $v$-valued field in which the product is itself difficult to compute (just as difficult as in the field $\FF_q((\frac{1}{\theta}))^{sep}$, see \cite{KED}) nevertheless warrants the presence of an environment in which computations with our Drinfeld modular forms are virtually possible. 

Thanks to this formalism we are able, without much additional effort, to reach most of the results of the first part of the present paper.
The reader may find the preliminary material \S \ref{preliminaries-etc}, \ref{tameseriestheory} and \ref{quasiperiodicfunctions} heavy  but this reflect the complexity of the given settings. 
It is perhaps possible to get rid of the field $\mathfrak{K}$ and work more directly, starting with Definition \ref{modularform} but $\mathfrak{K}$ is the natural field in which one can study series expansions at infinity of our modular forms and also allows to introduce notions of rationality and integrality of the coefficients etc. for modular forms and reduction. The difficulty of multiplying formal series in $\mathfrak{K}$ mirrors the complexity of the behavior at the cusp infinity of Drinfeld modular forms in our generalized setting.

\subsection{Results of the text}

The paper is organized in ten sections. These sections can be roughly divided in three principal parts.
\begin{itemize}
\item[I.] Sections \ref{preliminaries-etc} to \ref{differential-perkins}. We present the foundations of the theory.
\item[II.] Sections \ref{eisensteinseries} and \ref{modular-certain-repr}. We study modular forms for the representations $\rho^*_\Sigma$. We discuss the structure of {\em strongly regular modular forms}.
\item[III.] Sections \ref{shufflerelations} and \ref{Some-conjectures}. We discuss arguments related to the harmonic product for multiple sums and we present open problems.
\end{itemize}
Part III can be read quite independently of the previous ones. Reading Part II is possible without reading all proofs in Part I. The following synthesis summarizes the content of the paper and our results (more precise statements will be formulated along the text). 
We proceed in the order suggested by Parts I to III. 

\subsubsection*{Content of Part I}
The key environment is the {\em field of uniformizers} $\mathfrak{K}$ (remember \S \ref{necessity}) with valuation $v$, additive valuation group 
$\ZZ[\frac{1}{p}]$, residual field $\cup_\Sigma\KK_\Sigma$, valuation ring $\mathfrak{O}$ and maximal ideal $\mathfrak{M}$, to which the entire \S \ref{tameseriestheory} is devoted. 
The field $\mathfrak{K}$ is constructed explicitly in \S \ref{tameseriestheory} by taking the fraction field of an integral ring of entire functions that we call the {\em ring of tame series}. 
The next result is proved:

\medskip

\noindent{\bf Theorem A.}
{\em Let $\Sigma\subset\NN^*$ be a finite subset and $\rho:\Gamma\rightarrow\GL_N(\FF_q(\underline{t}_\Sigma))$ be a representation of the first kind, let $w\in\ZZ$ be such that $(\gamma,z)\mapsto J_\gamma(z)^w\rho(\gamma)$ is a factor of automorphy for $\widetilde{\Gamma}$.
The following properties hold.
\begin{enumerate}
\item There is a natural embedding of $\KK_\Sigma$-vector spaces $M^!_w(\rho;\KK_\Sigma)\xrightarrow{\iota_\Sigma}\mathfrak{K}^{N\times 1}$.

\item The image by $\iota_\Sigma$ of the $\KK_\Sigma$-vector space of modular forms $M_w(\rho;\KK_\Sigma)$ can be identified with $\iota_\Sigma(M_w^!(\rho;\KK_\Sigma))\cap\mathfrak{O}^{N\times 1}$. 
\item The vector space of cusp forms $S_w(\rho;\KK_\Sigma)$ can be identified with the sub-vector space of $M_w(\rho;\KK_\Sigma)$ which is sent to $\mathfrak{M}^{N\times 1}$ by the embedding $\iota_\Sigma$.
\item We have that $\CC_\infty((u))$ naturally embeds in $\mathfrak{K}$ and $v$ restricts to the 
$u$-adic valuation.
\item The vector spaces $M_w(\rho;\KK_\Sigma),S_w(\rho;\KK_\Sigma)$ are endowed with Hecke operators $T_{\mathfrak{a}}$ associated to ideals $\mathfrak{a}$ of $A$, which provide a totally multiplicative system of
endomorphisms reducing, in the case $\Sigma=\emptyset$, to the classical Hecke operators acting on classical scalar Drinfeld modular forms and cusp forms.
\item We have $\KK_\Sigma$-linear maps $\partial_w^{(n)}:M_w(\rho;\KK_\Sigma)\rightarrow S_{w+2n}(\rho\det^{-n};\KK_\Sigma)$, defined for all $n\geq 0$ and generalizing Serre's derivatives.
\end{enumerate}
}

\medskip

The corresponding results in the body of the text are more precise and cover a wider spectrum of applications. 
The main examples of modular forms (construction of Poincar\'e series etc.) and the basic results concerning the spaces $M_w(\rho;\KK_\Sigma)$ and $S_w(\rho;\KK_\Sigma)$ are contained in \S \ref{structureofvmf}. 
Parts (1), (2), (3) will be proved in Theorem \ref{theorem-u-expansions} and (4) is an obvious consequence of the above (so, when $\rho=\boldsymbol{1}$ is the trivial representation (sending every element of $\Gamma$ to $1\in\GL_1$), our construction specialises to the known setting, and $M=\oplus_wM_w(\boldsymbol{1};\CC_\infty)$ is the well known algebra of $\CC_\infty$-valued Drinfeld modular forms for $\Gamma$ (of type $0$). 
We will introduce Poincar\'e series in \S \ref{poincareseries} as a first non-trivial class of modular forms.
Part (5) is our Theorem \ref{theo-hecke-operators}; the proof is very simple, thanks to the flexibility of the use of the field of uniformizers, and we can say the same about part (6), which corresponds to our Theorem \ref{serre-derivatives}. 

A non-complete field $\LL_\Sigma$ intermediate between $\KK_\Sigma$ and the fraction field of $\TT_\Sigma$ will be needed in the next Theorem; it is defined in \S \ref{non-completefields}. 

\medskip

\noindent{\bf Theorem B.}
{\em The following properties hold, for $\rho$ a representation of the first kind.
\begin{enumerate}
\item For all $w\in\ZZ$, the $\LL_\Sigma$-vector space $M_w(\rho;\LL_\Sigma)$ has finite dimension. The dimension is zero if $w<0$.
\item The dimension of the space $M_1(\rho;\LL_\Sigma)$ does not exceed the dimension of  the $\LL_\Sigma$-vector space of common eigenvectors in $\LL_\Sigma^{N\times 1}$ of all the matrices $\rho(\gamma)$ with $\gamma$ in the Borel subgroup of $\Gamma$.
\end{enumerate}
}

\medskip

The matrices $\rho(\gamma)$ have all the eigenvalues equal to $1$.
Note that (1) of Theorem B only deals with modular forms with values in $\LL_\Sigma$. One reason for this restriction comes from the fact that we use, in the proof, a specialisation property at roots of unity which is unavailable in the general case of $\KK_\Sigma$-valued functions. This result corresponds to Theorem \ref{theorem-structure-tate-algebras}. 


\subsubsection*{Content of Part II}

As we have mentioned, a scalar Drinfeld modular form for $\Gamma$ as in Definition \ref{classical-Drinfeld-modular} has a unique $u$-expansion (\ref{regularity}) in $\CC_\infty[[u]]$ and combining part (2) of Theorem A and Proposition \ref{uniquedigitexpansion}, one sees that every entry $f$ of a given element of $M_w(\rho;\KK_\Sigma)$ has a unique series expansion
$$f=\sum_{i\geq 0}f_iu^i$$
where for all $i\geq 0$, $f_i$ is an entire function $\CC_\infty\rightarrow\KK_\Sigma$ of the variable $z\in\Omega$ of tame series described in \S \ref{tameseriessect} (and additionally, $f_0$ is constant in $\KK_\Sigma$). This generalizes the case of Definition \ref{classical-Drinfeld-modular}, where the coefficients $f_i$ are all constant functions, in $\CC_\infty$. 
It is in general very difficult to describe the coefficients $f_i$ but we make some attempts in this part. For instance, 
something can be done with Eisenstein series for the representations $\rho^*_\Sigma$ (see \S \ref{eisensteinseries}) by using {\em Perkins' series} as in \S \ref{theweightofperkins}; see Proposition \ref{seriesofEisensteinseries}. 

We fix a subset $\Sigma\subset\NN^*$ of cardinality $s$ and we consider, for all $i\in\Sigma$, \begin{equation}\label{varphidefi}\rho^*_{t_i}(\gamma)={}^t\begin{pmatrix}a(t_i) & b(t_i) \\ c(t_i) & d(t_i)\end{pmatrix}^{-1},$$ and $$\rho^*_\Sigma:=\bigotimes_{i\in\Sigma}\rho^*_{t_i}.\end{equation}
This is indeed a representation of the first kind of degree $s$ where $N=2^{|\Sigma|}$. Additionally,
$\rho^*_\Sigma$ is an irreducible representation of $\Gamma$ in $\operatorname{GL}_N(\FF_q[\underline{t}_\Sigma])$ (see \cite{PEL2}). 
An important feature of this class of representations is that it allows to construct certain Eisenstein series in \S \ref{eisensteinseries}. If $s\equiv w\pmod{q-1}$ and $w>0$ we have the Eisenstein series of weight $w$:
\begin{equation}\label{Eisenstein-series-definition}
\mathcal{E}(w;\rho^*_\Sigma)(z):=\sum_{(a,b)\in A^2\setminus\{(0,0)\}}(az+b)^{-w}\bigotimes_{i\in\Sigma}\binom{a(t_i)}{b(t_i)},\end{equation}
which is a non-zero holomorphic function $\Omega\rightarrow \TT_\Sigma^{N\times 1}$, where 
$\TT_\Sigma=\widehat{\CC_\infty[\underline{t}_\Sigma]}$ is the standard Tate algebra in the variables $\underline{t}_\Sigma$. These series generalize the usual scalar Eisenstein series for $\Gamma$ (case of $\Sigma=\emptyset$).
We have $\mathcal{E}(w;\rho^*_\Sigma)\in M_w(\rho^*_\Sigma;\TT_\Sigma)\setminus S_w(\rho^*_\Sigma;\TT_\Sigma)$. Writing
$\mathcal{E}(w;\rho^*_\Sigma)={}^t(\mathcal{E}_1,\ldots,\mathcal{E}_N)\in\mathfrak{O}_\Sigma^{N\times 1}$ we can prove that $\mathcal{E}_1,\ldots,\mathcal{E}_{N-1}\in\mathfrak{M}$
and $\mathcal{E}_{N}\in\mathfrak{O}\setminus\mathfrak{M}$ (we recall that $\mathfrak{O}$ and $\mathfrak{M}$ are respectively the valuation ring and the maximal ideal of the field of uniformizers). It turns out that $$\mathcal{E}_N\equiv-\zeta_A(1;\sigma_\Sigma)\pmod{\mathfrak{M}}$$ where 
\begin{equation}\label{zeta-value-tate}
\zeta_A(n;\sigma_\Sigma)=\sum_{a\in A^+}a^{-n}\sigma_\Sigma(a),\quad n\in\NN^*
\end{equation}
$\sigma_\Sigma(a)=\prod_{i\in\Sigma}\chi_{t_i}(a)$, are the zeta values in Tate algebras introduced in \cite{PEL1} and studied in 
\cite{ANG&PEL,ANG&PEL2,APTR} as well as in other papers.
These Eisenstein series seem to be the crossroad of several interesting features that we gather in the next result (but see the text for more precise results). To begin, we must point out that in \S \ref{v-valuation-eisenstein}, we construct an indexation $(\mathcal{E}^J)_{J\subset\Sigma}$ of the entries $\mathcal{E}_i$ of an Eisenstein series $\mathcal{E}=\mathcal{E}(w;\rho^*_\Sigma)$ by the subsets $J$ of $\Sigma$. 
With this indexation, the first entry $\mathcal{E}_1$ of $\mathcal{E}$ equals $\mathcal{E}^\emptyset$
and the last entry $\mathcal{E}_N$ equals $\mathcal{E}^\Sigma$. We have the next result.

\medskip

\noindent{\bf Theorem C. }{\em The following properties hold for the Eisenstein series $\mathcal{E}(w;\rho^*_\Sigma)$:
\begin{enumerate}
\item If $w=1$ and $J\subsetneq\Sigma$ is such that $|J|=(m-1)(q-1)+l$ with $m>0$ and $1\leq l\leq q-1$
or $m=0$ and $l=q-1$ then $v(\mathcal{E}^J)=1-q^{-m}(q-l)$ and $v(\mathcal{E}^\Sigma)=0$.
\item If $w>0$, $\mathcal{E}(w;\rho^*_\Sigma)$ is $\mathfrak{v}$-integrally definable  (after Definition \ref{integral-forms}) for valuations $\mathfrak{v}$ of $K(\underline{t}_\Sigma)$ associated with a non-zero prime ideal $\mathfrak{p}$ of $A$, and this for all but finitely many $\mathfrak{p}$.
\item Evaluating the first entry of $\mathcal{E}(w;\rho^*_\Sigma)$ at $t_i=\theta^{q^{k_i}}$ for all $i\in\Sigma$ with $k_i\in\NN$ yields, up to a scalar factor, a Drinfeld quasi-modular form in the sense of \cite{BOS&PEL} with an $A$-expansion as in \cite{PET} and all these series occur in this process.
\end{enumerate}
}

Part (1) can be generalized to some cases in which $\ell_q(w)<q$ (the sum of the digits of the $q$-ary expansion of $w$ is $<q$) thanks to Theorem \ref{moregeneralcomputation}, a result that describes the $v$-valuation of Perkins series as in \S \ref{theweightofperkins}. The question of the computation
of these $v$-valuations in full generality, related to the computation of the $v$-valuation of all Perkins' series is, we should say, not easy, and still open. It is related to a similar question on $v$-valuations of Perkins' series and therefore of generalizations of Goss' polynomials. The recent work of Gekeler \cite{GEK2} let think that this is accessible but difficult.

Part (2) generalizes the properties of integrality of the coefficients of the $u$-expansion of scalar Eisenstein series as in \cite[(6.3)]{GEK}. Note that our result is more recondite in the case $\Sigma\neq\emptyset$. Indeed a notion of integrality of the coefficients of a series $\sum_{i}f_iu^i$ with coefficients $f_i$ which are tame series has to be introduced, and this is exactly what we do, and it is not a triviality. Hence, Theorem C would not be meaningful without our investigations of \S \ref{tameseriestheory}. As for part (3), it answers a question by Goss. A quick description of properties related to $\mathfrak{v}$-adic modular forms is given in \S \ref{v-adic-modular}.

In general, we do not control the dimensions and we are unable to construct bases of the spaces $M_w(\rho;\KK_\Sigma)$ except when $w=1$ and $\rho=\rho^*_\Sigma$. We have proved:

\medskip

\noindent{\bf Theorem D. }
{\em If $|\Sigma|\equiv1\pmod{q-1}$ the vector space $M_1(\rho_\Sigma^*;\LL_\Sigma)$ is one-dimensional, generated by $\mathcal{E}(1;\rho^*_\Sigma)$.}

\medskip

This is Theorem \ref{coro-pel-per}. Part (2) of Theorem B (see Theorem \ref{rankoneforweighone}) also includes an upper bound for the dimensions of the $\LL_\Sigma$-vector spaces, and implies a positive answer to the question raised by \cite[Problem 1.1]{PEL&PER3} thanks to Theorem D.
The proofs of (2) of Theorem B and of Theorem D are easy but use a natural isomorphism between (scalar) Drinfeld modular forms for congruence subgroups of $\Gamma$ and spaces of {\em automorphic functions} (harmonic cocycles) over the Bruhat-Tits tree of 
$\Omega$, and the same specialisation properties in terms of the variables $t_i$ used in the proof of (1). When we do this with the entries of the elements of $M_1(\rho;\LL_\Sigma)$ span scalar Drinfeld modular forms of weight one for congruence subgroups of $\Gamma$. The proof of this result is thus based on a crucial earlier remark of Gekeler (which can be found in Cornelissen's paper  \cite{COR}).

From \S \ref{modular-certain-repr} on, the paper exclusively focuses on structure properties of modular forms for the representations $\rho^*_\Sigma$. We introduce here the notion of {\em strongly regular modular form} (see Definition \ref{defstronglyholomorphic}). A strongly regular modular form $f={}^t(f_1,\ldots,f_N)$ (transpose) is a Drinfeld modular form (in our generalized setting) which satisfies certain conditions on the $v$-valuations of its entries. Theorem \ref{theorem2} allows a completely explicit structure description for these modular forms which can be stated as follows
(more precise results can be found in the text).

\medskip

\noindent{\bf Theorem E. }{\em Every strongly regular modular form associated to the representation $\rho_\Sigma^*$ can be constructed combining Eisenstein series $\mathcal{E}(1;\rho_{t_i})$ and $\mathcal{E}(q;\rho_{t_i})$ for $i\in\Sigma$ by using the Kronecker product, and scalar Eisenstein series. In particular, the $M\otimes_{\CC_\infty}\KK_\Sigma$-module of $\KK_\Sigma$-valued strongly regular modular forms is free of rank $N=2^s$ where $s=|\Sigma|$.}

\medskip

The advantage of focussing on strongly regular modular forms is that to study them we do not need the full strength of the tools developed in Part I of this text, namely, the field of uniformizers and the theory of quasi-periodic matrix functions. To prove Theorem E, we only need appropriate generalizations of the arguments of
\cite{PEL&PER3}.

The continuous $\FF_q(\underline{t}_\Sigma)$-linear automorphism $\tau$ of $\KK_\Sigma$ induced by the automorphism $c\mapsto c^q$ of $\KK_\Sigma$ induces injective maps $M_w(\rho^*_\Sigma;\KK_\Sigma)\rightarrow
M_{qw}(\rho^*_\Sigma;\KK_\Sigma)$.
We show, in Theorem \ref{regularitystrongregularity} that for every $w$ there exists $k\in\NN$ such that $\tau^k(f)$ is {\em strongly regular} for every $f\in M_w(\rho^*_\Sigma;\KK_\Sigma)$. This shows that 
Drinfeld modular forms in $M_w(\rho^*_\Sigma;\KK_\Sigma)$ are not too distant from strongly regular modular forms and this allows to deduce:

\medskip

\noindent{\bf Theorem F. }{\em The $\KK_\Sigma$-vector spaces $M_w(\rho^*_\Sigma;\KK_\Sigma)$ have finite dimensions.}

\medskip

Note that the functions of Theorem F have values in $\KK_\Sigma^{N\times 1}$, not just in $\LL_\Sigma^{N\times 1}$ so that the methods of proof of Theorems B and D do not apply for Theorem F.
After Theorem E for every modular form $f\in M_w(\rho_\Sigma^*;\KK_\Sigma)$ there is $k$ such that $\tau^{k}(f)$ can be constructed
combining Eisenstein series, and the coefficients in the construction are in $\KK_\Sigma$. In full generality, it seems difficult to overcome the use of the field $\mathfrak{K}$ and prove Theorem F for any representation of the first kind.

\subsubsection*{Content of Part III}
This work ends with \S \ref{shufflerelations} and \S \ref{Some-conjectures} which are more speculative and contain a description of further perspectives of research. This part can be read quite independently of the previous ones. We present here the harmonic product for multiple sums, the interaction with multiple sums \`a la Thakur, multiple Eisenstein series, and we propose conjectures based on identities between Eisenstein series and many explicit formulas.

In \S \ref{shufflerelations} we prove (see Theorem \ref{anewsumshuffle}) a variant of a harmonic product formula for certain $A$-periodic multiple sums and we apply it to compute several explicit formulas relating Eisenstein series for $\rho^*_\Sigma$. Some of these formulas have been conjectured in earlier works. In \S \ref{shufflerelations} we state Conjecture \ref{conj-mult-eisentein}, where we evoke the potential existence of an $\FF_p$-algebra of {\em multiple Eisenstein series} and an $\FF_p$-isomorphism with an $\FF_p$-algebra of {\em multiple zeta values in Tate algebras}. Additionally, we speculate that 
a multiple Eisenstein series is a modular form for $\rho^*_\Sigma$ in our settings if and only if the 
multiple zeta values in Tate algebras corresponding to it, which also is related to its constant term, is {\em eulerian} following our Definition \ref{eulerianity}. We describe in \S \ref{conjecture-tate-algebras} a conjecture on certain identities involving zeta values in Tate algebras a particular case of which has been recently proved by Hung Le and Ngo Dac in \cite{HUN&NGO} and we end the work with analogue conjectural identities involving our Eisenstein series $\mathcal{E}(w;\rho_\Sigma^*)$. These identities are so complicate that are essentially undetectable by numerical experiments. They do not seem to have analogues in the classical setting of 
$\CC$-valued Eisenstein series for the group $\operatorname{SL}_2(\ZZ)$.

\subsubsection*{Acknowledgements} This section will be completed later. But I'm very grateful to the colleagues that  already sent me feedbacks on this paper, allowing me to improve it.


\section{Preliminaries}\label{preliminaries-etc}

\subsubsection*{Most commonly used notation}

\begin{itemize}
\item $\NN=\{0,1,\ldots\}$ the set of natural integers.
\item $\NN^*=\{1,\ldots\}$ the set of positive natural integers.
\item $B^{M\times N}$: $M$-row, $N$-column arrays with coefficients in the set $B$.
\item $I_r$: the $r\times r$ identity matrix.
\item $\sqcup$ disjoint union.
\item $\operatorname{Diag}(*,\ldots,*)$ diagonal matrix.
\item $\ell_q(n)$ sum of the digits of the base-$q$ expansion of the positive integer $n$.
\item $A=\FF_q[\theta]$, $K=\FF_q(\theta)$, $K_\infty=\FF_q((\frac{1}{\theta}))$, $\CC_\infty=\widehat{K_\infty^{sep}}$
\item $\Gamma=\operatorname{GL}_2(A)$
\item $\widetilde{\Gamma}=\Gamma/\FF_q^\times$.
\item $\boldsymbol{1}$ the trivial representation sending $\Gamma$ to $1\in\FF_q^\times=\GL_1(A)$.
\item $J_\gamma(z)$ the usual factor of automorphy.
\item $\Omega=\CC_\infty\setminus K_\infty$ the Drinfeld half-plane.
\item $u$ the uniformizer at infinity of $\Omega$.
\item $S_w,M_w$, spaces of cusp forms and modular forms of weight $w$.
\item $\Sigma$ a finite subset of $\NN^*$.
\item $\TT_\Sigma$ Tate algebra in the variables $\underline{t}_\Sigma=(t_i:i\in\Sigma)$.
\item $\KK_\Sigma$ the completion of the fraction field of $\TT_\Sigma$.
\item $\LL_\Sigma$ an intermediate field $\TT_\Sigma\subset\LL_\Sigma\subset\KK_\Sigma$ (see \S \ref{non-completefields}).
\item $\mathfrak{K}$ field of uniformizers, with valuation $v$, valuation ring $\mathfrak{O}$, maximal ideal $\mathfrak{M}$, residual field $\cup_\Sigma\KK_\Sigma$.
\item $\Tamecirc{B}$ the $B$-module of tame series with coefficients in $B$.
\item $\omega$ the function of Anderson and Thakur.
\end{itemize}

\subsubsection*{Overview of the section}
In this section we collect the basic objects over which we are going to build our theory. In \S \ref{completion-fraction} and \ref{non-completefields} we describe the fields $\KK_\Sigma,\LL_\Sigma$ used in the introduction, we give in \S \ref{holomorphicfunc} a quick account of analytic functions with values in certain non-archimedean countably cartesian Banach algebras, of which we are going to study the first examples, useful for what follows. For example, Proposition \ref{entire} is an analogue in our settings of Liouville's Theorem stating that a bounded entire function is constant.
In \S \ref{carlitzexttate} we also provide the reader with the basic tools related to Drinfeld modules and allied functions, such as the exponential and the logarithm. In \S \ref{Carlitz-exponential} we discuss other relevant functions, notably certain generalizations of Anderson and Thakur omega function, and generalizations 
of the entire map $\chi_t:\CC_\infty\rightarrow\TT$ that interpolates the map $A\ni a\mapsto a(t)\in\FF_q[t]$. These functions arise naturally when one studies quasi-periodic matrix functions in \S \ref{quasiperiodicfunctions}.

We make use of the following basic notations.
In all the following, $\NN$ denotes the set of non-negative integers, and $\NN^*$ the subset of positive integers. We choose once and for all independent variables $t_i$ with $i\in\NN^*$ and we work in rings such as $R[t_i:i\in\NN^*]$ where $R$ is a commutative ring with unit. If $\Sigma$ is a finite subset of $\NN^*$ of cardinality $|\Sigma|=s$ we denote by 
$\underline{t}_\Sigma$ the family of variables $(t_i:i\in\Sigma)$. Then, $R[\underline{t}_\Sigma]$
denotes the $R$-algebra $R[t_i:i\in\Sigma]$ in the $s$ variables $\underline{t}_\Sigma$, embedded in $R[t_i:i\in\NN^*]$ in the usual way. If $\Sigma=\{i\}$ is a singleton, then we will often simplify 
our notations by writing $t$ instead of $t_i$.

\subsection{Rings, fields, modules}\label{Rings-fields-modules}

For the general settings on valued rings and fields and local fields, we refer to the author's \cite[\S 2]{PEL3}, from which we borrow the basic notation, and the books \cite{CAS,SER1}.
Let $L$ be a valued field of positive characteristic which is complete for a multiplicative valuation $$L\xrightarrow{|\cdot|}\RR_{\geq 0}.$$ Equivalently, we can consider an additive valuation
$$L\xrightarrow{v}\RR\cup\{\infty\}$$ with the property that $|\cdot|=c^{-v(\cdot)}$, for some $c>1$.
We denote by $\mathcal{O}_L,\mathcal{M}_L$ and $k_L$ respectively the valuation ring, the maximal ideal, and the residual field of $L$, that we suppose to be countable in all the following. If $x\in \mathcal{O}_L$ we denote by $\overline{x}$
its image in $k_L$ by the morphism of reduction modulo $\mathcal{M}_L$. Since $L$ is of positive characteristic we can view $k_L$ as a subfield of $L$. More precisely, we have an embedding $k_L\hookrightarrow \{x\in L:|x|=1\}\cup\{0\}$ that we fix.

\subsubsection{Banach $L$-vector spaces and algebras}

\begin{Definition}\label{Banach-vector-space}
{\em A {\em Banach $L$-vector space} $(B,|\cdot|_B)$ is the datum of an $L$-vector space $B$ together with a map $$|\cdot|_B:B\rightarrow\RR_{\geq 0}$$
such that
\begin{enumerate}
\item for all $x,y\in B$, $|x+y|_B\leq \max\{|x|_B,|y|_B\}$,
\item for all $x\in B$ and $\lambda\in L$, $|\lambda x|_B=|\lambda||x|_B$,
\item If $x\in B$, $|x|_B=0$ if and only if $x=0$,
\end{enumerate}
and such that $B$ is complete for the metric induced by $|\cdot|_B$.

We say that two Banach $L$-vector spaces $(B_1,|\cdot|_{B_1})$ and $(B_2,|\cdot|_{B_2})$
are isometrically isomorphic if there exists an isomorphism of vector spaces $\varphi:B_1\rightarrow B_2$
such that $|\varphi(x)|_{B_2}=|x|_{B_1}$ for all $x\in B_1$.
}
\end{Definition}

\subsubsection*{The spaces $c_I(L)$} Let $I$ be a countable set. We denote by $c_I(L)$
the set of sequences $(x_i)_{i\in I}\in L^I$ such that $x_i\rightarrow 0$
where the limit is for the Fr\'echet filter of $I$, that is, the filter of the complements of finite subsets of $I$ (we shall more simply write $i\rightarrow\infty$). The set $c_I(L)$ is an $L$-vector space. We set $\|(x_i)_{i\in I}\|=\sup_{i\in I}\{|x_i|\}$ for $(x_i)_{i\in I}\in c_I(L)$. Then,
the supremum is a maximum and $(c_I(L),\|\cdot\|)$ carries a structure of Banach $L$-vector space. Note that $\|c_I(L)\|=|L|$; the image of $\|\cdot\|$ equals the image of $|\cdot|$ in $\RR_{\geq 0}$.

\begin{Definition}\label{countably-cartesian}
{\em 
A Banach $L$-vector space $B$ is said {\em countably cartesian} if it is isometrically isomorphic to a space $c_I(L)$ with $I$ countable. Let $\mathcal{B}=(b_i)_{i\in I}$ be a family of elements of $B$. We say that $\mathcal{B}$ is an {\em orthonormal basis} if $|b_i|_B=1$ for all $i$ and if every element 
$f\in B$ can be expanded in a unique way in a series 
$$f=\sum_{i\in I}f_ib_i,\quad f_i\in L,\quad f_i\rightarrow0,$$ 
so that $|f|_B=\max_i|f_i|.$}
\end{Definition}  
Compare with \cite[Chapter 2]{BGR}.
\begin{Definition}\label{algebra-countably-cartesian}
{\em 
A Banach $L$-vector space $(B,|\cdot|_B)$ with a structure of commutative unital $L$-algebra   
 is said to be a {\em Banach $L$-algebra} if $|1|_B=1$ and $|\cdot|_B$ is sub-multiplicative:
 for all $x,y\in B$, $|xy|_B\leq |x|_B|y|_B$. We identify $L$ with $L\cdot 1 \subset B$. A Banach $L$-algebra is {\em countably cartesian}
 if the underlying Banach $L$-vector space is so.
}
\end{Definition}  

\subsubsection{Tate algebras, affinoid algebras}

We consider $\Sigma$ a finite subset of $\NN^*$ and a sub-multiplicative norm $|\cdot|'$ on $L[\underline{t}_\Sigma]$ which restricts to $|\cdot|$ on $L\subset L[\underline{t}_\Sigma]$. 
We denote by $$\widehat{L[\underline{t}_\Sigma]}_{|\cdot|'}\text{ or }L[\underline{t}_\Sigma]_{|\cdot|'}^\wedge$$ the completion of $L[\underline{t}_\Sigma]$ for $|\cdot|'$ (\footnote{The last notation is used in those circumstances where the hat in the first displayed formula is too large.}). It is a Banach $L$-algebra in the sense of Definition \ref{algebra-countably-cartesian}.

For example, we can take $|\cdot|'=\|\cdot\|$ the {\em Gauss valuation} over 
 $L[\underline{t}_\Sigma]$, that is, the unique norm of $L[\underline{t}_\Sigma]$ which extends $|\cdot|$, such that 
 $$\|\underline{t}_\Sigma^{\underline{i}}\|=1$$ for all $\underline{i}=(i_j:j\in\Sigma)\in\NN^\Sigma$, where $$\underline{t}_\Sigma^{\underline{i}}=\prod_{j\in\Sigma}t_j^{i_j}.$$ It is easy to see that $\|\cdot\|$ is multiplicative. In this case we write $$\TT_{L,\Sigma}:=\widehat{L[\underline{t}_\Sigma]}_{\|\cdot\|}.$$ We usually drop the reference to $L$ if it is algebraically closed or if its choice is clear in the context, hence writing in a more compact way $\TT_\Sigma$. This is the {\em Tate algebra} (or {\em standard affinoid algebra}) of dimension 
$s=|\Sigma|$. If $\Sigma=\{i\}$ is a singleton we prefer the simpler notation $\TT_L$ or $\TT$ for this algebra, with variable $t$. Note that if $\Sigma'\subset\Sigma$ then the natural embedding $L[\underline{t}_{\Sigma'}]\subset L[\underline{t}_{\Sigma}]$ induces an embedding
$\TT_{L,\Sigma'}\subset\TT_{L,\Sigma}$. 

The Tate algebra $\TT_{L,\Sigma}$ is isomorphic to the sub-$L$-algebra of the formal series
$$f=\sum_{\begin{smallmatrix}i_j\geq0\forall j\in\Sigma \\ \underline{i}=(i_j:j\in\Sigma)\end{smallmatrix}}f_{\underline{i}}\underline{t}_\Sigma^{\underline{i}}\in L[[\underline{t}_\Sigma]]$$
which satisfy $$\lim_{\min\{i_j:j\in\Sigma\}\rightarrow\infty}f_{\underline{i}}=0.$$ Thus, we have, for $f$ a formal series of $\TT_{L,\Sigma}$ expanded as above
and non-zero, that
$$\|f\|=\sup_{\underline{i}}|f_{\underline{i}}|=\max_{\underline{i}}|f_{\underline{i}}|$$
and $\TT_{L,\Sigma}$ is countably cartesian (Definition \ref{algebra-countably-cartesian}). 
 It is well known that $\TT_{L,\Sigma}$ is a ring which is Noetherian and it is also a unique factorisation domain, normal, of Krull dimension $s$ (see \cite[Theorem 3.2.1]{FP} for a more general collection of properties, see \cite{BGR} for the general theory of these algebras). We will also 
 use the $L$-sub-algebra $\EE_{L,\Sigma}$ of $\TT_{L,\Sigma}$ of the series 
 $f$ as above with the property that for all $r\in|L^\times|$, 
 $$\lim_{\min\{i_j:j\in\Sigma\}\rightarrow\infty}|f_{\underline{i}}|r^{i_1+\cdots+i_n}=0.$$
 If $L$ is complete and algebraically closed, this can be identified with the $L$-algebra of {\em entire functions} in the variables $\underline{t}_\Sigma$. If $\Sigma$ is a singleton, we will write
 $\EE_L$ or $\EE$ for this algebra, and we will use the variable $t$. 
 
 An {\em affinoid $L$-algebra} $\mathcal{A}$ is the datum of a topological $L$-algebra $\mathcal{A}$ together with a surjective $L$-algebra morphism
 \begin{equation}\label{affinoid-algebras}
 \TT_{L,U}\xrightarrow{\psi}\mathcal{A}.
 \end{equation} Every affinoid $L$-algebra comes equipped with a Banach
 $L$-algebra structure, with the norm $$\|g\|=\inf_{\psi(f)=g}\|f\|,\quad g\in\mathcal{A}.$$
 The kernel of $\psi$ is closed and we have the next result where we assume that $L$ is algebraically closed.
 \begin{Lemma}\label{affinoid-is-cartesian}
 Every affinoid $L$-algebra is countably cartesian.
 \end{Lemma}
 \begin{proof} We consider $\mathcal{A}$ and affinoid algebra, with $\psi$ and $\TT_{L,U}$ as
 in (\ref{affinoid-algebras}).
 If $L$ is algebraically closed and $\mathfrak{J}$ is an ideal of $\TT_{L,U}$, by \cite[\S 1.3 Theorem 6]{BOSC}, there exists an orthonormal basis 
 $(\mathcal{b}_i)_{i\in I}$ of $\TT_{L,U}$ and a subset $J\subset I$ such that 
 $(b_j)_{j\in J}$ is an orthonormal basis of $\mathfrak{J}$. Then, $(\psi(b_i))_{i\in I\setminus J}$ defines an orthonormal basis of $\mathcal{A}$. 
\end{proof}
The general case is also true, where $L$ is not necessarily algebraically closed. 
 Note that if $\mathcal{A}$ is the affinoid algebra associated to an affinoid subset of $\mathbb{P}^{1,an}_L$ (with $\mathbb{P}^{1,an}_L$ the rigid analytic affine line over $L$), with its spectral norm, then it is countably cartesian also as an easy consequence of the Mittag-Leffler decomposition \cite[Proposition 2.2.6]{FP}.

\subsubsection{The completion $\KK_{L,\Sigma}$ of the fraction field of $\TT_{L,\Sigma}$}\label{completion-fraction} Let $L$ be a valued field, complete, containing $\FF_q$.
The fraction field of $\TT_{L,\Sigma}$ is not complete, unless $\Sigma=\emptyset$. We write $\KK_{L,\Sigma}$ for its completion. It is easy to see that this is also equal to the completion of $L(\underline{t}_\Sigma)$, for the extension of the Gauss norm. If $L$ is a local field, so that
$L=\FF((\pi))$ with $\FF$ a finite field containing $\FF_q$ and $\pi$ a uniformiser, then
$\KK_\Sigma=k_L(\underline{t}_\Sigma)((\pi))$.
The residual field $k_{\KK_{L,\Sigma}}$ of $\KK_{L,\Sigma}$ is $k_L(\underline{t}_\Sigma)$. If ${\Sigma'}\subset\Sigma$, we have an isometric embedding $\KK_{L,\Sigma'}\subset\KK_{L,\Sigma}$. 

\begin{Lemma}\label{lemma-di-serre}
Let ${\Sigma'}$ be a subset of $\Sigma$. 
Let $\mathcal{B}=(b_i)_{i\in I}$ be a family of elements of $\mathcal{O}_{\KK_{L,\Sigma}}$ such that 
$(\overline{b}_i)_{i\in I}$ is a basis of the $k_L(\underline{t}_{\Sigma'})$-vector space $k_L(\underline{t}_\Sigma)$. Then, every element $f$ of $\KK_{L,\Sigma}$ can be expanded, in a unique way, as a converging series
$$f=\sum_{i\in I}f_ib_i,\quad f_i\in\KK_{\Sigma'},\quad f_i\rightarrow0,$$
and $\|f\|=\max_{i\in I}\|f_i\|$.
\end{Lemma}
In the above lemma $I$ is countable (this follows from the fact that $k_L$ is countable).
If we choose ${\Sigma'}=\emptyset$ we see that Lemma \ref{lemma-di-serre} implies that 
$\KK_{L,\Sigma}$ is countably cartesian as in Definition \ref{algebra-countably-cartesian}. In other words, the Banach $L$-vector space $\KK_{L,\Sigma}$ is endowed with an orthonormal basis providing us with an isometric isomorphism with a Banach $L$-space $c_{I}(L)$. The proof that we present is essentially the same as Serre's in \cite[Lemma 1, Proposition 1]{SER}.

\begin{proof}[Proof of Lemma \ref{lemma-di-serre}] One sees easily that $\|\KK_{L,\Sigma}\|=|L|$, therefore it suffices to show the lemma for $f\in \KK_{L,\Sigma}$ with $\|f\|=1$. Let us consider $\alpha\in L[\underline{t}_\Sigma]$ with $\|\alpha\|=1$. We can decompose (in a unique way) $\alpha=\alpha_0+\alpha_1$ with $\alpha_i\in L[\underline{t}_\Sigma]$, 
$\overline{\alpha}_1\in k_L[\underline{t}_\Sigma]\setminus\{0\}$, and $\|\alpha_0\|<1$. For any multi-index $\underline{k}=(k_i:i\in\Sigma)\in\NN^{\Sigma}$ we have, in $\KK_{L,\Sigma}$ (with $\underline{t}_\Sigma^{\underline{k}}=\prod_{i\in\Sigma}t_i^{k_i}$):
$$\underline{t}_\Sigma^{\underline{k}}\alpha^{-1}=\frac{\underline{t}_\Sigma^{k}}{\alpha_1}\left(1-\frac{\alpha_0}{\alpha_1}+\frac{\alpha_0^2}{\alpha_1^2}-\cdots\right)$$
(the series converges because $\|\alpha_0\|<1$). For every $\underline{k}$ and $j\geq 0$, the image of $\underline{t}_\Sigma^{\underline{k}}\alpha_1^{-j}$ in $k_L(\underline{t}_\Sigma)$ for the reduction map can be expanded in the basis $(\overline{b}_i)_{i\in I}$. We deduce that any element $f=\frac{\beta}{\alpha}\in L(\underline{t}_\Sigma)$, $\alpha\neq0$, can be expanded as a convergent series:
$$f=\sum_{i\in I}f_ib_i,\quad f_i\rightarrow0,\quad f_i\in\KK_{\Sigma'}.$$ This expansion is unique because otherwise, there would exist a non-trivial relation
$$0=\sum_{i\in I}f_ib_i$$ such that for some $i\in I$, $\|f_i\|=1$, in contradiction with the fact that $(\overline{b}_i)_{i\in I}$ is a basis of $k_L(\underline{t}_\Sigma)$ over $k_L(\underline{t}_{\Sigma'})$.
This means that there is an isometric embedding $L(\underline{t}_\Sigma)\rightarrow c_{I}(\KK_{L,\Sigma'})$. Completing, we are left with an isometric isomorphism of Banach $L$-vector spaces $\KK_{L,\Sigma}\cong c_{I}(\KK_{L,\Sigma'})$ which terminates the proof.
\end{proof}

\subsubsection{The non-complete fields $\LL_{L,\Sigma}$}
\label{non-completefields}
Let $\Sigma,L,\ldots$ as in \S \ref{completion-fraction}. 
In this paper we also need certain fields intermediate between the fraction field of $\TT_{L,\Sigma}$ and $\KK_{L,\Sigma}$.
For any $d\in k_L[\underline{t}_\Sigma]\setminus\{0\}$ we have the affinoid $L$-algebra (completion for the Gauss norm)
$\widehat{\TT_{L,\Sigma}[d^{-1}]}$ which is a Banach $L$-sub-algebra of $\KK_{L,\Sigma}$ which also is countably cartesian. 
 We consider
 $$\LL_{L,\Sigma}=\bigcup_{d\in k_L[\underline{t}_\Sigma]\setminus\{0\}} \widehat{\TT_{L,\Sigma}[d^{-1}]}.$$
 \begin{Lemma}
 $\LL_{L,\Sigma}$ is a subfield of $\KK_{L,\Sigma}$.
 \end{Lemma}
\begin{proof}
The relation of divisibility in $k_L[\underline{t}_\Sigma]$ induces a filtration of $\LL_{L,\Sigma}$
by Banach $L$-sub-algebras of the form $\widehat{\TT_{L,\Sigma}[d^{-1}]}$ so that $\LL_{L,\Sigma}$ is an 
$L$-sub-algebra of $\KK_{L,\Sigma}$. We still need to show that every non-zero element $f$ of $\LL_{L,\Sigma}$ is invertible; we follow the same ideas of Lemma \ref{lemma-di-serre}; there is no loss of generality if we suppose that $\|f\|=1$. There exists $d\in k_L[\underline{t}_\Sigma]\setminus\{0\}$ such that 
$f\in\TT_{L,\Sigma}[d^{-1}]^{\wedge}$. We can write
$f=\alpha_1-\alpha_0$ where $\alpha_1\in k_L[\underline{t}_\Sigma][d^{-1}]\setminus\{0\}$ and where $\alpha_0\in\TT_{L,\Sigma}[d^{-1}]^\wedge$
is such that $\|\alpha_0\|<1$. Therefore, in $\KK_{L,\Sigma}$:
$$\frac{1}{f}=\frac{1}{\alpha_1}\left(1-\frac{\alpha_0}{\alpha_1}\right)^{-1}=
\frac{1}{\alpha_1}\sum_{i\geq0}\left(\frac{\alpha_0}{\alpha_1}\right)^i$$
 and the series converges in $\TT_{L,\Sigma}[\widetilde{d}^{-1}]^{\wedge}\subset\LL_{L,\Sigma}$, for some element $\widetilde{d}\in k_L[\underline{t}_\Sigma]$.
\end{proof} 
Note that $\LL_{L,\Sigma}$ contains 
the fraction field of $\TT_\Sigma$ and is not complete, unless $\Sigma=\emptyset$. The fields $\LL_{L,\Sigma}$ and $\KK_{L,\Sigma}$ both have residual field 
$k_L(\underline{t}_\Sigma)$ and $\KK_{L,\Sigma}$ is the completion of $\LL_{L,\Sigma}$ for the Gauss norm.

\subsection{Analytic functions with values in non-archimedean Banach algebras}\label{holomorphicfunc}

In this subsection we suppose that $L$ is an algebraically closed valued field with multiplicative valuation $|\cdot|$, complete for this valuation. We choose $(B,|\cdot|_B)$ a Banach $L$-algebra which is countably cartesian in the sense of Definition \ref{algebra-countably-cartesian}. 

Let $X/L$ be a rigid analytic variety which is the datum 
of $(X,T,\mathcal{O}_X)$ with $X$ a set, a $G$-topology $T$ and a structure sheaf $\mathcal{O}_X$ of $L$-algebras.
In all the following, we denote by $\mathcal{O}_{X/B}$ the presheaf of $B$-algebras defined, for $\mathcal{U}=(U_i)_i$
an affinoid covering of $X$, by $$\mathcal{O}_{X/B}(U_i)=\widehat{\mathcal{O}_X(U_i)\otimes_{L}B}\cong\widehat{\mathcal{O}_X(U_i)\otimes_{L}c_I(L)},$$
the completion being taken for the spectral (sub-multiplicative) norm on $U_i$ (see \cite[\S 3.2]{BGR}), and where $\cong$ indicates an isometric isomorphism of Banach $L$-vector spaces. 

An {\em analytic  function} (also called {\em holomorphic function}) from $X$ to $B$ is by definition
an element of $\mathcal{O}_{X/B}(X)$. Equivalently, an analytic function $f:X\rightarrow B$ is a function such that for every rational subset $Y\subset X$, 
the restriction $f|_Y$ is the uniform limit over $Y$ of a sequence of elements of $\mathcal{O}_X(Y)\otimes_{L}B$. As an alternative notation, we write $$f\in\operatorname{Hol}_B(X\rightarrow B).$$
Let $\mathcal{B}=(b_i:i\in\mathcal{I})$ be a countable orthonormal basis of $B$. In other words, for all $i\in I$ we have $|b_i|_B=1$, and moreover, for all $g\in B$ there is a unique convergent series expansion
$$g=\sum_{i\in\mathcal{I}}g_ib_i,\quad g_i\in L,$$
with $|g_i|\rightarrow0$ as $i\rightarrow\infty$ and $|g|_B=\max_i|g_i|$.
Then, every element $f\in \operatorname{Hol}_B(X\rightarrow B)$ can be expanded, in a unique way, as
$$f=\sum_{i\in I}f_ib_i$$
where $f_i\rightarrow0$ for the spectral norm associated to any rational subset $Y$ of $X$. For example, we can take $B=\KK_\Sigma$ or $B=\TT_\Sigma[d^{-1}]^\wedge$ with $d\in k_L[\underline{t}_\Sigma]\setminus\{0\}$.
Let $C$ be a sub-$L$-algebra of $B$ (not necessarily complete). We write $$\operatorname{Hol}_B(X\rightarrow C)$$ for the 
$C$-algebra of holomorphic (or analytic) functions from $X$ to $B$ such that the image is contained in $C$, and we omit the subscript if $B=C$. For instance, 
we can take $C=\LL_\Sigma\subset\KK_\Sigma=B$. We denote by $$\mathcal{O}_{X/B/C}$$ the 
presheaf of $C$-algebras determined by $\mathcal{O}_{X/B/C}(Y)=\operatorname{Hol}_B(Y\rightarrow C)$ for $Y$ rational subset of $X$. Since $C$ is an $L$-algebra, for every $U$ affinoid subdomain, $\mathcal{O}_{X/B/C}(U)$ is a $\mathcal{A}$-module
and we can define, for $M$ a finitely generated $\mathcal{A}$-module, the pre-sheaf $\mathcal{M}_{B/C}$ on $X$ by 
$$\mathcal{M}_{B/C}(U)=M\otimes_{\mathcal{A}}\mathcal{O}_{X/B/C}(U).$$

Tate's acyclicity theorem \cite[Theorem 4.2.2]{FP} is easily seen to extend to this framework and we have the next result which we do not prove in full (because anyway it will not be exploited in the paper) but which is useful to understand the nature of the class of analytic functions we use:

\begin{Lemma}\label{Tate-acyclicity}
The presheaf $\mathcal{M}_{B/C}$ is a sheaf of $C$-algebras.
\end{Lemma}

We omit the details of the proof because the proof of Tate's acyclicity theorem given in the above reference can be easily adapted to our framework,  thanks to the hypothesis that $B$ is countably cartesian. We will be content to focus on few aspects, in the case of $\mathcal{M}$ trivial. If $X=\operatorname{Spm}(\mathcal{A})$ with $\mathcal{A}$ an affinoid $L$-algebra, $U$ an affinoid subdomain of $X$ and $(U_j)_{j\in J}$ an admissible covering of $U$ (with $J$ a finite set), saying that $\mathcal{O}_X$ is a sheaf of $L$-algebras amounts to saying that there is an exact sequence of $L$-algebras
$$0\rightarrow\mathcal{O}_X(U)\xrightarrow{\alpha}\prod_{j\in J}\mathcal{O}_X(U_j)\xrightarrow{\beta}\prod_{j,k\in J}\mathcal{O}_X(U_j\cap U_k)$$ where $\alpha$ is
defined by the restrictions on the $U_j$'s and $\beta((f_j)_{j\in J})=(f_j|_{U_j\cap U_k}-f_k|_{U_j\cap U_k})_{j,k\in J}$. Taking $(\cdot)\widehat{\otimes}_{L}B$ determines an exact sequence of $B$-algebras because, denoting by $\alpha$ and $\beta$ the resulting maps, with $|\cdot|_U$ the spectral norm over $U$, $\sup_{j}|\alpha(f)|_{U_j}=\max_{j}|\alpha(f)|_{U_j}=|f|_U$
($\alpha$ is an isometry) so that if $(f_j)_{j\in J}$ is an element of $\prod_j\mathcal{O}_X(U_j)\widehat{\otimes}_LB$
such that $\beta((f_j)_j)=0$ then, writing $f_j=\sum_{i\in I}f_j^{(i)}b_i$ with $f_j^{(i)}\rightarrow0$ as $i\rightarrow\infty$ (expansion in the orthonormal basis $(b_i)_{i\in I}$ of $B$), for all $i\in I$ there exists $f^{(i)}\in\mathcal{O}_{X}(U)$
with $\alpha(f^{(i)})=(f_j^{(i)})_{j\in J}$ for all $i$, and $f^{(i)}\rightarrow0$ for $|\cdot|_U$ and therefore, $f=\sum_if^{(i)}b_i$ defines an element of $\mathcal{O}_{X}(U)\widehat{\otimes}_LB$
such that $\alpha(f)=(f_j)_j$. Now, the maps $\alpha$ and $\beta$ define $C$-algebra
maps between $\mathcal{O}_{X/B/C}(U),\prod_j\mathcal{O}_{X/B/C}(U_j)$ etc. and the map resulting from $\alpha$ is injective, while the element $f\in\mathcal{O}_{X/B}(U)$ constructed above clearly belongs to $\mathcal{O}_{X/B/C}(U)$ if $f_j\in \mathcal{O}_{X/B/C}(U_j)$ for all $j$.

\subsubsection{Structure of $\mathcal{O}_{X/B/C}$ with $X$ a curve}\label{structure-XBC}

We consider $B$ a Banach $L$-algebra which is countably cartesian and we suppose that $\Lambda$ is a partially ordered countable set, with partial order $\prec$ such that there is a family $(B_\lambda)_{\lambda\in\Lambda}$ of Banach sub-$L$-algebras of $B$ with the following two properties:
\begin{enumerate}
\item If $\lambda\prec\lambda'$ then $B_\lambda\subset B_{\lambda'}$,
\item For all $\lambda,\lambda'\in\Lambda$ such that $\lambda\prec\lambda'$ there exists an orthonormal basis $(b_i)_{i\in I}$ of $B$ (depending on $\lambda$) and subsets $J\subset J'\subset I$ with $(b_i)_{i\in J}$ an orthonormal basis of $B_\lambda$ and $(b_i)_{i\in J'}$
an orthonormal basis of $B_{\lambda'}$. 
\end{enumerate}
We set $C=\cup_\lambda B_\lambda$. This is a sub-$L$-algebra of $B$. We have the next Lemma.

\begin{Lemma}\label{precise-image}
Let $X$ be a rigid analytic curve over $L$. The following identity holds:
$$\operatorname{Hol}_{B}(X\rightarrow C)=\bigcup_{\lambda\in\Lambda}\operatorname{Hol}_{B_\lambda}(X\rightarrow B_\lambda).$$
\end{Lemma}

\begin{proof}
We first show the lemma when $X=\operatorname{Spm}(\mathcal{A})$ where $\mathcal{A}$ is an integral affinoid $L$-algebra. In this case, if $f:X\rightarrow L$ is holomorphic with infinitely many zeroes, then it is identically zero. Now, let $f$ be a global section of $\mathcal{O}_{X/B/C}$. For all $x\in X$ there exists $\lambda\in\Lambda$ such that 
$f(x)\in B_\lambda$. Therefore, there exists 
a map $$X\xrightarrow{\Phi}\Lambda,$$
defined by associating to every $x\in X$ a choice of $\lambda\in\Lambda$ such that 
$f(x)\in B_\lambda$.

 Since the underlying set of $X$ is uncountable (because $L$ is uncountable, due to the fact that it is complete) while the target set is countable, there exists an infinite subset $X_0\subset X$ and $\lambda\in\Lambda$, such that
$\Phi(x)=\lambda$ for all $x\in X_0$. Then $f(X_0)\subset B_\lambda$. We expand $f$ in an 
orthonormal basis $(b_i)_{i\in I}$ of $B$ such that for some $J\subset I$, $(b_j)_{j\in J}$ is an orthonormal basis of $B_\lambda$:
$$f=\sum_{j\in J}f_jb_j+\sum_{i\in I\setminus J}f_ib_i$$ (with $f_j\rightarrow0$ as $j\rightarrow\infty$).
Since for all $i\in I\setminus J$, $f_i(x)=0$ for all $x\in X_0$, $f_i\in \mathcal{O}_{X}(X)$ has infinitely zeroes and therefore vanishes identically and we deduce that 
$f\in\mathcal{O}_{X/B_\lambda}(X)$.

Now, suppose that $X$ is an affinoid subdomain of the curve $Y=\operatorname{Spm}(\mathcal{A})$. Let $f$ be in $\operatorname{Hol}_{B}(Y\rightarrow C)$. Then we can find 
$\lambda,\lambda'\in\Lambda$ such that $\lambda\prec\lambda'$ and 
$f\in\mathcal{O}_{Y/B_\lambda'}(Y)$, $f|_X\in \mathcal{O}_{Y/B_\lambda}(X)$. Writing 
$$f=\sum_{j'\in J'\setminus J}f_{j'}b_{j'}+\sum_{j\in J}f_jb_j$$ we note that for all $j\in J'\setminus J$, $f_{j'}(x)=0$ for all $x\in X$ which is infinite, and $f_{j'}$ vanishes identically. This means that
$f\in\mathcal{O}_{Y/B_\lambda}(Y)$. The lemma follows easily working on an admissible covering of a rigid analytic curve by affinoids.
\end{proof}

\subsubsection{Entire functions} We look at $B$-valued analytic functions on polydisks, where $(B,|\cdot|_B)$ is a Banach $L$-algebra which is countably cartesian.
If $X$ is the polydisk $$D_{L}(0,r)^n=\{\underline{x}=(x_1,\ldots,x_n)\in L^n;|x|\leq r\}$$ with $r\in|L|$ and with the usual structure sheaf of converging series, then $\operatorname{Hol}_B(X\rightarrow B)$ equals the ring of series $\sum_{\underline{i}\geq 0}f_{\underline{i}}\underline{x}^{\underline{i}}$ where $\underline{i}=(i_1,\ldots,i_n)$ with $i_j\geq 0$ for all $j$, where $\underline{x}^{\underline{i}}=x_1^{i_1}\cdots x_n^{i_n}$, and 
where $f_{\underline{i}}\in B$ are such that $|f_{\underline{i}}|_Br^{i_1+\cdots+i_n}\rightarrow0$ as $\underline{i}\rightarrow\infty$. We deduce that the
$B$-algebra $\operatorname{Hol}_B(\mathbb{A}^{n,an}_L\rightarrow B)$, with $\mathbb{A}^{n,an}_L$ the analytic $n$-dimensional affine space over $L$, is equal to the $B$-algebra of 
the functions $L^n\rightarrow B$ defined by the 
formal series $\sum_{\underline{i}\geq 0}f_{\underline{i}}\underline{x}^{\underline{i}}\in B[[x_1,\ldots,x_n]]$ such that $|f_{\underline{i}}|_Br^{i_1+\cdots+i_n}\rightarrow0$ for all $r\in|B|_B$.
It is also easy to see that a function
$f:L^n\rightarrow B$ belongs to $\operatorname{Hol}_B(\mathbb{A}^{n,an}_L\rightarrow B)$ 
if, on every bounded subset $U$ of $L$,
$f$ can be obtained as a uniform limit of polynomial functions $f_i\in B[x_1,\ldots,x_n]$,
$f_i: U\rightarrow B$. These functions are called $B$-\emph{entire} (or simply \emph{entire} if the reference to $B$ is understood). 
The following property is easily checked. Let $(f_i)_{i\geq 0}$ be a sequence of $B$-entire functions. If for every such $r$, the sequence $(f_i)_{i\geq 0}$ converges uniformly over $D_L(0,r)^n$, then the limit function $L^n\rightarrow B$ is a $B$-entire function.

The next result is a simple generalization of the analogue of Liouville's theorem which can be found in Schikhof's \cite[Theorems 42.2 and 42.6]{SCH}. See also \cite[Proposition 8]{PEL&PER}.

\begin{Proposition}[$B$-analogue of Liouville's Theorem]\label{entire}
Assuming that the Banach $L$-algebra $B$ is countably cartesian, any bounded $B$-entire function is constant.
\end{Proposition}

Although the principles of the proof are completely elementary, we prefer to give all the details. Let $n$ be a positive integer and $f:D_{L}(0,1)^n\rightarrow B$ a $B$-analytic function, so that, with 
$\underline{x}=(x_1,\ldots,x_n)\in D_{L}(0,1)^n$,
$$f(\underline{x})=\sum_{\underline{i}}f_{\underline{i}}\underline{x}^{\underline{i}},\quad f_{\underline{i}}\in B,$$
where $\underline{x}=x_1^{i_1}\cdots x_n^{i_n}$ if $\underline{i}=(i_1,\ldots,i_n)$ and $|f_{\underline{i}}|_B\rightarrow0$ as $\underline{i}\rightarrow\infty$. We set 
$$|f|_{B,\sup}:=\sup_{\underline{x}\in D_L(0,1)^n}|f(\underline{x})|_B.$$ We also set 
$\|f\|_B=\sup\{|f_{\underline{i}}|_B:\underline{i}\in\NN^n\}=\max\{|f_{\underline{i}}|_B:\underline{i}\in\NN^n\}$.

\begin{Lemma}\label{in-disks}
We have $|f|_{B,\sup}=\|f\|_B$.
\end{Lemma}

\begin{proof} There is no loss of generality to suppose that $\|f\|_B=1$. Indeed, $|B|_B=|L|$ because $B$ is countably cartesian.
It is easy to see that $|f|_{B,\sup}\leq\|f\|_B$ and we only need to prove the opposite inequality.
We proceed by induction on $n>0$. Let us write $\underline{x}=(x_1,\underline{x}')$ (concatenation). We note that
$$|f|_{B,\sup}\geq\sup_{x_1\in D_L(0,1)}\Bigg(\sup_{\underline{x}'\in D_L(0,1)^{n-1}}|f(x_1,\underline{x}')|_B\Bigg)=\sup_{x_1}|f(x_1,\cdot)|_{B,\sup}=\sup_{x_1}\|f(x_1,\cdot)\|_B$$
by the induction hypothesis. Let $B'$ be the $L$-algebra 
$$\operatorname{Hol}_B(D_L(0,1)^{n-1}\rightarrow B)$$ with the norm $|\cdot|_{B,\sup}=\|\cdot\|_B$. It is easy to see that $B'$ is a Banach $L$-algebra which is countably cartesian. Then, we can identify $f$ with a $B'$-analytic
function $\widetilde{f}:D_L(0,1)\rightarrow B'$, where $\widetilde{f}=\sum_{i\geq 0}f_ix_1^i$, $f_i\in B'$, $f_i\rightarrow0$. We see that $\sup_{x_1}\|f(x_1,\cdot)\|_B=\|\widetilde{f}\|_{B'}$, the latter norm equals $\|f\|_B$. Hence $|f|_{B,\sup}=\|f\|_B$. It remains to prove the case $n=1$ of the Lemma. For this, we follow \cite[Lemma 42.1]{SCH}.

Let us therefore consider an element $f\in\operatorname{Hol}_B(D_L(0,1)\rightarrow B)$ with $f(x)=\sum_{i\geq 0}f_ix^i$, $f_i\rightarrow0$. Of course $|f|_{B,\sup}\leq \|f\|_B$ and we can 
again suppose that $\|f\|_B=1$. If $|f_0|_B=1$ then 
$1=|f(0)|_B\leq|f|_{B,\sup}\leq \|f\|_B=1$ and we are done. Otherwise, let $N$ be the smallest integer $j$ such that $|f_j|_B=1$. We have $N>0$. Let $\epsilon>0$ be such that $\epsilon<1-\max\{|f_i|_B:0\leq i<N\}$. Since $|L^\times|$ is dense in $\RR_{>0}$ there exists $x\in L^\times$ such that $1-\epsilon<|x^N|<1$. We claim that $|f(x)|_B=|x^N|>1-\epsilon$. To see this note that 
$\max\{|f_i|_B:0\leq i<N\}<1-\epsilon$ so that $|f_0+\cdots+f_{N-1}x^{N-1}|_B<1-\epsilon$. On the other hand the sequence $(|x^i|)_{i\geq N}$ is strictly decreasing so that $|\sum_{i\geq N}f_ix^i|_B=|f_Nx^N|_B=|x^N|$. Hence $$|f(x)|_B=\max\Bigg\{\Big|f_0+\cdots+f_{N-1}x^{N-1}\Big|_B,\Big|\sum_{i\geq N}f_ix^i\Big|_B\Bigg\}=\Big|\sum_{i\geq N}f_ix^i\Big|_B>1-\epsilon.$$ The claim follows by letting $\epsilon$ tend to $0$ and the proof of the lemma is complete.
\end{proof}
\begin{Remark}{\em If $B$ is an algebraically closed field, Lemma \ref{in-disks} is contained in the arguments of \cite[\S 5.1.4]{BGR}.}\end{Remark}
\begin{proof}[Proof of Proposition \ref{entire}] Let $f$ be $B$-entire (in $n$ variables). If $r\in|L^\times|$ we can choose $\underline{\alpha}=(\alpha_1,\ldots,\alpha_n)\in (L^\times)^n$ so that
$|\alpha_1|=\cdots=|\alpha_n|=r$ and apply Lemma \ref{in-disks} to the $B$-entire function
$f(\alpha_1x_1,\cdots,\alpha_nx_n)$. We deduce that 
$$\sup_{\underline{x}\in D_L(0,r)^n}|f(\underline{x})|_B=\max_{\underline{i}}|f_{\underline{i}}|_Br^{i_1+\cdots+i_n}.$$ Assume now that $|f|_B$ is bounded, say, by $M>0$. Then
$\max_{\underline{i}}|f_{\underline{i}}|_Br^{i_1+\cdots+i_n}\leq M$ for all $r\in|L^\times|$. This means that $|f_{\underline{i}}|_B=0$ for all $\underline{i}\neq 0$ and $f$ is a constant map $\mathbb{A}^{n,an}_N\rightarrow B$ that can be identified with its constant term $f_{\underline{0}}$.
\end{proof}

\subsection{Drinfeld modules and exponential functions}\label{carlitzexttate}

We denote by $\FF_q$ the finite field with $q=p^e$ elements and characteristic $p$. Let $\theta$ be an indeterminate. We write $A$ for $\FF_q[\theta]$, the $\FF_q$ algebra of polynomials in $\theta$. We denote by $K$ its fraction field and by $K_\infty=\FF_q((\frac{1}{\theta}))$ the local field which is its completion at the infinity place or, which is the same, the completion for $|\cdot|$ the multiplicative valuation of $K$ normalised by $|\theta|=q$. Finally, we denote by $\CC_\infty$ the completion of an algebraic closure $K^{ac}$ of $K$. We recall that the residual field $k_{\CC_\infty}$ of $\CC_\infty$ is $\FF_q^{ac}$, an algebraic closure of $\FF_q$, that we can view as a subfield of $\CC_\infty$. From now on we set $L=\CC_\infty$ and we consider $\TT_{L,\Sigma}=\TT_\Sigma,\LL_{L,\Sigma}=\LL_\Sigma,\KK_{L,\Sigma}=\KK_\Sigma,$ etc.

The automorphism $c \mapsto c^q$ of $\CC_\infty$ extends in a unique way to an $\FF_q[\underline{t}_\Sigma]$-linear automorphism $\tau$ of $\CC_\infty[\underline{t}_\Sigma]$ and therefore, to each of the three $\CC_\infty$-algebras $\TT_\Sigma\subset\LL_\Sigma\subset\KK_\Sigma$ defined in \S \ref{Rings-fields-modules}, being continuous and open on the first and the third. Recall that $\|\cdot\|$ denotes the unique extension of the Gauss norm to $\KK_\Sigma$. Recall that by Lemma \ref{lemma-di-serre}, $(\KK_\Sigma,\|\cdot\|)$ is a Banach $\CC_\infty$-algebra which is countably cartesian.
For all $f\in \KK_\Sigma$, we have that $\|\tau(f)\|=\|f\|^q$. The sub-ring $\TT_\Sigma^{\tau=1}$ of the elements $f\in\TT_\Sigma$ such that $\tau(f)=f$ is the polynomial ring $\FF_q[\underline{t}_\Sigma]$, and we have the identities of fixed subfields 
\begin{equation}\label{fixed-subfields}
\FF_q(\underline{t}_\Sigma)=\LL_\Sigma^{\tau=1}=\KK_\Sigma^{\tau=1}.
\end{equation}
We also consider the non-commutative $\KK_\Sigma$-algebras
$\KK_\Sigma[\tau]$ and $\KK_\Sigma[[\tau]]$ (the multiplication is defined by the commutation rule $\tau f=\tau(f)\tau$ for $f\in \KK_\Sigma$). 

\subsubsection{The exponential of a Drinfeld module}

For the background on Drinfeld modules, lattices and exponential functions we refer to \cite{GOS} and \cite[\S 3]{PEL5}. Let $\phi$ be a Drinfeld $A$-module of rank $r$ defined over $\CC_\infty$, let $\exp_\phi$ be its exponential function and $\Lambda_\phi=\operatorname{Ker}(\exp_\phi)\subset\CC_\infty$
be its lattice period, a free, rank $r$ module over $A$ which is discrete for the metric of $\CC_\infty$ induced by $|\cdot|$. We recall that $\exp_\phi$ is an $\FF_q$-linear entire function
$\CC_\infty\rightarrow\CC_\infty$ that can be computed by means of the following everywhere convergent 
Weierstrass product
\begin{equation}\label{weierstrassproduct}\exp_\phi(Z)=Z\sideset{}{'}\prod_{\lambda\in\Lambda_\phi}\left(1-\frac{Z}{\lambda}\right),\quad Z\in\CC_\infty\end{equation}
(the dash $'$ indicates that the product runs over $\Lambda_\phi\setminus\{0\}$).
This product expansion also shows that locally at $0$, $\exp_\phi$ induces an isometric $\FF_q$-linear isomorphism. Indeed, if $\rho_\phi:=\min_{\lambda\in\Lambda_\phi\setminus\{0\}}|\lambda|$, 
$\exp_\phi$ induces an $\FF_q$-linear automorphism of $$D^\circ_{\CC_\infty}(0,\rho_\phi)=\{z\in\CC_\infty:|z|<\rho_\phi\}$$
such that for all $z\in D^\circ_{\CC_\infty}(0,\rho_\phi)$, $|\exp_\phi(z)|=|z|$. In fact
it can be proved that $\exp_\phi$ induces an isomorphism of $\CC_\infty$-rigid analytic
spaces $\mathbb{A}^{1,an}_{\CC_\infty}/\Lambda_\phi\cong\mathbb{A}^{1,an}_{\CC_\infty}$.
With $\phi(\CC_\infty)$ the $A$-module induced
by $\phi$, there is an exact sequence of $A$-modules
$$0\rightarrow\Lambda_\phi\rightarrow\CC_\infty\xrightarrow{\exp_\phi}\phi(\CC_\infty)\rightarrow0$$
($\exp_\phi$ is uniquely determined by the condition of being an entire $A$-module morphism 
with first dernivative $\exp_\phi'=1$).
We fix a finite subset $\Sigma\subset\NN^*$ and a Drinfeld module $\phi$ defined over $\CC_\infty$. There is a unique structure of $A\otimes_{\FF_q}\FF_q(\underline{t}_\Sigma)$-module $\phi(\KK_\Sigma)$ over $\KK_\Sigma$ which is defined by extending the operators $\phi_a$ (of multiplication by $a\in A$ in the $A$-module $\phi(\CC_\infty)$) $\FF_q(\underline{t}_\Sigma)$-linearly to $\KK_\Sigma$ along the extension of the map $(x\mapsto x^q):\CC_\infty\rightarrow\CC_\infty$ to the map $\tau:\KK_\Sigma\rightarrow\KK_\Sigma$.
Since $\KK_\Sigma$ is complete we have a $\FF_q(\underline{t}_\Sigma)$-linear map $\exp_\phi:\KK_\Sigma\rightarrow\KK_\Sigma$ which induces a morphism of $A\otimes_{\FF_q}\FF_q(\underline{t}_\Sigma)$-modules
$$\KK_\Sigma\xrightarrow{\exp_\phi}\phi(\KK_\Sigma)$$ such that
$\Lambda_\phi\otimes_{\FF_q}\FF_q(\underline{t}_\Sigma)\subset\operatorname{Ker}(\exp_\phi)$. 
It is obvious that $\exp_\phi$ induces an isometric $\FF_q(\underline{t}_\Sigma)$-linear automorphism of
$D^\circ_{\KK_\Sigma}(0,\rho_\phi)=\{f\in\KK_\Sigma:\|f\|<\rho_\phi\}.$
Moreover, we have:

\begin{Proposition}\label{propositionontatealgebras-ksigma}
Let $\phi$ be a Drinfeld $A$-module with exponential $\exp_\phi$. The map $\exp_\phi$ induces an exact sequence of $A\otimes_{\FF_q}\FF_q(\underline{t}_\Sigma)$-modules:
\begin{equation}\label{exactsequencetatemodule}
0\rightarrow\Lambda_\phi\otimes_{\FF_q}\FF_q(\underline{t}_\Sigma)\rightarrow\KK_\Sigma\xrightarrow{\exp_\phi} \phi(\KK_\Sigma)\rightarrow0.\end{equation}
\end{Proposition}

To prove this we need a small intermediate result. Using Lemma \ref{lemma-di-serre} we can choose an 
$\FF_q^{ac}$-basis $\mathcal{B}=(b_i)_{i\in I}$ of $\FF_q^{ac}(\underline{t}_\Sigma)$ 
determining an orthonormal basis of the Banach $\CC_\infty$-algebra $\KK_\Sigma$. 
For any $i$, there exists $d>0$ such that $\tau^d(b_i)=b_i$.
Let $\widetilde{J}\subset I$ be a finite subset. Considering the orbit under the action of the group $\operatorname{Gal}(\FF_q(b_j:j\in \widetilde{J})/\FF_q(\underline{t}_\Sigma))$ we see that there exists $J$ finite, with $\widetilde{J}\subset J\subset I$ with the following property.
There is a matrix 
$M_J\in\GL_{|J|}(\FF_q^{ac})$ such that, writing $\underline{b}_J$ for the column matrix $(b_i)_{i\in J}$,
$$\tau(\underline{b}_J)=M_J\underline{b}_J,$$ and moreover, there is a decomposition
\begin{equation}\label{decomposition-tau}
\FF_q^{ac}(\underline{t}_\Sigma)=\operatorname{Vect}_{\FF_q^{ac}}(\underline{b}_J)\oplus\operatorname{Vect}_{\FF_q^{ac}}((b_i)_{i\in I\setminus J})\end{equation} which splits the action of  
$\tau$. 

\begin{Lemma}\label{locally-we-have-this}
For any $\widetilde{J}\subset J\subset I$ as above, the exponential map $\exp_\phi$ induces a surjective $\FF_q$-linear endomorphism of $\oplus_{j\in J}\CC_\infty b_j$ with kernel 
$\Lambda_\phi\otimes_{\FF_q}\FF_q(\underline{t}_\Sigma)^{|J|\times 1}$.
\end{Lemma}

\begin{proof}
Since $J$ is fixed in the proof, let us write more simply $\underline{b}=\underline{b}_J$ and $M=M_J$.
Also, if $X$ is any matrix with entries in $\KK_\Sigma$, we set $X^{(i)}=\tau^i(X)$ (coefficient-wise application of $\tau^i$). Note that since $\underline{b}^{(1)}=M\underline{b}$, we have 
$\underline{b}^{(i)}=M^{(i-1)}\cdots M^{(1)}M\cdot\underline{b}$ for all $i\geq 0$.
By Lang's theorem \cite[Corollary p. 557]{LAN} there exists $U\in\GL_{|J|}(\FF_q^{ac})$ such that
$\tau(U)=MU$. Hence, $U^{(i)}=M^{(i-1)}\cdots M^{(1)}MU$ for all $i\geq0$. We deduce 
$$\underline{b}^{(i)}=U^{(i)}U^{-1}\underline{b},\quad i\geq 0\quad\text{ and }\quad(U^{-1}\underline{b})^{(1)}=U^{-1}\underline{b}.$$
Hence $U^{-1}\underline{b}\in \FF_q(\underline{t}_\Sigma)^{|J|}$ by (\ref{fixed-subfields}).

Let us compute, for $\underline{a}\in\CC_\infty^{|J|\times 1}$ (column vector), $({}^t\underline{a}\cdot
\underline{b})^{(i)}$, $i\geq 0$. We immediately see:
$({}^t\underline{a}\cdot
\underline{b})^{(i)}=({}^t\underline{a}\cdot U
)^{(i)}U^{-1}\underline{b}.$ Transposing we get:
$${}^t({}^t\underline{a}\cdot
\underline{b})^{(i)}={}^t\underline{b}\cdot{}^tU\cdot({}^tU\cdot\underline{a})^{(i)},$$
hence, if $f={}^t\underline{a}\cdot
\underline{b}\in \oplus_{j\in J}\CC_\infty b_j$,
$$\exp_\phi(f)=\exp_\phi({}^t\underline{a}\cdot
\underline{b})={}^t\underline{b}\cdot{}^t(U^{-1})\exp_\phi({}^tU\cdot \underline{a})\in\oplus_{j\in J}\CC_\infty b_j.$$ Since the map $\exp_\phi:\CC_\infty^{|J|\times 1}\rightarrow\CC_\infty^{|J|\times 1}$ is surjective, $\exp_\phi:\oplus_{j\in J}\CC_\infty b_j\rightarrow \oplus_{j\in J}\CC_\infty b_j$
is surjective. Now consider an element $f={}^t\underline{a}\cdot\underline{b}\in \oplus_{j\in J}\CC_\infty b_j$ such that $\exp_\phi(f)=0$. By the above computation, 
this is equivalent to $\exp_\phi({}^tU\cdot\underline{a})=0$, so that
$$\underline{a}\in{}^t(U^{-1})\cdot\Lambda_\phi^{|J|\times1}.$$
But ${}^t\underline{a}\cdot\underline{b}\in\Lambda_\phi^{1\times|J|}U^{-1}\cdot\underline{b}$
and we have seen that $U^{-1}\cdot\underline{b}\in \FF_q(\underline{t}_\Sigma)^{|J|\times 1}$. The lemma follows.
\end{proof}

\begin{proof}[Proof of Proposition \ref{propositionontatealgebras-ksigma}]
We first show that $\exp_\phi$ is surjective. Let us consider $g\in\KK_\Sigma$. There exists $J\subset I$ finite with $\underline{b}=(b_j)_{j\in J}=\underline{b}_J$ with $(\underline{b})^{(1)}=M\cdot\underline{b}$ as in the proof of Lemma \ref{locally-we-have-this}, and 
additionally, we can decompose
$$g=g_0+g_1$$
with $\|g_0\|<\rho_\phi$ and $g_1\in\oplus_{j\in J}\CC_\infty b_j$. By Lemma \ref{locally-we-have-this} there exists $f_1\in \oplus_{j\in J}\CC_\infty b_j$ such that 
$\exp_\phi(f_1)=g_1$ and since $\exp_\phi$ induces an isometry over 
$D^\circ_{\KK_\Sigma}(0,\rho_\phi)$, there also exists $f_0\in D^\circ_{\KK_\Sigma}(0,\rho_\phi)$
such that $\exp_\phi(f_0)=g_0$. Setting $f=f_0+f_1$ we deduce $\exp_\phi(f)=g$.

It remains to compute the kernel of $\exp_\phi$ over $\KK_\Sigma$. Let $f\in\KK_\Sigma$ be 
such that $\exp_\phi(f)=0$. Again, we can write $f=f_0+f_1$ with $\|f_0\|<\rho_\phi$ and 
$f_1\in \oplus_{j\in J}\CC_\infty b_j$. We write $f_0=f_0^0\oplus f_0^1$ where $f_0^0$
belongs to the Banach $\CC_\infty$-sub-vector space of $\KK_\Sigma$ generated by 
$(b_{i})_{i\in I\setminus J}$ and $f_0^1\in \oplus_{j\in J}\CC_\infty b_j$. By the hypothesis on $J$ we see that 
$\exp_\phi(f_0^0)=\sum_{i\in I\setminus J}c_ib_i$ while $\exp_\phi(f_0^1+f_1)\in\oplus_{j\in J}\CC_\infty b_j$. Hence, again by the fact that $\exp_\phi$ induces an isometry over 
$D^\circ_{\KK_\Sigma}(0,\rho_\phi)$, we can suppose that $f_0=0$. We can conclude by using Lemma \ref{locally-we-have-this}.\end{proof}

Let $\delta$ be an element of $\FF_q(\underline{t}_\Sigma)^\times$.
From the proof of Proposition \ref{propositionontatealgebras-ksigma} one deduces that the exponential function $\exp_\phi$  of a Drinfeld $A$-module $\phi$ also induces an $\FF_q[\underline{t}_\Sigma][\delta]$-linear surjective endomorphism of $\TT_\Sigma[\delta]^\wedge\subset\KK_\Sigma$, and we deduce the next result (compare with  \cite{ANG&PEL2}):

\begin{Corollary}\label{propositionontatealgebras}
For any $\delta\in \FF_q(\underline{t}_\Sigma)$ the map $\exp_\phi$ induces an exact sequence of $A[\underline{t}_\Sigma][\delta]$-modules:
\begin{equation}\label{exactsequencetatemodule}
0\rightarrow\Lambda_\phi\otimes_{\FF_q}\FF_q[\underline{t}_\Sigma][\delta]\rightarrow\widehat{\TT_\Sigma[\delta]}\xrightarrow{\exp_\phi} \phi(\widehat{\TT_\Sigma[\delta]})\rightarrow0.\end{equation}
Hence, we also have 
an exact sequence of $\FF_q(\underline{t}_\Sigma)[\theta]$-modules:
\begin{equation}\label{exactsequencetatemodule2}
0\rightarrow\Lambda_\phi\otimes_{\FF_q}\FF_q(\underline{t}_\Sigma)\rightarrow\LL_\Sigma\xrightarrow{\exp_\phi} \phi(\LL_\Sigma)\rightarrow0.\end{equation}
\end{Corollary}

The proof of the next Lemma is easy and left to the reader. 

\begin{Lemma}\label{generation-of-entire-functions}
Let $\phi$ be a Drinfeld module over $\CC_\infty$ and $\exp_\phi$ be its exponential function. 
Let $f:\KK_\Sigma\rightarrow\KK_\Sigma$ be a $\KK_\Sigma$-entire function. Then the composition 
$\exp_\phi\circ f$ is a $\KK_\Sigma$-entire function. Additionally, if $f(\LL_\Sigma)$ is contained in $\LL_\Sigma$, then the image of $\exp_\phi\circ f$ is contained in $\LL_\Sigma$. Finally, if 
$f(z)=\lambda z$ with $\lambda\in\KK_\Sigma$, then $\exp_\phi\circ f$ is a $\FF_q(\underline{t}_\Sigma)[\theta]$-module morphism $\KK_\Sigma\rightarrow\phi(\KK_\Sigma)$.
\end{Lemma}

\subsection{Some relevant functions associated with the Carlitz module}\label{Carlitz-exponential}

The functions mentioned in the title of the present subsection, and that will be described here, are required as basic tools to describe the analogues of Fourier series for our modular forms. 
One of the simplest examples of Drinfeld $A$-modules is the {\em Carlitz module} $C$. For the background on the Carlitz module,  together with its exponential, read \cite[Chapter 3]{GOS} and \cite[\S 4]{PEL5}. The Carlitz module $C(\KK_\Sigma)$ over $\KK_\Sigma$ is the 
$\FF_q(\underline{t}_\Sigma)$-algebra morphism 
$$A\otimes_{\FF_q}\FF_q(\underline{t}_\Sigma)\xrightarrow{C}\operatorname{End}_{\FF_q(\underline{t}_\Sigma)-\text{lin.}}(\KK_\Sigma)$$
defined by $C(\theta)=C_\theta=\theta+\tau$, the multiplication by $\theta$ (we prefer to adopt from now on the notation $C_a$ for the multiplication by $a\in A$ in $\phi(\KK_\Sigma)$). More generally,
$C$ can be viewed as a functor from the category of $\FF_q(\underline{t}_\Sigma)[\theta][\tau]$-modules to the category of $\FF_q(\underline{t}_\Sigma)[\theta]$-modules (with appropriate morphisms) so that we can define
$C(\LL_\Sigma),C(A),\ldots$ as well.
To describe the associated Carlitz exponential, we introduce the analogue of the sequence of numbers $q^n!$ in the following way:
$$d_n=\prod_a a,$$
where the product runs over the monic polynomials $a$ of $A$ of degree $n$. 
It can be proved  (see \cite[Proposition 3.1.6]{GOS}) that 
\begin{equation}\label{dn}d_n=\Big(\theta^{q^n}-\theta\Big)\cdots\Big(\theta^{q^n}-\theta^{q^{n-1}}\Big),\quad n\geq 0.\end{equation}
Then, $d_n\in K^\times$ for all $n$ and the map $\exp_C:\KK_\Sigma\rightarrow\KK_\Sigma$
defined by $$\exp_C(z)=\sum_{i\geq 0}d_i^{-1}\tau^i(z)$$
is the exponential function associated to the Carlitz module, which is a continuous, open $\FF_q(\underline{t}_\Sigma)$-linear endomorphism $\KK_\Sigma\rightarrow\KK_\Sigma$
to which we can apply Proposition \ref{propositionontatealgebras-ksigma} and Corollary \ref{propositionontatealgebras}. 
In particular, the kernel of $\exp_C$ (over $\LL_\Sigma$ or $\KK_\Sigma$) is equal to 
$\widetilde{\pi}\FF_q(\underline{t}_\Sigma)[\theta]$ where
\begin{equation}\label{pitilde}
\widetilde{\pi}=\theta(-\theta)^{\frac{1}{q-1}}\prod_{i=1}^\infty\Big(1-\theta^{1-q^i}\Big)^{-1},
\end{equation}
which belongs to $K_\infty((-\theta)^{\frac{1}{q-1}})\setminus K_\infty$ (we make a choice of a $(q-1)$-th root of $-\theta$, and we note that $(-\theta)^{\frac{1}{q-1}}=\exp_C(\widetilde{\pi}\theta^{-1})$).
From this product expansion one immediately sees that $|\widetilde{\pi}|=|(-\theta)^{\frac{1}{q-1}}|=|\theta|^{\frac{q}{q-1}}$.
It can be proved that $\widetilde{\pi}$ is transcendental over $K$; there are several ways that lead to this result, using the above product expansion. See \cite{PEL&BOU,PEL5} for an overview.

Occasionally,
we use the notation $\exp_C$ for the {\em Carlitz exponential operator} which is formal series $\sum_{n\geq 0}d_n^{-1}\tau^n\in K[[\tau]]\subset\KK_\Sigma[[\tau]],$ the unique one
such that the first term for $n=0$ is $1=\tau^0$ (normalized), satisfying, for the product rule of $\KK_\Sigma[[\tau]]$, $C_\theta\exp_C=\exp_C\theta$. 

The inverse of the Carlitz exponential $\exp_C$ for the composition is the {\em Carlitz logarithm} 
defined by the locally convergent series
$$\log_C(z)=\sum_{i\geq 0}l_i^{-1}\tau^i(z),$$
where $l_n$ is equal to $(-1)^n$ times the monic least common multiple of all polynomials of $A$ of degree $n$. It can be proved (see again \cite[Proposition 3.1.6]{GOS}) that 
\begin{equation}\label{ln}
l_n=\Big(\theta-\theta^q\Big)\cdots\Big(\theta-\theta^{q^n}\Big).\end{equation}
More precisely, $\log_C$ induces an isometric $\FF_q(\underline{t}_\Sigma)$-linear automorphism
$$D_{F}^\circ(0,|\widetilde{\pi}|)\xrightarrow{\log_C} D_{F}^\circ(0,|\widetilde{\pi}|),$$
where $F=\LL_\Sigma$ or $F=\LL_\Sigma$, and similar properties occur with $F=\TT_\Sigma[\delta]$ with $\delta\in \FF_q(\underline{t}_\Sigma)$ etc. We also identify, sometimes, $\log_C$ with the {\em Carlitz logarithm operator} $\sum_{n\geq 0}l_n^{-1}\tau^n\in K[[\tau]].$

\subsubsection{Omega matrices}\label{omegamatrices} We need certain matrix-valued
maps. Let $$A\xrightarrow{\chi} \FF_q(\underline{t}_\Sigma)^{n\times n}$$ be an injective $\FF_q$-algebra morphism.
We set
$$\vartheta:=\chi(\theta)\in \FF_q(\underline{t}_\Sigma)^{n\times n}.$$ Let $d\in\FF_q[\underline{t}_\Sigma]\setminus\{0\}$ be such that $d\vartheta\in\FF_q[\underline{t}_\Sigma]^{n\times n}$. Then, the image of $\chi$ lies in $\FF_q[\underline{t}_\Sigma][\frac{1}{d}]^{n\times n}$. We set 
$$\omega_\chi:=\sum_{i\geq 0}\exp_C\left(\frac{\widetilde{\pi}}{\theta^{i+1}}\right)\vartheta^i=\exp_C\Big(\widetilde{\pi}(\theta I_n-\vartheta)^{-1}\Big)\in\widehat{\TT_\Sigma[d^{-1}]}^{n\times n}\subset\LL_\Sigma^{n\times n},$$
where the map $\exp_C$ is applied coefficientwise on the entries of the matrix $\widetilde{\pi}(\theta I_n-\vartheta)^{-1}\in\KK_\Sigma^{n\times n}$. We have, for all $a\in A$, with $C_a\in K[\tau]^{n\times n}$ the multiplication by $a$ over $C(\KK_\Sigma)^{n\times n}$: 
\begin{equation}\label{Caomega}
C_a(\omega_\chi)=\exp_C\Big(\widetilde{\pi}a(\theta I_n-\vartheta)^{-1}\Big)=
\exp_C\Big(\widetilde{\pi}(a I_n-\chi(a))(\theta I_n-\vartheta)^{-1}\Big)+\chi(a)\omega_\chi=\chi(a)\omega_\chi,\end{equation}
because $a I_n-\chi(a)=(\theta I_n-\vartheta)H$ with $H\in A[\vartheta]^{n\times n}$, so that
$\widetilde{\pi}(a I_n-\chi(a))\in\operatorname{Ker}(\exp_C)^{n\times n}$.
\begin{Lemma}\label{lemmaomegachi}
We have $\omega_\chi\in \operatorname{GL}_n(\TT_\Sigma[\frac{1}{d}]^\wedge)$ and $\omega_\chi$ is solution of the linear $\tau$-difference system $$\tau(X)=(\vartheta-\theta I_n)X.$$ Moreover, every solution $X$ in $\KK_\Sigma^{n\times 1}$ of this difference system is of the form $X=\omega_\chi m$, with $m\in \FF_q(\underline{t}_\Sigma)^{n\times 1}$.
\end{Lemma}
\begin{proof}
observe that 
$$\omega_\chi=\exp_C\Big(\widetilde{\pi}(\theta I_n-\vartheta)^{-1}\Big)=\exp_C\Big(\widetilde{\pi}\theta^{-1}(I_n-\vartheta\theta^{-1})^{-1}\Big)=\exp_C(\widetilde{\pi}\theta^{-1})I_n+R$$ where
$R\in\KK_\Sigma^{n\times n}$ is such that $\|R\|<|\theta|^{\frac{1}{q-1}}=|\widetilde{\pi}\theta^{-1}|=|\exp_C(\widetilde{\pi}\theta^{-1})|$. This proves that $\omega_\chi\in \operatorname{GL}_n(\TT_\Sigma[\frac{1}{d}]^\wedge)$. The fact that $\omega_\chi$ is a matrix solution of the system indicated above follows directly from (\ref{Caomega}) with $a=\theta$. Finally, if $X$ is a column solution of the system above, we have that $\omega_\chi^{-1}X$ has entries in the constant subfield of $\KK_\Sigma$ which is $\FF_q(\underline{t}_\Sigma)=\FF_q(\underline{t}_\Sigma)$, and this proves the last assertion. 
\end{proof}

We denote by $\EE_\Sigma[\frac{1}{d}]^\wedge$ the $\CC_\infty$-algebra generated 
by all the series 
$$\sum_{i\geq 0}f_id^{-i},\quad f_i\in\EE_\Sigma,$$
with the property that $\|f_i\|r^i\rightarrow0$ for all $r\in|\CC_\infty|$. We have the next:

\begin{Corollary}\label{coroomegachi}
We have the identity 
$$\omega_\chi=(-\theta)^{\frac{1}{q-1}}\prod_{i\geq 0}\left(I_n-\vartheta\theta^{-q^i}\right)^{-1},$$
up to the choice of an appropriate root $(-\theta)^{\frac{1}{q-1}}$. Moreover, $\omega_\chi^{-1}\in\GL_n(\TT_\Sigma[\frac{1}{d}]^\wedge)\cap(\EE_\Sigma[\frac{1}{d}]^\wedge)^{n\times n}$.
\end{Corollary}
Note that the factors of the infinite product commute each other. 
\begin{proof}[Proof of Corollary \ref{coroomegachi}] First of all note that 
$$F:=(-\theta)^{\frac{1}{q-1}}\prod_{i\geq 0}\left(I_n-\vartheta\theta^{-q^i}\right)^{-1}\in(\TT_\Sigma[d^{-1}]^\wedge)^{n\times n}$$
is a matrix $X$ solution of the difference system $\tau(X)=(\vartheta-\theta I_n)X$, in $\GL_n(\KK_\Sigma)$. Lemma \ref{lemmaomegachi} applies and there exists a matrix $V\in\GL_n(\FF_q(\underline{t}_\Sigma))$ such that $F=V\omega_\chi$. 
Now we proceed to prove that $V=I_n$. We recall that $(-\theta)^{\frac{1}{q-1}}=\exp_C(\frac{\widetilde{\pi}}{\theta})$ for a unique choice of $(-\theta)^{\frac{1}{q-1}}$.
We have seen, in the proof of Lemma \ref{lemmaomegachi}, that  
$\omega_\chi=\exp_C(\widetilde{\pi}\theta^{-1})I_n+R$ where $R\in\KK_\Sigma^{n\times n}$ is such that $\|R\|<|\theta|^{\frac{1}{q-1}}$. We also have 
$F=\exp_C(\widetilde{\pi}\theta^{-1})I_n+R'$, $R'\in\KK_\Sigma$ such that $\|R'\|<q^{\frac{1}{q-1}}$. Hence $V=I_n$.
Additionally, note that $\widetilde{\pi}(\theta I_n-\vartheta)^{-1}\in(\TT_\Sigma[\frac{1}{d}]^\wedge)^{n\times n}$ so that, by Corollary \ref{propositionontatealgebras}, 
$\omega_\chi$ has entries in $\TT_\Sigma[\frac{1}{d}]^\wedge$. Also, $F$ in this case is an element of $\GL_n(\TT_\Sigma[\frac{1}{d}]^\wedge)$. Writing $\vartheta=d^{-1}\nu$ with $\nu\in\FF_q[\underline{t}_\Sigma]^{n\times n}$, we see that 
$F=\sum_{i\geq 0}c_i\nu^id^{-i}$ with $c_i\in \CC_\infty$ such that $|c_i|r^i\rightarrow0$ for all $r\in|\CC_\infty|$. But then $c_i\nu^i\in\CC_\infty[\underline{t}_\Sigma]^{n\times n}$ with 
$\|c_i\nu^i\|r^i\rightarrow 0$ and therefore, the entries of $\omega_\chi^{-1}$ belong to 
$\EE_\Sigma[\frac{1}{d}]^\wedge$.
\end{proof}
 
\subsubsection{A class of entire functions}\label{class-of-entire-functions}

We recall that we have set $\vartheta=\chi(\theta)\in\FF_q(\underline{t}_\Sigma)^{n\times n}$ and that $d\in\FF_q[\underline{t}_\Sigma]\setminus\{0\}$ is such that 
$d\vartheta\in\FF_q[\underline{t}_\Sigma]$.
For $z\in\CC_\infty$, we set (\footnote{Note that the factors commute.}):
\begin{equation}\label{chidefi}
\widetilde{\chi}(z):=\exp_C\Big(\widetilde{\pi}z(\theta I_n-\vartheta)^{-1}\Big)\omega_\chi^{-1},\end{equation}
where $\omega_\chi\in\operatorname{GL}_n(\TT_\Sigma[\frac{1}{d}]^\wedge)$ has been introduced in \S \ref{omegamatrices}. By Lemma \ref{generation-of-entire-functions}, this is 
an entire function in $\operatorname{Hol}_{\TT_\Sigma[\frac{1}{d}]^\wedge}(\CC_\infty\rightarrow\TT_\Sigma[\frac{1}{d}]^\wedge)^{n\times n}$. 
We now use the material developed in this section to show the following
(compare with \cite[Lemmas 15, 17]{PEL&PER}).

\begin{Proposition}\label{firstpropertieschi}
The function $\widetilde{\chi}$ satisfies the following properties.
\begin{enumerate}
\item It has image in $(\EE_\Sigma[\frac{1}{d}]^\wedge)^{n\times n}$,
\item it satisfies $\widetilde{\chi}(a)=\chi(a)$ for all $a\in A$,
\item it satisfies the $\tau$-difference system
$\tau(X)=X+\exp_C(\widetilde{\pi}z)\omega_\chi^{-1}.$
\end{enumerate}
\end{Proposition}

\begin{proof}
(1) Since
$$\exp_C\Big(\widetilde{\pi}z(\theta I_n-\vartheta)^{-1}\Big)=\sum_{i\geq 0}d_i^{-1}(\widetilde{\pi}z)^{q^i}(\theta^{q^i} I_n-\vartheta)^{-1},\quad z\in\CC_\infty$$
and $\|d_i^{-1}(\theta^{q^i} I_n-\vartheta)^{-1}\|=|\theta|^{(i-1)q^i}$ for all $i\geq 0$,
the image of the map $\widetilde{\chi}$ is contained in $(\EE_\Sigma[\frac{1}{d}]^\wedge)^{n\times n}$ (we recall from Corollary \ref{coroomegachi} that $\omega_\chi^{-1}$ has entries in $\EE_\Sigma[\frac{1}{d}]^\wedge$). (2)
Observe that 
if $a\in A$,
\begin{eqnarray*}
\widetilde{\chi}(a)&=&\exp_C\Big(\widetilde{\pi}a(\theta I_n-\vartheta)^{-1}\Big)\omega_\chi^{-1}\\
&=&C_a(\omega_\chi)\omega_\chi^{-1}\\
&=&\chi(a).
\end{eqnarray*}
(3) We set $F=\exp_C\left(\widetilde{\pi}z(\theta I_n-\vartheta)^{-1}\right)$. Then,
\begin{eqnarray*}
\tau(F)&=&- \theta F +\exp_C\Big(\widetilde{\pi}z(\theta I_n-\vartheta+\vartheta)(\theta I_n-\vartheta)^{-1}\Big)\\
&=&-\theta F+\exp_C\Big(\widetilde{\pi}(z(\theta I_n-\vartheta)^{-1})\vartheta\Big)+\exp_C(\widetilde{\pi}z)I_n\\
&=&F\cdot(\vartheta-\theta I_n)+\exp_C(\widetilde{\pi}z)I_n.
\end{eqnarray*}
 \end{proof}
 
 From now on, we will 
denote both maps, $A\xrightarrow{\chi}\FF_q(\underline{t}_\Sigma)^{n\times n}$ and
$\CC_\infty\xrightarrow{\widetilde{\chi}}\KK_\Sigma^{n\times n}$, with $\chi$ to simplify our notations.
 
 \subsubsection{An example with $n=1$}\label{a-class-chi}
 
 We consider, to illustrate an example, the above picture in the case when 
 $\chi=\chi_t$, where $\chi_t$ is the unique $\FF_q$-algebra map $$A\xrightarrow{\chi_t}\FF_q[t]$$
defined by $\theta\mapsto t$ (therefore, $n=1$). In this case $\omega_\chi$ is the 
 {\em function of Anderson and Thakur} $\omega$. It is likely that this function appeared for the first time in the literature in the paper of Anderson and Thakur \cite[Proof of Lemma 2.5.4 p. 177]{AND&THA}. 
We have:
$$\omega(t)=\exp_C\left(\frac{\widetilde{\pi}}{\theta-t}\right).$$
Corollary \ref{coroomegachi} implies that
\begin{equation}\label{definitionomega}
\omega(t)=(-\theta)^{\frac{1}{q-1}}\prod_{i\geq0}\left(1-\frac{t}{\theta^{q^{i}}}\right)^{-1}\in\TT^\times,\end{equation} for a fixed choice of the $(q-1)$-th root, and its inverse is an entire function in $\EE$. The element $\omega$ can be also viewed as a function of the variable $t\in\CC_\infty$, because the infinite product 
converges for all $$t\in\CC_\infty\setminus\{\theta^{q^k};k\geq 0\}$$ and defines a meromorphic function over the above set, with simple poles at $\theta^{q^k}$,
$k\geq 0$.
The element $\omega$ is a $(\theta-t)$-torsion point in the Carlitz $A[t]$-module $C(\TT)$.
In particular, $\omega$ is a generator of the free sub-$\FF_q[t]$-module of rank one of $\TT$, kernel of the evaluation of the operator $$C_{\theta-t}=\tau+\theta-t\in K[t][\tau],$$
so that $\omega$ is a solution of the linear homogeneous $\tau$-difference equation of order $1$ (see also \cite[Proposition 3.3.6]{PAP0}):
\begin{equation}\label{tauomega}\tau(\omega)(t)=(t-\theta)\omega(t).\end{equation}
All these properties easily follow from Corollary \ref{coroomegachi}.

For the function $\chi_t:\CC_\infty\rightarrow\TT$ we note that explicitly, $$\chi_t(z):=\frac{\exp_C\left(\frac{\widetilde{\pi}z}{\theta-t}\right)}{\omega(t)},\quad z\in \CC_\infty.$$ 
We deduce that $\chi_t$ defines an entire function $\CC_\infty\rightarrow\EE$ which satisfies $\chi_t(a)=a(t)$ for all $a\in A$, and the $\tau$-difference equation
\begin{equation}\label{taudifferechit}
\tau(\chi_t(z))=\chi_t(z)+\frac{\exp_C(\widetilde{\pi}z)}{\tau(\omega)}.
\end{equation}
 To mention an additional property of the entire function $\chi_t$, it can be proved that the function $z\mapsto\frac{\chi_t(z)}{z}\in\EE$ is non-constant, entire, without zeroes.  

\section{Field of uniformizers}\label{tameseriestheory}

The crucial feature of the modular forms we study in the present text is that their entries can be identified with certain formal series generalizing the Fourier series of classical Drinfeld modular forms $f:\Omega\rightarrow\CC_\infty$ for $\Gamma$. These formal series can be seen as elements of the {\em field of uniformizers} $\mathfrak{K}$ (Definition \ref{definition-field-uniformizers}) which provides a natural environment to do computations 
and to prove our results. Roughly speaking, if $f:\Omega\rightarrow\KK_\Sigma^{N\times 1}$
is a modular form of weight $w$ for a representation of the first kind $\rho$, then the entries of 
$f$ can be viewed as elements of an algebraically closed field of generalized formal series in the sense of Kedlaya \cite{KED}, containing the valued field $\KK_\Sigma((u))$
with $u$ the uniformiser defined in (\ref{Goss-uniformiser}). We need to be a bit more precise however, as in practice, these series span a much smaller field and in the sequel, we need to gain a certain control on their coefficients. The main results in this section are Propositions \ref{descriptionfieldofunif-bis} and \ref{descriptionfieldofunif} where the reader can find an explicit description of the elements of $\mathfrak{K}$ as formal Laurent series with coefficients which are 
{\em tame series}, certain entire functions defined in \S \ref{tameseriessect}. Similar constructions have also been considered in \cite{PEL4}. We begin with \S \ref{algebrasettings}, where we introduce some algebraic settings.

\subsection{Some algebras and fields}\label{algebrasettings}
In this subsection, we consider an integral commutative $A$-algebra $B$ with the structure induced by a morphism 
$$A\xrightarrow{\iota}B.$$
Additionally, we suppose that $B$ is endowed with an $\FF_q$-algebra endomorphism $\tau$ which acts as the map $c\mapsto c^q$ over $\iota(A)$ so that  $(B,\tau)$ is a difference ring. We set $$\Theta=\iota(\theta).$$ In the paper, we are going to use this ring $B$ mainly in the case of $\iota$ injective. In this case, we identify $\Theta$ with $\theta$ but in the first general discussions, we prefer to keep $\Theta$ and $\theta$ distinct. 

We consider, further, the polynomial $B$-algebra $$\mathcal{R}=B[X_i;i\in\ZZ]$$ in infinitely many variables $X_i$, and the ideal $\mathcal{P}$ generated by the polynomials $$X_i^q+\Theta X_i-X_{i-1},\quad i\in\ZZ.$$
Then, with $\underline{X}$ the collection $(X_i:i\in\ZZ)$ and $j=\sum_{i\in\ZZ}j_iq^{-i}\in\ZZ[\frac{1}{p}]_{\geq 0}$ expanded in base $q$ (so that only finitely many terms occur),
we set 
$$\Angle{X}^j=\prod_{i\in\ZZ}X_i^{j_i}\in\mathcal{R}/\mathcal{P}.$$
The quotient $B$-algebra $\mathcal{R}/\mathcal{P}$ can be identified with the ring 
$\Ring{B}{X}$ whose elements $F$ are formal {\em finite} sums in the indeterminates $X_i$, $i\in\ZZ$:
\begin{equation}\label{prototypeF}
F=\sum_{j\in\ZZ[\frac{1}{p}]_{\geq 0}}F_j\Angle{X}^j=\sum_{j\in\ZZ[\frac{1}{p}]_{\geq 0}}F_j\prod_{k\in\ZZ}X_k^{j_k},\quad F_j\in B,\end{equation}
where we have expanded the indices $j=\sum_{k\in\ZZ}j_kq^{-k}$ in base $q$ (the coefficients $j_i$ are almost all zero and belong to $\{0,\ldots,q-1\}$). Note that a product over $\Ring{B}{X}$ is well defined in virtue of the rules $X_i^q=X_{i-1}-\Theta X_i$. 
We have thus identified, after a mild abuse of notation, $\Ring{B}{X}$ with a 
complete system of representatives of $\mathcal{R}$ modulo $\mathcal{P}$ and we have defined over it, a product which makes it isomorphic to the quotient $\mathcal{R}/\mathcal{P}$.

\subsubsection*{Examples} If $B=A$ and $\iota$ is the identity, since the multiplication by $\theta$ of the Carlitz $A$-module is given by $C_\theta={\theta}+\tau$, we have $X_{i-1}=C_\theta(X_i)$ in $C(\Ring{B}{X})$. If $B=\CC_\infty$ and $\iota$ is the inclusion $A\subset\CC_\infty$, the substitution $X_i\mapsto e_C(\frac{z}{\theta^i})$, where $e_C$ is defined by
\begin{equation*}\label{def-eC}e_C(z)=\exp_C(\widetilde{\pi}z),\end{equation*}
yields a $\CC_\infty$-algebra homomorphism $$\Ring{\CC_\infty}{X}\rightarrow\operatorname{Map}\left(K\rightarrow \CC_\infty\right).$$ 

We come back to the general settings of this \S \ref{algebrasettings}.
We define a map $$\Ring{B}{X}\xrightarrow{v}\ZZ[p^{-1}]_{\leq 0}\cup\{\infty\}$$ in the following way. We define $v(0):=\infty$ and we set $v(B\setminus\{0\})=\{0\}$. Further, for a monomial $\Angle{X}^j=\prod_{i\in\ZZ}X_i^{j_i}$ (so only finitely many factors satisfy $j_i>0$), we set
$v(\Angle{X}^j)=-j$. Note that distinct monomials $\Angle{X}^j$ correspond to distinct values in $\ZZ[\frac{1}{p}]_{\leq 0}$ so that $v$ is injective over $\{\Angle{X}^j:j\in\ZZ[\frac{1}{p}]_{\geq 0}\}$. If $F$ is non-zero as in (\ref{prototypeF}), then we set
$$v(F)=\inf\{v(\Angle{X}^j):F_j\neq0\};$$ the infimum is a minimum. 
\begin{Lemma}\label{valuation-BB}
The map $v$ is an additive valuation.
\end{Lemma}

\begin{proof}
With $j,k\in\ZZ[\frac{1}{p}]_{\geq 0}$ and by the definition of the ideal $\mathcal{P}$,
$\Angle{X}^j\Angle{X}^k=\Angle{X}^{j+k}+F$ where $F\in \Ring{B}{X}$
satisfies $v(F)>v(\Angle{X}^{j+k})$, so that if $F,G\in \Ring{B}{X}$, $v(FG)=v(F)+v(G)$.
\end{proof}

\begin{Remark}\label{no-carry-over}
{\em Note that in general,
$$\Angle{X}^i\Angle{X}^j\neq\Angle{X}^{i+j},\quad i,j\in\ZZ[p^{-1}]_{\geq 0}.$$
The equality holds if there is no base-$q$ carry over in the sum $i+j$. For example, the reader can 
verify the formula:
\begin{equation}\label{aformula1}
\Angle{X}^{(q-1)(\frac{1}{q}+\cdots+\frac{1}{q^n})}\Angle{X}^{\frac{1}{q^n}}=\Angle{X}^1-{\Theta}\sum_{i=0}^{n-1}\Angle{X}^{(q-1)(\frac{1}{q}+\cdots+\frac{1}{q^i})}\Angle{X}^{\frac{1}{q^{i+1}}},\quad \forall n\geq 1.\end{equation}
}
\end{Remark}

Since $\Ring{B}{X}$ is a valued ring by Lemma \ref{valuation-BB}, it is integral and we deduce that $\mathcal{P}$ is a prime ideal. The residual ring of $\Ring{B}{X}$ is $B$. Further, defining $$\tau(X_i)=X_i^q\equiv X_{i-1}-{\Theta}X_i\pmod{\mathcal{P}}$$ induces an endomorphism of $\Ring{B}{X}$ and the subring $\Ring{B}{X}^{\tau=1}$ of the elements $F$ such that $\tau(F)=F$ is equal to $B^{\tau=1}$.
Note that even in the case of $(B,\tau)$ inversive, $\tau$ does not extend to an automorphism of $\Ring{B}{X}$.

\subsubsection{The algebra $\Rring{B}{X}$}
We analyze another difference $B$-algebra containing $\Ring{B}{X}$ (it is not complete, but it is {\em inversive}, that is, $\tau$ induces an automorphism). 

\begin{Definition}\label{def-BX}
{\em We define
$\Rring{B}{X}$ to be the $B$-module of formal series as in (\ref{prototypeF}), without the condition of finiteness of the sums, and such that the following conditions hold: 
\begin{enumerate}
\item There exists $L\geq 0$ (depending on $F$) such that if $F_j\neq 0$, then $\ell_q(j)\leq L$, with $\ell_q(j)$ denoting the sum of digits of $j$ in base $q$ (which means that the length of the base-$q$ expansions of the exponents $j$ involved is bounded). 
\item If $F_j\neq0$, then
$j\geq M$ with a constant $M$ depending on $F$ (which means that only the variables $X_i$ with $i$ in a subset of $\ZZ$ which has a lower bound in $\ZZ$ occur).
\end{enumerate}
It is clear that there is an inclusion of $B$-modules $\Ring{B}{X}\subset
\Rring{B}{X}$. 
The first condition also means that the number of factors of the monomials occurring in $F\in \Rring{B}{X}$ is bounded. If $F\in \Rring{B}{X}$ is non-zero, we call {\em depth} of $F$ the smallest integer $L$ satisfying the condition (1) above. We denote it by $\lambda(F)$.} \end{Definition}
We have that $F\in \Ring{B}{X}\setminus\{0\}$ has depth $0$ if and only if $F\in B$.

\subsubsection{Product in $\Rring{B}{X}$}\label{product-in-B}

Let $F\in \Ring{B}{X}\setminus\{0\}$ be as in (\ref{prototypeF}). We denote by $\mu(F)$ the largest $m\in\ZZ$ such that the variable $X_m$ occurs in at least one non-zero monomial of $F$ (remember that the elements of $\Ring{B}{X}$ are polynomials so that $\mu(F)$ is well defined). Similarly, we denote by $\nu(F)$ the smallest $n\in\ZZ$ such that the variable $X_n$ occurs in at least one non-zero monomial of $F$.  Clearly, the function $\mu$ dominates the function $\nu$ over $\Ring{B}{X}$ (in the natural ordering of $\ZZ$)
and $\nu$ is also defined on $\Rring{B}{X}$. The proof of the next result is elementary and left to the reader.

\begin{Lemma}\label{lemmanumu}
For two monomials $\Angle{X}^i$ and $\Angle{X}^j$ in $\Ring{B}{X}$ the following properties hold:
\begin{enumerate}
\item $\lambda(\Angle{X}^i\Angle{X}^j)\leq\lambda(\Angle{X}^i)+\lambda(\Angle{X}^i)$,
\item $\mu(\Angle{X}^i\Angle{X}^j)=\max\{\mu(\Angle{X}^i),\mu(\Angle{X}^j)\}$,
\item $\nu(\Angle{X}^i\Angle{X}^j)\in\{\min\{\nu(\Angle{X}^i),\nu(\Angle{X}^j)\},\min\{\nu(\Angle{X}^i),\nu(\Angle{X}^j)\}-1\}.$
\end{enumerate}
\end{Lemma}

We are ready to show the next result:
\begin{Proposition}\label{B-alg-structure}
The $B$-module $\Rring{B}{X}$ is endowed with the structure of a difference $B$-algebra with endomorphism $\tau$, extending that of the difference algebra $(\Ring{B}{X},\tau)$. This difference algebra $(\Rring{B}{X},\tau)$ carries a unique extension of the valuation $v$ for which the residual ring is $B$, and for all $F\in \Rring{B}{X}$, $v(\tau(F))=qv(F)$. Furthermore, if $(B,\tau)$ is inversive, then $(\Rring{B}{X},\tau)$ is inversive and if $F\in \Rring{B}{X}$ there exists $G\in \Rring{B}{X}$ such that $\tau(G)-G=F$.
\end{Proposition}

\begin{proof}
Let $F$ be an element of $\Rring{B}{X}\setminus\{0\}$. By the condition (1) of Definition \ref{def-BX}, there exists $h\in\NN^*$ and $j_1,\ldots,j_h\in\ZZ[\frac{1}{p}]_{\geq 0}$ such that
\begin{equation}\label{a-special-type}
F=\sum_{n=1}^hF_n,\end{equation}
where $F_n\in \Rring{B}{X}\setminus\{0\}$ is of the form
$$F_n=\sum_{i\geq 0}F_{i,n}\Angle{X}^{p^{-i}j_n},\quad F_{i,n}\in B.$$
Consider two elements $r,s\in \ZZ[\frac{1}{p}]_{\geq 0}$ and $i,j\in\ZZ$. By Lemma \ref{lemmanumu}, $\lambda(\Angle{X}^{p^{-i}r}\Angle{X}^{p^{-j}s})$ is uniformly bounded in the dependence of $i,j$. This means that the product of two series as in (\ref{a-special-type}) defines an element of $\Rring{B}{X}$ and therefore, 
$\Rring{B}{X}$ carries a structure of $B$-algebra,
extending the structure of $\Ring{B}{X}$. Also the fact that 
the valuation $v$ extends to a valuation $\Rring{B}{X}$, with same image, is easily verified.

Assuming now that $(B,\tau)$ is inversive, we observe that $Y_j:=\sum_{i\geq j}{\Theta}^{\frac{i-j}{q}}X_{i+1}\in \Rring{B}{X}$ for all $j\in\ZZ$, with ${\Theta}^{\frac{1}{q}}$ the $q$-th root of ${\theta}$ in $B$ (which is inversive by hypothesis, and therefore contains $\{x^{\frac{1}{p}}:x\in\iota(A)\}$), satisfies 
$Y_j^q=X_j$. Indeed, 
\begin{eqnarray*}
Y_j^q &=&\sum_{i\geq j}\Theta^{i-j}X_{i+1}^q\\
&=&\sum_{i\geq j}\Theta^{i-j}(X_i-\Theta X_{i+1})\\
&=&X_j.
\end{eqnarray*}
Therefore, inductively, if we set:
\begin{equation}\label{seriesF}
Y_{j,r}:=\sum_{i_1>\cdots>i_r\geq j}\Theta^{\frac{i_1-i_2-1}{q}+\frac{i_2-i_3}{q^2}\cdots+\frac{i_r-j}{q^r}}X_{i_1+1},
\end{equation}
then $Y_{j,r}^{q^r}=X_j$ for all $r\geq 0$ and $j\in\ZZ$.
Now consider 
$F=\sum_{j\in\ZZ[\frac{1}{p}]_{\geq 0}}F_j\Angle{X}^j\in \Rring{B}{X}$, with $\lambda(F)\leq L$.
Let $j$ be such that $F_j\neq0$ and write $$\langle\underline{Y}_r\rangle^j=\prod_{k\in \ZZ}Y_{k,r}^{j_k}.$$ Note that this is a well defined element of $\Rring{B}{X}$, due to the fact that $\lambda(Y_{j,r})=1$ for all $r,j$. It is therefore easy to see that the formal series 
$$g_r=\sum_{j\in\ZZ[\frac{1}{p}]_{\geq 0}}\tau^{-r}(F_j)\langle \underline{Y}_r\rangle^j$$
defines an element of $\Rring{B}{X}$ and $\tau^r(g_r)=f$. The last assertion of the proposition follows easily from the fact that if $F\in \Rring{B}{X}$, then
$$G=\sum_{i\geq 0}\tau^{-i}(F)$$
defines an element of $\Rring{B}{X}$ satisfying $\tau(G)-G=F$.
\end{proof}

\subsubsection{Depth homogeneity}

We denote by $\Rring{B}{X}_s$ the $B$-submodule of $\Rring{B}{X}$ whose elements
are the formal series $F$ as in (\ref{prototypeF}) such that 
if $F_j\neq0$, then $\Angle{X}^j$ has depth equal to $s$, i.e. $\ell_q(j)=s$.
It is easy to see that 
\begin{equation}\label{grading-T}
\Rring{B}{X}=\bigoplus_{s\geq 0}\Rring{B}{X}_s
\end{equation}
 as a $B$-module. If $F\in\Rring{B}{X}$, we
can expand in finite sum and in a unique way 
\begin{equation}\label{decomposition-f-T}
F=\sum_{s\geq 0}F^{[s]},\end{equation} where $F^{[s]}\in\Rring{B}{X}_s$. 
Moreover, we have the next Lemma, not used in the present text, the proof of which is left to the reader.

\begin{Lemma} For any $s\geq 0$, $\tau$ induces an endomorphism of the $B$-module $\Rring{B}{X}_s$.
\end{Lemma}

\begin{Remark}{\em 
The $B$-algebra $\Rring{B}{X}$ is not graded by the depths. Instead, we have that
$$\Rring{B}{X}_s\Rring{B}{X}_{s'}\subset\bigoplus_{j\geq 0}\Rring{B}{X}_{s+s'-j(q-1)},$$
where we set $\Rring{B}{X}_s=\{0\}$ if $s<0$.}\end{Remark}

\subsubsection{The case of $B$ a field}\label{case-B-field}

We suppose here that $B=L$ is a field together with an embedding $A\rightarrow L$. We write $$\Rringcirc{L}{X}=\{F\in \Rring{L}{X}:\nu(F)\geq 1\}\cup L.$$ Only the variables $X_1,X_2,\ldots$ occur in the series defining 
$\Rringcirc{L}{X}$.
This is an $L$-vector space and we have $$v(\Rringcirc{L}{X})=\ZZ[p^{-1}]\cap]-1,0].$$ In particular, $\Rringcirc{L}{X}$ is not a ring.
We set $$X:=X_0.$$ One sees that for any $f\in \Rring{L}{X}$ there exist
$n$ and $f_0,\ldots,f_n\in \Rringcirc{L}{X}$ such that
$f=f_0+f_1X+\cdots+f_nX^n$, and this expression is unique. We can write, loosely:
$$\Rring{L}{X}=\Rringcirc{L}{X}[X].$$

We now consider $\widehat{\operatorname{Frac}(\Rring{L}{X})}_v$, the completion for the valuation $v$ of the fraction field of $\Rring{L}{X}$. 
\begin{Proposition}\label{descriptionfieldofunif-bis}
Every element $f$ of $\widehat{\operatorname{Frac}(\Rring{L}{X})}_v$ can be expanded in a unique way as a sum
$$f=\sum_{i\geq i_0}f_iX^{-i},\quad f_i\in \Rringcirc{L}{X}.$$
\end{Proposition}

\begin{Remark}{\em Note that in the above expansion the depths of the coefficients $f_i$ may be unbounded in their dependence on $i$.}
\end{Remark}

We can write $\Rringcirc{L}{X}((X^{-1}))$ for the set of the formal series $f=\sum_{i\geq i_0}f_iX^{-i}$ as above, with $f_i\in \Rringcirc{L}{X}$ for all $i$, with the warning that this is not a field for the usual Cauchy product rule of formal series, since, as pointed out previously, $\Rringcirc{L}{X}$ is not a ring but just an $L$-vector space. The proposition tells us that this set in fact
carries a structure of complete field, and equals $\widehat{\operatorname{Frac}(\Rring{L}{X})}_v$, but the product rule is not the Cauchy's one. To prove the proposition we will need the next two Lemmas. The first one describes the valued ring structure of $\Rringcirc{L}{X}((X^{-1}))$.

\begin{Lemma}\label{lemmanaturalstructure}
The set $\Rringcirc{L}{X}((X^{-1}))$ has a structure of commutative ring with unit, over which the valuation $v$ extends in a unique way from $\Rringcirc{L}{X}$, and which is complete for it.
\end{Lemma}

\begin{proof}
Since $\Rringcirc{L}{X}$ is an $L$-vector space, in order to show that $\Rringcirc{L}{X}((X^{-1}))$ is a ring, all we need to do is to show that the product of $\Rring{L}{X}$ extends to a product structure
on $\Rringcirc{L}{X}((X^{-1}))$. Let $f=\sum_{i\geq i_0}f_iX^{-i}$ and $g=\sum_{j\geq j_0}g_jX^{-j}$ be two elements of $\Rringcirc{L}{X}((X^{-1}))$. We note that $h_k:=\sum_{i+j=k}f_ig_j\in \Rring{L}{X}$
has valuation in $]-2,0]\cup\{\infty\}$ and we can write $h_k=\alpha_k X+\beta_k$, with $\alpha_k,\beta_k\in \Rringcirc{L}{X}$. We define
$$h=fg=\sum_{k\geq k_0:=i_0+j_0}X^{-k}h_k=\sum_{k\geq k_0}\alpha_kX^{1-k}+\sum_{k\geq k_0}\beta_kX^{-k}\in \Rringcirc{L}{X}((X^{-1})).$$ From this, we obtain the required ring structure.
If $f=\sum_{i\geq i_0}f_iX^{-1}\in \Rringcirc{L}{X}((X^{-i}))$ is such that $f_{i_0}\in\Rringcirc{L}{X}\setminus\{0\}$, then we set $v(f):=v(f_{i_0})+i_0\in ]i_0-1,i_0]$ and it is plain that $v$ defines a valuation over the ring 
$\Rringcirc{L}{X}((X^{-1}))$ and that every such series of $\Rringcirc{L}{X}((X^{-1}))$ converges for this valuation.

Note that $f=\sum_if_iX^{-i}\in \Rringcirc{L}{X}((X^{-1}))$ is such that $v(f)>N$ where $N$ is characterised by the following condition: the smallest $i_0$
such that $f_{i_0}\neq0$ is such that $i_0\geq N+1$. This is meaningful, indeed, if $f_{i_0}\in \Rringcirc{L}{X}\setminus\{0\}$, $v(f_{i_0}X^{-i})\in]i_0-1,i_0]$. Thus, if $(F_k)_k$ is a Cauchy sequence of $\Rringcirc{L}{X}((X^{-1}))$, the sequence
$(F_0-F_k)_k=(\sum_{i=1}^k(F_{i-1}-F_i))_k$ converges to an element of $\Rringcirc{L}{X}((X^{-1}))$ which is then complete.
\end{proof}
We introduce the ring:
$$\Rringbullet{L}{X}:=\Rring{L}{X}\Big[(\Angle{X}^{j})^{-1}:j\in\ZZ[p^{-1}]_{\geq 0}\Big]=\Rring{L}{X}[X_i^{-1}:i\in\ZZ],$$ which contains $\Rring{L}{X}$. 
Every element $f$
of $\Rringbullet{L}{X}$ has a well defined valuation $v(f)$ in $\ZZ[\frac{1}{p}]$. To see this we note that for every $g\in \Rringbullet{L}{X}$, there exists $j\in\ZZ[\frac{1}{p}]_{\geq 0}$ such that $\Angle{X}^jg\in \Rring{L}{X}$ and this provides the unique extension of the valuation map over $\Rringbullet{L}{X}$.  

\begin{Lemma}\label{firstapproach}
We have $\Rringbullet{L}{X}\subset \Rringcirc{L}{X}((X^{-1}))$.
\end{Lemma}

\begin{proof}
If $n>0$ we have $X_{-n}^{-1}\in X^{-q^n}(1+X^{-1}A[[X^{-1}]])$ (recall that $C_a(e_C(z))=e_C(az)$ for all $a\in A$) and therefore, $X_{-n}^{-1}\in \Rringcirc{L}{X}((X^{-1}))$ for all $n>0$. Now, we show that $X_i^{-1}\in \Rringcirc{L}{X}((X^{-1}))$
for all $i\geq 0$. Since this is clear for $i=0$, let us assume (induction) that $X_0^{-1},\ldots,X_{i-1}^{-1}\in \Rringcirc{L}{X}((X^{-1}))$. We observe, in the fraction field of $\Rring{L}{X}$:
$$\frac{1}{X_i}=\frac{X_i^{q-1}}{X_i^q}=\frac{X_i^{q-1}}{X_{i-1}-\Theta X_i}=\frac{X_i^{q-1}}{X_{i-1}\left(1-\frac{\Theta X_i}{X_{i-1}}\right)}.$$
Since $v(\Theta X_i/X_{i-1})>0$, the series $\sum_{j\geq0}(\frac{\Theta X_{i}}{X_{i-1}})^j$
converges in $(\Rring{L}{X}[X_{i-1}^{-1}])^\wedge_v$ (completion at $v$) to an element $h$ such that $h(1-\frac{\Theta X_i}{X_{i-1}})=1$.
Now, we have $\Rring{L}{X}[X_{i-1}^{-1}]^\wedge_v\subset \Rringcirc{L}{X}((X^{-1}))$ by our induction hypothesis. Since 
$\Rringcirc{L}{X}((X^{-1}))$ is a ring, $\frac{1}{X_i}=X_i^{q-1}\cdot\frac{1}{X_{i-1}}\cdot h\in \Rringcirc{L}{X}((X^{-1}))$, and more generally, $\frac{1}{\Angle{X}^j}\in \Rringcirc{L}{X}((X^{-1}))$ for all $j$, and the lemma follows remembering that $\Rring{L}{X}=\Rringcirc{L}{X}[X]$.
\end{proof}

\begin{proof}[Proof of Proposition \ref{descriptionfieldofunif-bis}] 
We show that $\operatorname{Frac}(\Rring{L}{X})$ embeds in $(\Rringbullet{L}{X})^\wedge_v$ (completion). To see this, we only need to show that if $f\in \Rring{L}{X}$ is not proportional by an element of $L^\times$ to $\Angle{X}^j$ for some $j\in\ZZ[\frac{1}{p}]_{\geq 0}$, then there exists $g\in(\Rringbullet{L}{X})^\wedge_v$
such that $fg=1$. Now, write $f=\alpha \Angle{X}^j-h$ for some $j$, where $\alpha\in L^\times$ and where $h\in
\Rring{L}{X}$ is such that $v(h)>-j$. Then, the series $\sum_{i\geq 0}(\frac{h}{\alpha \Angle{X}^j})^i$ converges in $(\Rringbullet{L}{X})^\wedge_v$ and we can set $$g=\frac{1}{\alpha \Angle{X}^j}\sum_{i\geq 0}\left(\frac{h}{\alpha \Angle{X}^j}\right)^i\in \widehat{\Rringbullet{L}{X}}_v.$$
By Lemma \ref{firstapproach}, $(\operatorname{Frac}(\Rring{L}{X}))^\wedge_v\subset \Rringcirc{L}{X}((X^{-1}))$ which is complete. On the other hand, any series
$\sum_{i\geq i_0}f_iX^{-i}$ with the coefficients $f_i$ in $\Rringcirc{L}{X}$ converges (for $v$)
and the partial sums are elements of $\Rringbullet{L}{X}[X^{-1}]\subset(\operatorname{Frac}(\Rring{L}{X}))^\wedge_v$ from which we can conclude that
$(\operatorname{Frac}(\Rring{L}{X}))^\wedge_v=\Rringcirc{L}{X}((X^{-1}))$ and also, we note that in this way, $\Rringcirc{L}{X}((X^{-1}))$ carries the structure of a complete, valued field.  
\end{proof}

Note that the field $\Rringcirc{L}{X}((X^{-1}))$ has valuation ring $$L\oplus\widehat{\bigoplus_{i>0}\Rringcirc{L}{X}X^{-i}}$$
and maximal ideal 
$$\widehat{\bigoplus_{i>0}\Rringcirc{L}{X}X^{-i}}.$$
The residual field is $L$.
\subsubsection{The case of $B$ inversive}\label{case-b-inversible}

We suppose here that, in addition to the hypotheses of \S \ref{case-B-field}, $B=L$ is an inversive field
containing $\iota(A)$ such that $\tau(x)=x^q$
for $x\in\iota(A)$. By Proposition \ref{B-alg-structure}, $\Rring{L}{X}$
is inversive. We give some complements on the structure of $\Rring{L}{X}$.

We consider the following set of {\em generalized formal series} in the sense of Kedlaya, \cite{KED}:
$$L^\circ\{\!\{\underline{X}\}\!\}=\left\{f=\sum_{i\in\ZZ[p^{-1}]_{\geq 0}}f_iX^{-i}:f_i\in L\text{ and there exists }c\geq 0\text{ such that }\ell_p(i)\leq c\right\},$$ where $\ell_p(\cdot)$ denotes the sum of the digits in 
the base-$p$ expansion of an integer. Equivalently, $L^\circ\{\!\{\underline{X}\}\!\}$ can be 
described as the set of all the generalized formal series in the indeterminate $t=X^{-1}$ which are supported by the sets
$S_{a,b,c}$ of \cite[\S 3]{KED} with $a=1$, $b=0$ and $c\geq 0$. 

\begin{Lemma}
$L^\circ\{\!\{\underline{X}\}\!\}=\Rringcirc{L}{X}$.
\end{Lemma}

\begin{proof}
It follows easily from the proof of Proposition \ref{B-alg-structure}.
\end{proof}

\begin{Corollary}
If $L$ is perfect
the completion of the fraction field of $\Rring{L}{X}$ for the valuation $v$ is perfect and has no non-trivial Artin-Schreier extensions.
\end{Corollary}

\begin{proof}
The only property that we need to show is that if $f\in \Rringcirc{L}{X}((X^{-1}))$
then there exists $g\in \Rringcirc{L}{X}((X^{-1}))$ such that $g^p-g=f$. But this is immediate from Proposition \ref{descriptionfieldofunif-bis}, the direct sum decomposition of $L$-vector spaces:
$$\Rringcirc{L}{X}((X^{-1}))=\Rringcirc{L}{X}[X]\oplus \widehat{\bigoplus_{i\geq 1}
\Rringcirc{L}{X} X^{-i}}$$ and Proposition \ref{B-alg-structure}.
\end{proof}

\begin{Remark}{\em The reader should compare Proposition \ref{descriptionfieldofunif-bis} with \cite[Lemma 7]{KED}. By Theorem 6 ibid., if $L$ is algebraically closed, the field
$$\bigcup_{n\geq 1}\widehat{\operatorname{Frac}(\Rring{L}{X})}_v((X^{-\frac{1}{n}}))$$
contains an algebraic closure of $L((X^{-1}))$.
 }\end{Remark}
 
 \subsubsection{Some automorphisms} We keep assuming the hypotheses of \S \ref{case-b-inversible}.
 It is easy to see that there is a natural embedding
$$L((\theta^{-1}))^\times\xrightarrow{\varphi}\operatorname{Aut}_L\big(\Rringcirc{L}{X}((X^{-1}))\big).$$
To define this map we recall that $\Theta:=\iota(\theta)$. We consider
$$\alpha=\sum_{j\geq j_0}\alpha_j\theta^{-j}\in L((\theta^{-1}))=\widehat{L\otimes_{\FF_q}K_\infty}_{|\cdot|},$$
where the valuation $|\cdot|$ is extended trivially on $L\otimes 1$. Then, setting 
\begin{equation}\label{automorphisms-alpha}
\varphi_\alpha(X_j):=\sum_{i\geq i_0}\alpha_iX_{i+j}
\end{equation}
defines an automorphism of $\Rringcirc{L}{X}((X^{-1}))$ as expected, and the map is $\alpha\mapsto\varphi_\alpha$. 
More precisely, the reader can verify the next result:
\begin{Lemma}\label{alpha-lambda-v}
For all $f\in\Rringcirc{L}{X}((X^{-1}))$ and for all $\alpha$ as above we have  $v(\varphi_\alpha(f))=q^{\deg_\theta(\alpha)}v(f)$. Moreover, if $f\in \Rring{L}{X}$ then $\lambda(\varphi_\alpha(f))\leq \lambda(f)$.
\end{Lemma}

In particular, choosing
\begin{equation}\label{p-root}
\alpha=\frac{\theta\Theta^{\frac{1}{q}}}{\theta-\Theta^{\frac{1}{q}}}=\sum_{i\geq 0}\Theta^{\frac{i+1}{q}}\theta^{-i}
\end{equation}
we can reconstruct the map $f\mapsto\tau^{-1}(f)$ of Proposition \ref{B-alg-structure}.

\subsection{Tame series}\label{tameseriessect}

Unless otherwise specified, we shall fix, throughout this subsection, a $\tau$-difference sub-$A$-algebra $B$ of $\KK_\Sigma$, for some $\Sigma$.
We denote by $\Rring{B}{X}^b$ the sub-$B$-algebra of $\Rring{B}{X}$ formed by the series as in (\ref{prototypeF}), satisfying $\sup_{j}\|F_j\|<\infty$ ($(\cdot)^b$ stands for 'bounded'). We leave to the reader the proof of the following:
\begin{Lemma}\label{lemma15} $\Rring{B}{X}^b$ is a difference sub-$B$-algebra of $\Rring{B}{X}$ containing $\Ring{B}{X}$.\end{Lemma}
We consider the map $\Ring{B}{X}\xrightarrow{J}\operatorname{Hol}(\CC_\infty\rightarrow \KK_\Sigma)$ defined by $J(X_i)=e_i$, where
$$e_i:=e_C\left(\frac{z}{\theta^i}\right)=\exp_C\left(\frac{\widetilde{\pi}z}{\theta^i}\right)$$ for all $i\in\ZZ$. It is easy to see that $J$ is a $B$-algebra morphism and defines an algebra map
from $\Rring{B}{X}^b$ to the maps from $\CC_\infty$ to $\KK_\Sigma$; this follows from the fact that, for all $z\in\CC_\infty$, $|e_i(z)|=|\frac{\widetilde{\pi}z}{\theta^i}|$ for all $i$ sufficiently large (depending on $z$).
We set $\underline{e}=(e_i:i\in\ZZ)$.
We denote by $\Tame{B}$ the image $J(\Rring{B}{X}^b)$ of $J$ in the $\KK_\Sigma$-valued maps. We call it the $B$-algebra of  {\em tame series}.
Explicitly, if we set
$$\Angle{e}^j=J(\Angle{X}^j)=\prod_{i\in\ZZ}e_i^{j_i},\quad j=\sum_{i\in\ZZ}j_iq^{-i}\in\ZZ[p^{-1}]_{\geq 0},\quad j_k\in\{0,\ldots,q-1\},$$
we can make the next:
\begin{Definition}\label{deftameseries}
{\em A {\em tame series} with coefficients in $B$ is a map $\CC_\infty\rightarrow\KK_\Sigma$ which is defined by an everywhere converging series $f$ of the type
\begin{equation}\label{prototypeftame}
f(z)=\sum_{j\in\ZZ[p^{-1}]_{\geq 0}}f_j\Angle{e}^j,\quad f_j\in B,\end{equation}
satisfying the following properties.
\begin{enumerate}
\item There exists an integer $L\geq 0$ such that if $f_j\neq 0$, then
$\ell_q(j)\leq L$.
\item There exists $M>0$ such that, for all $j\in\ZZ[p^{-1}]_{\geq 0}$, $f_j\in B$ satisfies $\|f_j\|\leq M$.
\item There exists $N\in\NN$ such that if $j\in\ZZ[\frac{1}{p}]_{\geq 0}$ is such that $f_j\neq 0$, then $j<N$.
\end{enumerate}}
\end{Definition}

\begin{Proposition}\label{everytameboundedisentire}
The map $J$ extends to a $B$-algebra morphism $$\Rring{B}{X}^b\xrightarrow{J} \operatorname{Hol}(\CC_\infty\rightarrow \KK_\Sigma)$$
and this is a morphism of $\tau$-difference rings.
\end{Proposition}

\begin{proof}
Let us consider a series $f$ such as in (\ref{prototypeftame}).
Observe that for all $j\in\ZZ[\frac{1}{p}]_{\geq 0}$, the function 
$z\mapsto \Angle{e}^j$ is $\KK_\Sigma$-entire.
It suffices to show that, for all $R\in|\CC_\infty|$, the series defining $f$ converges 
uniformly over the disk $D_{\CC_\infty}(0,R)$. 
One immediately sees that $f(z)$ is a tame series if and only if $f(\theta^{-1}z)$ is a tame series. Hence, we are reduced to prove the above property in the case $R=1$.
Now, observe that the set $\{j\in\ZZ[p^{-1}]_{\geq 0}:f_j\neq 0\text{ and }j\geq1\}$ is finite (because of the conditions (1) and (3) of Definition \ref{deftameseries}). Hence, 
we can decompose 
\begin{equation}\label{decompotsition-f}
f=\sum_{j\geq1}f_j\Angle{e}+\sum_{0\leq j<1}f_j\Angle{e}.\end{equation}
 The first sum is finite and therefore defines an entire function. Note now that if $j=\sum_{k}j_kq^{-k}<1$ then we can write
$$\Angle{e}^j=e_{i_1}(z)^{j_1}\cdots
e_{i_l}(z)^{j_l}$$
 where
 $\underline{i}=(i_1,\ldots,i_l)\in(\NN^*)^l$. Then, for $|z|\leq 1$, by the fact that $\exp_C$ is locally an isometry,
$$|\Angle{e}^j|=|e_{i_1}(z)^{j_1}\cdots
e_{i_l}(z)^{j_l}|\leq |\widetilde{\pi}|^{\ell_q(j)}|\theta|^{-(i_1j_1+\cdots+i_lj_l)}.$$
Hence 
$$\|f_j\Angle{e}^j\|\leq M|\widetilde{\pi}|^{L}|\theta|^{-(i_1j_1+\cdots+i_lj_l)}\rightarrow0$$ where $L,M$ are as in (1) and (2) of Definition \ref{deftameseries},
the limit being considered for the Fr\'echet filter over the set of couples $(\underline{i},\underline{j})$ with $\underline{j}=(j_1,\ldots,j_l)$. This means that in the above decomposition (\ref{decompotsition-f}), the second series defines a $\KK_\Sigma$-entire function and the series defining $f$ converges to a $\KK_\Sigma$-entire function.\end{proof}

\subsubsection{Asymptotic behavior of tame series}\label{asymptotic-beha}

For $j\in\ZZ[\frac{1}{p}]_{\geq 0}$ we call $\Angle{e}^j$ a {\em monic tame monomial}.
Its {\em depth} is the integer $\lambda(\Angle{e}^j)=\ell_q(j)$ and its {\em weight} is $j$. To fix the ideas, the weight of $e_0=e_C(z)$ is one and the weight of $1$ or of a non-zero constant is $0$. Distinct tame monomials have distinct weights. The condition of finite depth ensures that the supremum of the weights of the monomials composing a non-zero tame series is a maximum.
In the following, we call the unique tame monomial of maximal weight in a non-zero tame series $f$, the {\em leading tame monomial}. The weight $w(f)$ of $f$ is by definition equal to the weight of the leading tame monomial. The weight $-\infty$ is assigned to the zero tame series. 
We now discuss the question on whether, assigning to a non-zero tame series $f$ the weight $w(f)$, we have defined a degree map
$$\Tame{B}\xrightarrow{w}\ZZ[p^{-1}]_{\geq 0}\cup\{-\infty\},$$ that is, the opposite of a valuation. 
Of course, this is related to the uniqueness of the tame expansion of a function such as in (\ref{prototypeftame}), entire after Proposition \ref{everytameboundedisentire}; we are going to focus on these questions now.

\begin{Lemma}\label{distinctweights} We consider a monic tame monomial
$f(z)=\Angle{e}^j=e_{i_1}(z)^{j_1}\cdots
e_{i_l}(z)^{j_l}$ with $i_1>\cdots>i_l$ and $j_1,\ldots,j_l\in\{0,\ldots,q-1\}$. Let $z\in\CC_\infty$ be such that $|z|\not\in |\theta|^\ZZ$. If $|z|>|\theta|^{i_l}$, we have 
$|f(z)|=|e_C(z)|^{j}.$
\end{Lemma}
\begin{proof}
Let $z\in\CC_\infty$ be such that $|\theta|^{n-1}<|z|<|\theta|^n$, for $n\in\ZZ$. Let us suppose that $n\geq 1$.
From the Weierstrass product expansion of the function
$e_A(z)=\widetilde{\pi}^{-1}\exp_C(\widetilde{\pi}z)$:
 \begin{equation}\label{Weierstrass}
 e_A(z)=z\prod_{a\in A\setminus\{0\}}\left(1-\frac{z}{a}\right),\end{equation}
 we see that
 $$|e_A(z)|=|z|\prod_{a\neq0}\left|1-\frac{z}{a}\right|=|z|\prod_{0<|a|<|z|}\left|\frac{z}{a}\right|=|z|^{q^n}\prod_{0<|a|\leq |\theta|^{n-1}}|a|^{-1}.$$ Therefore
\begin{eqnarray*}
|e_C(z)|&=&|\widetilde{\pi}||z|^{q^n}\prod_{0<|a|<|\theta|^{n-1}}|a|^{-1},\\
|e_C(z\theta^{-i})|^{q^i}&=&|\widetilde{\pi}|^{q^i}\Big|\frac{z}{\theta^i}\Big|^{q^n}\prod_{0<|a|<|\theta|^{n-1-i}}|a|^{-q^i}.
\end{eqnarray*}
One computes easily $\prod_{0<|a|\leq q^{n-1}}a^{-1}=\frac{l_n}{D_n}$ with $D_n$ and $l_n$ defined in (\ref{dn}) and (\ref{ln}) and $|\frac{l_n}{D_n}|=|\theta|^{q\frac{q^n-1}{q-1}-nq^n}$ from which we deduce
$$\frac{\Big|e_C\Big(\frac{z}{\theta^i}\Big)^{q^i}\Big|}{|e_C(z)|}=|\widetilde{\pi}|^{q^i-1}|\theta^i|^{-q^n}\Bigg|\frac{l_{n-i}^{q^i}D_n}{D_{n-i}^{q^i}l_n}\Bigg|=1.$$
To resume, if $i$ is a non-negative integer and $n>i$ (note that $|\frac{z}{\theta^{i}}|\not\in |\theta|^\ZZ$), then $$|e_i(z)|=\left|e_C\left(\frac{z}{\theta^i}\right)\right|=|e_C(z)|^{\frac{1}{q^i}}.$$ This suffices to complete the proof of the Lemma.
\end{proof}

\begin{Proposition}\label{leading}
Let us consider a non-zero tame series $f$ as in (\ref{prototypeftame}) and let $\Angle{e}^{j_0}$ be its leading tame monomial. Then, for all $z\in\CC_\infty$ such that
$|z|\not\in |\theta|^\ZZ$ and with $|z|$ large enough depending on $f$,
$\|f(z)\|=\|f_{j_0}\||e_C(z)|^{j_0}$.
\end{Proposition}

\begin{proof}
Let $z\in\CC_\infty$ be such that $|\theta|^{n-1}<|z|<|\theta|^n$, for $n\in\ZZ$.
Let us suppose that $n\leq i$. Then, $|z|< |\theta|^{i}$ and $|z/\theta^i|< 1$. In this case 
the product expansion (\ref{Weierstrass}) yields
$|e_C\left(\frac{z}{\theta^i}\right)|=\left|\widetilde{\pi}\frac{z}{\theta^i}\right|$.

We consider an arbitrary tame monomial $\Angle{e}^j$, and $z$ as above.
Writing $j=j_1q^{-i_1}+\cdots+j_lq^{-i_l}$ with $i_l>\cdots>i_1$ and $j_i\in\{0,\ldots,q-1\}$,
we can set $$j_{\leq n}=\sum_{\begin{smallmatrix}m\text{ such that}\\
n\leq i_m\end{smallmatrix}}j_mq^{-i_m},\quad 
j_{>n}=\sum_{\begin{smallmatrix}m\text{ such that}\\
n> i_m\end{smallmatrix}}j_mq^{-i_m}\in\ZZ[p^{-1}]_{\geq0}$$
so that
$$j=j_{\leq n}+j_{>n}$$
without carrying over in the base-$q$ sum. Then,
$$\Angle{e}^j=\Angle{e}^{j_{\leq n}}\Angle{e}^{j_{>n}}.$$ By Lemma  \ref{distinctweights}
we have $|\Angle{e}^{j_{>n}}|=|e_C(z)|^{j_{>n}}$. On the other hand, writing 
$j_{\leq n}=j_{k+1}q^{-i_{k+1}}+\cdots+j_{l}q^{-i_{l}}$ (hence $j_{<n}=j_{1}q^{-i_{1}}+\cdots+j_{k}q^{-i_{k}}$), 
we see that $$|\Angle{e}^{j_{\leq n}}|=\frac{(\widetilde{\pi}z)^{\ell_q(j_{\leq n})}}{|\theta|^{\delta_n}}\leq \frac{|\widetilde{\pi}|^{\ell_q(j)}}{|\theta|^{\delta_n}}|z|^{\ell_q(j)},$$ where $\delta_n:=i_{k+1}j_{k+1}+\cdots+i_lj_l$.
Then, we see that 
$$|\Angle{e}^j|\leq |e_C(z)|^{j}|\theta|^{-\delta_n}|\widetilde{\pi}z|^{L}.$$
Let us choose $\widetilde{w}\in\ZZ[\frac{1}{p}]$, positive. Then, for $|z|\geq R_0$ with $R_0\in|\CC_\infty|$ large enough, depending 
only on $\widetilde{w}$ and $L$, we have that 
$|\widetilde{\pi}z|^{L}\leq |e_C(z)|^{\widetilde{w}}$, so that $$|\Angle{e}^j|\leq |e_C(z)|^{j+\widetilde{w}}|\theta|^{-\delta_n}.$$
Now, let us consider a non-zero tame series $f$ that we can write in the following way
$$f=f_{j_0}\Angle{e}^{j_0}+\sum_{j\neq j_0}f_{j}\Angle{e}^j$$ with $f_{j_0}\neq 0$. 
There exists $\widetilde{w}\in\ZZ[\frac{1}{p}]_{\geq 0}$ such that if $j\neq j_0$ is such that $f_j\neq 0$, then
$j<j+\widetilde{w}<j_0$. Hence:
$$\|f_j\Angle{e}^j\|\leq C_1|e_C(z)|^{j+\widetilde{w}}|\theta|^{-\delta_n},\quad |z|\geq R_0,$$ where 
$C_1$ is an upper bound for the absolute values $\|f_j\|$. Since
$\delta_n\rightarrow\infty$, we have that
$$\left\|\sum_{j\neq j_0}f_j\Angle{e}^j\right\|\leq C_2|e_C(z)|^{w'},\quad |z|\geq R_0,$$
for $w'\in\ZZ[\frac{1}{p}]_{\geq 0}$, $0\leq w'<j_0$ and $R_0$ depending on $f$.
Hence,
$$\left\|f(z)-f_{j_0}\Angle{e}^{j_0}\right\|\leq C_3|e_C(z)|^{w'}$$
and if $|z|\geq R_1$ depending on $C_3$ and $w'$, we get $$\|f(z)\|=\|f_{j_0}\|\cdot|\Angle{e}^{j_0}|=\|f_{j_0}\|\cdot|e_C(z)|^{j_0}$$ ($C_2,C_3$ are constants depending on $f$).
\end{proof}

\begin{Remark}\label{remarkIm}{\em 
We define, for $z\in\CC_\infty 
$, $|z|_\Im=\inf\{|z-l|:l\in K_\infty\}=\min\{|z-l|:l\in K_\infty\}$ (see \cite[\S 5]{PEL5}).
The statement of Proposition \ref{leading} holds under the weaker condition that 
$|z|_\Im$ is large enough. We leave the details to the reader.}\end{Remark}

We have the following important consequence of Proposition \ref{leading}.
\begin{Corollary}\label{embeddingcorollary}
If $f$ is an entire function which belongs to $\Tame{B}$, then its tame series expansion is unique.
\end{Corollary}

\begin{proof}
It suffices to show that a tame series as in (\ref{prototypeftame}) cannot vanish identically, if not trivially. But otherwise, such a series 
would then have a unique leading tame monomial, which would contradict the property of Proposition \ref{leading}.
\end{proof}

Thanks to the above Corollary, $J$ is injective, $\Tame{B}$ has a structure of $B$-algebra,
the map $w\circ J$ is the opposite of the valuation $v$ and the depth $\lambda(f)$ of a tame series $f$ defined as the depth of $g\in \Rring{B}{X}^b$ such that $J(g)=f$ becomes a well defined invariant of the entire function it represents. 

\begin{Remark}{\em The opposite of the weight is an additive valuation on tame series that we denote by $v$. While a tame series as in (\ref{prototypeftame}) in general diverges for the $v$-valuation, it converges for the inf-valuation associated to any disk $D_{\CC_\infty}(0,R)$, $R\in |\CC_\infty^\times|$.}\end{Remark}

\subsubsection{The field of uniformizers}\label{tbulletalgebras}

Several constructions of \S \ref{algebrasettings} can be reproduced in connection with the 
$B$-algebra $\Tame{B}$, with very little changes. We set $\Tamecirc{B}=J(\Rringcirc{B}{X}^b)$. 
Explicitly, $\Tamecirc{B}$ is the $B$-module of the series satisfying the items (1) to (3) of Definition \ref{deftameseries} with the additional property that only the functions $e_1,e_2,\ldots$ occur, just as for the indeterminates $X_1,X_2,\ldots$ in the definition of $\Rringcirc{L}{X}$ at the beginning of \S \ref{case-B-field}.
The reader can easily check, writing
$$e:=e_0,$$
the next result:
\begin{Lemma}\label{uniquedigitexpansion}
Every element $f\in\Tame{B}$ can be expanded, in a unique way, as 
$$f=\sum_{i=0}^rf_ie^i,\quad f_i\in\Tamecirc{B}.$$
\end{Lemma}

If $B=L$ is a field, We set
$$\mathfrak{K}_L=\widehat{\operatorname{Frac}(\Tame{L})}_v$$
($v$-adic completion); we call this the {\em field of uniformizers over $L$}. The next proposition provides a simple way to represent the elements of $\mathfrak{K}_L$. Its proof closely follows that of Proposition \ref{descriptionfieldofunif-bis} and we omit it.

\begin{Proposition}[$u$-expansions]\label{descriptionfieldofunif}
Every element $f$ of $\mathfrak{K}_L$ can be expanded in a unique way as a sum
$$f=\sum_{i\geq i_0}f_ie^{-i},\quad f_i\in\Tamecirc{L}.$$
\end{Proposition}

We also need to introduce the valuation ring $\mathfrak{O}_L$ and the maximal ideal $\mathfrak{M}_L$ of $\mathfrak{K}_L$. The residual field is $L$. 
We have, as $L$-vector spaces:
$$\mathfrak{M}_L=\widehat{\bigoplus_{i>0}\Tamecirc{L}e^{-i}},\quad \mathfrak{O}_L=L\oplus\mathfrak{M}_L.$$
We write, for simplicity, $\mathfrak{K}_\Sigma$
for $\mathfrak{K}_{\KK_\Sigma}$ etc.
\begin{Definition}\label{definition-field-uniformizers}
{\em The {\em field of uniformizers} is the complete $v$-valued field 
$$\mathfrak{K}=\widehat{\bigcup_\Sigma\mathfrak{K}_\Sigma}.$$} We denote by $\mathfrak{O},\mathfrak{M}$ the valuation ring and the maximal ideal of $v$.
\end{Definition}

\subsubsection{Some continuous automorphisms of $\Tame{B}$ and $\Tamecirc{B}((e^{-1}))$} We continue to assume that $B$ is a subalgebra of $\KK_\Sigma$ containing $A$. Some automorphisms of the type (\ref{automorphisms-alpha}) give rise to automorphisms of 
$\Tame{B}$ and $\Tamecirc{B}((e^{-1}))$. 
We recall that in \S \ref{carlitzexttate} we have studied the extension of the Carlitz exponential $\exp_C$ to an $\FF_q(\underline{t}_\Sigma)$-linear endomorphism of $\KK_\Sigma$. In particular, we can view the functions $e_i:\CC_\infty\rightarrow\CC_\infty$ (for $i\in\ZZ$) as 
$\FF_q(\underline{t}_\Sigma)$-linear endomorphisms of $\KK_\Sigma$.
Consider $\alpha=\sum_{i\geq i_0}\alpha_i\theta^{-i}$ where $\alpha_i\in B$, such that the set $\{\|\alpha_i\|:i\in\ZZ\}$ is bounded. Additionally, suppose that $\alpha_{i_0}\in B^\times$. Then, it is easy to see that 
$\varphi_\alpha$ is well defined and induces a continuous automorphism of $\Tame{B}$. Moreover,
we can set, for $f=\sum_{i\geq i_0}f_ie^{-i}\in\Tamecirc{B}((e^{-1}))$:
$$\varphi_\alpha(f)=\sum_{i\geq i_0}\varphi_\alpha(f_i)e(\alpha z)^{-i}.$$
By using Lemma \ref{lemma15}, we see that the series defining $\varphi_\alpha(f)$ converges 
in $\Tamecirc{B}((e^{-1}))$ to an element of weight $q^{\deg_\theta(\alpha)}w(f)$; this is the extension by continuity of the previous automorphism of $\Tame{B}$. Similarly,
if $B$ is complete and $f\in\Tame{B}$ is such that $f:\CC_\infty\rightarrow \KK_\Sigma$ is entire, and if $\beta\in B$,
then we have that the entire function $f(z+\beta)$ defines an element of $\Tame{B}$ of same weight and same depth as $f$. Hence, we can define, for $f=\sum_{i\geq i_0}f_ie^{-i}\in\Tamecirc{B}((e^{-1}))$, 
$$\psi_\beta(f)=\sum_{i\geq i_0}f_i(z+\beta)(e+e(\beta))^{-i}$$
and again, this series converges in $\Tamecirc{B}((e^{-1}))$ to an element which has the same weight as $f$. These properties can be combined to yield the next lemma, which will be needed later.

\begin{Lemma}\label{aendomorphisms}
Let $f$ be an element of $\Tamecirc{\KK_\Sigma}((e^{-1}))$ and $\alpha\in\FF_q(\underline{t}_\Sigma)((\theta^{-1}))^\times$, $\beta\in\FF_q(\underline{t}_\Sigma)((\theta^{-1}))$. Then,
$h:=(\varphi_\alpha\circ\psi_\beta)(f)\in\Tame{\KK_\Sigma}$ satisfies $w(h)=q^{\deg_\theta(\alpha)}w(f)$.
\end{Lemma}

\subsubsection{Some remarks}

There are entire functions $\CC_\infty\rightarrow\CC_\infty$ which are not tame series. One of them is the identity map $z\mapsto z$.
Indeed, one sees easily that for all $w\in\QQ$,
$$\lim_{|z|_\Im\rightarrow\infty}\frac{|z|}{|e_C(z)|^w}\in\{0,\infty\}.$$ therefore, 
$(z\mapsto z)\not\in\Tame{\CC_\infty}$ as otherwise, we could assign a well defined weight in $\ZZ[\frac{1}{p}]$ to it.

To define $\Tame{B}$, we have used formal series with bounded coefficients in $B$ (in Definition \ref{deftameseries}).
One of the reasons for this choice is that the isomorphism $J$ of Proposition \ref{everytameboundedisentire}
is likely not to extend to a larger sub-algebra of $\Rring{B}{X}$.
We illustrate the problem for $B=\CC_\infty$.

We set $$G=\sum_{i\geq 0}\theta^{\frac{i}{q}}X_{i+1}=\varphi_\alpha(X_0)\in \Rringcirc{\CC_\infty}{X},$$ where $\alpha$ is as in (\ref{p-root}).
Then, we have the identities in $\Rring{\CC_\infty}{X}$ (we have used the following computation to show that $\Rring{\CC_\infty}{X}$ is inversive for $\tau$):
\begin{eqnarray*}
G^q & = & \left(\sum_{i\geq 0}\theta^{\frac{i}{q}}X_{i+1}\right)^q\\
&=&\sum_{i\geq 0}\theta^iX_{i+1}^q\\
&=&\sum_{i\geq 0}\theta^i(C_\theta(X_{i+1})-\theta X_{i+1})\\
&=&\sum_{i\geq 0}\theta^i(X_{i}-\theta X_{i+1})\\
&\overset{!}{=}&\sum_{i\geq 0}\theta^iX_{i}-\sum_{i\geq 0}\theta^{i+1}X_{i+1}\\
&=&X_0.
\end{eqnarray*}
Note the exclamation mark over the next to the last equality.
In parallel, let us set 
$$g=\sum_{i\geq 0}\theta^{\frac{i}{q}}e_{i+1}.$$
This is not an element of $\Tame{\CC_\infty}$ because the sequence $(|\theta^{\frac{i}{q}}|)_i$ is not bounded.
We claim that $g$ defines an entire function. Indeed, for all $R\in |\CC_\infty|$
and all $z\in D(0,R)$, we have, for any $i$ large enough,
$|e_{i+1}(z)|=|\widetilde{\pi}||z||\theta|^{-i-1}$ so that $|\theta^{\frac{i}{q}}e_{i+1}(z)|\leq|\widetilde{\pi}||\theta|^{\frac{i}{q}-i-1}R\rightarrow0$ which implies the uniform convergence of the series defining 
$g$ over any disk $D(0,R)$.

Now, $g^q\neq e$. One way to see this is by observing that $e=\widetilde{\pi}z+h^q$,
with $h$ an entire function. If $g^q=e$, the identity map $z\mapsto z$ would be equal to the 
$q$-th power of an entire function, which is impossible. To compute $g^q-e$ we cannot 
use the argument we applied to show the identity $G^q=X_0$; this argument breaks at the level of the equality $\overset{!}{=}$ because the series of functions $\sum_{i\geq 0}\theta^{i}e_{i+1}$ is divergent outside $0$ although the series $\sum_{i\geq 0}\theta^iX_{i+1}$ defines an element of 
$\Rring{\CC_\infty}{X}$.

To compute $g^q$ we proceed in the following way. We set $$\phi=e_C\Big(\frac{\widetilde{\pi}z}{\theta-t}\Big)=\sum_{i\geq 0}t^ie_{i+1}\in\Tamecirc{\FF_q[t]}.$$ It is easy to see that 
$\lim_{t\rightarrow\theta}(\theta-t)\phi=\widetilde{\pi}z$. But
$$e_C(z)=C_{\theta-t}(\phi)=(\theta-t)\phi+\tau(\phi)$$
so that $e=e_C(z)=\widetilde{\pi}z+\lim_{t\rightarrow\theta}\tau(\phi)=\widetilde{\pi}z+\sum_{i\geq 0}\theta^{i}e_{i+1}^q=\widetilde{\pi}z+g^q$. We thus obtain:
$$g^q-e=\widetilde{\pi}z.$$
From this identity we deduce (1) that $g\not\in\Tame{\CC_\infty}$ (because $z$ is not tame) and (2) the map $J$ does not extend to a $\CC_\infty$-algebra map over $\Tame{\CC_\infty}[G]$. 

Also, note that the condition 
of finite depth in Definition \ref{deftameseries} is necessary. It is not difficult to show that there is a uniformly convergent series expansion (in any bounded subset of $\CC_\infty$)
 $$\widetilde{\pi}z=\sum_{i\geq 0}c_ie_i^{q^i},$$ with $c_0=1$ and $c_i\in K_\infty$ such that 
 $c_i\rightarrow0$ so that the sequence $(c_i)_{i\geq 0}$ is bounded. The reader can compute the coefficients $c_i$ inductively.
 
 \subsubsection{Examples of tame series}
 
 To conclude this section, we give examples of tame series of the kind which will be used in the present paper.
Following \S \ref{class-of-entire-functions}, we consider, in the notations introduced there, a 
function
$$\chi\in\operatorname{Hol}_{\TT_\Sigma[d^{-1}]^\wedge}\Big(\CC_\infty\rightarrow(\EE_\Sigma[d^{-1}]^\wedge)^{n\times n}\Big)$$
analytically extending an $\FF_q$-algebra morphism $A\rightarrow \FF_q(\underline{t}_\Sigma)^{n\times n}$
(see Proposition \ref{firstpropertieschi}, where $\widetilde{\chi}=\chi$). We now use the properties of tame series that we know to show the following result.

\begin{Proposition}\label{chi-is-tame}
The function $\chi$ is the unique entire function $f:\CC_\infty\rightarrow\KK_\Sigma^{n\times n}$
such that $f(a)=\chi(a)$ for all $a\in A$ with $\|\exp_C(\widetilde{\pi}z)^{-\frac{1}{q}}f(z)\|\rightarrow0$
as $\exp_C(\widetilde{\pi}z)\rightarrow0$.
\end{Proposition}

\begin{proof}
We have already seen in Proposition \ref{firstpropertieschi} that the entire function $\chi$ interpolates the map $\chi:A\rightarrow \FF_q(\underline{t}_\Sigma)^{n\times n}$. We now prove the growth estimate. But note that
$$\chi(z)=\exp_C\Big(\widetilde{\pi}(\theta I_n-\Theta)^{-1}z\Big)\omega_\chi^{-1}
=\omega_\chi^{-1}\sum_{i\geq 0}e_{i+1}\Theta^{-i}\in\Big(\Tamecirc{{\TT_\Sigma[d^{-1}]^{\wedge}}}\Big)^{n\times n}.$$ We deduce that $w(\chi)=w(e_1)=\frac{1}{q}$. Hence, by Proposition \ref{leading}, we have that the function
$\|\exp_C(\widetilde{\pi}z)^{-\frac{1}{q}}f(z)\|$ is bounded as $\exp_C(\widetilde{\pi}z)\rightarrow0$.

It remains to show uniqueness. Consider $f\in\operatorname{Hol}_{\KK_\Sigma}(\CC_\infty\rightarrow\KK_\Sigma^{n\times n})$ such that $f(a)=\chi(a)$ for all $a\in A$.
Then the function $g=f-\chi$ is in  $\operatorname{Hol}_{\KK_\Sigma}(\CC_\infty\rightarrow\KK_\Sigma^{n\times n})$ and vanishes on $A\subset\CC_\infty$.
Therefore $\frac{g(z)}{\exp_C(\widetilde{\pi}z)}$ is entire and $\lim_{\exp_C(\widetilde{\pi}z)\rightarrow0}\Big\|\frac{g(z)}{\exp_C(\widetilde{\pi}z)}\Big\|=0$. By Proposition \ref{entire}, $g$ vanishes identically.
\end{proof}

\section{Quasi-periodic matrix functions}\label{quasiperiodicfunctions}

One of the basic observations in the theory of modular forms for the full modular group 
$\operatorname{SL}_2(\ZZ)$ is that they are $\ZZ$-periodic, so that they have a Fourier series development, also called $q$-expansion. There is a very similar feature for (scalar) Drinfeld modular forms for the full modular group $\Gamma=\GL_2(A)$ which are $A$-periodic, and indeed we have in this case $u$-expansions, which is the appropriate structure to study their behaviour at the cusp infinity as well as a large part of their theory.

A similar feature holds for the vector-valued modular forms which are studied in the present work, associated to higher dimensional representations of $\Gamma$. For a special class of representations called {\em representations of the first kind}, introduced below (Definition \ref{firstkindrepresentations}) we are able to expand entries of modular forms as certain formal sums in tame series as in \S \ref{tameseriessect}. However, the problem that we tackle is more involved than the special case of scalar Drinfeld modular forms (case of powers of the determinant representations). Therefore crucial becomes the notion of {\em quasi-periodic function}, equally introduced in this section (Definition
\ref{quasipermatrix}). In this section we study these functions, which can also be understood as a kind of generalization of Goss polynomials. The terminology chosen comes from Gekeler's paper \cite{GEK01} (see \S 2). Gekeler uses what he calls quasi-periodic functions to construct an analogue of the De Rham isomorphism associated to a Drinfeld module (between a `De Rham module' of classes of biderivations and a `Betti module'). More precisely, he constructs (in his \S 4) certain Poincar\'e series to show that the map is surjective (while injectivity follows essentially from the fact the the logarithm series does not extend to an entire function). These Poincar\'e series have inspired the construction of Perkins' series and are similar to the quasi-periodic functions we study in the present paper. It is possible to use them to prove an appropriate version of the De Rham isomorphism for the Carlitz functor evaluated on certain difference algebras, but this theme will not be pursued in the present paper. 

What is important to us is another property: that quasi-periodic functions allow us to make a bridge between modular forms and the tame series of \S \ref{tameseriessect} (see also the motivations in \S \ref{necessity}).
The central result obtained here is Theorem \ref{theorem-u-expansions}, which asserts that every modular form in the sense of Definition
\ref{modularform} can be expanded as a formal series in the field of uniformizers $\mathfrak{K}$. We also give an application of these structures in Theorem \ref{theo-hecke-operators}, where we show that the spaces of our modular forms and cusp forms are endowed with Hecke endomorphisms, generalizing \cite[Proposition 5.12]{PEL&PER3}, which deals with the very special case of $N=2$ and $\rho=\rho_t^*$ (with an ad hoc proof unfortunately very hard to generalize to our more general settings).

\subsection{Quasi-periodic functions}\label{Quasi-periodic-functions}
Let $k$ be any field, and $R$ a commutative $k$-algebra. We denote by $B(R)$ the Borel subgroup $\{(\begin{smallmatrix}* & *\\ 0 & *\end{smallmatrix})\}\subset\GL_2(R)$ and by $U(R)$ the unit upper-triangular subgroup
$\{(\begin{smallmatrix}1 & *\\ 0 & 1\end{smallmatrix})\}\subset\GL_2(R)$. Let $T$ be an indeterminate and $E/k(T)$ 
be a field extension. Suppose we are given 
$$GL_2(k)\xrightarrow{\mu}\GL_N(E)\xleftarrow{\nu}U(k[T])$$
two representations such that $\mu|_{U(k)}=\nu|_{U(k)}$ and such that for all $\lambda\in k^\times$ and $a\in k[T]$, 
$$\mu(\begin{smallmatrix}\lambda & 0\\ 0 & 1\end{smallmatrix})\nu(\begin{smallmatrix}1 & a\\ 0 & 1\end{smallmatrix})\mu(\begin{smallmatrix}\lambda^{-1} & 0\\ 0 & 1\end{smallmatrix})=\nu(\begin{smallmatrix}1 & \lambda a\\ 0 & 1\end{smallmatrix}).$$
Then, there is a unique representation $\rho:\GL_2(k[T])\rightarrow\GL_N(E)$
which restricts to $\mu,\nu$ respectively on $GL_2(k)$ and $U(k[T])$.

Indeed, see \cite{NAG,SER0}, we have that 
$$\GL_2(k[T])=\GL_2(k)*_{B(k)}B(k[T]),$$ which means that 
$\GL_2(k[T])$ is the amalgamated product of $\GL_2(k)$ and $B(k[T])$ along the common subgroup $B(k)$. By Bruhat's decomposition $\GL_2(k)=B(k)(\begin{smallmatrix}0 & 1\\ 1 & 0\end{smallmatrix})U(k)\sqcup B(k)$ this implies that every element $\gamma\in\GL_2(k[T])$ can be written in a unique way
$$\gamma=A_1B_1\cdots A_lB_l$$
for some $l$, where $A_i\in B(k)(\begin{smallmatrix}0 & 1\\ 1 & 0\end{smallmatrix})U(k)$
and $B_i\in B(k[T])$. Therefore, the identities
$$(\begin{smallmatrix}\lambda & 0\\ 0 & 1\end{smallmatrix})(\begin{smallmatrix}1 & a\\ 0 & 1\end{smallmatrix})(\begin{smallmatrix}\lambda^{-1} & 0\\ 0 & 1\end{smallmatrix})=(\begin{smallmatrix}1 & \lambda a\\ 0 & 1\end{smallmatrix})$$
are the gluing condition for $\mu,\nu$ giving rise to a unique representation $\rho$ of $\Gamma$.

We now take $k=\FF_q$ and $T=\theta$ and we recall that we write $\Gamma=\operatorname{GL}_2(A)$ with $A=\FF_q[\theta]$.
We also recall that $\Omega$ denotes the rigid analytic space whose underlying set is $\CC_\infty\setminus K_\infty$ as defined, for instance, in \cite{GEK} (see also \cite[\S 5, 6]{PEL5}). 
We set, for $a\in A$, $$T_a=(\begin{smallmatrix}1 & a \\ 0 & 1\end{smallmatrix}),\quad S=(\begin{smallmatrix}0 & -1 \\ 1 & 0\end{smallmatrix})$$ (in $\Gamma$). 
The discussion above suggests us, in order to study a representation
\begin{equation}\label{representation-in-B}
\Gamma\xrightarrow{\rho} \operatorname{GL}_N(B)\end{equation}
with $(B,|\cdot|_B)$ a countably cartesian Banach $\CC_\infty$-algebra, that we first analyse its restriction to $U(A)$.
This brings us to the next definition.

\begin{Definition}\label{quasipermatrix}{\em (a) Let $\rho$ be a representation as in (\ref{representation-in-B}).
An analytic function
$$\Omega\xrightarrow{f}B^{N\times1}$$
such that
\begin{equation}\label{semiperiodic}
f(z+a)=\rho(T_a)f(z)\quad\forall a\in A,\end{equation}
is called a {\em $\rho$-quasi-periodic function.} We say that $f$ is {\em tempered} if there exists $M\in\ZZ$ such that $$\lim_{|z|=|z|_\Im\rightarrow\infty}f(z)u(z)^M=\underline{0}$$ where $u$ is defined in (\ref{Goss-uniformiser}). We further say that 
$f$ is {\em regular} if there exists a constant $c>0$ (depending on $f$) such that
the set $\{|f(z)|_B:|z|_\Im\geq c\}$ is bounded (remember that $|\cdot|_\Im$ has been introduced in Remark \ref{remarkIm}).

\noindent (b) Let $$f:\Omega\rightarrow B^{N\times N}$$ be an analytic matrix function 
such that its columns are $\rho$-quasi-periodic in the sense of the point (a) above, so that
$$f(z+a)=\rho(T_a)f(z)\quad \forall a\in A.$$
We say that $f$ is of {\em type} $l\in\ZZ/(q-1)\ZZ$ if for all $\nu\in\FF_q^\times$, we have 
$$f(\nu z)=\nu^{-l}\rho(\begin{smallmatrix} \nu & 0 \\ 0 & 1\end{smallmatrix})f(z)\rho(\begin{smallmatrix} \nu & 0 \\ 0 & 1\end{smallmatrix})^{-1}.$$
(c) We denote by 
$\mathcal{QP}_l^!(\rho;B)$ the $B$-module of tempered $\rho$-quasi-periodic functions 
$$\Omega\rightarrow B^{N\times N}$$ of type $l$, and by $\mathcal{QP}_l(\rho;B)$ the sub-module  of quasi-periodic regular functions.}
\end{Definition}
If $n=1$ and $\rho=\boldsymbol{1}$ (with $\boldsymbol{1}$ the trivial map which sends every element of $\Gamma$ to $1\in\FF_q^\times$), then a quasi-periodic function is a holomorphic function
$f:\Omega\rightarrow B$ such that $f(z+a)=f(z)$ for all $a\in A$. Explicit examples are $e_C(z)=\exp_C(\widetilde{\pi}z)$ and $$u(z)=\frac{1}{\widetilde{\pi}}\sum_{a\in A}\frac{1}{z-a}=\frac{1}{e_C(z)}.$$
Both functions are obviously tempered. The function $e_C(z)$ is of type $-1$ and the function
$u(z)$ is of type $1$. For further use, we record the next Proposition.
\begin{Proposition}\label{propositionperiodic}
Let $f:\Omega\rightarrow B$ be rigid analytic, such that $f(z+a)=f(z)$ for all $a\in A$. Then, the following properties hold:
\begin{itemize}
\item[(a)] There is a unique series expansion
\begin{equation}\label{u-expansion-of-f}
f=\sum_{n\in\ZZ}f_nu(z)^n,\quad f_n\in B,\end{equation}
convergent if $z\in\Omega$ is such that $|z|_\Im>c$ for some $c\in|\CC_\infty^\times|$. 
\item[(b)] If $|f(z)|_B$ is bounded for $|u(z)|<c$ for some $c\in|\CC_\infty^\times|$, then 
$f_n=0$ for all $n<0$.
\item[(c)] If $f$ extends to an entire function over $\CC_\infty$, and there exists $M\in\ZZ$ such that 
$$|u(z)^Mf(z)|_B\rightarrow0$$ as $|u(z)|\rightarrow0$, then $f\in B[u(z)^{-1}]$.
\end{itemize}
\end{Proposition}

\begin{proof}[Sketch of proof.] This result is basically well known but there is a lack of complete reference in the literature. Let us give some details. 

\noindent (a) The proof of \cite[Proposition 6.1]{PEL5} can be adapted to our setting. We recall from ibid. that for an integer $n$ we define
$$\mathcal{B}_n=D_{\CC_\infty}^\circ(0,|\theta|^n)\setminus\bigcup_{a\in A(n)}D_{\CC_\infty}^\circ(a,1),\quad \mathcal{C}_n=D_{\CC_\infty}^\circ(0,|l_n|)\setminus D_{\CC_\infty}^\circ(0,1),$$
which are filtered unions of affinoid subsets of $\CC_\infty$ ($A(n)$ denotes the $\FF_q$-vector space of all the elements of $A$ which are of degree $<n$ in $\theta$). One can verify that, for all $n$,
\begin{multline*}
\mathcal{O}_{\mathcal{C}_n/B}(\mathcal{C}_n)=\Big\{\sum_{k\in\ZZ}f_ku^k:f_k\in B\text{ for all $k$, }f_{-k}\rightarrow0\text{ as }k\rightarrow\infty,f_kl_n^{(1-\epsilon)k}\rightarrow0\\
\text{ as }k\rightarrow\infty,\text{ for all }\epsilon>0\Big\}.\end{multline*}
This follows from the explicit use of an orthonormal basis of $\mathcal{O}_{\mathcal{C}_n}(\mathcal{C}_n)$ and yields an explicit description of the sheaf $\mathcal{O}_{\mathcal{C}_n/B}$.
Similarly, the sub-sheaf of $\mathcal{O}_{\mathcal{B}_n/B}$ whose global sections $g$ are such that $g(z+a)=g(z)$ for all $a\in A(n)$ equals the pull back $\mathcal{E}^*_n\mathcal{O}_{\mathcal{C}_n/B}$
where $\mathcal{E}_n(z)=l_nE_n(z)$, $E_n$ being the $n$-th Carlitz polynomial (see \cite[\S 4.2]{PEL5}). This follows from an application of Proposition 6.2 of ibid. After these observations, the proof of Proposition 6.1 can be slightly modified to yield the existence of the expansion (\ref{u-expansion-of-f}). Uniqueness follows easily from the connectedness of $\Omega$.

Before considering the point (b) of our proposition, we define, after (\ref{u-expansion-of-f}):
$$F(u):=\sum_{n\in\ZZ}f_nu^n,\quad F^-(u):=\sum_{n\leq 0}f_nu^n,\quad F^+(u):=\sum_{n>0}f_nu^n.$$
We have that $F$ converges for all $u\in\dot{D}_{\CC_\infty}(0,c):=\{z\in\CC_\infty:0<|z|\leq c\}$ where $c\in|\CC_\infty^\times|$, $c<1$ and $f(z)=F(u(z))$. The series $F^-(u)$ converges for all $0<|u|\leq c$ and $c<1$. In other words, $|f_k|_Bc^{-k}\rightarrow0$, which implies that the sequence $f_{-k}$ tends to zero as $k\rightarrow\infty$. In particular, $F^-(u)$ converges for every $u\neq0$.

\noindent (b) Applying (a), $\lim_{u\rightarrow0}F(u)$ exists and $|F(u)|_B$ is bounded on 
$\dot{D}_{\CC_\infty}(0,c)$. We write $f_n=\sum_{i\in I}f_{n,i}b_i$ with $f_{n,i}\in\CC_\infty$, where $(b_i)_{i\in I}$ is an 
orthonormal basis of $B$. We note that
$$|f_{n,i}|\max\{r_1^n,r_2^n\}\rightarrow0\text{ as }i\rightarrow\infty,$$
for all $r_1,r_2\in |\CC_\infty^\times|$ such that $r_1<r_2\leq c$. Therefore we have unconditional convergence with $u$ in $\dot{D}_{\CC_\infty}(0,c)$ for an appropriate choice of $c$ and we can write:
$$F(u)=\sum_{n\in\ZZ}\left(\sum_{i\in I}f_{n,i}b_i\right)u^n=\sum_{i\in I}\left(\sum_{n\in\ZZ}f_{n,i}u^n\right)b_i.$$
We get that for all $i\in I$, the limit for $u\rightarrow0$ of $\sum_nf_{n,i}u^n$ exists. By 
\cite[\S 3, Theorem (Riemann I)]{BAR}, $f_{n,i}=0$ for all $i$, $n<0$ and $F-F^+\in B$.

\noindent (c) Let $f:\CC_\infty\rightarrow B$ be entire, such that 
$f(z+a)=f(z)$ for all $a\in A$. Then, by (a) of this proposition, 
$$f(z)=F(u)=\sum_{k\in\ZZ}f_ku^k,\text{ with }f_k\in B,\quad \forall k\in\ZZ.$$ By the above remarks, 
setting $f^-(z)=F^-(u(z))$ if $z\not\in A$ and $0$ otherwise, $f^-$ defines a $B$-entire function. hence
$f^+(z)=F^+(u(z))=f(z)-f^-(z)$ is $B$-entire and at once, bounded at infinity. By Proposition \ref{entire}, it is constant, hence identically zero; We conclude that 
$f(z)=f^-(z)=F^-(u(z)).$ Now, assume that there exists $M$ such that $|u^MF^-(z)|_B$
is bounded in $B$ as $u\rightarrow0$ (i.e. as $|z|_\Im=|z|\rightarrow\infty$). Then, by (b), 
we have that $G:=u^MF^-$ is such that $G=G^+$ (in the above notations). This suffices to conclude.
\end{proof}

Proposition \ref{propositionperiodic} implies that $\mathcal{QP}_l^!(\boldsymbol{1};B)$ can be embedded in $B[[u]][u^{-1}]^{N\times N}$ and for all $l\in\ZZ/(q-1)\ZZ$ and a representation $\rho$ as in (\ref{representation-in-B}), $\mathcal{QP}_l^!(\rho;B)$ is a module over $\mathcal{QP}_0^!(\boldsymbol{1};B)^{N\times N}$, and a similar property holds for the regular quasi-periodic functions. Of course, we can specify the target space; the meaning of 
$\mathcal{QP}_l^!(\rho;\LL_\Sigma)$ etc. is therefore understood.

\subsubsection{The series $\Psi_m(\rho)$}\label{section-Psi}

There are three types of quasi-periodic functions that are needed in the present work. They are denoted by $\Psi_(\rho),\Xi_\rho$ and $\Phi_\rho$. Here we study the first type.
We consider a representation $\rho:\Gamma\rightarrow\operatorname{GL}_N(B)$. We additionally suppose that:
\begin{equation}\label{rhota}
|a^{-1}\rho(T_a)|_B\rightarrow0,\quad \text{ as $|a|\rightarrow\infty$ with $a\in A$}.\end{equation} 
\begin{Lemma}\label{lemmapsil}
Let $l$ be a positive integer. The function
$\Psi_m(\rho)$
defined, for all $z\in\CC_\infty\setminus A$, by 
$$\Psi_m(\rho)(z)=\sum_{a\in A}(z-a)^{-m}\rho(T_a),$$
determines a non-zero element of $\mathcal{QP}_m(\rho;B)$.
\end{Lemma}
\begin{proof}
It is easy to show that $\Psi_m(\rho)$ converges uniformly for $z\in\CC_\infty\setminus(\sqcup_{a\in A}D(a,r))$ with $r\in|\CC_\infty^\times|$, $0<r<1$. This implies that 
$\Psi_m(\rho)$ defines a holomorphic function $\Omega\rightarrow B^{N\times N}$, and this function is non-zero because it has, in any disk $D_{\CC_\infty}(0,r)$ with $r\in|\CC_\infty^\times|$, a meromorphic 
extension which has poles of order $m$ at every $a\in D(0,r)\cap A$. 
Moreover, we have, for all $z\in\CC_\infty\setminus A$ and $b\in A$, writing $\Psi$ for $\Psi_m(\rho)$:
\begin{eqnarray*}
\Psi(z-b)&=&\sum_{a\in A}(z-a-b)^{-m}\rho(T_a)\\
&=&\sum_{a\in A}(z-a-b)^{-m}\rho(T_{a+b})\rho(T_{-b})\\
&=&\Psi(z)\rho(T_{-b})=\rho(T_{-b})\Psi(z).
\end{eqnarray*}
so that
\begin{equation}\label{qpbehavior}\Psi(z+a)=\Psi(z)\rho(T_a)=\rho(T_a)\Psi(z),\quad \forall a\in A.\end{equation}
Since $$T_a=\begin{pmatrix}\lambda & 0 \\ 0 & 1\end{pmatrix}T_{\lambda^{-1}a}\begin{pmatrix}\lambda^{-1} & 0 \\ 0 & 1\end{pmatrix},\quad \forall a\in A,\quad \lambda\in\FF_q^\times,$$
for all $\lambda\in\FF_q^\times$:
\begin{eqnarray*}
\Psi(\lambda z)&=&\sum_{a\in A}(\lambda z-a)^{-m}\rho(T_a)\\
&=&\lambda^{-m}\rho\begin{pmatrix}\lambda & 0 \\ 0 & 1\end{pmatrix}\Psi(z)\rho\begin{pmatrix}\lambda^{-1} & 0 \\ 0 & 1\end{pmatrix},
\end{eqnarray*}
and the type is $m$.
Now, as $|z|_\Im\rightarrow\infty$, we get $\Psi(z)\rightarrow0$ so that $\Psi_m(\rho)\in\mathcal{QP}_m(\rho;B)$.
\end{proof}

\subsubsection*{Growth in annuli} What follows will be used in our study of Poincar\'e series \S \ref{poincareseries}, notably in proving non-vanishing properties. This part may be skipped in a first reading.
We suppose that $B\subset\KK_\Sigma$.
We study the series $\Psi_m(\rho)$ in the annuli $\mathcal{C}_0=\{z\in\CC_\infty:0<|z|<1\}$
and $\mathcal{C}_n=\{z\in\CC_\infty:|\theta|^{n-1}<|z|<|\theta|^n\}$, for $n>0$. The representation
$\rho$ being fixed, we now write $\Psi_m$ instead of $\Psi_m(\rho)$. We also write:
\begin{eqnarray*}
\Psi_m^{<}(z)&=&I_Nz^{-m}\quad\text{ if $n=0$},\\
&=&\sum_{\begin{smallmatrix}a\in A\\ |a|<|\theta|^n\end{smallmatrix}}(z-a)^{-m}\rho(T_a)\quad\text{ if $n>0$},\\
\Psi_m^{\geq}(z)&=&\sum_{\begin{smallmatrix}a\in A\\ |a|\geq|\theta|^n\end{smallmatrix}}(z-a)^{-m}\rho(T_a)\quad\text{ for all $n$}.
\end{eqnarray*}
Note that $\Psi_m^{<}\in \KK_\Sigma(z)^{N\times N}$ and that $\Psi_m=\Psi_m^{<}+\Psi_m^{\geq}$. Also, if $\mathcal{D}_k$ denotes the higher divided derivative of order $k$ in the variable $z$ applied coefficientwise, we have
\begin{equation}\label{higher-deri-psi}
\Psi_m=(-1)^{m-1}\mathcal{D}_{m-1}(\Psi_1)=(-1)^{m-1}\mathcal{D}_{m-1}(\Psi^{<}_1)+(-1)^{m-1}\mathcal{D}_{m-1}(\Psi^{\geq}_1).\end{equation}
We begin by studying the case $m=1$. We note that if $a\in A$ is such that 
$|a|<|\theta|^n$ then $|\frac{a}{z}|<1$ and
$$\frac{1}{z-a}=\frac{1}{z}\frac{1}{1-\frac{a}{z}}=z^{-1}\Bigg(1+\sum_{i\geq 0}\Big(\frac{a}{z}\Big)^i\Bigg).$$
Hence, we get
\begin{equation}\label{Psi-<}
\Psi_1^{<}(z)=z^{-1}\Bigg(\sum_{i\geq 0}H_{-i}(\rho)z^{-i}\Bigg),\quad |z|>|\theta|^{n-1},
\end{equation}
where
$$H_{-i}(\rho)=\sum_{\begin{smallmatrix}a\in A\\ |a|<|\theta|^n\end{smallmatrix}}a^i\rho(T_a),\quad i\geq 0,$$ where we adopt the convention $a^0=1$ including when $a=0$, so that
$H_0(\rho)=\sum_{|a|<|\theta|^n}\rho(T_a)$.

Similarly, if $|a|\geq|\theta|^n$ then $|\frac{z}{a}|<1$ and
$$\frac{1}{z-a}=-\frac{1}{a}\frac{1}{1-\frac{z}{a}}=-\frac{1}{a}\sum_{i\geq 0}\Big(\frac{z}{a}\Big)^i=-z^{-1}\sum_{j\geq1}\Big(\frac{z}{a}\Big)^j,$$
and we derive the expansion
\begin{equation}\label{Psi-geq}\Psi_1^{\geq}(z)=-z^{-1}\sum_{j\geq 1}H_j(\rho)z^j,\quad |z|<|\theta|^n,\end{equation}
where 
$$H_{j}(\rho)=-\sum_{\begin{smallmatrix}a\in A\\ |a|\geq|\theta|^n\end{smallmatrix}}a^{-j}\rho(T_a),\quad j\geq 1.$$
We deduce that
$$\Psi_1(z)=z^{-1}\sum_{i\in\ZZ}H_i(\rho)z^i,\quad z\in\mathcal{C}_n.$$

\subsubsection*{The case $\rho=\boldsymbol{1}$}
We suppose that $\rho=\boldsymbol{1}$ so that $N=1$. It is easy to show that, in this case, writing 
$S_d(i)=\sum_{a\in A^+(d)}a^{-i}\in K$, $S_{<d}(i)=\sum_{0\leq k<d}S_k(i)\in K$ and $\zeta_A(i)=\sum_{a\in A^+}a^{-i}\in K_\infty$, we get $H_i(\boldsymbol{1})=0$ if $q-1\nmid i$, and if $q-1\mid i$ then
$H_i(\boldsymbol{1})=-S_{<n}(-i)$ if $i<0$, $H_0(\boldsymbol{1})=1$ if $n=0$ and $H_0(\boldsymbol{1})=0$ if $n>0$, and $H_i(\boldsymbol{1})=\zeta_A(i)-S_{<n}(i)\in K_\infty$. If $n=0$ we get
$H_i(\boldsymbol{1})=0$ for all $i<0$ and therefore we conclude with the well known identity:
\begin{equation}\label{tilde-psi-Carlitz}
\Psi_1(z)=\frac{1}{z}\Bigg(1+\sum_{\begin{smallmatrix}j>0\\ q-1\mid j\end{smallmatrix}}\zeta_A(j)z^j\Bigg),\quad 0<|z|<1.\end{equation}

If we choose $n=1$ and $z\in\mathcal{C}_1$ (i.e. $1<|z|<|\theta|$) we easily verify:
$$\Psi_1^{<}(z)=\sum_{\lambda\in\FF_q}\frac{1}{z-\lambda}=\frac{-1}{z^q-z}=-\sum_{\begin{smallmatrix}i\geq q\\ q-1\mid i\end{smallmatrix}}z^{-i},\quad |z|>1.$$
Similarly, we compute
$$\Psi_1^{\geq}(z)=z^{-1}\sum_{\begin{smallmatrix}i\geq 1\\ q-1\mid i\end{smallmatrix}}\Big(\zeta_A(i)-1\Big)z^i,\quad |z|<|\theta|.$$
In other words, to construct the formal series which represents $\Psi_1$ on $\mathcal{C}_1$ it suffices to compute 
$$\widetilde{\Psi}-z^{-1}\sum_{\begin{smallmatrix}j\in\ZZ\\ q-1\mid j\end{smallmatrix}}z^j,$$
where $\widetilde{\Psi}$ is the formal series (\ref{tilde-psi-Carlitz}) which represents $\Psi_1$ on $\mathcal{C}_0$ (note that the last series is nowhere converging). This example when $n=1$ is useful in computations related to Poincar\'e series, see \S \ref{poincareseries}, so we keep developing it a little bit further but the reader can skip it at the first reading.

Since $\zeta_A(i)-1\equiv\theta^{-i}\pmod{\mathcal{M}_\infty^{i+1}}$ for all $i>0$ where $\mathcal{M}_\infty=\frac{1}{\theta}\FF_q[[\frac{1}{\theta}]]$ is the maximal ideal of $K_\infty$, we observe that 
the $\infty$-adic Newton polygon of $\Psi_1$ over $\mathcal{C}_1$ has three slopes. 
If $z\in\mathcal{C}_1$, we have $|\Psi_1^{<}(z)|=|z|^{-q}$ and $|\Psi_1^{\geq}(z)|=|z|^{q-2}|\theta|^{1-q}$. We therefore have that 
$|\Psi_1^{<}(z)|=|\Psi_1^{\geq}(z)|$ if and only if $|z|=|\theta|^{\frac{1}{2}}$ and if $1<|z|<|\theta|^{\frac{1}{2}}$ we have 
$|\Psi_1(z)|=|\Psi_1^{<}(z)|=|z|^{-q}$ while if $|\theta|^{\frac{1}{2}}<|z|<|\theta|$ we have 
$|\Psi_1(z)|=|\Psi_1^{\geq}(z)|=|z|^{q-2}|\theta|^{1-q}$. 

In view of our further investigations (related to Poincar\'e series), we need to generalize the above computations to the study of $\Psi_m$ with $m>0$. But using (\ref{higher-deri-psi}) and the fact that
$(-1)^{m-1}\mathcal{D}_{m-1}(\Psi^{<}_1)=\Psi^{<}_m$ and $(-1)^{m-1}\mathcal{D}_{m-1}(\Psi^{\geq}_1)=\Psi^{\geq}_m$ we deduce that if $z\in\mathcal{C}_1$:
\begin{equation}\label{precise-estimate}
|\Psi_m^{<}(z)|=|z|^{1-q-m-\omega_1},\quad |\Psi_m^{\geq}(z)|=|\theta|^{-m}\Big|\frac{z}{\theta}\Big|^{\omega_2},
\end{equation}
where $-1+q+m+\omega_1$ is the order of $\Psi_m^{<}(z)$ in $z^{-1}$ with $\omega_1\geq 0$, and 
$\omega_2\geq 0$ is the order of $\Psi_m^{\geq 0}(z)$ in $z$. 
Indeed, the reader can easily verify that $\theta^m\mathcal{D}_{m-1}(\Psi_1^{\geq})=\sum_k\alpha_kz^k$
where $|\alpha_k|=|\theta|^{-k}$ for all $k$.
The computation of $\omega_1$ and $\omega_2$ and their dependence in $m$ is a combinatorial problem which goes beyond our scopes but fortunately, we do not need to solve it.
We see that the $\infty$-adic Newton polygon has three slopes in this case too.
Note that $|\Psi_m^{<}(z)|=|\Psi_m^{\geq}(z)|$ if and only if 
$|z|^{1-q-m-\omega_1}=|\theta|^{-m}\Big|\frac{z}{\theta}\Big|^{\omega_2}$
which is equivalent to
$|z|=|\theta|^{\frac{m+\omega_2}{\omega_1+\omega_2+m+q-1}}.$
Now, 
$$\frac{m+\omega_2}{\omega_1+\omega_2+m+q-1}=1-\frac{\omega_1+q-1}{\omega_1+\omega_2+m+q-1}\in]0,1[.$$
We deduce the next result.
\begin{Lemma}\label{lemma-omega12}
There exists an element $\kappa_1\in]1,|\theta|[\cap|\CC_\infty^\times|$ and two non-negative integers $\omega_1,\omega_2$ such that if $1<|z|<\kappa_1$, then 
$\Psi_m(z)=|z|^{1-q-m-\omega_1}$ and if $\kappa_1<|z|<|\theta|$, then
$\Psi_m(z)=|\theta|^{-m}|\frac{z}{\theta}|^{\omega_2}$.
\end{Lemma}
This result is used in the proof of Proposition \ref{proposition-poincare-bis} which deals with non-vanishing properties of Poincar\'e series.

\subsection{Representations of the first kind}\label{representations-first-kind}

We now introduce a class of representations of $\Gamma$ for which we can construct explicitly entire non-zero quasi-periodic functions in several ways. First of all, we introduce a useful technical definition.

\begin{Definition}\label{degree-of-a-representation}
{\em We say that a representation $\rho:\Gamma\rightarrow\GL_N(\FF_q(\underline{t}_\Sigma))$ is of {\em degree} $l\in\ZZ/(q-1)\ZZ$ if for all $\mu\in\FF_q^\times$, $\rho(\mu I_2)=\mu^{-l}I_N$.}
\end{Definition}
We recall that after (\ref{condition-degree}), $J_\gamma(z)^w\rho(\gamma)$ is a factor of 
automorphy if and only if $\rho$ is of degree $w$.
For example, $\det^{-m}$ is of degree $2m$ (the double of the type). The identity map over $\Gamma$ is of degree $-1$. All the representations that we consider in this text have a well defined degree.

\begin{Definition}\label{firstkindrepresentations}
{\em Let $\chi:A\rightarrow\FF_q(\underline{t}_\Sigma)^{n\times n}$
be an injective $\FF_q$-algebra morphism, let $d\in\FF_q[\underline{t}_\Sigma]\setminus\{0\}$ be such that 
$d\chi(\theta)\in\FF_q[\underline{t}_\Sigma]^{n\times n}$. Then the map 
$$\rho_\chi:\Gamma\rightarrow\operatorname{GL}_{2n}\left(\FF_q[\underline{t}_\Sigma][d^{-1}]\right)\subset\operatorname{GL}_{2n}(\FF_q(\underline{t}_\Sigma))$$
defined, with $\gamma=(\begin{smallmatrix}a & b \\ c & d\end{smallmatrix})\in\Gamma$,
by $$
\rho_\chi(\gamma):=\begin{pmatrix}\chi(a) & \chi(b)\\ \chi(c) & \chi(d)\end{pmatrix},
$$
is a representation of degree $-1$, called the {\em basic representation} associated to $\chi$.
Note also that
$$\det(\rho_\chi(\gamma))=\det(\chi(ad-bc))=\det(\gamma)^n.$$
If $\rho$ is a representation, we write $$\rho^*:={}^t\rho^{-1}$$ for its contragredient representation.
If $\rho$ is of degree $l$, $\rho^*$ is of degree $-l$.
Let $\rho:\Gamma\rightarrow\operatorname{GL}_N(\underline{t}_\Sigma)$ be a representation.
We say that $\rho$ is a {\em representation of the first kind} if $\rho$ can be obtained from basic representations by finitely many iterated 
applications of the following {\em elementary operations}: $(\cdot)^*$, direct sums $\oplus$, Kronecker products $\otimes$, symmetric powers $S^m$, exterior powers $\wedge^m$, in such a way that $\rho$ has a well defined degree. For further use, we will call these operations {\em admissible operations}.}\end{Definition}

Note that if $\rho$ and $\psi$ are two representations such that $\rho$ has degree $l$ and $\psi$ has degree $m$, then:
\begin{eqnarray*}
\rho\oplus\psi& \text{has degree} & l \text{ (if $l=m$)}\\
\rho\otimes\psi& &l+m,\\
S^r(\rho)& & rl,\\
\wedge^r\rho & & rl,\\
\rho^*& & -l,
\end{eqnarray*}
where in the right, $(\cdot)^*,\oplus, \otimes, S^r$ and $\wedge^r$ denote respectively the contragredient, direct sum, Kronecker product, $r$-th symmetric power and the $r$-th exterior power, of representations.

\begin{Remark}\label{remark-monoid}{\em For basic representations $\rho_1,\ldots,\rho_k$, any representation of the first kind 
$\rho:\Gamma\rightarrow\GL_N(\FF_q(\underline{t}_\Sigma))$ constructed combining them with the admissible operations $\oplus,\otimes,\wedge^r,S^r$ extend to monoid homomorphisms 
$A^{2\times 2}\rightarrow\FF_q(\underline{t}_\Sigma)^{N\times N}$. The operation $(\cdot)^*$ is excluded. However, the comatrix representation $\operatorname{Co}(\rho):=\det(\rho)\otimes\rho^*$ also extends to
a monoid homomorphism.}\end{Remark}

\subsubsection{The functions $\Xi_\rho$}\label{here-defi-phi-rho}
For any representation of the first kind $\rho$, we can canonically associate a quasi-periodic function $\Xi_\rho$. This allows to show that, for $L\subset\KK_\Sigma$ a field extension of $\CC_\infty$, the $\KK_{\Sigma}((u))^{N\times N}$-module $\mathcal{QP}_m^!(\rho;\KK_\Sigma)$ is free of rank one. 
Additionally, $\Xi_\rho$ has entries in $\Tamecirc{(\EE_\Sigma[\frac{1}{d}]^\wedge)}$.
Let us first assume that $\rho=\rho_\chi$ is a basic representation.
We denote by $\chi$ the function $\widetilde{\chi}$ of Proposition \ref{firstpropertieschi}. 

By using Proposition  
\ref{firstpropertieschi} and the identity $\chi(z+a)=\chi(z)+\chi(a)$ for $z\in\CC_\infty$ and $a\in A$, we see that the function
\begin{equation}\label{def-xi}
\Xi_\rho(z)=\begin{pmatrix} I_n & \chi(z)\\ 0& I_n\end{pmatrix},\end{equation}
belongs to $\mathcal{QP}^!_{0}(\rho;\EE_\Sigma[\frac{1}{d}]^\wedge)$ (with $d\chi(\theta)\in\FF_q[\underline{t}_\Sigma]\setminus\{0\}$). 
In fact, we have more. Indeed, since $\chi(z+a)=\chi(z)+\chi(a)=\chi(z+a)=\chi(a)+\chi(z)$,
we have $\Xi_\rho(z)=\rho(T_a)\Xi_\rho(z)=\Xi_\rho(z)\rho(T_a)$ for all $a\in A$.

If now $\rho$ is a representation of the first kind, by definition it can be constructed from basic representations $\rho_1,\ldots,\rho_m$ by finitely many iterated 
applications of direct sums, Kronecker products, exterior and symmetric powers, and the  operation of taking the contragredient, and following the same 
process, we can combine the functions $\Xi_{\rho_1},\ldots,\Xi_{\rho_m}$ to construct 
a quasi-periodic matrix function $\Xi_\rho\in\mathcal{QP}^!_{0}(\rho;\widehat{\EE_\Sigma[\frac{1}{d}]})$ for some $d$. More precisely, we set, for $\rho,\psi$ two representations of the first kind:
\begin{eqnarray}
\Xi_{\rho\oplus\psi}&=&\Xi_\rho\oplus\Xi_\psi,\label{procedures}\\
\Xi_{\rho\otimes\psi}&=&\Xi_\rho\otimes\Xi_\psi,\nonumber\\
\Xi_{S^r(\rho)}&=&S^r(\Xi_\rho),\nonumber\\
\Xi_{\wedge^r\rho}&=&\wedge^r\Xi_\rho,\nonumber\\
\Xi_{\rho^*}&=&(\Xi_\rho)^*\nonumber.
\end{eqnarray}
We thus get:
\begin{equation}\label{xi-left-right}
\Xi_\rho(z+a)=\rho(T_a)\Xi_\rho(z)=\Xi_\rho(z)\rho(T_a), \quad a\in A.
\end{equation}
To simplify our notations we write, in the following,
$$\mathfrak{E}:=\EE_\Sigma[d^{-1}]^\wedge,$$ where $\EE_\Sigma[d^{-1}]^\wedge$ has been introduced before Corollary \ref{coroomegachi}.
\begin{Proposition}\label{somepropertiesxirho}
If $\rho$ is a representation of the first kind then we have: (1)
$$\Xi_\rho\in\mathcal{QP}^!_{0}\Big(\rho;\widehat{\EE_\Sigma[d^{-1}]}\Big)\cap(\Tame{\KK_\Sigma})^{N\times N},$$ (2) $\Xi_\rho\in\GL_N(\Tame{\KK_\Sigma})$ and $\Xi_\rho^p=I_N$ and (3)  
$\oplus_{m}\mathcal{QP}^!_{m}(\rho;\KK_\Sigma)\subset\mathfrak{K}_\Sigma^{N\times N}$ 
is both a left and a right $\KK_\Sigma((u))^{N\times N}$-module, free of rank one.
\end{Proposition}
\begin{proof}
The fact that $\Xi_\rho$ is quasi-periodic is clear from (\ref{xi-left-right}).
Moreover, it is easy to see that $\Xi_\rho$ is of type $0$. It suffices to check this for basic representations.
For this note that, for $\nu\in\FF_q^\times$, and for any $\FF_q$-algebra morphism $\chi:A\rightarrow\FF_q(\underline{t}_\Sigma)$, 
$(\begin{smallmatrix}I_n & \chi(\nu z)\\ 0 & I_n\end{smallmatrix})=(\begin{smallmatrix}I_n & \nu\chi( z)\\ 0 & I_n\end{smallmatrix})=
(\begin{smallmatrix}\nu I_n & 0\\ 0 & I_n\end{smallmatrix})(\begin{smallmatrix}I_n & \chi( z)\\ 0 & I_n\end{smallmatrix})(\begin{smallmatrix}\nu^{-1}I_n & 0\\ 0 & I_n\end{smallmatrix})$. But
since $\rho=\rho_\chi$, we have $\rho(\begin{smallmatrix}a & b \\ c & d\end{smallmatrix})=
(\begin{smallmatrix}aI_n & bI_n \\ cI_n & dI_n\end{smallmatrix})$ for all $(\begin{smallmatrix}a & b \\ c & d\end{smallmatrix})\in\operatorname{GL}_2(\FF_q)$, and therefore,
\begin{equation}\label{suffering1}
\Xi_\rho(\nu z)=
\rho\begin{pmatrix} \nu & 0 \\ 0 & 1\end{pmatrix}\Xi_\rho (z)\rho\begin{pmatrix} \nu & 0 \\ 0 & 1\end{pmatrix}^{-1}.
\end{equation}
Additionally, since the entries of the function $\chi$ are tame series in virtue of Proposition \ref{chi-is-tame}, $\Xi_\rho$  is tempered thanks to Proposition \ref{leading}. Now, note that $\det(\Xi_\rho)=1$ due to the fact that this equality holds true for $\rho$ a basic representation.
Hence $\Xi_\rho^{-1}\in (\Tame{\KK_\Sigma})^{N\times N}$ which confirms (1). 
For (2), note that $\Xi_\rho\in\GL_N(\Tame{\KK_\Sigma})$ (with determinant one) and $\Xi_\rho^p=I_N$ for $\rho$ a basic representation, just by construction. The general case follows easily. 
Finally for (3),
note that by (\ref{xi-left-right}), for all $a\in A$,
$$\Xi_\rho(z+a)^{-1}=\Xi_\rho(z)^{-1}\rho(T_a)^{-1}=\rho(T_a)^{-1}\Xi_\rho(z)^{-1}.$$
Let $\Phi$ be an element of $\mathcal{QP}^!_{m}(\rho;\KK_\Sigma)$ for some $m$. Then
$U_1:=\Xi_\rho^{-1}\Phi$ and $U_2:=\Phi\Xi_\rho^{-1}$ are both $A$-periodic and tempered. By Proposition \ref{propositionperiodic} we see that
$U_1,U_2$ belong to $\KK_\Sigma((u))^{N\times N}$. Hence $\Phi=\Xi_\rho U_1=U_2\Xi_\rho\in\mathfrak{K}_\Sigma^{N\times N}$. A simple computation indicates that $U_1,U_2$ are both of type $m$.
\end{proof}

Along with (\ref{procedures}) we also define, with $\rho_\chi:\Gamma\rightarrow\GL_{2n}(\FF_q(\underline{t}_\Sigma))$ a basic representation and $\omega_\chi$ as in (\ref{omegamatrices}):

\begin{eqnarray}
\omega_{\rho_\chi}&=&\Big(\begin{smallmatrix} \omega_\chi & 0_n\\ 0_n & I_n\end{smallmatrix}\Big)\\
\omega_{\rho\oplus\psi}&=&\omega_\rho\oplus\omega_\psi,\label{procedures-omega}\\
\omega_{\rho\otimes\psi}&=&\omega_\rho\otimes\omega_\psi,\nonumber\\
\omega_{S^r(\rho)}&=&S^r(\omega_\rho),\nonumber\\
\omega_{\wedge^r\rho}&=&\wedge^r\omega_\rho,\nonumber\\
\omega_{\rho^*}&=&(\omega_\rho)^*\nonumber.
\end{eqnarray}
This allows to associate, in a unique way, to every representation of the first kind $\rho$ of dimension $N$, 
an element $\omega_\rho\in\GL_N(\LL_\Sigma)$.
We have:
\begin{Lemma}\label{rationality-of-xi}
If $\rho$ is a representation of the first kind there exist $\vartheta_1,\ldots,\vartheta_r\in \FF_q(\underline{t}_\Sigma)$ such that $\omega_\rho\Xi_\rho\omega_\rho^{-1}\in (\Tame{A[\vartheta_1,\ldots,\vartheta_r]})^{N\times N}$.
\end{Lemma}

\begin{proof}
This follows, with $\chi$ basic, from $\Big(\begin{smallmatrix}\omega_\chi & 0_n\\ 0_n & I_n\end{smallmatrix}\Big)
\Big(\begin{smallmatrix}I_n & \chi(z)\\ 0_n & I_n\end{smallmatrix}\Big)\Big(\begin{smallmatrix}\omega_\chi^{-1} & 0_n\\ 0_n & I_n\end{smallmatrix}\Big)=\Big(\begin{smallmatrix}I_n & \omega_\chi\chi(z)\\ 0_n & I_n\end{smallmatrix}\Big)$, and the property that, with $\vartheta=\chi(\theta)$, $\omega_\chi\chi(z)=\exp_C\Big(\widetilde{\pi}z(\theta I_n-\vartheta)^{-1}\Big)=\sum_{i\geq 0}\vartheta^ie_{i+1}\in(\Tame{\FF_q[\vartheta]})^{n\times n}$ and (\ref{chidefi}).
\end{proof}

\subsubsection{The functions $\Phi_\rho$}\label{section-Phi-rho}
Another important class of matrix-valued functions is the following one, that we are going to study now:
$$\Phi_\rho=e_A\Psi_1(\rho),$$ where we recall that $\Psi_1(\rho)=\sum_{a\in A}(z-a)^{-1}\rho(T_a)$, depending on the choice of a representation of the first kind $\rho$.

\begin{Proposition}\label{propgeneralitiestame}
The following properties hold:
\begin{itemize}
\item[(a)] The function $\Phi_\rho$ extends to an entire function $\CC_\infty\rightarrow\mathfrak{E}^{N\times N}$.
\item[(b)] We have that $\Phi_\rho\in\mathcal{QP}^!_0(\rho;\mathfrak{E})$.
\item[(c)] 
There exist two matrices $U_1,U_2\in(\mathfrak{E}[e_C(z)])^{N\times N}$ of type $0$ with $$U_i-I_N\in e_C(z)(\mathfrak{E}[e_C(z)])^{N\times N},\quad i=1,2$$ which are $p$-nilpotent, uniquely determined depending on $\rho$,
such that 
$$\Phi_\rho=U_1\Xi_\rho=\Xi_\rho U_2.$$\item[(d)] We have $\Phi_\rho\in(\Tamecirc{\mathfrak{E}})^{N\times N}$ and this is the unique element $f$ of $(\Tamecirc{\mathfrak{E}})^{N\times N}$ such that $f(a)=\rho(T_a)$ for all $a\in A$.
\end{itemize}
\end{Proposition}

Note that if $\rho=\boldsymbol{1}$ is the trivial representation, with $N=1$, then we have $\Xi_\rho=1$ and $\Phi_\rho=1$, because $\Psi_1(\rho)=\sum_{b\in A}\frac{1}{z-b}$.

\begin{proof}[Proof of Proposition \ref{propgeneralitiestame}]
(a). In any disk $D(0,r)$ with $r\in|\CC_\infty^\times|$, the product $e_A(z)\Psi_1(\rho)(z)$ extends to a holomorphic matrix-valued function
because of the Weierstrass factorization $$e_A(z)=z\prod_{a\in A\setminus\{0\}}\left(1-\frac{z}{a}\right).$$ This immediately implies that $\Phi_\rho$ has entire entries, and the target space is easily determined.

\noindent (b). Since $\rho$ is a representation of the first kind, $\Xi_\rho$ can be constructed applying finitely many operations as in (\ref{procedures}) to finitely many functions $\Xi_{\rho_i}$ associated to 
basic representations $\rho_i$, which take the elements $T_a$ with $a\in A$ to unipotent matrices (in fact, upper triangular with one on the diagonals). Therefore $\Xi_\rho^{-1}$ defines an entire function $\CC_\infty\rightarrow\mathfrak{E}^{N\times N}$. Hence, $U_1(z):=\Phi_\rho(z)\Xi_\rho(z)^{-1}$ has entries which are holomorphic $\Omega\rightarrow\mathfrak{E}^{N\times N}$, and $U_1(z+a)=U_1(z)$ for all $a\in A$, by (\ref{qpbehavior}). Moreover, since $\Xi_\rho$
is tempered and $\|\Psi_1(\rho)(z)\|$ tends to zero as $|z|_\Im=|z|\rightarrow\infty$, there exists $L\in\ZZ$ such that
$u(z)^LU_1(z)\rightarrow\underline{0}$ as $|z|_\Im\rightarrow\infty$. By (b) of Proposition \ref{propositionperiodic}, $U_1$ can be identified with an element of $\mathfrak{E}[[u]][u^{-1}]^{N\times N}$ and we easily check that
$\Phi_\rho\in\mathcal{QP}^!_0(\rho;\mathfrak{E})$. 

\noindent (c). By (2) of Proposition \ref{somepropertiesxirho} we see that $\Xi_\rho\in\GL_N(\Tame{\mathfrak{E}})$ therefore by the arguments of the point (b), we additionally observe that $U_1=\Phi_\rho(z)\Xi_\rho^{-1}\in\mathfrak{E}[[u]][u^{-1}]^{n\times n}$ extends to an entire matrix function which, in virtue 
of (c) of Proposition \ref{propositionperiodic}, belongs to $\mathfrak{E}[e_C(z)]^{N\times N}$. Note that for all $a\in A\setminus\{0\}$, $U_1(a)=\Phi_\rho(a)\Xi_\rho(a)^{-1}=\Phi_\rho(a)\rho(T_{-a})$.
Now, $\Phi_\rho(a)=\lim_{z\rightarrow a}e_A(z)\Psi_1(\rho)(z)=\lim_{z\rightarrow a}e_A(z)(z-a)^{-1}\rho(T_a)=\rho(T_a)$ because $e_A'=1$. 
Hence, $U_1(a)=I_N$ and the various properties claimed for $U_1$ follow. Similar arguments hold for $U_2$. 

\noindent (d). From (c) above, $\Phi_\rho\in(\Tame{\mathfrak{E}})^{N\times N}$. We denote by $w\in\ZZ[\frac{1}{p}]_{\geq 0}$ the supremum of the weights of all the entries of $\Phi_\rho$. 
Then $\Psi_1(\rho)\in u(\Tame{\mathfrak{E}})^{N\times N}$ and since we have the obvious limit $\lim_{|z|_\Im=|z|\rightarrow\infty}\|\Psi_1(\rho)(z)\|=0$ we note that $w<1$ so that $\Phi_\rho\in(\Tamecirc{\mathfrak{E}})^{N\times N}$.
An element $f\in\Tamecirc{\KK_\Sigma}$ satisfies $\|u(z)f(z)\|\rightarrow0$ as $|z|_\Im=|z|\rightarrow\infty$. By Proposition \ref{entire}, for any map $g:A\rightarrow \KK_\Sigma$ there exists at most one element $f\in\Tamecirc{\mathfrak{E}}$ such that $f(a)=g(a)$ for all $a\in A$.
Consequently, if $f$ is an element of $(\Tamecirc{\mathfrak{E}})^{N\times N}$ such that
$f(a)=\rho(T_a)$ for all $a\in A$, then, $\Phi_\rho=f$. 
\end{proof}
We have the next corollary, where $\rho$ is a representation of the first kind.
\begin{Corollary}\label{previouscorollary}
The tame series expansion of $\Phi_\rho$ is provided by the unique representative in 
the $\mathfrak{E}$-module $(\Tamecirc{\mathfrak{E}})^{N\times N}$ of the matrix $\Xi_\rho$ in the quotient of $(\Tame{\mathfrak{E}})^{N\times N}$ by the principal ideal generated by $e_0I_N$.
Moreover, we have $\det(\Phi_\rho)=1$, $\Phi_\rho-I_N$ is $p$-nilpotent and $\Phi_\rho^{-1}\in(\Tamecirc{\mathfrak{E}})^{N\times N}$. If $\omega_\rho$ is the matrix
introduced in (\ref{procedures-omega}), then $\omega_\rho\Phi_\rho\omega_\rho^{-1}\in(\Tamecirc{A[\vartheta]})^{N\times N}$ for an element $\vartheta\in\FF_q[\underline{t}_\Sigma]$.
\end{Corollary}\label{corollarycomputationtame}
\begin{proof}
The first property follows directly from Proposition \ref{propgeneralitiestame} (c), (d). To show the second property we first note that the matrices $\rho(T_a)$, $a\in A$, can be simultaneously (upper) triangularised over an algebraic closure 
$\FF_q(\underline{t}_\Sigma)^{ac}$ of $\FF_q(\underline{t}_\Sigma)$, and the diagonal entries are all equal to one because $T_a^p=I_2$ for all $a$. Hence, $\Psi_1(\rho)$ is conjugated over $\FF_q(\underline{t}_\Sigma)^{ac}$ to an upper triangular matrix having $e_A(z)^{-1}$ as diagonal entries. This implies that $\Phi_\rho$ is conjugated over $\FF_q(\underline{t}_\Sigma)^{ac}$ to an upper triangular matrix having $1$ in the diagonal. Hence,
$\det(\Phi_\rho)=1$, $(\Phi_\rho-I_N)^p=0$ and $\Phi_\rho^{-1}\in(\Tamecirc{\mathfrak{E}})^{N\times N}$. The last property follows easily from Lemma \ref{rationality-of-xi}.
\end{proof}

Let $\chi:A\rightarrow\FF_q(\underline{t}_\Sigma)^{n\times n}$ be an $\FF_q$-algebra morphism 
and denote by $\rho$ the basic representation $\rho_\chi:\Gamma\rightarrow\operatorname{GL}_N(\FF_q(\underline{t}_\Sigma))$ defined by $\rho(\begin{smallmatrix}a & b \\ c & d\end{smallmatrix})=(\begin{smallmatrix}\chi(a) & \chi(b) \\ \chi(c) & \chi(d)\end{smallmatrix})$, with $N=2n$. For a matrix $f\in\mathfrak{K}_\Sigma^{N\times N}$, $v(f)$ denotes the infimum of the $v$-valuations 
of the entries of $f$ (where $v$ is the valuation defined after Proposition \ref{descriptionfieldofunif}).

\begin{Corollary}\label{caseofthebasicrepresentation}
We have $\Phi_\rho=\Xi_\rho$, $v(\Phi_\rho)=-\frac{1}{q}$ and $v(\Phi_\rho-\omega_{\chi}^{-1}e_1)>-\frac{1}{q}$.
\end{Corollary} 

\begin{proof}
By definition, $\Xi_\rho=(\begin{smallmatrix}I_n & \chi \\ 0 & I_n\end{smallmatrix})$
and $\chi(z)=e_C(z(\theta I_n-\vartheta)^{-1})\omega_\chi^{-1}$ (with $\vartheta=\chi(\theta)$) has entries in $\Tamecirc{\KK_\Sigma}$ so we have already $\Phi_\rho=\Xi_\rho$ by Corollary \ref{previouscorollary}.
Moreover, the tame series expansion of $e_C(z(\theta I_n-\vartheta)^{-1})$ is
$e_C(z(\theta I_n-\vartheta)^{-1})=e_1I_n+$terms of smaller weight, which implies the remaining properties.
\end{proof}

\subsubsection{Application to column quasi-periodic functions}

We consider, in this subsection, a representation of the first kind $\Gamma\xrightarrow{\rho}\operatorname{GL}_N(\FF_q(\underline{t}_\Sigma)).$ Recall the notation $\mathfrak{K}_\Sigma=
\mathfrak{K}_{\KK_\Sigma}$ where, for a subfield $L$ of $\KK_\Sigma$, $\mathfrak{K}_L$ has been defined after Proposition \ref{descriptionfieldofunif}. We recall that the $v$-valuation ring is denoted by $\mathfrak{O}_\Sigma$, the maximal ideal is denoted by $\mathfrak{M}_\Sigma$, and the residual field is denoted by $\FF_q(\underline{t}_\Sigma)$. 

\begin{Proposition}\label{quasiperiodictempered}
If $f:\Omega\rightarrow\KK_\Sigma^{N\times 1}$ is $\rho$-quasi-periodic and tempered, we can identify it with an element of $\mathfrak{K}_\Sigma^{N\times 1}$. If additionally $f$ is regular, then we can identify it with an element of $\mathfrak{O}_\Sigma^{N\times 1}$. In the latter case, we can expand
in a unique way
\begin{equation}\label{expansion-ff}
f=f_0+\sum_{i>0}f_iu^i,\quad f_0\in\KK_\Sigma^{N\times 1},\quad f_i\in(\Tamecirc{\KK_\Sigma})^{N\times 1},\quad i>0,\end{equation}
and the coefficients $f_i$ are $\KK_\Sigma$-linear combinations of the columns of $\Phi_\rho$.
\end{Proposition}

\begin{proof}
In the proof of part (c) of Proposition \ref{propgeneralitiestame}, we have seen that $\Phi_\rho$ can be identified with an element of $\operatorname{GL}_N(\Tame{\KK_\Sigma})$. Hence, the function $\Phi_\rho^{-1}f:\Omega\rightarrow\KK_\Sigma^{N\times 1}$ has entries which are all $A$-periodic and tempered. By part (b) of Proposition \ref{propositionperiodic}, the entries are thus elements of $\KK_\Sigma((e_0^{-1}))=\KK_\Sigma((u))$ and the entries of $f=\Phi_\rho\Phi_\rho^{-1}f$ are therefore in $\Tamecirc{\KK_\Sigma}((e_0^{-1}))$ which is equal, by Proposition \ref{descriptionfieldofunif}, to $\mathfrak{K}_\Sigma$. This proves the first part of the proposition.

Since $\Phi_\rho$ is a matrix function which is quasi-periodic we have $f=\Phi_\rho g$ where $g\in\KK_\Sigma((u))^{N\times 1}$. Corollary \ref{previouscorollary} implies that $\Phi_\rho\in\GL_N(\Tamecirc{\KK_\Sigma})$. Namely, $\det(\Phi_\rho)=1$ and $\Phi_\rho^{-1}\in(\Tamecirc{\KK_\Sigma})^{N\times N}$.
Observe that $g=\Phi_\rho^{-1}f$. Since the entries of $\Phi_\rho^{-1}$ are in $\Tamecirc{\KK_\Sigma}$,
for $|z|_\Im\geq c_1$ for some constant $c_1\in|\CC_\infty^\times|$, we have 
$\|\Phi_\rho^{-1}f\|\leq c_2|e_A(z)|^w$ by Proposition \ref{leading}, where $w\in\ZZ[\frac{1}{p}]\cap[0,1[$, for some $c_2>0$. This means that $\|u^wg\|\leq c_2$ as $|z|_\Im$ is large. Let $\alpha>0$ be such that 
$p^\alpha w\in\ZZ$. Then $\|u^{p^\alpha w}g^{p^\alpha}\|$ is bounded at infinity and $u^{p^\alpha w}g^{p^\alpha}\in\KK_\Sigma((u))^{N\times 1}$. Therefore, $u^{w}g\in\KK_\Sigma[[u^{\frac{1}{p^\alpha}}]]^{N\times 1}$ by Proposition \ref{propgeneralitiestame} (b) and we deduce that, necessarily, $g\in\KK_\Sigma[[u]]^{N\times 1}$. Hence, 
$f=\Phi_\rho g\in\mathfrak{O}_\Sigma^{N\times 1}$. This yields (\ref{expansion-ff}) because
$f_0\in\KK_\Sigma^{N\times 1}$, if non-vanishing, has weight zero. Considering again
$$g=\Phi_\rho^{-1}f=\Phi_\rho^{-1}f_0+\sum_{i>0}\Phi_\rho^{-1}f_iu^i\in\KK_\Sigma[[u]]^{N\times 1}$$
The coefficients $\Phi_\rho^{-1}f_i$ are in $(\Tamecirc{\KK_\Sigma})^{N\times 1}$ and $A$-periodic, hence in $\KK_\Sigma^{N\times 1}$. 
\end{proof}

Taking into account Definition \ref{modularform}, we deduce parts (1), (2), (3) of Theorem A in the introduction, where the hypothesis that $\rho$ is of the first kind is essential:
 \begin{Theorem}\label{theorem-u-expansions}
For all $w\in\ZZ$, there is a natural embedding $$M^!_w(\rho;\KK_\Sigma)\xrightarrow{\iota_\Sigma}\mathfrak{K}_\Sigma^{N\times 1}$$ such that $$M_w(\rho;\KK_\Sigma)=\iota_\Sigma^{-1}\Big(\iota_\Sigma(M^!_w(\rho;\KK_\Sigma))\cap\mathfrak{O}_\Sigma^{N\times 1}\Big)$$ and 
$$S_w(\rho;\KK_\Sigma)=\iota_\Sigma^{-1}\Big(\iota_\Sigma(M^!_w(\rho;\KK_\Sigma))\cap\mathfrak{M}_\Sigma^{N\times 1}\Big).$$
 \end{Theorem}
 \begin{proof}
 Since a weak modular form is also a tempered quasi-periodic (column) function and a modular form is 
 a regular quasi-periodic function, the first part of the result follows directly from Proposition \ref{quasiperiodictempered}.
 To prove the two other parts of the statement, namely the characterisation of the image of
 $M_w(\rho;\KK_\Sigma)$ and $S_w(\rho;\KK_\Sigma)$, we combine Proposition \ref{descriptionfieldofunif} with Proposition \ref{leading},
which allows to derive, from the fact that $f$ is bounded at infinity (resp. has zero limit at infinity)
that valuations of the entries of $f$ are non-negative (resp. positive).
 \end{proof}

\subsection{Hecke operators}\label{section-Hecke-operators}

We show here part (5) of Theorem A in the introduction.
As an immediate consequence of the above investigations, we will now define Hecke operators 
acting on the spaces $M_w(\rho;\KK_\Sigma),M_w(\rho;\LL_\Sigma),S_w(\rho;\KK_\Sigma)$ and $S_w(\rho;\LL_\Sigma)$, with $w\in\ZZ,$ when $\Gamma\xrightarrow{\rho}\operatorname{GL}_N(\FF_q(\underline{t}_\Sigma))$ is a representation of the first kind. Although not explicitly considered in the general purposes of it, Miyake's book \cite{MIY} essentially contains everything we need to set up the basis of the present discussion. Following \cite[\S 2.7 and \S 4.5]{MIY} we consider the Hecke algebra 
$\mathcal{R}_{A}(\Gamma,\Delta)$ where $\Delta=(\begin{smallmatrix} * & * \\ 0 & *\end{smallmatrix})\cap A^{2\times 2}\cap\GL_2(K)$ is the semigroup generated by the elements 
of $G=\GL_2(K)$ with entries in $A$ and with the lower left coefficient equal to zero. Explicitly,
$\mathcal{R}_{A}(\Gamma,\Delta)$ is the free $A$-module generated by the double cosets
$\Gamma \delta\Gamma$ with $\delta$ in $\Delta$, endowed with the structure of $A$-algebra
induced by ibid. (2.7.2), after reduction modulo $p$ of the integral coefficients. It is easy to see, using \cite[Theorem 2.7.8]{MIY}, that 
$\mathcal{R}_{A}(\Gamma,\Delta)$ is commutative. For $a\in A$, we 
set $T(a)=\Gamma(\begin{smallmatrix} 1 & 0 \\ 0 & a\end{smallmatrix})\Gamma\in\mathcal{R}_{A}(\Gamma,\Delta)$. The proof of ibid. Lemma 4.5.7
can be easily modified to show that, if $P\in A$ is irreducible, then
$$T(P)T(P^n)=T(P^{n+1})+q^{\deg_\theta(P)}T(P,P)T(P^{n-1}),\quad n\geq1,$$
where $T(P,P)=\Gamma(\begin{smallmatrix} P & 0 \\ 0 & P\end{smallmatrix})\Gamma$ (compare with ibid. (4.5.15)). But $K$ has characteristic $p\mid q$ so that $T(P)T(P^n)=T(P^{n+1})$. Similarly, the proof of Lemma 4.5.8
in Miyake's book implies that if $a,b\in A$ are relatively prime, then $T(a)T(b)=T(ab)$
in $\mathcal{R}_{A}(\Gamma,\Delta)$. The map $A\mapsto T(a)$ is therefore totally multiplicative.
Also, given any right action of $\Delta$ on a $B$-module $\mathcal{M}$, $\mathcal{R}(\Gamma,\Delta)$ acts on $\mathcal{M}^\Gamma=\{m\in \mathcal{M}:m|\gamma=m,\forall \gamma\in \Gamma\}$, as described in \cite[Lemma 2.7.2]{MIY}, where we denoted by $m|\gamma$ the action of $\gamma$ on $m$.

We consider 
$\rho:\Gamma\rightarrow\operatorname{GL}_N(\FF_q(\underline{t}_\Sigma))$ a representation of the first kind.
Then, $\rho$ can be extended in a unique way to a faithful representation of $G=\GL_2(K)$
in $\GL_2(K)$ and there exists $d\in \FF_q[\underline{t}_\Sigma]\setminus\{0\}$ such that
$\rho(\Delta)\subset\FF_q[\underline{t}_\Sigma][d^{-1}]$.
Let $w$ be an integer and $B$ a $\CC_\infty$-algebra contained in $\KK_\Sigma$ such that
it contains $\TT_\Sigma[d^{-1}]^\wedge$.
We set $\mathcal{M}_B:=\operatorname{Hol}_{\KK_\Sigma}(\Omega\rightarrow B^{N\times 1})$. Let $f$ be in $\mathcal{M}_B$.
The {\em Petersson slash operator} $f|_{w,\rho}\gamma$ on $f$ is defined, for any $\gamma\in\GL_2(K)$, by 
\begin{equation}\label{Petersson}
(\mathcal{G}|_{w,\rho}\gamma)(z):=J_\gamma(z)^{-w}\rho(\gamma)^{-1}\mathcal{G}(\gamma(z)).\end{equation}
It is easily seen that this gives rise to an action of $\Delta$ over 
$\mathcal{M}_w(\rho;B)$, the $B$-module of the modular-like functions of weight $w$ for $\rho$ of 
Definition \ref{modularform}. For instance, the reader can easily check that 
$(f|_{w,\rho}\gamma)|_{w,\rho}\delta=f|_{w,\rho}\gamma\delta$ for any $\gamma,\delta\in\GL_2(K)$. By the above discussion, we have a well defined $\mathcal{R}_A(\Gamma,\Delta)$-module structure on $\mathcal{M}_w(\rho;B)$. If $\Gamma \delta\Gamma$ is a double coset in $\mathcal{R}_A(\Gamma,\Delta)$ we can expand in a finite sum $\Gamma \delta\Gamma=\sum_i\Gamma \delta_i$
with $\delta_i\in\Delta$ for all $i$ as depicted in \cite[Lemma 2.7.3]{MIY} and the action is given by $$(\Gamma \delta\Gamma,f)\mapsto \sum_if|_{w,\rho}\delta_i.$$ We also denote by 
$T_a(f)$ the image of the action of $T(a)$
on $f$, with $a\in A$. Then, $T_{a}(T_b(f))=T_{ab}(f)$ for all $a,b\in A$. 
For example, since for $P\in A$ irreducible,
$$T(P)=\Gamma\begin{pmatrix} P & 0 \\ 0 & 1\end{pmatrix}\sqcup\bigsqcup_{\begin{smallmatrix}b\in A\\ |b|<|P|\end{smallmatrix}}\Gamma\begin{pmatrix} 1 & b \\ 0 & P\end{pmatrix}$$
(see very similar computations in \cite[Lemma 4.5.6]{MIY}), we have, for $f\in\mathcal{M}_w(\rho;B)$:
\begin{equation}\label{TP}
T_P(f)(z)=\rho\begin{pmatrix} P & 0 \\ 0 & 1\end{pmatrix}^{-1}f(Pz)+P^{-w}\sum_{|b|<|P|}\rho\begin{pmatrix} 1 & b \\ 0 & P\end{pmatrix}^{-1}f\left(\frac{z+b}{P}\right),\quad z\in \Omega.\end{equation}
Comparing with \cite[(7.1)]{GEK} we have here a different normalisation for these operators. In the case
of $\rho=\boldsymbol{1}$ so that $N=1$, denoting by $\mathcal{T}_P$ the weight $w$ operator of ibid., 
we have $T_P=P^{-w}\mathcal{T}_P$.

The following result holds:

\begin{Theorem}\label{theo-hecke-operators}
Assuming that $\rho$ is of the first kind, we have that for all $a\in A$ and $w\in \ZZ$, $T_a$ defines a $B$-linear endomorphism of $M^!_w(\rho;B)$ which induces endomorphisms of $M_w(\rho;B)$ and $S_w(\rho;B)$.
\end{Theorem}

\begin{proof}
Thanks to the above observations it suffices to prove the result for $a=P$ irreducible. Lemma
\ref{aendomorphisms} implies that $T_P$ operates, via the slash operator of weight $w$ associated with $\rho$, on $\mathfrak{K}_\Sigma^{N\times 1}$ and furthermore, it leaves 
$\mathfrak{O}_\Sigma^{N\times 1}$ and $\mathfrak{M}_\Sigma^{N\times 1}$ invariant. 
\end{proof}
This generalizes \cite[Proposition 5.12]{PEL&PER3} (which deals with the very special case of $N=2$ and $\rho=\rho_t^*$, with an ad hoc proof hard to generalize to our settings).

\subsubsection*{Example} Assume that $\rho=\rho_\Sigma^*={}^t\rho_\Sigma^{-1}$ for a finite subset $\Sigma$ of $\NN^*$ and consider
$f={}^{t}(f_1,\ldots,f_N)\in M_w(\rho;B)$. Then the first entry $(T_P(f))_1$ in
(\ref{TP}) is 
\begin{equation}\label{Hecke-first-entry}
(T_P(f))_1=\sigma_\Sigma(P)f_1(Pz)+P^{-w}\sum_{|b|<|P|}f_1\left(\frac{z+b}{P}\right).\end{equation}
The last entry is also interesting but slightly more involved. We have:
\begin{equation}\label{last-entry}
(T_P(f))_N=f_N(Pz)+P^{-w}\sum_{|b|<|P|}\Bigg(\bigotimes_{i\in\Sigma}\Big(\chi_{t_i}(b),\chi_{t_i}(P)\Big)\Bigg)\cdot f\left(\frac{z+b}{P}\right).
\end{equation} Note that the whole column vectors $f(\frac{z+b}{P})$ occur in the right-hand side.

\section{Structure results for modular forms}\label{structureofvmf}

We consider, in this section, a representation $$\Gamma\xrightarrow{\rho}\operatorname{GL}_N(\FF_q(\underline{t}_\Sigma)).$$ 
We recall that $M^!_w(\rho;\LL_\Sigma)$,  $M_w(\rho;\LL_\Sigma)$, $S_w(\rho;\LL_\Sigma)$ denote respectively, the $\LL_\Sigma$-vector spaces of weak modular forms, modular forms, and cusp forms in $\operatorname{Hol}_{\KK_\Sigma}(\Omega\rightarrow\LL_\Sigma^{N\times 1})$ of weight $w$ for $\rho$ (in the sense of Definition
\ref{modularform}), so that $S_w(\rho;\LL_\Sigma)\subset M_w(\rho;\LL_\Sigma)\subset M_w^!(\rho;\LL_\Sigma)$. 
The operator $\tau$ induces 
$\FF_q(\underline{t}_\Sigma)$-linear injective maps
$$M_{w}(\rho;\LL_\Sigma)\xrightarrow{\tau}M_{qw}(\rho;\LL_\Sigma),$$
and similarly for $M^!_{w}(\rho;\LL_\Sigma),S_{w}(\rho;\LL_\Sigma)$ etc.
Of course, this depends on the choice of $\Sigma$. To simplify, we will sometimes also write $M_w(\rho)$ for $M_w(\rho;\LL_\Sigma)$ etc. when the reference to the field $\LL_\Sigma$ is clear. The next sub-section also allows to justify this abuse of notation.

\subsection{Changing the coefficient field}\label{changing}

We have defined, for $\rho:\Gamma\rightarrow\GL_N(\FF_q(\underline{t}_\Sigma))$ a representation,
the $\LL_\Sigma$-vector space of modular forms $W_w(\rho;\LL_\Sigma)$ and the $\KK_\Sigma$-vector space of modular forms $W_w(\rho;\KK_\Sigma)$ (with $W$ a symbol such that $W\in\{M^!,M,S\}$). Let $\Sigma'$ be finite such that $\Sigma\subset\Sigma'\subset\NN^*$. Then,
we also have the spaces $W_w(\rho;\LL_{\Sigma'})$ and $W_w(\rho;\KK_{\Sigma'})$. The next result allows to compare these spaces for $\Sigma'\supset\Sigma$. It is important in that 
it confirms that there are bases of these spaces which depend on the representation only. The notation $W_w$ stands for $M^!_w,M_w,S_w$ respectively.

\begin{Proposition}\label{proposition-scalars} Assuming that $\rho$ is of the first kind we have that
$$W_w(\rho;\KK_{\Sigma'})=W_w(\rho;\KK_{\Sigma})\widehat{\otimes}_{\KK_{\Sigma}}\KK_{\Sigma'}$$ where $\widehat{\otimes}_{\KK_{\Sigma}}$ means that
every element $f$ of $W_w(\rho;\KK_{\Sigma'})$ can be expanded as a series 
$f=\sum_{i}a_if_i$ where $a_i\in\KK_{\Sigma'}$, $f_i\in W_w(\rho;\KK_{\Sigma})$ for all $i$, and 
$a_if_i\rightarrow0$ for the supremum norm of every affinoid subdomain of $\Omega$.
Moreover, If $\dim_{\LL_\Sigma}(M_w(\rho;\LL_\Sigma))<\infty$, then
$$M_w(\rho;\LL_{\Sigma'})=M_w(\rho;\LL_\Sigma)\otimes_{\LL_\Sigma}\LL_{\Sigma'},\quad S_w(\rho;\LL_{\Sigma'})=S_w(\rho;\LL_\Sigma)\otimes_{\LL_\Sigma}\LL_{\Sigma'}.$$
\end{Proposition}

\begin{proof} 
Let $(b_i)_{i\in I}$ be a basis of $\FF_q^{ac}(\underline{t}_{\Sigma'})$ over $\FF_q^{ac}(\underline{t}_{\Sigma})$. By Lemma \ref{lemma-di-serre}, $(b_i)_{i\in I}$ is an orthonormal basis. In other words, every element $\kappa\in\KK_{\Sigma'}$ can be expanded, in a unique way, as a series $\kappa=\sum_{i}\kappa_ib_i$ with $\kappa_i\in\KK_\Sigma$ such that $\kappa_i\rightarrow0$. Let us choose a basis $(c_j)_{j\in J}$ of 
$\FF_q^{ac}(\underline{t}_{\Sigma})$ over $\FF_q^{ac}$ so that $(b_ic_j)_{i,j}$ is an orthonormal 
basis of $\KK_{\Sigma'}$ over $\CC_\infty$.
Now consider $f\in W_w(\rho;\KK_{\Sigma'})$. We can expand 
$$f(z)=\sum_{i,j}f_{i,j}(z)b_ic_j$$
where $f_{i,j}\in \operatorname{Hol}_{\CC_\infty}(\Omega\rightarrow\CC_\infty)$ for all $i,j$,
with the property that $f_{i,j}\rightarrow0$ with respect to the supremum norm relative to any choice of an affinoid subdomain of $\Omega$. Observe that
$$f(\gamma(z))=J_\gamma(z)^w\sum_i\left(\rho(\gamma)\sum_jf_{i,j}(z)c_j\right)b_i,$$
because $f$ is modular-like.
Since $\rho(\gamma)\sum_jf_{i,j}(z)c_j\in \KK_\Sigma$ and $(b_i)_i$ is an orthonormal 
basis of $\KK_{\Sigma'}$ over $\KK_\Sigma$, we deduce that for all $i\in I$, setting 
$f_i=\sum_jf_{i,j}(z)c_j$,
$$f_i(\gamma(z))=J_\gamma(z)^w\rho(\gamma)f_i(z),$$
and one sees that $f_i\in W_w(\rho;\KK_{\Sigma})$ from which the first part of the Proposition follows.

The proof of the second part of the proposition is similar but we restrict to $W_w=M_w,S_w$. First notice that by Lemma \ref{precise-image} which can be applied to $B=\LL_{\Sigma'}$ (it satisfies the conditions at the beginning of \S \ref{structure-XBC}),
if $f\in W_w(\rho;\LL_{\Sigma'})$ then there exists $d'\in\FF_q[\underline{t}_{\Sigma'}]\setminus\{0\}$ such that $f\in W_w(\rho;\TT_{\Sigma'}[\frac{1}{d'}]^\wedge)$. We can even choose $d,d'$ with
$d\in \FF_q[\underline{t}_{\Sigma'}]\setminus\{0\}$ such that $d\mid d'$ and such that the image 
of $\rho$ is contained in $\GL_{N}(\FF_q[\underline{t}_{\Sigma}][\frac{1}{d}])$. The proof of the first part of the proposition can be modified to obtain that 
$f$ can be expanded as a series $f=\sum_ka_kf_k$ where $a_k\in\TT_{\Sigma'}[\frac{1}{d'}]^\wedge$ and $f_i\in W_w(\rho;\TT_{\Sigma}[\frac{1}{d}])$,
and $a_if_i\rightarrow0$ for the supremum norm associated to any affinoid subset of $\Omega$. If now $\dim_{\LL_{\Sigma}}M_w(\rho;\LL_{\Sigma})$, we deduce the result. We have excluded $W_w=M^!_w$ because in general, 
$\dim_{\LL_{\Sigma}}M^!_w(\rho;\LL_{\Sigma})=\infty$.
\end{proof}

\subsection{Finiteness results}\label{structure-tate}

In this subsection we suppose that the representation $\rho:\Gamma\rightarrow\GL_N(\FF_q(\underline{t}_\Sigma))$ is of the first kind.
We also recall that $\mathfrak{K}_\Sigma$ is the completion of the fraction field of $\Tame{\KK_\Sigma}$ for the valuation $v$, and that $\mathfrak{O}_\Sigma$, $\mathfrak{M}_\Sigma$ are respectively the valuation ring and the maximal ideal of $v$.
 We have the following results which correspond to part (1) of Theorem B in the introduction:

\begin{Theorem}[Finiteness Theorem]\label{theorem-structure-tate-algebras}
The $\LL_\Sigma$-vector space $M_w(\rho;\LL_\Sigma)$ has finite dimension $r_\rho(w)$ and we have 
$r_\rho(w)\leq(1+\lfloor\frac{w}{q+1}\rfloor)N$ if $q>2$ and $r_\rho(w)\leq2(1+\lfloor\frac{w}{q+1}\rfloor)N$ if $q=2$.
\end{Theorem}

In particular, if $w<0$, then $r_\rho(w)=0$ and $M_w(\rho;\LL_\Sigma)=\{0\}$ but this property will be actually proved separately to obtain the general result. The proof of this theorem
makes use of an important feature of our Drinfeld modular forms when they take values in $\LL_\Sigma$; the possibility of evaluating the variables $t_i$ ($i\in\Sigma$) at roots of unity. This will the subject of the next subsection. In \S \ref{Proof-Finiteness-Theorem} we prove Theorem \ref{theorem-structure-tate-algebras} by using that the spaces of modular forms of negative weight are trivial.
This is a consequence of the fact that classical negative weight (scalar) Drinfeld modular forms for congruence subgroups of $\Gamma$ are zero. The upper bound for the dimensions 
in Theorem \ref{theorem-structure-tate-algebras} can be slightly refined, but our methods do not allow an explicit computation.

\subsubsection{Evaluating at roots of unity}\label{evaluation-banach}
The representation of the first kind $\rho$ is constructed starting from a finite set of basic representations $\rho_i$ associated with injective $\FF_q$-algebra morphisms
$\chi_i:A\rightarrow \FF_q(\underline{t}_\Sigma)$ ($i=1,\ldots,r$). If $d_1,\ldots,d_r\in\FF_q[\underline{t}_\Sigma]\setminus\{0\}$ are such that the entries of $d_i\chi_i(\theta)$ are in $\FF_q[\underline{t}_\Sigma]$ then the image of $\rho$ is in $\GL_N(\FF_q[\underline{t}_\Sigma][\frac{1}{d_1},\ldots,\frac{1}{d_r}])\subset\GL_N(\FF_q[\underline{t}_\Sigma][\frac{1}{d}])$ for some $d\in\FF_q[\underline{t}_\Sigma]\setminus\{0\}$. We thus get, after Proposition \ref{propgeneralitiestame}, that $$\Xi_\rho,\Phi_\rho\in\operatorname{Hol}(\CC_\infty\rightarrow\widehat{\EE_\Sigma[d^{-1}]}^{N\times N}).$$

Let $\Sigma=U\sqcup V$ be a finite subset of $\NN^*$ written as a disjoint union of subsets $U,V$, with $U$ non-empty.
The set $$\mathcal{V}_U(d)=\{\underline{\zeta}\in(\FF_q^{ac})^U:d(\underline{\zeta})=0\}$$ is contained in a proper hypersurface of $(\FF_q^{ac})^{U}$ and therefore, $(\FF_q^{ac})^U\setminus\mathcal{V}_U(d)$ is Zariski-dense in $(\FF_q^{ac})^{U}$. Let $\underline{\zeta}=(\zeta_i:i\in U)$ be an element of $(\FF_q^{ac})^U\setminus\mathcal{V}_U(d)$.

The {\em evaluation map}
$$\operatorname{ev}_{\underline{\zeta}}:\widehat{\TT_\Sigma[d^{-1}]}\rightarrow\widehat{\TT_V[\operatorname{ev}_{\underline{\zeta}}(d)^{-1}]}$$ is the $\TT_V$-algebra morphism uniquely determined by the assignment $t_i\mapsto\zeta_i$
for $i\in U$. If there is no possibility of confusion, we write $f(\underline{\zeta})$ in place of $\operatorname{ev}_{\underline{\zeta}}(f)$.
We extend this map to matrices with entries in $\TT_\Sigma[d^{-1}]^\wedge$. It is easy to see that if $X$ is an analytic space over $\CC_\infty$ and $f\in\operatorname{Hol}(X\rightarrow\TT_\Sigma[d^{-1}]^\wedge)$, then $\operatorname{ev}_{\underline{\zeta}}(f)\in \operatorname{Hol}(X\rightarrow\TT_V[d(\underline{\zeta})^{-1}]^\wedge)$. Moreover:

\begin{Lemma}\label{locally-constant-implies-constant} 
Let $X$ be a rigid analytic space over $\CC_\infty$.
If $f\in\operatorname{Hol}(X\rightarrow\TT_\Sigma[\frac{1}{d}]^\wedge)$ and 
if for all $\underline{\zeta}\in(\FF_q^{ac})^U\setminus\mathcal{V}_U(d)$, $\operatorname{ev}_{\underline{\zeta}}(f)\in
\operatorname{Hol}(X\rightarrow\TT_V[\frac{1}{d(\underline{\zeta})}]^\wedge)$ is constant, then $f$ is constant.
\end{Lemma}

\begin{proof}
It is enough to prove the result for $X=\operatorname{Spm}(\mathcal{A})$ affinoid. By Lemma \ref{affinoid-is-cartesian} we can choose an orthonormal basis $(a_i)_{i\in I}$ of the Banach $\CC_\infty$-algebra $\mathcal{A}$. 
We can even assume, without loss of generality, that $a_{i_0}$ is the constant function equal to one for an index $i_0\in I$. Then, for all $i\neq i_0$, $a_i$ is non-constant over $X$.
We can expand every element 
$f$ of $\operatorname{Hol}(X\rightarrow\TT_\Sigma[\frac{1}{d}]^\wedge)$ as
$f=\sum_{i\in I}f_ia_i$ with $f_i\in \TT_\Sigma[\frac{1}{d}]^\wedge,$
where the series converges for the supremum norm of $X$. Hence,
$$\operatorname{ev}_{\underline{\zeta}}(f)=\sum_{i\in I}\operatorname{ev}_{\underline{\zeta}}(f_i)a_i,$$ and
$\operatorname{ev}_{\underline{\zeta}}(f_i)=0$ for all $i\neq i_0$. Since 
this happens for all $\underline{\zeta}\in(\FF_q^{ac})^U\setminus\mathcal{V}_U(d)$ which is Zariski-dense, we obtain $f_i=0$ for all $i\neq i_0$ and $f$ is constant over $X$.
\end{proof}

Let $\mathfrak{n}$ be a non-zero ideal of $A$.
We denote by $\Gamma(\mathfrak{n})$
the associated principal congruence subgroup of $\Gamma$:
$$\Gamma(\mathfrak{n})=\{\gamma\in\Gamma:\gamma\equiv(\begin{smallmatrix}1 & 0 \\ 0 & 1\end{smallmatrix})\pmod{\mathfrak{n}}\}.$$
We recall that $\rho:\Gamma\rightarrow\GL_N(\FF_q[\underline{t}_\Sigma][d^{-1}])$ is a representation of the first kind.
\begin{Lemma}\label{lemma-reduction-n}
Let $\underline{\zeta}=(\zeta_i:i\in\Sigma)$ be an element of $(\FF_q^{ac})^\Sigma\setminus\mathcal{V}_\Sigma(d)$. There exists a non-zero ideal $\mathfrak{n}$ of $A$ such that for all $\gamma\in\Gamma(\mathfrak{n})$,
$\operatorname{ev}_{\underline{\zeta}}\Big(\rho(\gamma)\Big)=I_N$.
\end{Lemma}

\begin{proof}
There exist basic representations $\rho_{\chi_1},\ldots,\rho_{\chi_r}$, associated to $\FF_q$-algebra morphisms $\chi_i:A\rightarrow \FF_q(\underline{t}_\Sigma)^{n_i\times n_i}$ ($i=1,\ldots,r$) such that $\rho$ can be constructed applying 
admissible operations finitely many times (as in Definition \ref{firstkindrepresentations}). We fix $\underline{\zeta}\in(\FF_q^{ac})^\Sigma\setminus\mathcal{V}_\Sigma(d)$ where $d\in\FF[\underline{t}_\Sigma]\setminus\{0\}$ is such that 
$d\chi_i(\theta)\in\FF_q[\underline{t}_\Sigma]^{n_i\times n_i}$. 
We denote by 
$\mathfrak{n}$ the ideal generated by $P_1|_{X=\theta},\ldots,P_r|_{X=\theta}\in A\setminus\{0\}$, where $P_i\in\FF_q[X]$ is  the minimal polynomial
 of $\eta_i=\chi_i(\theta)|_{\underline{t}_\Sigma=\underline{\zeta}}$ (for all $i$), which are well defined. Then, if $a\in\mathfrak{n}$, we have 
 $\operatorname{ev}_{\underline{\zeta}}(\chi_{t_i}(a))=0$ for all $i$ so that
 $\operatorname{ev}_{\underline{\zeta}}(\rho(\gamma))=I_N$ for all $\gamma\in\Gamma(\mathfrak{n})$ due to the fact that
 the admissible operations (which construct all the representations of the first kind) stabilise the set of identity matrices.
\end{proof}

We now introduce a slightly more general notion of vector-valued modular form for a congruence subgroup of $\Gamma$.
Let $G$ be a congruence subgroup of $\Gamma$. The quotient space $G\backslash\Omega$ carries a natural structure of analytic curve $Y_G$ with compactification $X_G$ obtained by adding finitely many points to $Y_G$ called {\em cusps}. We can consider neighbourhoods of a cusp of $G\backslash\Omega$ in $\Omega$ in the usual way and therefore, there is a natural notion of modular-like forms $f:\Omega\rightarrow \LL_\Sigma^{N\times 1}$ of weight $w$ for $\rho$, seen as a representation of $G$ by restriction, namely, satisfying the collection of 
functional equations 
\begin{equation}\label{functional-bis}
f(\gamma(z))=J_\gamma(z)^w\rho(\gamma)f(z)\quad \forall z\in\Omega,\quad \forall \gamma\in G.\end{equation}   Let $c$ be a cusp of $X_G$ and let us consider $\delta\in\Gamma$
such that $\delta(\infty)=c$. If $f:\Omega\rightarrow \LL_\Sigma^{N\times 1}$ is a map and $w$ an integer, we set
$$f^\delta(z):=f|_{w,\rho}\delta=J_{\delta}(z)^{-w}\rho(\delta)^{-1}f(\delta(z))$$ (Petersson slash operator as in (\ref{Petersson})).
A simple computation shows that if $f$ is modular-like of weight $w$ for the restriction $\rho|_G$ of $\rho$ on $G$, then $f^\delta:\Omega\rightarrow \LL_\Sigma^{N\times 1}$ is modular-like of weight $w$ for $\rho|_{G^\delta}$ where $G^\delta:=\delta^{-1}G\delta$ (in particular, if $f$ is modular-like for
the group $\Gamma$, then $f=f^\delta$).

\begin{Definition}\label{refined-modularity} Let $w$ be in $\ZZ$.
{\em We say that a modular-like function $\Omega\xrightarrow{f} \LL_\Sigma^{N\times 1}$ 
of weight $w$ for $\rho|_G$ is:
\begin{enumerate}
\item A {\em weak Drinfeld modular form of weight $w$ for $\rho|_G$} if there exists $H\in\ZZ$ such that $$\|u(z)^Hf^\delta(z)\|\rightarrow0$$ as $z\in\Omega$ is such that $|z|=|z|_\Im\rightarrow\infty$, and this, for 
all $\delta\in\Gamma$.
\item A {\em Drinfeld modular form of weight $w$ for $\rho|_G$,} if $\|f^\delta(z)\|$ is bounded as $|u(z)|<c$
for some constant $c<1$, for all $\delta\in\Gamma$.
\item A {\em cusp form of weight $w$ for $\rho|_G$} if $\|f^\delta(z)\|\rightarrow0$ as $z\in\Omega$ is such that $|z|=|z|_\Im\rightarrow\infty$ for 
all $\delta\in\Gamma$.
\end{enumerate}
We denote by $M^!_w(G;\rho;\LL_\Sigma)$ (resp. $M_w(G;\rho;\LL_\Sigma),S_w(G;\rho;\LL_\Sigma)$) the $\LL_\Sigma$-vector spaces of weak modular forms (resp. modular forms, cusp forms) of weight $w$ for $\rho$. More generally, if $B$ is a $\CC_\infty$-subalgebra of $\KK_\Sigma$,
we write $M_w(G;\rho;B)$ for the corresponding $B$-module of modular forms.}
\end{Definition}
It is easy to see that the $\CC_\infty$-vector space $M_w(G;\boldsymbol{1};\CC_\infty)$
is equal to the $\CC_\infty$-vector space of the classical (scalar) Drinfeld modular forms of weight $w$ for $G$
and a similar property holds for weak modularity and cuspidality of a form. In the next proposition, $W_w$ stands for 
$M_w^!,M_w,S_w$ (so the proposition is in fact equivalent to three distinct statements).

\begin{Proposition}\label{evaluation-of-modular}
Let $f$ be in $W_w(\rho;\LL_\Sigma)$.
Then, there exists $d\in\FF_q[\underline{t}_\Sigma]\setminus\{0\}$ such that $f\in W_w(\rho;\widehat{\TT_\Sigma[\frac{1}{d}]})$. Let us consider, further, $\underline{\zeta}\in(\FF_q^{ac})^\Sigma\setminus\mathcal{V}_\Sigma(d)$. We have $\ev_{\underline{\zeta}}(f)\in W_w(\Gamma(\mathfrak{n});\boldsymbol{1};\CC_\infty)^{N\times 1}$ where $\mathfrak{n}$ is any ideal as in Lemma \ref{lemma-reduction-n}.
\end{Proposition}
Hence, the evaluations of the $N$ entries of $f\in M_w(\rho;\LL_\Sigma)$ are scalar Drinfeld modular forms of weight $w$ for $\Gamma(\mathfrak{n})$.
\begin{proof}[Proof of Proposition \ref{evaluation-of-modular}]
By Lemma \ref{lemma-reduction-n}, for all $\gamma\in\Gamma(\mathfrak{n})$ and $z\in\Omega$, 
$\ev_{\underline{\zeta}}(f)(\gamma(z))=J_\gamma(z)^w\ev_{\underline{\zeta}}(f)(z)$ and also, it is easy to see that $\ev_{\underline{\zeta}}(f)$ has rigid analytic entries. It remains to show that the entries of $\operatorname{ev}_{\underline{\zeta}}(f)$ have the decay properties of Definition \ref{refined-modularity} which is guaranteed if we show regularity at all cusps of $G\backslash\Omega$. In more detail, 
if $f$ has image defined over $\widehat{\TT_\Sigma[\frac{1}{d}]}$,
we show that the map $\operatorname{ev}_{\underline{\zeta}}(\cdot)$ defines maps ($\CC_\infty$-linear maps)
\begin{eqnarray}
M_{w}^!(\rho;\widehat{\TT_\Sigma[d^{-1}]})&\rightarrow& M^!_{w}(\Gamma(\mathfrak{n});\boldsymbol{1};\CC_\infty)^{N\times 1},\label{propev1}\\
M_{w}(\rho;\widehat{\TT_\Sigma[d^{-1}]})&\rightarrow& M_{w}(\Gamma(\mathfrak{n});\boldsymbol{1};\CC_\infty)^{N\times 1},\label{propev2}\\
S_{w}(\rho;\widehat{\TT_\Sigma[d^{-1}]})&\rightarrow& S_{w}(\Gamma(\mathfrak{n});\boldsymbol{1};\CC_\infty)^{N\times 1}.\label{propev3}
\end{eqnarray}
First of all, a holomorphic function $f:\Omega\rightarrow\CC_\infty$ satisfying $f(\gamma(z))=J_\gamma(z)^wf(z)$ for all $\gamma\in\Gamma(\mathfrak{n})$ is a weak modular form of weight 
$w$ for $\Gamma(\mathfrak{n})$ if for all $\delta\in\Gamma$, the function
$f^\delta(z)$ can be expanded as a series of $\CC_\infty((u(\frac{z}{\boldsymbol{n}})))$ in the neighborhood of the cusp $\delta(\infty)$, where $\boldsymbol{n}$ is a generator of $\mathfrak{n}$. We deduce that 
$f^\delta(z)$ is a weak modular form of weight $w$ for the group $\delta^{-1}\Gamma(\mathfrak{n})\delta$. Note indeed that $u(\frac{z}{\boldsymbol{n}})$ is a uniformiser at
$\infty$ for the action of $\Gamma(\mathfrak{n})$ over $\Omega$ in virtue of the fact that the group $(\begin{smallmatrix}1 & \mathfrak{n} \\ 0 & 1\end{smallmatrix})$ is contained in $\delta^{-1}\Gamma(\mathfrak{n})\delta$
for all $\delta\in\Gamma$. 

Let $f$ be in $M^!_{w}(\rho;\TT_\Sigma[d^{-1}]^\wedge)$. Then, $\operatorname{ev}_{\underline{\zeta}}(f)$ has all the entries which are $\mathfrak{n}$-periodic and $\operatorname{ev}_{\underline{\zeta}}(f^\delta)$ is tempered for all $\delta\in\Gamma$. This implies that $\operatorname{ev}_{\underline{\zeta}}(f)\in M^!_{w}(\Gamma(\mathfrak{n});\boldsymbol{1};\CC_\infty)^{N\times 1}$ which proves (\ref{propev1}). Now assume that 
$f$ is, additionally, a modular form in $M_{w}(\rho;\TT_\Sigma[d^{-1}]^\wedge)$. Then, all the entries $b^\delta$ of $\operatorname{ev}_{\underline{\zeta}}(f^\delta)$ satisfy $b^\delta\in \CC_\infty[[u(\frac{z}{\boldsymbol{n}})]]$ for all $\delta\in\Gamma$, which yields (\ref{propev2}). Similarly, if $f$ is in $S_{w}(\rho;\TT_\Sigma[d^{-1}]^\wedge)$, we see that all the entries of $\operatorname{ev}_{\underline{\zeta}}(f)$ vanish at all the cusps of $X(\mathfrak{n})$
hence confirming (\ref{propev3}) and completing the proof of the Proposition.
\end{proof}

At this point, we would like to ask a question. The next definition prepares it.

\begin{Definition}
{\em Let $\mathfrak{n}$ be a non-zero ideal of $A$, let $g$ be a Drinfeld modular form of weight $w$ for $\Gamma(\mathfrak{n})$. We say that $g$ {\em lifts to a modular form for the full modular group} if there exist: (1) a representation of the first kind $\rho:\Gamma\rightarrow\GL_N(\FF_q(\underline{t}_\Sigma))$ and $\underline{\zeta}\in(\FF_q^{ac})^\Sigma$ such that the evaluations $\operatorname{ev}_{\zeta}(\rho(\gamma))$ are well defined for every $\gamma\in\Gamma$, and (2) an element $f={}^t(f_1,\ldots,f_N)\in M_w(\rho;\LL_\Sigma)$ such that $g=\operatorname{ev}_{\zeta}(f_i)$ for some $i\in\{1,\ldots,N\}$.}
\end{Definition}

\begin{Question} Compute the $\CC_\infty$-span in $M_{w}(\Gamma(\mathfrak{n});\boldsymbol{1};\CC_\infty)$ of the modular forms which lift to modular forms for the full modular group. For which $\mathfrak{n}$ do we obtain the whole space?
\end{Question}

\subsubsection{Proof of the Finiteness Theorem}\label{Proof-Finiteness-Theorem}

We first study the structure of the space $M_0(\rho;\LL_\Sigma)$.

\begin{Lemma}\label{weight-zero}
We have $M_0(\rho;\LL_\Sigma)\subset\LL_\Sigma^{N\times 1}$.
\end{Lemma}

\begin{proof}
Let $f$ be an element of $M_0(\rho;\LL_\Sigma)$. By Lemma \ref{precise-image} there exists $d\in\FF_q[\underline{t}_\Sigma]\setminus\{0\}$ such that the image of $f$ is defined over $\widehat{\TT_\Sigma[\frac{1}{d}]}$. By Proposition \ref{evaluation-of-modular}, for all $\underline{\zeta}\in(\FF_q^{ac})^\Sigma\setminus\mathcal{V}_\Sigma(d)$ there exists a non-zero ideal $\mathfrak{n}$ of $A$ such that $\operatorname{ev}_{\underline{\zeta}}(f)\in M_0(\Gamma(\mathfrak{n});\boldsymbol{1};\CC_\infty)^{N\times 1}$. A scalar Drinfeld modular form of weight zero is constant. Hence, for all $\underline{\zeta}$ as above, $\operatorname{ev}_{\underline{\zeta}}(f)\in\CC_\infty^{N\times 1}$. Therefore, $f$ is a constant map by Lemma \ref{locally-constant-implies-constant} with $X=\Omega$.\end{proof}

We recall from \S \ref{carlitzexttate} the $\FF_q(\underline{t}_\Sigma)$-linear automorphisms $\tau:\KK_\Sigma\rightarrow\KK_\Sigma$, $\tau:\LL_\Sigma\rightarrow\LL_\Sigma$. Since the image of a representation of the first kind $\rho$ lies in $\FF_q(\underline{t}_\Sigma)^{N\times N}$ for some $N$, we have injective
$\FF_q(\underline{t}_\Sigma)$-linear maps
$$W_w(\rho;\KK_\Sigma)\xrightarrow{\tau} W_{qw}(\rho;\KK_\Sigma),\quad W_w(\rho;\LL_\Sigma)\xrightarrow{\tau} W_{qw}(\rho;\LL_\Sigma),$$ where $W_w=M^!_w,M_w,S_w$. With this, we can prove the next corollary to Lemma \ref{weight-zero}.

\begin{Corollary}\label{corollary-negative}
If $w<0$, $M_w(\rho;\LL_\Sigma)=\{0\}$.
\end{Corollary}

\begin{proof}
Let $f$ be an element of $M_w(\rho;\LL_\Sigma)$ with negative $w$. For all $k,\alpha,\beta\in\NN$
with $\beta>0$, 
$\widetilde{f}:=g^\alpha h^\beta\tau^k(f)\in S_{q^kw+\alpha(q-1)+\beta(q+1)}(\rho\det^{-\beta};\LL_\Sigma)$, where $g$ is the normalised Eisenstein series in $M_{q-1}(\boldsymbol{1};\CC_\infty)$ and $h$ is $-1$ times
the normalised generator of $S_{q+1}(\det^{-1};\CC_\infty)$ (we are adopting Gekeler's notations in \cite{GEK}, see also \S \ref{Gekeler-Poincare}).
We show that there exist $k,\alpha,\beta$ with $\beta>0$ such that \begin{equation}\label{kalphabeta}q^kw+\alpha(q-1)+\beta(q+1)=0.\end{equation} This is very easy but we give all the details. To find such $k,\alpha,\beta$, we first observe that we need $q^kw+\alpha(q-1)+\beta(q+1)\equiv0\pmod{q-1}$, and this is guaranteed by $w\equiv-2\beta\pmod{q-1}$. We must have:
\begin{eqnarray*}
\alpha&=&\frac{1}{q-1}(-wq^k-\beta(q+1))\\
&=&\frac{1}{q-1}(-wq^k-2\beta)+\beta.
\end{eqnarray*}
Assume first that $p\neq2$. Then, there exists $\beta\in\{1,\ldots,q-1\}$ such that $w\equiv-2\beta\pmod{q-1}$. We can choose $k$ large enough so that $-wq^k-2\beta$, divisible by $q-1$, is $\geq 0$. Therefore we can choose $\alpha\in\NN$ such that, with such $\beta$ and $k$, (\ref{kalphabeta}) holds.

If $p=2$ we can set $\beta=1$ and $k$ such that $\alpha=-2^kw-3\geq 0$. Since $\beta>0$ we see that $\widetilde{f}$ is a cusp form and Lemma \ref{weight-zero} now implies that $\widetilde{f}=0$; hence $f=0$ because $\tau$ is injective.
\end{proof}

\begin{proof}[Proof of Theorem \ref{theorem-structure-tate-algebras}]
The result is already proved in Lemma \ref{weight-zero} and Corollary \ref{corollary-negative} if $w\leq 0$. Now assume that $w>0$ and let $f$ be in $M_w(\rho;\LL_\Sigma)$. Again, we can suppose that 
$f\in M_w(\rho;\widehat{\TT_\Sigma[\frac{1}{d}]})$ for some $d\in\FF_q[\underline{t}_\Sigma]\setminus\{0\}$. 

We have that $f\in\mathfrak{O}_{\Sigma}^{N\times 1}$ by Theorem \ref{theorem-u-expansions}. In fact, the proof of Proposition \ref{quasiperiodictempered} allows to show that, more precisely, $f\in\mathfrak{O}_{\TT_\Sigma[\frac{1}{d}]^\wedge}^{N\times 1}$. Since $f$ is a regular $\rho$-quasi-periodic function (Definition \ref{quasipermatrix}), viewing the 
proof of Proposition \ref{quasiperiodictempered}, we obtain that 
$f=\Phi_\rho g$, where $\Phi_\rho$ has been defined in \S \ref{here-defi-phi-rho} and studied in Proposition \ref{propgeneralitiestame}, and where $g$ is in 
$\TT_\Sigma[\frac{1}{d}]^\wedge[[u]]^{N\times 1}$. We recall that from Corollary \ref{previouscorollary} that $\det(\Phi_\rho)=1$ and $\Phi_\rho,\Phi_\rho^{-1}\in(\Tamecirc{\mathfrak{E}})^{N\times N}$. We now study the association $f\mapsto g$ so that we write $g_f$ to stress the dependence of $g$ on $f$.

Let $\nu$ be in $\FF_q^\times$. We have 
$$\rho(\begin{smallmatrix}\nu & 0 \\ 0 & 1\end{smallmatrix})\Phi_\rho(z)\rho(\begin{smallmatrix}\nu & 0 \\ 0 & 1\end{smallmatrix})^{-1}g_f(\nu z)=f(\nu z)=\nu^w\rho(\begin{smallmatrix}\nu & 0 \\ 0 & 1\end{smallmatrix})f(z)=\nu^w\rho(\begin{smallmatrix}\nu & 0 \\ 0 & 1\end{smallmatrix})\Phi_\rho(z)g_f(z),\quad \forall z\in\Omega.$$
Since $\rho$ is of the first kind, $\rho(\begin{smallmatrix}\nu & 0 \\ 0 & 1\end{smallmatrix})$ is diagonal and we can write:
$$\rho(\begin{smallmatrix}\nu & 0 \\ 0 & 1\end{smallmatrix})=\begin{pmatrix}\nu^{-n_1} & & \\
& \ddots & \\ & & \nu^{-n_N}\end{pmatrix},\quad n_i\in\ZZ/(q-1)\ZZ,\quad \nu\in\FF_q^\times.$$
Writing additionally $g_f={}^t(g_1,\ldots,g_N)$, we deduce that $$g_i(\nu z)=\nu^{w-n_i}g_i(z)$$ for all $i=1,\ldots,N$, so that
$g_i\in u^{m_i}\TT_\Sigma[\frac{1}{d}]^\wedge[[u^{q-1}]]$ where $m_i$ is the unique representative of $n_i-w$ modulo $q-1$ in $\{0,\ldots,q-2\}$. This implies that the subspace $\mathcal{W}_w$ of 
$M_w(\rho;\LL_\Sigma)$ spanned by the forms $f$ with $g_f$ having entries of $v$-valuation 
in the set $\{0,1\}$ has dimension not exceeding $N$ if $q>2$ and $2N$ if $q=2$. On the other hand, if $f\in M_w(\rho;\LL_\Sigma)$ is such that $g_f$ is not in $\mathcal{W}_w$, that is, the $v$-valuations of its entries are $\geq 2$, then, by the fact that $\Phi_\rho^{-1}\in(\Tamecirc{\mathfrak{E}})^{N\times N}$, we deduce that
the $v$-valuations of the entries of $f$ are all $\geq 1$ and therefore $f\in hM_{w-(q+1)}(\rho\det;\LL_\Sigma)$ (where we recall that $h$ is the generator of $S_{q+1}(\det^{-1};\CC_\infty)$ normalised by 
the coefficient of $u$ in its $u$-expansion, which is set to $-1$).
We have proved that $$M_w(\rho;\LL_\Sigma)=hM_{w-(q+1)}(\rho\det;\LL_\Sigma)\oplus \mathcal{W}_w.$$ This implies
$$\dim_{\LL_\Sigma}\Big(M_w(\rho;\LL_\Sigma)\Big)\leq \dim_{\LL_\Sigma}\Big(M_{w-(q+1)}(\rho\det;\LL_\Sigma)\Big)+\left\{\begin{matrix}N & \text{ if }q>2\\ 
2N & \text{ if }q=2\end{matrix}\right.$$
The result follows by induction over $w$.
\end{proof}

\subsubsection{Modular forms of weight one}\label{weight-one-first-part}

We keep working with a representation of the first kind $\rho:\Gamma\rightarrow\GL_N(\FF_q(\underline{t}_\Sigma))$ and we set, with $L$ a field extension of $\FF_q(\underline{t}_\Sigma)$, $$H(\rho;L)=\{l\in L^{N\times 1}:\rho(T_a)l=l\text{ for all }a\in A\}.$$
This is equal to the $L$-vector space generated by the simultaneous eigenvectors of $\rho(T_a)$ in $\FF_q(\underline{t}_\Sigma)$,
with $a\in A$. Note indeed that for all $a\in A$, $T_a^p=I_2$ so that $1$ is the unique eigenvalue of 
$\rho(T_a)$ for all $a$. We denote by $\delta_\rho$ the dimension of $H(\rho;L)$ (independent on $L$).

Let us consider $f\in M_w(\rho;L)$ where $L=\KK_\Sigma$ or $L=\LL_\Sigma$. By Theorem \ref{theorem-u-expansions} we can identify $f={}^t(f_1,\ldots,f_N)$ with an element of $\mathfrak{O}_\Sigma^{N\times1}$. We denote by $\overline{f}_i$ the image of $f_i$ modulo $\mathfrak{M}_\Sigma$ for all $i$. This is an element of $\KK_\Sigma$ and we set $\overline{f}={}^t(\overline{f}_1,\ldots,\overline{f}_N)\in \KK_\Sigma^{N\times1}$. We easily see, by taking the limit for $z\in\Omega$, $|z|=|z|_\Im\rightarrow\infty$ that in fact, $\overline{f}\in L^{N\times1}$. Note that for every
$a\in A$, $f|_{w,\rho}T_a=\rho(T_{-a})f(z+a)$ equally belongs to $\mathfrak{O}_\Sigma^{N\times1}$
(by Lemma \ref{aendomorphisms}). Therefore $\overline{f}\in H(\rho;L)$. 
This means that 
$$M_w(\rho;L)=S_w(\rho;L)\oplus \mathcal{W}_w,$$
where the map $f\mapsto \overline{f}$ induces an embedding $\mathcal{W}_w\rightarrow H(\rho;L)$
so that $\delta_\rho$ is an upper bound for the dimension of $\mathcal{W}_w$.
We can now prove the following result which justifies part (2) of Theorem B in the introduction:
\begin{Theorem}\label{rankoneforweighone} We have $S_1(\rho;\LL_\Sigma)=\{0\}$ and the inequality
$\dim_{\LL_\Sigma}(M_1(\rho;\LL_\Sigma))\leq \delta_\rho$.\end{Theorem}
\begin{proof} It suffices to show that $S_1(\rho;\LL_\Sigma)=\{0\}$. 
Let $f$ be a cusp form of $S_1(\rho;\LL_\Sigma)$. In the settings of Proposition \ref{evaluation-of-modular}, for $\underline{\zeta}\in(\FF_q^{ac})^{\Sigma}\setminus\mathcal{V}_\Sigma(d)$ we get (after this proposition) that the evaluation $\operatorname{ev}_{\underline{\zeta}}(f)$ is well defined and its entries are cusp forms of $S_{1}(\Gamma(\mathfrak{n}))$. The latter space is zero as it was first noticed by Gekeler (see Cornelissen, in \cite[Theorem (1.10)]{COR}). Hence, for all $\underline{\zeta}$ as above, $\operatorname{ev}_{\underline{\zeta}}(f)=0$. By Lemma \ref{locally-constant-implies-constant}, $f$ vanishes identically.
\end{proof}
A more precise result in a particular case is Theorem \ref{coro-pel-per}.

\subsection{Poincar\'e series}\label{poincareseries} Here we construct explicit examples of modular forms in our generalized setting. We are mainly concerned with a class of matrix-valued Poincar\'e series. 

We consider a representation of the first kind $$\Gamma\xrightarrow{\rho}\operatorname{GL}_N(\FF_q(\underline{t}_\Sigma)),$$ of degree $l$.
Let $w$ be an integer and, with $L=\LL_\Sigma$ or $L=\KK_\Sigma$, let $G:\Omega\rightarrow L_\Sigma^{N\times N}$ be a tempered matrix $\rho$-quasi-periodic matrix function of type $m$, following Definition \ref{quasipermatrix}. We shall keep these settings all along \S \ref{poincareseries}.
We set, for $\gamma\in\Gamma$ and $z\in\Omega$:
$$S_\gamma(w,m;G)(z)=\det(\gamma)^mJ_\gamma(z)^{-w}\rho(\gamma)^{-1}G(\gamma(z))\rho(\begin{smallmatrix}\det(\gamma) & 0 \\ 0 & 1\end{smallmatrix}).$$
\begin{Lemma}\label{propertiesofSgamma}
Let $\gamma,\gamma'\in\Gamma$ be in the same left coset modulo $H:=\{(\begin{smallmatrix} * & * \\ 0 & 1\end{smallmatrix})\}\subset\Gamma$. Then we have the equality $S_\gamma(w,m;G)(z)=S_{\gamma'}(w,m;G)(z)$.
Moreover, for all $\delta\in\Gamma$,
$$S_\gamma(w,m;G)(\delta(z))=\det(\delta)^{-m}J_\delta(z)^{w}\rho(\delta)S_{\gamma\delta}(w,m;G)(z)
\rho(\begin{smallmatrix}\det(\delta)^{-1} & 0 \\ 0 & 1\end{smallmatrix}).$$
\end{Lemma}
\begin{proof} We simplify the notation: $S_\gamma(w,m;G)(z)=S_\gamma(z)$.
We prove the first property. Since $H$ is the semidirect product of $A$ by $\FF_q^\times$,
it suffices to show that: (1) for all $a\in A$, $S_{T_a\gamma}(z)=S_\gamma(z)$ and (2)
for all $\nu\in\FF_q^\times$, $S_{\delta\gamma}(z)=S_\gamma(z)$ if $\delta=(\begin{smallmatrix}\nu & 0 \\ 0 & 1\end{smallmatrix})$.
For (1), we observe, by the properties of $G$, that
\begin{eqnarray*}
S_{T_a\gamma}(z)&=&\det(T_a\gamma)^mJ_{T_a\gamma}(z)^{-w}\rho(T_a\gamma)^{-1}G(T_a(\gamma(z)))\rho(\begin{smallmatrix}\det(T_a\gamma) & 0 \\ 0 & 1\end{smallmatrix})\\
&=&\det(\gamma)^mJ_{\gamma}(z)^{-w}\rho(\gamma)^{-1}\rho(T_a)^{-1}\rho(T_a)G(\gamma(z))\rho(\begin{smallmatrix}\det(\gamma) & 0 \\ 0 & 1\end{smallmatrix})\\
&=&S_\gamma(z).
\end{eqnarray*}
For (2), we see, similarly, with $\delta=(\begin{smallmatrix}\nu & 0 \\ 0 & 1\end{smallmatrix})$ (here we use that $G$ has type $m$):
\begin{eqnarray*}
S_{\delta\gamma}(z)&=&\det(\delta\gamma)^mJ_{\delta\gamma}(z)^{-w}\rho(\delta\gamma)^{-1}G(\delta(\gamma(z)))\rho(\begin{smallmatrix}\det(\delta\gamma) & 0 \\ 0 & 1\end{smallmatrix})\\
&=&\det(\gamma)^mJ_{\gamma}(z)^{-w}\rho(\gamma)^{-1}\rho(\delta)^{-1}\det(\delta)^m\det(\delta)^{-m}\rho(\delta)G(\gamma(z))\rho(\delta)^{-1}\rho(\begin{smallmatrix}\det(\delta\gamma) & 0 \\ 0 & 1\end{smallmatrix})\\
&=&S_\gamma(z).
\end{eqnarray*}
This completes the proof of the first part of the Lemma. For the second, observe, if $\gamma'=\gamma\delta$ with $\delta\in\Gamma$:
\begin{eqnarray*}
S_{\gamma}(\delta(z))&=&\det(\gamma)^mJ_\gamma(\delta(z))^{-w}\rho(\gamma)^{-1}G(\gamma(\delta(z)))\rho(\begin{smallmatrix}\det(\gamma) & 0 \\ 0 & 1\end{smallmatrix})\\
&=&\det(\delta)^{-m}\det(\gamma')^mJ_\delta(z)^wJ_{\gamma'}(z)^{-w}\rho(\gamma'\delta^{-1})^{-1}G(\gamma'(z))\rho(\begin{smallmatrix}\det(\gamma') & 0 \\ 0 & 1\end{smallmatrix})
\rho(\begin{smallmatrix}\det(\delta) & 0 \\ 0 & 1\end{smallmatrix})^{-1}\\
&=&\det(\delta)^{-m}J_\delta(z)^w\rho(\delta)S_{\gamma'}(z)\rho(\begin{smallmatrix}\det(\delta) & 0 \\ 0 & 1\end{smallmatrix})^{-1}.
\end{eqnarray*}
\end{proof}
We consider the formal series (Poincar\'e series):
\begin{equation}\label{defpoincare}
\mathcal{P}_{w}(G)(z):=\sum_\gamma S_\gamma(w,m;G)(z),
\end{equation}
where the sum runs over a complete set of representatives of $H\backslash \Gamma$. Note that, if well defined, this is a matrix function. Compare this with Bruinier's definition of Poincar\'e series in \cite[\S 1.2, 1.3]{BRU}.
We have the next property.
\begin{Proposition}\label{preinvestigation}
If the series $\mathcal{P}_{w}(G)(z)$ converges to an element of 
$\operatorname{Hol}_{\KK_\Sigma}(\Omega\rightarrow L^{N\times N})$ then it satisfies, for all $z\in\Omega$ and 
$\gamma\in\Gamma$:
$$\mathcal{P}_{w}(G)(\gamma(z))=\det(\gamma)^{-m}J_\gamma(z)^w\rho(\gamma)\mathcal{P}_{w}(G)(z)\rho(\begin{smallmatrix}\det(\gamma) & 0 \\ 0 & 1\end{smallmatrix})^{-1}.$$
For each column $f$ of $\mathcal{P}_{w}(G)$ there exists $i\in\ZZ/(q-1)\ZZ$ such that
$$f(\delta(z))=\det(\delta)^{i-m}J_\delta(z)^w\rho(\delta)f(z),\quad \forall z\in\Omega,\quad \delta\in\Gamma.$$
\end{Proposition}
\begin{proof} We assume that the series converges, giving rise to an element of $\operatorname{Hol}_{\KK_\Sigma}(\Omega\rightarrow L^{N\times N})$. We note that for $\nu\in\FF_q^\times$, $\rho(\begin{smallmatrix}\nu & 0 \\ 0 & 1\end{smallmatrix})$ is diagonal in $\operatorname{GL}_N(\FF_q)$ and there are integers $n_i$ with $i=0,\ldots,q-2$ such that $\sum_in_i=N$ so that we can decompose 
\begin{equation}\label{blocks-P}
\mathcal{P}_{w}(G)=\bigoplus_{i=0}^{q-2}\mathcal{P}_{w}^{[i]}(G),
\end{equation}
where $\mathcal{P}_{w}^{[i]}(G):\Omega\rightarrow\KK_\Sigma^{N\times n_i}$ for all $i$,
and 
$$\mathcal{P}_{w}^{[i]}(G)(\delta(z))=\det(\delta)^{i-m}J_\delta(z)^w\rho(\delta)\mathcal{P}_{w}^{[i]}(G),\quad \forall z\in\Omega,\quad \delta\in\Gamma,\quad i=0,\ldots,q-2.$$
\end{proof}
In full generality (for any quasi-periodic function $G$), we do not have a good criterion of convergence for the series $\mathcal{P}_{w}(G)$. We discuss these series for two choices of $G$.

We will need the next Lemma in the book \cite{GER&PUT} of Gerritzen and van der Put.
\begin{Lemma}\label{lemma3}
There exists a complete set of representatives $\gamma_{c,d}=\big(\begin{smallmatrix}* & *\\ c & d\end{smallmatrix}\big)$ of $H\backslash\Gamma$ in which each matrix is of one of the following three types:
\begin{enumerate}
\item $\gamma_{0,\mu}=\big(\begin{smallmatrix}\mu^{-1} & 0\\ 0 & \mu\end{smallmatrix}\big)$ with $\mu\in\FF_q^\times$,
\item $\gamma_{\mu,\nu}=\big(\begin{smallmatrix}0 & -\mu^{-1}\\ \mu & \nu\end{smallmatrix}\big)$ with $\mu\in\FF_q^\times$ and $\nu\in\FF_q$,
\item $\gamma_{c,d}=\big(\begin{smallmatrix}a & b\\ c & d\end{smallmatrix}\big)$, with $a,b,c,d\in A$ such that $ad-bc=1$,  $|cd|>1$, $|a|<|c|$, $|b|<|d|$.
\end{enumerate}
\end{Lemma}
We note that the first two sets are finite. 
Let us look at the corresponding extracted series in the series
(\ref{defpoincare}) defining $\mathcal{P}_{w}(G)$; we denote them by $\mathfrak{A},\mathfrak{B},\mathfrak{C}$ (in agreement with the order of the types in the above set of representatives), so that $\mathfrak{A},\mathfrak{B}$ correspond to finite sums while $\mathfrak{C}$ is an infinite sum.
We set:
$$\epsilon(\rho):=\sum_{\mu\in\FF_q^\times}\mu^{2m-w+l}\rho\big(\begin{smallmatrix}\mu^{2} & 0 \\ 0 & 1\end{smallmatrix}\big)\in\FF_q^{N\times N}.$$
Note that this is a diagonal matrix with entries in $\{-1,0\}$.
For the first sub-sum we have, in virtue of the fact that $G$ is of type $m$ (second equality) and that $\rho$ is of degree $l$ (third equality):
\begin{eqnarray}
\mathfrak{A}:=\sum_{\mu\in\FF_q^\times}S_{\gamma_{0,\mu}}(z)
&=&\sum_{\mu\in\FF_q^\times}\mu^{-w}\rho\big(\begin{smallmatrix}\mu^{-1} & 0 \\ 0 & \mu\end{smallmatrix}\big)^{-1}
G(\mu^{-2}z)\label{frak-A}\\
&=&\sum_{\mu\in\FF_q^\times}\mu^{-w}\rho\big(\begin{smallmatrix}\mu^{-1} & 0 \\ 0 & \mu\end{smallmatrix}\big)^{-1}\mu^{2m}\rho\big(\begin{smallmatrix}\mu^{-2} & 0 \\ 0 & 1\end{smallmatrix}\big)
G(z)\rho\big(\begin{smallmatrix}\mu^{-2} & 0 \\ 0 & 1\end{smallmatrix}\big)^{-1}\nonumber\\
&=&G(z)\epsilon(\rho)\nonumber.
\end{eqnarray}
For the second sub-sum we have, similarly:
\begin{eqnarray}
\lefteqn{\mathfrak{B}:=\sum_{\begin{smallmatrix}\mu\in\FF_q^\times\\ \nu\in\FF_q\end{smallmatrix}}S_{\gamma_{\mu,\nu}}(z)=}\\ &=&\sum_{\mu\in\FF_q^\times}\mu^{2m-w+l}\sum_{\nu\in\FF_q}\left(z+\frac{\nu}{\mu}\right)^{-w}\rho
\Big(\begin{smallmatrix}\frac{\nu}{\mu} &-1 \\ 1 & 0\end{smallmatrix}\Big)G\left(\frac{-1}{z+\frac{\nu}{\mu}}\right)\rho
\Big(\begin{smallmatrix}\mu^2 & 0 \\ 0 & 1\end{smallmatrix}\Big)\label{frak-B}\\ 
&=&(-1)^m\left(\sum_{\beta\in\FF_q}(z+\beta)^{-w}\rho\Big(\begin{smallmatrix}-\beta & 1 \\ 1 & 0\end{smallmatrix}\Big)
G\left(\frac{1}{z+\beta}\right)\rho\Big(\begin{smallmatrix}-1 & 0 \\ 0 & 1\end{smallmatrix}\Big)\right)\epsilon(\rho)\nonumber.
\end{eqnarray}
We easily deduce that $\mathfrak{A}+\mathfrak{B}\in\operatorname{Hol}_{\KK_{\Sigma}}(\Omega\rightarrow L^{N\times N})$. We now make explicit choices for $G$.

\subsubsection{The case of $G$ entire} 

We suppose that $G\in\mathcal{QP}^!_m(\rho;\KK_\Sigma)$ extends to an entire function
$\CC_\infty\rightarrow L^{N\times N}$, where $L$ is a field extension of $\CC_\infty$ contained in $\KK_\Sigma$. In this part we study $\mathcal{P}_{w}(G)$ with $w\in\ZZ$, $w>0$. Let $M\geq 0$
be such that $\|G(z)u(z)^M\|$ is bounded for $|u(z)|<c$ for some $c<1$ (it exists as $G$ is tempered).

\begin{Lemma}\label{lemma-preparatory}
There are three constants $c_1,c_2,c_3\in|\CC_\infty^\times|$ such that $c_1\geq1$ and $\eta\in\ZZ[\frac{1}{p}]\cap[0,M+1[$ such that if $|z|_\Im\geq c_1$ then 
$\|G(z)\|\leq c_2|e_C(z)|^\eta$ and if $|z|_\Im\leq c_1$ then $\|G(z)\|\leq c_3$.
\end{Lemma}

\begin{proof}
We recall that $\Phi_\rho$, introduced at the end of \S \ref{Quasi-periodic-functions}, is 
entire (Proposition \ref{propgeneralitiestame} (a)), $\rho$-quasi-periodic of type $0$ (same proposition (b)) and that $\Phi_\rho^{-1}$ is entire (Corollary \ref{previouscorollary}) and has its entries which are
at once tame series of degrees in $[0,1[\cap\ZZ[\frac{1}{p}]\cup\{-\infty\}$. Then $G\Phi_\rho^{-1}$
is also entire and $A$-periodic. Therefore, by Proposition \ref{propositionperiodic} (c),
$G\Phi_\rho^{-1}\in L[e_C(z)]^{N\times N}$ and the degrees in $e_C$ of the entries of this matrix function, well defined, are $\leq M$ while the matrix function itself is of type $m$. We deduce that 
$$G\in L[e_C(z)]^{N\times N}\Phi_\rho.$$
By Proposition \ref{leading}, there exist constants $c_1\geq 1$ and $c_2$ with $c_1,c_2\in|\CC_\infty^\times|$, $\eta\in\ZZ[\frac{1}{p}]\cap[0,M+1[$ such that if $|z|_\Im\geq c_1$, then 
$\|G(z)\|\leq c_2|e_C(z)|^\eta$. Suppose now that $|z|_\Im\leq c_1$. 
There exists $\lambda\in K_\infty$ such that $|z-\lambda|=|z|_\Im\leq c_1$. We can write
$\lambda=a+m$ with $a\in A$ and $m\in\frac{1}{\theta}\FF_q[[\frac{1}{\theta}]]$. Then 
$|z-a|=|z-\lambda+m|\leq\max\{|z-\lambda|,|m|\}\leq c_1$ because $|m|<1\leq c_1$. Now, since
$G(z)$ is $\rho$-quasi-periodic, $\|G(z)\|\leq\|G(z-a)\|\leq c_3$ for some constant $c_3\in|\CC_\infty^\times|$, because the entries of $G$ are 
entire functions, hence bounded in the disk $D_{\CC_\infty}(0,c_1)$.
\end{proof}

\begin{Proposition} 
Let $w$ be a positive integer.
If $G$ is an entire tempered $\rho$-quasi-periodic function of type $m$ the series defining $\mathcal{P}_{w}(G)$ converges to an element of $\operatorname{Hol}_{\KK_\Sigma}(\Omega\rightarrow L^{N\times N})$ and the matrix functions $\mathcal{P}_{w}^{[i]}(G)$ defined in (\ref{blocks-P}) are elements of $M^!_w(\rho\det^{i-m};L)^{1\times n_i}$ for $i$ varying in $\ZZ/(q-1)\ZZ$. If the $i$-th block of $\epsilon(\rho)$
is non-zero, then the columns of $\mathcal{P}_{w}^{[i]}(G)$ are non-zero. Moreover, the matrix functions
$h^{M+1}\mathcal{P}^{[i]}_{w}(G)$ are elements of $S_{w+(M+1)(q+1)}(\rho\det^{i-m-M-1};L)^{1\times n_i}$.
\end{Proposition}

\begin{proof} 
Let $\gamma=(\begin{smallmatrix}a & b\\ c & d\end{smallmatrix})$ be in $\Gamma$, such that $c\neq0$ and let us consider $z\in\Omega$. Then:
\begin{equation}\label{decompositiongamma}
\gamma(z)=\frac{a}{c}-\frac{\det(\gamma)}{c(cz+d)}.
\end{equation}
We consider $c_1\in|\CC_\infty^\times|$ such that $c_1>1$ and we consider $z\in\Omega$ such that
$c_1^{-1}\leq |z|_\Im\leq |z|\leq c_1$. We note that if $\gamma$ is of type (2) or (3) as in Lemma \ref{lemma3}, then $|\gamma(z)|\leq c_1$. Since $G$ has entire entries, we therefore get that
the series defining $\mathcal{P}_w(G)$ converges uniformly over all the
affinoid subdomains of $\Omega$ of the type $\{z\in\Omega:c_3\leq |z|_\Im \leq |z|\leq c_4\}$ with $c_3,c_4\in|\CC_\infty^\times|$ hence defining an element of $\operatorname{Hol}_{\KK_\Sigma}(\Omega\rightarrow\LL_\Sigma^{N\times N})$. Now observe that if $|z|_\Im=|z|\rightarrow\infty$
and $\gamma$ is of type (2) or (3), then $|\gamma(z)|\rightarrow0$ uniformly on the set of representatives $\gamma$ of $H\backslash \Gamma$ and therefore, the sum $\mathfrak{B}+\mathfrak{C}$,
as a function of the variable $z$, is bounded as $|z|_\Im=|z|\rightarrow\infty$. By Lemma \ref{lemma-preparatory} and the expression we found for $\mathfrak{A}$, we therefore have that 
$\mathcal{P}_w(G)$ is tempered, because for $|z|_\Im$ large enough, $\|\mathcal{P}_{w}(G)(z)\|=\|G(z)\epsilon(\rho)\|$. More precisely, $|e_C(z)|^{-\eta}\|G(z)\|$ is bounded as 
$|z|_\Im=|z|\rightarrow\infty$ where $\eta$ is given in Lemma \ref{lemma-preparatory}.
Thanks to Proposition \ref{preinvestigation}, this suffices to show that $\mathcal{P}_{w}(G)$
has it columns in $M^!_w(\rho\det^{i-m};L)$. If $\epsilon(\rho)$ does not vanish identically, looking at the blocks which are not zero we deduce 
the properties regarding $\mathcal{P}_{w}^{[i]}(G)$. The last assertion of the proposition is verified
by noticing that $\|G(z)u(z)^\eta\|\rightarrow0$ as $|z|_\Im=|z|\rightarrow\infty$, and $0\leq \eta<M+1$. 
Therefore $h^{M+1}\mathcal{P}_{w}(G)$ vanishes at infinity because $v(h)=1$.
\end{proof}

\begin{Corollary}
If $G=\Phi_\rho$ and $\epsilon(\rho)\neq 0$ then there exists $i$ such that $$h\mathcal{P}^{[i]}_{w}(G)\in \Big(S_{w+q+1}(\rho\det{}^{i-1};L)\setminus\{0\}\Big)^{1\times n_i}.$$ 
\end{Corollary}

\subsubsection{The case $G=\Psi_m(\rho)$}
With $m\geq 1$, we study $\mathcal{P}_m(G)$ where:
\begin{equation}\label{defgengoss}
G=\Psi_m(\rho)=\sum_{a\in A}\frac{1}{(z-a)^m}\rho(T_a).
\end{equation} The functions $\Psi_m(\rho)$ have been introduced in \S \ref{section-Psi}.
By Lemma \ref{lemmapsil}
we have $\Psi_m(\rho)\in\mathcal{QP}_m(\rho;\LL_\Sigma)$.
If $\rho=\boldsymbol{1}:\Gamma\rightarrow\{1\}$ we recover the (scalar) sums $S_{m,\Lambda}$ for the lattice $\Lambda=A$ (see \cite[\S 6]{GOS2} and \cite[\S 3]{GEK}). In particular,
for any $m\geq1$ there exists a polynomial $G_{m}\in K[X]$ (called the {\em Goss' polynomial} of order $m$) such that \begin{equation}\label{goss-polynomials-identity}
G=\widetilde{\pi}^mG_m(u).
\end{equation}
The Goss' polynomials $G_m$ can be computed inductively by using the generating series:
\begin{equation}\label{goss-poly-relations}
\sum_{m\geq 1}G_m(u)X^m=\frac{uX}{1-u\exp_C(X)}.\end{equation}
See \cite[(3.6)]{GEK},\cite{GEK1}, and \cite[Theorem 3.2]{PAP&ZEN}, \cite{GEK2} for more recent results on these polynomials. See also our Lemma \ref{generalisation-gekeler}.

The next result holds:
\begin{Proposition}\label{proposition-convergence} Let us consider $w,m\in\NN^*$.
If $G=\Psi_m(\rho)$, the columns of $\mathcal{P}_{w}(G)$ are 
in $S_w(\rho\det^{-j};\LL_\Sigma)$ with $j$ varying in $\ZZ/(q-1)\ZZ$.
\end{Proposition}
\begin{proof} It suffices to show that the sum defining $\mathcal{P}_{w}(G)$ is uniformly convergent on affinoid subdomains of $\Omega$ of the type $\mathcal{C}:=\{z\in\Omega:c_1^{-1}\leq |z|_\Im \leq |z|\leq c_1\}$ with $c_1\in|\CC_\infty^\times|$ such that $c_1>1$. For this, we use the decomposition 
$\mathcal{P}_w(G)=\mathfrak{A}+\mathfrak{B}+\mathfrak{C}$. We need to show that the series $\mathfrak{C}$ converges uniformly over $\mathcal{C}$. We note that if $\gamma=\gamma_{c,d}$ is of type (3) as in Lemma \ref{lemma3}, then if $z\in\mathcal{C}$ we get $|\gamma(z)|\leq c_1$. In fact,
we have $\gamma(z)\rightarrow0$ by (\ref{decompositiongamma}) as $\gamma$ varies in the chosen 
representative set of $H\backslash\Gamma$ and $\gamma(\mathcal{C})\subset D_{\CC_\infty}(0,|\theta|^{-1})\cap\Omega$ for all but finitely many $\gamma$. If we denote by $\mathcal{E}$ the set of such homographies, we get $\|G(\gamma(z))\|\leq |z|^{-m}$ for all $z\in\mathcal{C}$ and for all $\gamma\in\mathcal{E}$. Therefore we can decompose $\mathfrak{C}=\mathfrak{C}_0+\mathfrak{C}_1$ where $\mathfrak{C}_1$ is a finite sum of holomorphic functions and $\mathfrak{C}_0=\sum_{\gamma\in\mathcal{E}}S_\gamma(G)(z)$ which converges uniformly on $\mathcal{C}$
in virtue of the fact that $w>0$. We deduce that $\mathcal{P}_w(G)$ defines a holomorphic function over $\Omega$, with values in $\LL_\Sigma^{N\times N}$. Since moreover,
$\|G(z)\|\rightarrow0$ as $|z|=|z|_\Im\rightarrow\infty$, we see that the columns of $\mathcal{P}_{w}(G)$ are cusp forms.
\end{proof}
Giving sufficiently general conditions for the non-vanishing of $\mathcal{P}_{w}(G)$ is more difficult in the case
$G=\Psi_m(\rho)$. We have the next proposition:

\begin{Proposition}\label{proposition-poincare-bis}
Assuming that $m,w$ are two positive integers such that $w>2m$, if  
$G=\Psi_m(\rho)$ and $\epsilon(\rho)\neq0$, then $\mathcal{P}_{w}(G)$ has a non-zero column in $S_w(\rho\det^{-i};\LL_\Sigma)$ for some $i$.
\end{Proposition}

\begin{proof} 
We need to analyze the various subsums $\mathfrak{A},\mathfrak{B}$ and $\mathfrak{C}$
of $\mathcal{P}_w(G)$ that we know being convergent series, by Proposition \ref{proposition-convergence}. 
We begin by studying the subsum $\mathfrak{A}$. 
Note that $\rho(T_a)-I_N$ is a nilpotent matrix having zeroes in the diagonal for all $a\in A$. The diagonal of $G=\Psi_m(\rho)$ is equal to 
$I_N\sum_{a\in A}(z-a)^{-m}$ and the hypothesis on $\epsilon(\rho)$ implies that 
$G\epsilon(\rho)$ has some non-zero elements on the diagonal of valuation $|\cdot|$ equal to 
$|\Psi_m(\boldsymbol{1})|$. By Lemma \ref{lemma-omega12} there exists $\kappa_1\in]1,|\theta|[\cap|\CC_\infty^\times|$ and a non-negative integer $\omega_2$ such that if $\kappa_1<|z|<|\theta|$, then
$\Psi_m(z)=|\theta|^{-m}|\frac{z}{\theta}|^{\omega_2}$. We deduce that
\begin{equation}\label{bound-A}
\|\mathfrak{A}(z)\|=|\theta|^{-m}\Big|\frac{z}{\theta}\Big|^{\omega_2},\quad \kappa_2<z<|\theta|.
\end{equation}

We now study the subsum $\mathfrak{B}$. To do this, we assume that $|z|>1$. By (\ref{frak-B}) and the definition of 
$G$:

\begin{multline}
\mathfrak{B}=\Bigg(\underbrace{\sum_{\beta\in\FF_q}(z+\beta)^{m-w}\rho\Big(\begin{smallmatrix}-\beta & 1 \\ -1 & 0\end{smallmatrix}\Big)}_{\mathfrak{B}_0}+\underbrace{\sideset{}{'}\sum_{a\in A}\sum_{\beta\in\FF_q}\frac{(z+\beta)^{m-w}}{(1-a(z+\beta))^m}\rho\Big(\begin{smallmatrix}-\beta & 1 \\ 1 & 0\end{smallmatrix}\Big)\rho(T_a)\rho\Big(\begin{smallmatrix}-1 & 0 \\ 0 & 1\end{smallmatrix}\Big)}_{\mathfrak{B}_1}\Bigg)\epsilon(\rho),
\end{multline}

where the sum is split in two pieces, the first sum corresponding to $a=0$, while the dash $'$ on the second sum designates the term corresponding to $a=0$ omitted. If $a\neq0$ we get
$|1-a(z+\beta)|=|a||z|\geq |z|=|z+\beta|$ and therefore, 
$\|\mathfrak{B}_1(z)\|\leq |z|^{-w}$ for $1<|z|.$
As for $\mathfrak{B}_0$, we see that 
$$\mathfrak{B}_0=\sum_{\beta\in\FF_q}(z+\beta)^{m-w}\rho\Big(\begin{smallmatrix}-\beta & 1\\ -1 & 0\end{smallmatrix}\Big)\epsilon(\rho).$$ Hence,
$\|\mathfrak{B}_0(z)\|\leq |z|^{m-w}$ again for $1<|z|.$
Thus,
\begin{equation}\label{bound-B}
\|\mathfrak{B}(z)\|\leq |z|^{m-w},\quad 1<|z|.
\end{equation}

It remains to handle the subsum $\mathfrak{C}$ and we consider, for this purpose, $z\in\Omega$
such that $1<|z|$ and $|z|\not\in|\theta|^\ZZ$. Suppose that $\gamma=\gamma_{c,d}=(\begin{smallmatrix}a & b\\ c & d\end{smallmatrix})$ is of type (3) as
in Lemma \ref{lemma3}. We notice that $|az+b|< |cz+d|$. This follows easily from the conditions on $a,b,c,d$ determining the type (3) and the fact that $|az+b|=\max\{|az|,|b|\}$ and $|cz+c|=\max\{|cz|,|d|\}$ because $|z|\not\in|\theta|^\ZZ$.

Then 
$$S_\gamma(G)=\widetilde{\pi}^{-m}J_\gamma(z)^{-w}\rho(\gamma)\sum_{\widetilde{b}\in A}(\gamma(z)-\widetilde{b})^{-m}\rho(T_{\widetilde{b}})\rho(\begin{smallmatrix}\det(\gamma) & 0\\ 0 & 1\end{smallmatrix}).$$
One sees easily that
$$(\gamma(z)-\widetilde{b})^{-m}=\frac{J_\gamma(z)^m}{(az+b-\widetilde{b}J_\gamma(z))^m}.$$
Note that
$az+b-\widetilde{b}J_\gamma(z)=z(a-\widetilde{b}c)+b-\widetilde{b}d$ so that, if $\widetilde{b}\neq0$,
$|az+b-\widetilde{b}J_\gamma(z)|=\max\{|z||a-\widetilde{b}c|,|b-\widetilde{b}d|\}=\max\{|z||c|,|d|\}|\widetilde{b}|=
|\widetilde{b}||J_\gamma(z)|$. Hence $\widetilde{b}\neq0$
implies that $|(\gamma(z)-\widetilde{b})^{-m}|\leq1$. If $\widetilde{b}=0$, since $|az+b|<|cz+d|$, 
we get $|\gamma(z)|^{-m}\leq|J_\gamma(z)|^m$.
Therefore, we deduce that $\|S_\gamma(G)\|\leq |J_\gamma(z)|^{m-w}$ for $\gamma$ of type (3) and we can conclude that
\begin{equation}\label{bound-C}
\|\mathfrak{C}(z)\|\leq |z|^{m-w},\quad \text{if }1<|z|,\quad |z|\not\in|\theta|^{\ZZ}.
\end{equation}

Assuming by contradiction that 
$\mathcal{P}_w(G)$ vanishes identically, we have that $\mathfrak{A}=-(\mathfrak{B}+\mathfrak{C})$. 
Looking at Lemma \ref{lemma-omega12} we observe that $|z|>|\theta|^{\frac{\omega_2+m}{\omega_2+w-m}}$ if and only if $|z|^{\omega_2+w-m}>|\theta|^{\omega_2+m}$, equivalent to $|\theta|^{-m}|\frac{z}{\theta}|^{\omega_2}>|z|^{m-w}$. But
$$\frac{\omega_2+m}{\omega_2+w-m}=1-\frac{w-2m}{\omega_2+w-m}$$ and the hypothesis 
$w>2m$ ensures that there exists $\kappa_2\in]1,|\theta|[\cap|\CC_\infty^\times|$ such that
for all $z\in\Omega$ such that $\kappa_2<|z|<|\theta|$,
$$\|\mathfrak{A}(z)\|\geq|\Psi_m(z)|=|\Psi_m^{\geq}(z)|=|\theta|^{-m}\Big|\frac{z}{\theta}\Big|^{\omega_2}>|z|^{m-w}\geq \|\mathfrak{B}(z)+\mathfrak{C}(z)\|,$$ by (\ref{bound-A}), (\ref{bound-B}) and (\ref{bound-C}) (more precisely, a non-zero column of $\mathfrak{A}$ has an entry which has $\|\cdot\|$ equal to $|\Psi_m(z)|$). This is impossible. Hence $\mathcal{P}_w(G)$ does not vanish identically.
\end{proof}

\subsubsection{Example: Poincar\'e series in a class introduced by Gekeler}\label{Gekeler-Poincare} We consider the case $N=1$, $\rho=\boldsymbol{1}$, we choose $G(z)=G_m(u)=\widetilde{\pi}^{-m}\Psi_m(\boldsymbol{1})$ the Goss' polynomial of order $m$ with $m>0$. Then, we see that 
$\epsilon(\rho)=\sum_{\mu\in\FF_q^\times}\mu^{2m-w}$ which is non-zero if and only if $w\equiv2m\pmod{q-1}$. We therefore reach the next result.

\begin{Corollary}
If $w\equiv2m\pmod{q-1}$ and $w>2m$ then, with $G(z)=G_m(u)$, the Poincar\'e series $\mathcal{P}_w(G)$ determines a non-zero element of $S_w(\det^{-m};\CC_\infty)$.
\end{Corollary}

This sharpens Petrov's \cite[Remark 4.1]{PET2} where the condition on $w\equiv2m\pmod{q-1}$ is  stronger: $w>(q+1)m$. Note that Petrov's condition is the same of Gerritzen and van der Put in \cite[pp. 304-307]{GER&PUT}. If we take
$w> 2m$, $m\in\{1,\ldots,q\}$, $w\equiv2m\pmod{q-1}$ and $G=u^m$, we see that
$$\mathcal{P}_{w}(G)=\sum_{\gamma\in H\backslash \Gamma}\det(\gamma)^mJ_\gamma^{-w}u(\gamma(z))=P_{w,m}\in S_w(\det{}^{-m};\CC_\infty)$$ in the notations of Gekeler, \cite[(5.11)]{GEK}. If $w=q+1$ and $m=1$ then $h=P_{q+1,1}$.

\subsubsection{Example: Poincar\'e series associated to the representations $\rho_\Sigma^*$} We consider $\rho=\rho_\Sigma^*$ which is of degree $s=|\Sigma|$, where $\Sigma\subset\NN^*$. A simple computation shows that
$$\rho(\begin{smallmatrix}\nu & 0 \\ 0 & 1\end{smallmatrix})=\operatorname{Diag}(\nu^{-s},\cdots,\nu^{-n_1},\nu^{-n_0})$$ where the integer sequence $(n_i)_{i\geq 0}$ does not depend on $s$ and coincides with the so-called one's-counting sequence, that is, the sequence which gives the number of one's in the binary expansion of $i$. We suppose $s$ fixed.
We consider $w,m>0$ such that $w>2m$ and we set $r=2m-w+s$. We have the next result, of which we omit the elementary proof, where we recall that $p$ is the prime dividing $q$.

\begin{Proposition}\label{tedious}
The following properties hold.
\begin{enumerate}
\item If $p=2$ we have $\epsilon(\rho)=I_N$.
\item If $p>2$ then $\epsilon(\rho)$ is non-zero if and only if $r$ is even.
\item If $p>2$ then $r\equiv0\pmod{q-1}$ implies that the last column of $\epsilon(\rho)$ is 
non-zero.
\item Assume that $p>2$. For all $i=1,\ldots,N$ there exists a unique $k\in\{0,2,\ldots,\frac{q+1}{2}\}$
such that if $r\equiv k\pmod{q-1}$, then the $i$-th column of $\epsilon(\rho)$ equals $-\varepsilon_i$
where $\varepsilon_i$ is the $i$-th column of $I_N$.
\end{enumerate}
\end{Proposition}

If $s\equiv1\pmod{q-1}$ then the smallest parameters allowable in the construction of a Poincar\'e series as above are $w=3$ and $m=1$. By Proposition \ref{tedious}, the last column of $\mathcal{P}_3(G)$ where $G=\Psi_1(\rho_\Sigma^*)$ is an element of
$S_3(\rho_\Sigma^*\det^{-1};\LL_\Sigma)\setminus\{0\}$. Note that if $\Sigma=\emptyset$ and $q=2$ then 
we get a multiple of Gekeler's function $h$.

\subsection{Eisenstein series}\label{intro-eisenstein-series}
The process that leads to the construction of Eisenstein series is different from that of Poincar\'e series and delivers, in general, vector-valued modular forms rather than matrix-valued modular forms. We describe it in our particular setting but the discussion that follows easily generalizes to e.g. the case of vector-valued modular forms for the group $\operatorname{SL}_2(\ZZ)$ etc. Let $\rho$ be a representation
$$\Gamma\xrightarrow{\rho}\GL_N(B),$$
with $(B,|\cdot|_B)$ a countably cartesian Banach $\CC_\infty$-algebra.
Suppose that there is a map 
\begin{equation}\label{the-map-mu}
A^{1\times 2}\xrightarrow{\mu}B^{N\times 1}\end{equation}
such that for all $\gamma\in\Gamma$, if $(a,b)\gamma=(a',b')$ in $A^{1\times 2}$, then
$${}^t\rho(\gamma)\mu(a,b)=(a',b').$$ Assume further that the image of $\mu$ is bounded, that is, there
is $c_1>0$ such that $|\mu(a,b)|_B\leq c_1$ for all $a,b\in A$. Then, for all $w>0$, the series
$$\mathcal{E}=\sideset{}{'}\sum_{a,b\in A}(az+b)^{-w}\mu(a,b)$$ (where the dash $'$ indicates that the term corresponding to $a=b=0$ is omitted)
converges to a rigid analytic map $$\Omega\rightarrow B^{N\times1}$$ and moreover:

\begin{Lemma}\label{defi-general-eisenstein-series}
We have that $\mathcal{E}\in M_w(\rho^*;B)$. If $\sum_{b\in A\setminus\{0\}}b^{-w}\mu(0,b)$ is non-zero, then $\mathcal{E}$ does not vanish identically.
\end{Lemma}

\begin{proof} 
We consider $\gamma\in\Gamma$. Then:
\begin{eqnarray*}
\mathcal{E}_w(\rho;\mu)&=&J_\gamma(z)^{w}\sideset{}{'}\sum_{a,b\in A}\Big((a,b)\gamma\Big(\begin{smallmatrix}z \\ \\1\end{smallmatrix}\Big)\Big)^{-w}\mu(a,b)\\
&=&J_\gamma(z)^{w}\sideset{}{'}\sum_{a',b'\in A}\Big((a',b')\Big(\begin{smallmatrix}z \\  \\1\end{smallmatrix}\Big)\Big)^{-w}\rho^*(\gamma)\mu(a',b')\\
&=&J_\gamma(z)^{w}\rho^*(\gamma)\mathcal{E}_w(\rho;\mu).
\end{eqnarray*}
Since $|(az+b)^{-w}\mu(a,b)|_B$ tends to zero, we easily conclude that $\mathcal{E}\in M_w(\rho^*;B)$ and the non-vanishing condition is clear.
\end{proof}

\begin{Definition}
{\em We call the function $\mathcal{E}$ of Lemma \ref{defi-general-eisenstein-series} the {\em Eisenstein series of weight $w$ associated with the data $(\rho^*,\mu)$} and we denote it by $\mathcal{E}_w(\rho^*;\mu)$ or more simply $\mathcal{E}_w(\rho^*)$ when the reference to $\mu$ is understood.}
\end{Definition}

Although we can always associate Poincar\'e series to representations of the first kind $\rho$
(it follows from Proposition \ref{proposition-poincare-bis} that for any representation of the first kind $\rho$ there exists $m\in\ZZ/(q-1)\ZZ$ and $w>0$ such that a column of a Poincar\'e series constructed there defines an non-zero element of $M_w(\rho)$) not every representation $\rho$ can be enriched by a map $\mu$ as above. The reader can check that if $\rho$ is a representation of the first kind that can be constructed by starting from basic representations by using only the elementary operations $\oplus,\otimes,S^m,\wedge^m$ (so the operation $(\cdot)^*$ is omitted) then maps like $\mu$ exist which are not zero and Lemma \ref{defi-general-eisenstein-series} can be applied to construct non-zero Eisenstein series in $M_w(\rho^*)$ for certain $w>0$. In this paper Eisenstein series will be studied in depth for specific choices of $\rho$ only. Namely, we will study in \S \ref{eisensteinseries} Eisenstein series associated to the representation $\rho_\Sigma^*$ with $\Sigma$ a finite subset of $\NN^*$.

\section{Differential operators on modular forms, Perkins' series}\label{differential-perkins}

A classical feature of modular forms for the group $\operatorname{SL}_2(\ZZ)$ is the existence of differential operators acting homogeneously on them (sending families of modular forms to modular forms). For instance, one can mention the so-called Serre's derivatives, Rankin-Cohen's brackets etc. For scalar Drinfeld modular forms associated to the characters $\det^{-m}$, similar structures exist and have been investigated (see \cite{BOS&PEL0,BOS&PEL,PAP&ZEN}). Here we describe the natural extension of Serre's derivatives over the Drinfeld modular forms for a representation of the first kind. In order to justify the existence of such operators, we need to first show that divided derivatives leave the fields of uniformizers invariant. 

In this section (see \S \ref{theweightofperkins}) we will also apply our results on quasi-periodic functions and higher derivatives to determine, in Theorems \ref{theopsi} and \ref{moregeneralcomputation}, the $v$-valuations of certain series introduced by Perkins in his Thesis \cite{PER0}, which turn out to be related to tame series. Perkins noticed that these series play a singular role in series expansions of Eisenstein series (see \S \ref{eisensteinseries}).

All along this section, we consider the divided higher derivatives:
$$\mathcal{D}_m(z^n)=\binom{n}{m}z^{n-m},\quad n,m\in\NN.$$
We choose $(B,|\cdot|_B)$ a Banach $L$-algebra which is countably cartesian in the sense of Definition \ref{algebra-countably-cartesian}. For all $n\geq 0$,
$\mathcal{D}_n$ defines a $B$-linear endomorphism of $\mathcal{O}_{\mathbb{A}^{1,an}_{\CC_\infty}/B}$. Note that these operators satisfy
Leibniz's rule $$\mathcal{D}_n(fg)=\sum_{i+j=n}\mathcal{D}_i(f)\mathcal{D}_j(g),$$ for 
$f,g$ analytic functions. To handle divided derivatives it is convenient to introduce the following map, where $x$ is an indeterminate and where $\underline{\mathcal{D}}$ denotes the family
of operators $(\mathcal{D}_n)_{n\geq 0}$ (Taylor's map):
$$\Exp{\mathcal{D}}{x}:\mathcal{O}_{\mathbb{A}^{1,an}_{\CC_\infty}/B}\rightarrow\mathcal{O}_{\mathbb{A}^{1,an}_{\CC_\infty}/B}[[x]],\quad \Exp{\mathcal{D}}{x}(f)=\sum_{i\geq 0}\mathcal{D}_i(f)x^i.$$
Then, $\Exp{\mathcal{D}}{x}$ induces $B$-algebras morphisms at the level of the sections, and Leibniz's rule is equivalent to the multiplicativity
$\Exp{\mathcal{D}}{x}(fg)=\Exp{\mathcal{D}}{x}(f)\Exp{\mathcal{D}}{x}(g)$. Let $Y$ be an affinoid subdomain of $\mathbb{A}^{1,an}_{\CC_\infty}/B$, $z\in Y$ and $x_0\in\CC_\infty$ such that $z+x_0\in Y$. If $f\in\mathcal{O}_{\mathbb{A}^{1,an}_{\CC_\infty}/B}$ then $\Exp{\mathcal{D}}{x}(f)_{x=x_0}=f(z+x_0)$. If $x,y$ are two indeterminates,
we therefore have 
$$\Exp{\mathcal{D}}{x}(\Exp{\mathcal{D}}{y}(f))=\Exp{\mathcal{D}}{x+y}(f).$$ This implies that the family of higher derivatives $\underline{\mathcal{D}}$ is {\em iterative}: 
$$\mathcal{D}_{m+n}=\binom{m+n}{m}\mathcal{D}_m\circ\mathcal{D}_n=\binom{m+n}{n}\mathcal{D}_n\circ\mathcal{D}_m,$$ for all $m,n\geq 0$. By an application of Lucas' formula, 
if $n=n_0+n_1q+\cdots+n_rq^r\in\NN$ with $n_0,\ldots,n_r\in\{0,\ldots,q-1\}$, we have the identity
$$\mathcal{D}_n=\mathcal{D}_{n_0}\circ\mathcal{D}_{n_1q}\circ\cdots\circ\mathcal{D}_{n_rq^r},$$
and the operators $\mathcal{D}_{n_iq^i}$ mutually commute, for $i=0,\ldots,r$.

\subsection{Higher derivatives on tame series}\label{sectionongosspoly}
We show that tame series are closed under higher derivations. The main result of this subsection is 
Proposition \ref{propdoublestructure} but we also present some auxiliary properties that can be of interest for the reader willing to do computations. Let $\Sigma$ be a finite subset of $\NN^*$ with $s$ elements. 
Let $m\geq 0$ be the unique integer such that $(m-1)(q-1)+1\leq s\leq m(q-1)$. 
If $s=0$ then $m=0$.
Let $l$ be the unique integer with $s=(m-1)(q-1)+l$ (so that $1\leq l\leq q-1$ and if $s=m=0$, then $l=q-1$). 
We set:
\begin{equation}\label{Ms}
M_s=e_1^{q-1}\cdots e_{m-1}^{q-1}e_m^l\in\Tamecirc{\FF_q}\end{equation}
(note that we can define the $B$-module $\Tamecirc{B}$ for any $\FF_q$-algebra $B$).
We clearly have, by the fact that $s=(m-1)(q-1)+l$: 
\begin{equation}\label{def-w-max}
\lambda(M_s)=s,\quad w(M_s)=(q-1)\sum_{i=1}^{m-1}\frac{1}{q^i}+\frac{l}{q^m}=:w_{\max}(s).
\end{equation}
We set, for $B$ as in \S \ref{tameseriessect}, $\Tamecirc{B}_s=\Tamecirc{B}\cap\Tame{B}_s$ (recall the graduation by depths (\ref{grading-T}) and $\lambda$ in Definition \ref{def-BX}). We have the direct sum of $B$-modules $\Tamecirc{B}=\oplus_{s\geq 0}\Tamecirc{B}_s$. 
We call $M_s$ 
 the {\em maximal} tame monomial, a terminology which is motivated by the following result which tells us that
in the homogeneous module $\Tamecirc{B}_s$, $M_s$ has maximal weight (the proof is easy and left to the reader).
\begin{Lemma}\label{for-the-next-proof}
For all 
$f\in \Tamecirc{\KK_\Sigma}_s$ we have $w(f)\leq w_{\max}(s)$.
\end{Lemma}

We have the next rather straightforward result, where 
$w_{\max}$ has been defined in (\ref{def-w-max}) (recall that if $f\in\Tame{B}$ then $f^{[i]}$ is the projection of $f$ on $\Tame{B}_i$ of (\ref{decomposition-f-T})):
\begin{Proposition}\label{propdoublestructure} The following properties hold.
(1) The operators $(\mathcal{D}_i)_{i\geq0}$ induce $B$-linear endomorphisms
of $\Tamecirc{B},\Tame{B},\Tame{B}[[u]]$.
(2) If $f=\sum_if^{[i]}\in\Tame{B}$ is of depth $\leq L$ we have, for all $n\geq 1$:
$$\mathcal{D}_n(f)=\sum_{L\geq i\geq \ell_q(n)}\mathcal{D}_n(f^{[i]}).$$
(3) For all $n\geq 0$ and for all $f\in\Tamecirc{B}$ of depth $\leq s$, $\mathcal{D}_n(f)\in\Tamecirc{B}$ is of depth $\leq s-\ell_q(n)$ and of weight $\leq w_{\max}(s-\ell_q(n))$. (4) We have the commutation rules 
\begin{equation}\label{commutationrules}
\mathcal{D}_n\tau=\left\{\begin{matrix}0\text{ if }q\nmid n\\
\tau\mathcal{D}_{\frac{n}{q}}\text{ if }q\mid n,\end{matrix}\right.\quad n\geq 1.
\end{equation}
\end{Proposition}

\begin{proof}[Sketch of proof.]

If $M\in\Tame{B}_s$ is a tame monomial of depth $s$ (as in \S \ref{asymptotic-beha}), then $\mathcal{D}_n(M)$ is a tame polynomial, and
$$\mathcal{D}_n(M)\in\bigoplus_{i\geq 0}\Tame{B}_{s-\ell_q(n)-i(q-1)}.$$
To see this consider more generally, for $i\in U$ with $U$ a finite subset of $\NN^*$ of cardinality $s$, $\FF_q$-linear functions $f_i\in\operatorname{Hol}(\CC_\infty\rightarrow B)$, so that we can write
$$f_i=\sum_{j\geq 0}f_{i,q^j}z^{q^j},\quad f_{i,q^j}\in B,\quad i\in U.$$
By Leibniz's rule we have for $n\geq 0$:
$$\mathcal{D}_n\left(\prod_{i\in U}f_i\right)=\sum_{i_1+\cdots+i_s=n}\prod_{k\in U}\mathcal{D}_{i_k}(f_k).$$ By $\FF_q$-linearity we have that 
$\mathcal{D}_k(f_i)=f_i$ if $k=0$, $\mathcal{D}_k(f_i)=f_{i,q^j}$ if $k=q^j$ with $j\in\NN$, and $0$ otherwise.
Hence, setting $f_{i,0}:=f_i$, we can write:
\begin{equation}\label{ausefulformula}
\mathcal{D}_n\left(\prod_{i\in U}f_i\right)=\sum_{\begin{smallmatrix}i_1+\cdots+i_s=n\\ 
i_k\in \{0\}\cup q^\NN;\forall k\end{smallmatrix}}\prod_{k\in U}f_{k,i_k},
\end{equation} if the subset of indices is non-empty, and $0$ otherwise, by the usual conventions on empty sums. 
Coming back to our elements of $\Tame{B}$, since for all $i$, $e_i$ is $\FF_q$-linear, we deduce that for all $n\geq 0$, $\mathcal{D}_n$ sends tame monomials on tame polynomials and therefore the 
operators $\mathcal{D}_i$ induce $B$-linear endomorphisms of $\Tame{B}$ as expected and the property corresponding to $\Tamecirc{B}$ follows easily. Now, it is easy to see that the operators $D_i$ extend to $B$-linear endomorphisms of $B[u^{-1}][[u]]$ so that we can also deduce the expected property for 
$\Tame{B}[[u]]$ and this suffices to justify (1). For (2), 
let $n$ be in $\NN^*$ and let us consider the set of decompositions of length $r\geq1$
$$n=\sum_{i=1}^rn_iq^i,\quad r\in\NN,\quad n_i\in\NN^*.$$
Then, the $q$-ary expansion of $n$ (the unique one which has the coefficients $n_i\in\{0,\ldots,q-1\}$) minimises the length $r=\ell_q(n)$. The reader can complete the verifications of the remaining properties of the proposition.
\end{proof}

\begin{Remark}
{\em The behavior of $v$ with respect to the action of the operator $\tau$ is multiplicative. On the other hand, it is difficult to make the interaction between $v$ and the collection of operators $\underline{\mathcal{D}}$ explicit which introduces a difficulty in handling our modular forms.}
\end{Remark}

\subsection{Divided higher derivatives of $\rho$-quasi-periodic functions}\label{divided-higher-derivatives}

We discuss here the problem of the computation of higher divided derivatives of the entries of the 
matrix functions $\Phi_\rho$ and $\Psi_m(\rho)$ for $m\geq 1$. We added this section to allow readers to perform explicit computations of higher derivatives of our modular forms. Indeed, the latter are all $\rho$-quasi-periodic and Proposition \ref{quasiperiodictempered} tells us that in order to explicitly compute higher derivatives of $\rho$-quasi-periodic functions, it suffices to explicitly compute higher derivatives of 
$u$ and $\Phi_\rho$. 

For this purpose it is convenient to choose a different normalisation for the higher derivatives. We set $D_n=(-\widetilde{\pi})^{-n}\mathcal{D}_n$
for all $n\geq 0$ and we write $\underline{D}=(D_i)_{i\geq 0}$. The formalism of the function
$\Exp{\mathcal{D}}{x}$ extends to $\underline{D}$ and matrix functions. We set, for $f$ 
an analytic function $\Omega\rightarrow\KK_\Sigma^{N\times N}$,
$$\Exp{D}{x}(f)=\sum_{i\geq0}D_i(f)x^i.$$
This defines, with $\mathbb{H}:=\operatorname{Hol}_{\KK_\Sigma}(\Omega\rightarrow\KK_\Sigma^{N\times N})$ a $\KK_\Sigma^{N\times N}$-algebra morphism 
$$\mathbb{H}\xrightarrow{\Exp{D}{x}}\mathbb{H}[[x]].$$
We also set
\begin{equation}\label{def-Gm}
G_m(\rho)=\widetilde{\pi}^{-m}\sum_{a\in A}(z-a)^{-m}\rho(T_a)=D_{m-1}(G_1(\rho)),\quad m\geq1,\end{equation}
and $G_0(\rho)=0$. The generating series of these functions is
\begin{equation}\label{G-rho}
\boldsymbol{G}(\rho):=\sum_{i\geq 0}G_i(\rho)x^i=x\Exp{D}{x}(G_1(\rho)).\end{equation}
We have the next lemma where we recall that $\exp_C(x)=\sum_{i\geq 0}d_i^{-1}x^{q^i}$ is Carlitz's exponential in $x$ (see \S \ref{Carlitz-exponential}).
\begin{Lemma}\label{generalisation-gekeler} The following formula holds:
$$\boldsymbol{G}(\rho)=\frac{ux}{1-u\exp_C(x)}\Exp{D}{x}(\Phi_{\rho}).$$
\end{Lemma}
\begin{proof}
It suffices to compute $\Exp{D}{x}(G_1(\rho))$. Since $G_1(\rho)=u\Phi_\rho$ the formula is obvious 
if we prove that $\Exp{D}{x}(G_1(\boldsymbol{1}))=\frac{u}{1-u\exp_C(x)}$. This is well known, see Gekeler \cite[(3.6)]{GEK}. Nevertheless, we recall the proof here. From $1=uu^{-1}$ we see that
$1=\Exp{D}{x}(u)\Exp{D}{x}(u^{-1})$. But $\Exp{D}{x}(u^{-1})=u^{-1}-\sum_{i\geq 0}d_i^{-1}x^{q^i}$ (recall that $u^{-1}$ is $\FF_q$-linear). Hence $$\Exp{D}{x}(u)=\frac{1}{u^{-1}-\sum_{i\geq 0}d_i^{-1}x^{q^i}}=\frac{u}{1-u\exp_C(x)}.$$
\end{proof}
If $\rho=\boldsymbol{1}$ then the formula of Lemma \ref{generalisation-gekeler} reduces to \cite[(3.6)]{GEK} because in this case $\Phi_\rho=1$. In general, the next Lemma can be helpful in determining some properties of $\Exp{D}{x}(\Phi_\rho)$.

\begin{Lemma}\label{matrix-M}
There exist $\vartheta\in\FF_q(\underline{t}_\Sigma)$ and a matrix $$M\in A[\vartheta,\theta^{-1}][[x]]^{N\times N}$$ such that 
$$\Exp{D}{x}(\Phi_\rho)=\Phi_\rho M.$$
\end{Lemma} 

\begin{proof}
Since $\Exp{D}{x}(\Phi_{\rho}(z+a))=\rho(T_a)\Exp{D}{x}(\Phi_{\rho}(z))$ for all $a\in A$, by Proposition \ref{propdoublestructure},
the coefficients $D_i(\Phi_{\rho})$ of the $x$-expansion of $\Exp{D}{x}(\Phi_{\rho})$ are, for $\vartheta\in\FF_q(\underline{t}_\Sigma)$, elements of $(\Tamecirc{A[\vartheta,\theta^{-1}]})^{N\times N}$ which are in $\mathcal{QP}^!_0(\rho;\KK_\Sigma)$, therefore of the form $\Phi_\rho M_i$ with $M_i\in A[\vartheta,\theta^{-1}]^{N\times N}$.
\end{proof}
For further use we also state and prove:
\begin{Corollary}\label{v-integrality}
There exists $\vartheta\in\FF_q(\underline{t}_\Sigma)$ such that $$\Exp{D}{x}\Big(\omega_\rho\boldsymbol{G}(\rho)\omega_\rho^{-1}\Big)\in\Big(\Tame{K[\vartheta]}[u][[x]]\Big)^{N\times N}$$ and every truncation of the power series in $x$ has coefficients in 
$\Tame{A[\vartheta,\Theta][u]}$ for some $\Theta\in K$ depending on the order of the truncation.
\end{Corollary}
\begin{proof}
We have $\frac{u}{1-u\exp_C(x)}\in K[u][[x]]$ and $$\Exp{D}{x}\Big(\omega_\rho\Phi_\rho\omega_\rho^{-1}\Big)\in\Big(\Tamecirc{A[\vartheta,\theta^{-1}]}[[x]]\Big)^{N\times N}$$ by Corollary \ref{previouscorollary} (the elements $\Theta$ are necessary in the statement because of the denominators of the coefficients of the series expansion of $\frac{u}{1-u\exp_C(x)}$ and the presence of $\theta$ in the denominators of the series expansion associated to $\omega_\rho\Phi_\rho\omega_\rho^{-1}$).
\end{proof}

For example, if $\rho=\rho_\chi$ is basic, we have seen in Corollary \ref{caseofthebasicrepresentation} that $\Phi_\rho=\Xi_\rho=\Big(\begin{smallmatrix}I_n & \chi(z) \\ 0_n & I_n\end{smallmatrix}\Big)$, with $N=2n$. By (\ref{chidefi}) $D_{q^k}(\chi)=D_k^{-1}\Big(\vartheta-\theta^{q^k} I_n\Big)^{-1}\omega_\chi^{-1}$ for $k\geq 0$, and $D_{j}(\chi)=0$ if $j>0$ is not a $q$-power. Hence in this case the matrix $M$ of Lemma \ref{matrix-M} is:
$$M=\begin{pmatrix}I_n & \chi(z)\\ 0_n & I_n\end{pmatrix}+
\omega_\chi^{-1}\exp_C\Big((\vartheta-\theta I_n)^{-1}x\Big)
\begin{pmatrix}0_n & I_n\\ 0_n & 0_n\end{pmatrix},$$ with $\tau(x)=x^q$.

\subsection{Serre's derivatives}\label{serre-derivatives}

In this subsection we prove part (6) of our Theorem A.
We discuss variants of {\em Serre's higher derivatives} introduced in \cite[\S 1.2.3]{BOS&PEL}. Following this reference, we set, for $n,w\in\NN$ and $f\in\operatorname{Hol}(\Omega\rightarrow \KK_\Sigma^{N\times 1})$:
\begin{equation}\label{ramanujan-derivatives}
\partial_n^{(w)}(f):=D_n(f)+\sum_{i=1}^n(-1)^i\binom{w+n-1}{i}D_{i-1}(E)D_{n-i}(f),\end{equation}
where $E$ is the normalised false Eisenstein series
of weight $2$ and type $1$ of Gekeler, defined in \cite[\S 8]{GEK}. We recall the definition here, for convenience of the reader. We can define $E$ by using the conditionally convergent lattice sum
$$E(z)=\widetilde{\pi}^{-1}\sum_{a\in A^+}\sum_{b\in A}\frac{a}{az+b}.$$ This defines a rigid analytic function
$E:\Omega\rightarrow\CC_\infty$ which satisfies
$$E(\gamma(z))=J_\gamma(z)^2\det(\gamma)^{-1}\Big(E(z)-\widetilde{\pi}\frac{c}{cz+d}\Big),\quad \gamma=(\begin{smallmatrix}* & * \\ c & d\end{smallmatrix})\in\Gamma$$
(a {\em Drinfeld quasi-modular form of weight $2$, type $1$} and depth $1$ in the terminology of \cite{BOS&PEL0}). We also recall the $u$-expansion, with $u_a=e_C(az)^{-1}$:
$$E=\sum_{a\in A^+}au_a.$$
Another property of $E$ is that it can be computed as a logarithmic derivative $E=\frac{D_1(\Delta)}{\Delta}$ of $\Delta$ the cusp form of weight $q^2-1$ defined in \cite[\S (6.4)]{GEK}. See also \S \ref{few-examples-quasimodular}.
Coming back to our modular forms, note that the case $n=1$ of (\ref{ramanujan-derivatives}) yields the operator $\partial_1^{(w)}=D_1-wEI_N$. This 
is the analogue of Ramanujan's derivative introduced by Gekeler in \cite[(8.5)]{GEK}.

\begin{Theorem}\label{serre-derivatives} Let $\rho:\Gamma\rightarrow\GL_N(\FF_q(\underline{t}_\Sigma))$ be a representation of the first kind.
The operator $\partial_n^{(w)}$ induces a $\KK_\Sigma$-linear map 
$M_{w}(\rho;\KK_\Sigma)\rightarrow S_{w+2n}(\rho\det^{-n};\KK_\Sigma)$ and an $\LL_\Sigma$-linear map 
$M_{w}(\rho;\LL_\Sigma)\rightarrow S_{w+2n}(\rho\det^{-n};\LL_\Sigma)$.
\end{Theorem}

\begin{proof}
If $f\in M_{w}(\rho;\KK_\Sigma)$ then $f$ can be identified with an element of 
$\mathfrak{O}_\Sigma^{N\times 1}$ (Theorem \ref{theorem-u-expansions}) which is
$\partial_n^{(w)}$-stable for all $n,w$. The same arguments of the proof of \cite[Theorem 4.1]{BOS&PEL} (which holds in a wider context of {\em Drinfeld quasi-modular forms}) imply that
$\partial_n^{(w)}(f)\in M_{w+2n}(\rho\det^{-n};\KK_\Sigma)$. Further, it is easy to see that 
$\partial_n^{(w)}(f)$ has entries in $\mathfrak{M}_\Sigma$ so it is a cusp form.
\end{proof}

 \subsection{Application to Perkins' series}\label{theweightofperkins}
In this subsection we present the series indicated in the title, originally introduced by Perkins in his Ph. D. Thesis \cite{PER0}, as generating series for certain zeta values in Tate algebras introduced by the author in \cite{PEL1}. These series define elements of $\mathfrak{O}_\Sigma$ and the problem of computing their $v$-valuations (or equivalently, weights) arises. This is quite an intricate problem that we partially solve here. One of the difficulties is that the matrix formalism of the preceding sections
does not seem suitable to extract this kind of information.

Let $U$ be a finite subset of $\NN^*$. We set $$\sigma_U=\prod_{i\in U}\chi_{t_i}.$$ 
Explicitly, $\sigma_U(a)=\prod_{i\in U}\chi_{t_i}(a)\in \FF_q[\underline{t}_U]$ for all $a\in A$. 

For further use, with $\Sigma$ a given finite subset of $\NN^*$:

\begin{Definition}\label{def-semicharacters}
{\em A {\em semicharacter} is a map $\sigma:A\rightarrow\FF_q[\underline{t}_\Sigma]$ defined, for $a\in A$, by
$\sigma(a)=\prod_{i\in \Sigma}\chi_{t_i}(a)^{\alpha_i}$ for integers $\alpha_i\geq 0$.}
\end{Definition}

We are interested in the following class of function.

\begin{Definition}\label{defperkinsseries}
{\em Let $U$ be a finite subset of $\NN^*$. The {\em Perkins series of order $n\geq 1$ associated to $\sigma_U$} is the series:
$$\psi(n;\sigma_U)=\sum_{a\in A}(z-a)^{-n}\sigma_U(a).$$}
\end{Definition}
For any $U$ and $n$ as above, the series converges for $z\in \CC_\infty\setminus A$ (with respect to the norm $\|\cdot\|$ of $\KK_\Sigma$, $\Sigma$ being a finite subset of $\NN^*$ containing $U$) and 
$z\mapsto e_A(z)^n\psi(n;\sigma_U)(z)$ define entire functions $\CC_\infty\rightarrow\EE_\Sigma$, as it is easily seen. 
If $U=\emptyset$ we have $\sigma_\emptyset=\boldsymbol{1}$ the {\em trivial semi-character}, and Perkins' generating series are related to Goss' polynomials associated to the lattice $A\subset\CC_\infty$ as in \cite[\S 6]{GOS2} and \cite[\S 3]{GEK}. Indeed,
\begin{equation}\label{connectionwithgosspoly}
\psi(n;\boldsymbol{1})=S_{n,A}=\sum_{b\in A}\frac{1}{(z-b)^{n}}=G_{n,A}(S_{1,A}),
\end{equation}
for polynomials $G_{n,A}\in K_\infty[X]$ (in the notations of \cite{GEK}.)
The functions $\psi(n;\sigma_U)$ with 
$U\subset \Sigma$ occur in the entries of $\Psi_n(\rho_\Sigma)$, where $\rho_\Sigma$ is the representation of the first kind
$$\rho_\Sigma=\bigotimes_{i\in\Sigma}\rho_{t_i},$$
where $\rho_{t_i}(\begin{smallmatrix}a & b \\ c & d\end{smallmatrix})=
\Big(\begin{smallmatrix}a(t_i) & b(t_i) \\ c(t_i) & d(t_i)\end{smallmatrix}\Big)$ (or alternatively, one can also use $\rho=\rho^*_\Sigma$).
Since $\Psi_n(\rho_\Sigma)\in\mathcal{QP}^!_n(\rho_\Sigma;\EE_\Sigma)$ by 
Proposition \ref{somepropertiesxirho}, Lemma \ref{generalisation-gekeler} implies:
\begin{Lemma} 
For all $U\subset\Sigma$ and $n\geq 1$ we have $\psi(n;\sigma_U)\in\mathfrak{K}_\Sigma$. Additionally, $\phi(1;\sigma_U):=e_0\psi(1;\sigma_U)\in\Tamecirc{\EE_\Sigma}$.
\end{Lemma}

\subsubsection{Perkins' series of order $n=1$}\label{perkins-series-order-1} 
We focus now on $\phi(1;\sigma_\Sigma)\in\Tamecirc{\EE_\Sigma}$. The next question is the computation of its weight.
We set, for $\Sigma$ non-empty with $s=|\Sigma|=(m-1)(q-1)+l$ with $m\geq 1$ and $l\in\{1,\ldots,q-1\}$:
 \begin{equation}\label{def-kappa}
 \kappa(\Sigma):=q^{-m}(q-l)\in]0,1[\cap\ZZ[p^{-1}].
 \end{equation} For $\Sigma=\emptyset$, we extend the definition to $\kappa(\emptyset):=1$. Note that $\kappa(\Sigma)$ defines a strictly decreasing function $|\Sigma|\mapsto\kappa(\Sigma)$, and $\lim_{|\Sigma|\rightarrow\infty}\kappa(\Sigma)=0$. 
 We prove:
\begin{Theorem}\label{theopsi} 
The function $\phi(1;\sigma_\Sigma)\in\Tamecirc{\EE_\Sigma}$ has weight
\begin{equation}\label{theformula}
w(\phi(1;\sigma_\Sigma))=1-\kappa(\Sigma)=1-q^{1-m}+lq^{-m}=w_{\max}(s).
\end{equation}
\end{Theorem}

\begin{proof} The identities connecting $\kappa$ and $w_{\max}$ are easily verified. If $\Sigma=\emptyset$, it is clear that $\phi(1;\sigma_\Sigma)$
has weight $0$ (it is in this case a constant function).
We suppose that $\Sigma$ is non-empty. We consider the unique representative $g_\Sigma\in\Tamecirc{\EE_\Sigma}$
of $\prod_{i\in\Sigma}\chi_{t_i}(z)$ (see \S \ref{a-class-chi} for the definition of $\chi_t(z)$) modulo the ideal of $\Tame{\EE_\Sigma}$ generated by $e_0$.
By Corollary \ref{previouscorollary}, we have $\phi(1;\sigma_\Sigma)=g_\Sigma$. We can write $g_\Sigma=\sum_{i=0}^sg_\Sigma^{[i]}$
with $g_\Sigma^{[i]}\in\Tamecirc{\EE_\Sigma}_i$ (see (\ref{decomposition-f-T})) (\footnote{In fact, one sees that if $i\not\equiv s\pmod{q-1}$, then $g_\Sigma^{[i]}=0$.}). 
We note that 
\begin{equation}\label{expansionofthetamepart}
g_\Sigma^{[s]}=\left[\prod_{i\in\Sigma}\chi_{t_i}(z)\right]^{[s]}=\underbrace{e_1^{q-1}\cdots e_{m-1}^{q-1}e_m^l}_{\text{Tame monomial $M_s$}}\mathcal{P}_\Sigma+\Phi,\end{equation}
with $w(M_s)=w_{\max}(s)$, $\Phi\in\Tame{\EE_\Sigma}$, with $w(\Phi)<w_{\max}(s)$, and where
$$\mathcal{P}_\Sigma:=\sum_{\begin{smallmatrix}I_0\sqcup I_1\sqcup\cdots\sqcup I_m=\Sigma\\
|I_0|=\cdots=|I_{m-1}|=q-1\\
|I_m|=l\end{smallmatrix}
}\left(\prod_{i_1\in I_1}t_{i_1}\right)\cdots\left(\prod_{i_m\in I_m}t_{i_m}^{m-1}\right)\in\FF_p[\underline{t}_\Sigma].$$ This polynomial is non-zero as it is easily verified 
by choosing a subset $\widetilde{I}\subset \Sigma$ such that $|\Sigma\setminus\widetilde{I}|=l$.
Substituting $t_i$ by $1$ if $i\in \Sigma\setminus\widetilde{I}$ and by $0$ if 
$i\in \widetilde{I}$, we get the value $1$.
By Lemma \ref{for-the-next-proof}:
$$w\Big(g_\Sigma^{[s]}-\mathcal{P}_\Sigma M_s\Big)<w_{\max}(s).$$
This implies the theorem because the map $s\mapsto w_{\max}(s)$ is a strictly increasing function ($s>0$) so that 
$$w(\phi(1;\sigma_\Sigma))=w(g_\Sigma)=w(g_\Sigma^{[s]})=w_{\max}(s).$$
\end{proof} 

For all $\Sigma\subset\NN^*$ a finite subset, the above proof yields the next corollary:
\begin{Corollary}
We have 
$$\lim_{\begin{smallmatrix}|z|_\Im\rightarrow\infty\end{smallmatrix}}e_A(z)^{\kappa(\Sigma)}\psi(1;\sigma_\Sigma)=\mathcal{P}_\Sigma.$$
\end{Corollary}

\subsubsection*{Example} If $\Sigma$ is a singleton we can work with one variable $t$ and we have the explicit formula, due to Perkins, a simple proof of which can be found in \cite{PEL&PER} (combine (3) and Theorem 1):
\begin{equation}\label{Perkins-formula}
\psi(1;\chi)=\widetilde{\pi}u(z)\chi_t(z).
\end{equation}
Let $\Sigma$ be a subset of $\NN^*$ of cardinality $q$. Developing the product $\prod_{k\in\Sigma}e_C\left(\frac{z}{\theta-t_k}\right)$ we get, after elimination of the $q$-th powers:
\begin{multline*}
\prod_{k\in\Sigma}e_C\left(\frac{z}{\theta-t_k}\right)=e_0-\sum_{j\geq 0}\left(\theta\prod_{i\in\Sigma}t_i^j-\prod_{i\in\Sigma}t_i^{j+1}\right)e_{j+1}+\\ +\sum_{\begin{smallmatrix}0\leq i_1\leq \cdots\leq i_q\\ i_k\text{ not all equal}\end{smallmatrix}}e_{i_1+1}\cdots e_{i_q+1}\sum_{\begin{smallmatrix}
\underline{\alpha}=(\alpha_i:i\in\Sigma)\in\NN^{|\Sigma|}\\ |\underline{\alpha}|=q^{i_1}+\cdots+q^{i_q}
\end{smallmatrix}}\prod_{k\in\Sigma}t_k^{\alpha_k}.\end{multline*}
from this tame series expansion (of depth $q$) we deduce that the leading tame monomial of the tame series $\prod_{i\in\Sigma}\chi_{t_i}(z)$ is $e_0$. Hence, $\prod_{k\in\Sigma}e_C\left(\frac{z}{\theta-t_k}\right)-e_0\in\Tamecirc{A[\underline{t}_\Sigma]}$ and we get an explicit computation of $\phi(1;\sigma_\Sigma)$ for this choice of $\Sigma$.

\subsubsection{Perkins' series of higher order}

In this part we are interested in the following question:
\begin{Question}\label{questionorder}
Compute the valuation $v(\psi(n;\sigma_\Sigma))\in\ZZ[\frac{1}{p}]_{\geq 0}$ explicitly in terms of $l,m,n$. 
\end{Question} 
The case $\Sigma=\emptyset$, where $N=1$ was partially settled by Gekeler in \cite{GEK1}. The complete 
solution is now available in Gekeler's manuscript \cite{GEK2}. In Theorem \ref{moregeneralcomputation} we give a partial answer in the several variables case.
We suppose that $s=|\Sigma|\neq0$. 
We recall that by Proposition \ref{propdoublestructure}, $\mathcal{D}_n$ induces $\KK_\Sigma$-linear endomorphisms of $\Tamecirc{\KK_\Sigma}$ and $\Tame{\KK_\Sigma}$ for all $n$. We also recall that $w$ denotes the opposite of the valuation $v$ (degree).
\begin{Proposition}\label{quiteageneralproposition}
Let $i$ be a non-negative integer, let $r\geq 0$ be such that $\mathcal{D}_{(i+1)q^r-1}(f_\Sigma)\neq0$. Then,
$$w(\psi(1+i;\sigma_\Sigma))=\frac{1}{q^r}w(\mathcal{D}_{(i+1)q^r-1}(f_\Sigma))-\frac{1}{q^r}\in\left[-\frac{1}{q^r},0\right[.$$
\end{Proposition}
\begin{proof}
We observe that $\tau^r(\psi(i+1;\sigma_\Sigma))\in\mathfrak{K}_\Sigma$. Further, we have:
\begin{equation}\label{formulaigeneral}\tau^r(\mathcal{D}_i(\psi(1;\sigma_\Sigma)))=(-1)^i\tau^r(\psi(i+1;\sigma_\Sigma))=\mathcal{D}_{q^r(i+1)-1}(\psi(1;\sigma_\Sigma)),\quad i,r\geq 0.\end{equation}
We are interested in the computation of the weight of $\tau^r(\psi(i+1;\sigma_\Sigma))$ (it is equal to $q^r$ times the weight of $\psi(i+1;\sigma_\Sigma)$, which is the quantity we ultimately want to compute). 
We set $f_i=\mathcal{D}_i(e_C(z)\psi(1;\sigma_\Sigma))\in\Tamecirc{\KK_\Sigma}$. 
In particular, $f_0=f_\Sigma=e_C(z)\psi(1;\sigma_\Sigma)\in\Tamecirc{\KK_\Sigma}$.
By Leibniz's rule,
we have
\begin{equation}\label{fiintermsofpsi}
f_i=e_C(z)\mathcal{D}_i(\psi(1;\sigma_\Sigma))+\underbrace{\sum_{\begin{smallmatrix}\alpha+\beta=n\\ \alpha>0\end{smallmatrix}}\mathcal{D}_\alpha(e_C(z))\mathcal{D}_\beta(\psi(1;\sigma_\Sigma))}_{=:\Xi}.
\end{equation} 
All terms of the above sum are in $\mathfrak{K}_\Sigma$. Since the higher derivatives of positive order of $e_C(z)$ are constant and all the functions $\psi(1+\beta;\sigma_\Sigma)$ for $\beta\geq0$ tend to zero as $|z|=|z|_\Im\rightarrow\infty$, the weight of the above defined term $\Xi$ is $< 0$. We apply the operator $\tau^r$. We get, by (\ref{formulaigeneral}):
\begin{equation}\label{ideinvfi}
\tau^r(f_i)=e_C(z)^{q^r}\mathcal{D}_{(i+1)q^r-1}(\psi(1;\sigma_\Sigma))+\tau^r(\Xi).\end{equation}
We have that $\tau^r(\Xi)\in \mathfrak{K}$ and the weight is $\leq 0$; we also set $n=(i+1)q^r-1$. Then, 
\begin{eqnarray*}
\mathcal{D}_{n}(\psi(1;\sigma_\Sigma))&=&\mathcal{D}_{n}(uf_\Sigma)\\
&=&u\mathcal{D}_n(f_\Sigma)+\underbrace{\sum_{\begin{smallmatrix}\alpha+\beta=n\\ \alpha>0\end{smallmatrix}}\mathcal{D}_\alpha(u)\mathcal{D}_\beta(f_\Sigma)}_{=:\Upsilon}.\end{eqnarray*}
If $\alpha>0$, $\mathcal{D}_\alpha(u)\in\CC_\infty[u]\subset\mathfrak{K}$ which is of weight $\leq -2$ as the reader can easily check. Since $f_\Sigma\in\Tamecirc{\KK_\Sigma}$, the weights of all its higher derivatives are in $\{-\infty\}\cup[0,1[$
and thus, the weight of the term $\Upsilon$ above defined is $<-1$. Let us suppose that $\mathcal{D}_n(f_\Sigma)$ is non-zero. Then, its weight belongs to $[0,1[$ and the weight of 
$u\mathcal{D}_n(f_\Sigma)$ belongs to $[-1,0[$. We deduce that, under this hypothesis of non-vanishing, 
the weight of $\mathcal{D}_{n}(\psi(1;\sigma_\Sigma))$ is equal to the weight of $u\mathcal{D}_n(f_\Sigma)$, belonging to the interval $[-1,0[$. Coming back to the identity (\ref{ideinvfi}) and recalling that $\tau^r(\Xi)$ has negative weight, we deduce that
$\tau^r(f_i)$ and $e_C(z)^{q^r-1}\mathcal{D}_n(f_\Sigma)$ have the same weight, belonging to the interval
$[q^r-1,q^r[$, and the weight of $f_i$ satisfies:
\begin{equation}\label{idweightfi}
w(f_i)=1+\frac{1}{q^r}w(\mathcal{D}_{q^r(i+1)-1}(f))-\frac{1}{q^r}\in\left[1-\frac{1}{q^r},1\right[,\quad r\geq 0,\quad i\geq 1.
\end{equation}
Coming back to (\ref{fiintermsofpsi}), we have noticed that the term $\Xi$ has weight $<0$. But $f_i$ has non-negative weight by (\ref{idweightfi}). Hence, the weight of the first term in the right-hand side of (\ref{fiintermsofpsi}) has the same weight as $f_i$ and the result follows.
\end{proof}

We recall that if $s=|\Sigma|=(m-1)(q-1)+l$ with $m\geq 1$ and $l\in\{1,\ldots,q-1\}$, then $w(\psi(1;\sigma_\Sigma))=lq^{-m}-q^{1-m}$ (see Theorem \ref{theopsi}). We want to compute the weight of $\psi(1+n;\sigma_\Sigma)$ for $n\geq 0$ and this allows to compute the $v$-valuation of these elements. The following Theorem generalizes Theorem \ref{theopsi}:
\begin{Theorem}\label{moregeneralcomputation}
Let $\Sigma$, $s,m,l$ as above and 
let $n$ be $\geq0$ such that $\ell_q(n)\leq l$. Then, 
$$v(\psi(1+n;\sigma_\Sigma))=q^{1-m}-(l-\ell_q(n))q^{-m}.$$
\end{Theorem}
\begin{proof}
We choose $i=n$ and $r=0$ in Proposition \ref{quiteageneralproposition} (note that in this case $n=(i+1)q^r-1$). We show that $\mathcal{D}_n(f_\Sigma)\neq0$ and we compute its depth.
To construct $f_\Sigma$, we have applied the rule 
$e_{i-1}=C_\theta(e_i)$ to the product 
$\prod_{i\in\Sigma}e_C\left(\frac{z}{\theta-t_i}\right)$ which implies that 
$$f_\Sigma=f_\Sigma^{[s]}+f_\Sigma^{[s-q+1]}+\cdots.$$ 
We recall that we have already seen that $f_\Sigma^{[s]}$ is equal to $\chi(\sigma_\Sigma)^{[s]}$ and has the monic maximal tame monomial $M_s$ as a non-zero term of its tame expansion. Further, $w(f_\Sigma^{[s-j(q-1)]})<w_{\max}(s)$ for all $j>0$.
Hence, we can write $f_\Sigma=M_s+g$ with $w(g)<w_{\max}(s)$ and $\mathcal{D}_{n}(g)$ has weight strictly less than $$w_{\max}(s-\ell_q(n))=1-q^{1-m}+(l-\ell_q(n))q^{-m}.$$ We now claim that
$w(\mathcal{D}_n(f_\Sigma))=w_{\max}(s-\ell_q(n))$. If this is true, we deduce, from Proposition \ref{quiteageneralproposition}, the formula 
$w(\psi(1+n;\sigma_\Sigma))=(l-\ell_q(n))q^{-m}-q^{1-m}$ hence completing the proof of the Theorem. 
\end{proof}
The claim is the object of the next Lemma, where $M_s$ is defined in (\ref{Ms}):
\begin{Lemma}\label{asimpletechnicallemma}
With $s=|\Sigma|$ equal to $(m-1)(q-1)+l$, $m\geq 1$ and $1\leq l\leq q-1$,
let $n\in\NN$ be such that $\ell_q(n)\leq l$. Then, 
$$\mathcal{D}_n(M_s)=\kappa_nM_{s-\ell_q(n)}+h,$$
with $$\kappa_n:=\left(\frac{\widetilde{\pi}}{\theta^m}\right)^{n_0}
\left(\frac{\widetilde{\pi}^q}{\theta^{mq}d_1}\right)^{n_1}\cdots\left(\frac{\widetilde{\pi}^{q^r}}{\theta^{mq^r}d_r}\right)^{n_r}\in \CC_\infty^{\times},$$
where $n=n_0+n_1q+\cdots+n_rq^r$ is the base-$q$ expansion of $n$, and with $h\in\mathcal{T}^\circ(\CC_\infty)$ of weight $<w_{\max}(s-\ell_q(n))$. 
\end{Lemma}
\begin{proof}
We write $M_s=FG$ with $F=\left(e_1\cdots e_{m-1})\right)^{q-1}$ and $G=e_m^{l}$. By Leibniz's rule $\mathcal{D}_n(M_s)=\sum_{\alpha+\beta=n}\mathcal{D}_\alpha(F)\mathcal{D}_\beta(G)$. If $\alpha>0$, then $w(\mathcal{D}_\alpha(F)\mathcal{D}_\beta(G))$ is strictly smaller than $w_{\max}(s-\ell_q(n))$. Now, we consider the term
with $\alpha=0$. Note, by the formula (\ref{ausefulformula}) applied to the product of $\FF_q$-linear maps $G=e_m^l$, that 
$$\mathcal{D}_n(G)=(\mathcal{D}_{n_0}\circ\mathcal{D}_{n_1q}\circ \cdots\circ\mathcal{D}_{n_rq^r})(e_m^l)=\kappa_n e_m^{l-\ell_q(n)}.$$ The result follows.
\end{proof}

\section{Eisenstein series for $\rho^*_\Sigma$}\label{eisensteinseries}

This section contains the proofs of the various items of Theorem C in the introduction.
We present several aspects of Eisenstein series for the representation $\rho=\rho_\Sigma^*$, with $N=2^s$. These functions provide important examples of the modular forms we consider (see also \cite{PEL2}).  
We set, for $w\in\NN^*$:
$$\mathcal{E}(w;\rho_\Sigma^*):=\sideset{}{'}\sum_{(a,b)\in A}(az+b)^{-w}\bigotimes_{i\in\Sigma}\binom{\chi_{t_i}(a)}{\chi_{t_i}(b)},$$  where the sum runs over the 
$a,b\in A$ which are not both zero. This series corresponds to the choice
$$\mu(a,b)=\bigotimes_{i\in\Sigma}\binom{\chi_{t_i}(a)}{\chi_{t_i}(b)}$$
in (\ref{the-map-mu}) (this is the transposition of the first line of $\rho_\Sigma\Big(\begin{smallmatrix}a & b \\ * & *\end{smallmatrix}\Big)$) so that by Lemma \ref{defi-general-eisenstein-series} $\mathcal{E}_w(\rho_\Sigma^*)\in
M_w(\rho_\Sigma^*)\setminus\{0\}$ if $s=|\Sigma|\equiv w\pmod{q-1}$ (see also \cite[\S 5]{PEL2}). Note also that this series defines a holomorphic function
$\Omega\rightarrow \EE_\Sigma^{N\times 1}$. We call $\mathcal{E}(w;\rho^*_\Sigma)$ the {\em Eisenstein series of weight $w$ associated to $\rho^*_\Sigma$}. 

Here is the plan of this section. In \S \ref{expansion-eisenstein}, Corollary \ref{valuations-eisenstein}, we compute the $v$-valuation of the entries of $\mathcal{E}(1;\rho^*_\Sigma)$. The computation uses results of \S \ref{differential-perkins} on Perkins' series. The general problem of the computation of the $v$-valuation of the entries of $\mathcal{E}(m;\rho^*_\Sigma)$ for $m>0$ is likely to be a difficult problem. Some partial results can be obtained applying Theorem \ref{moregeneralcomputation}. 
In \S \ref{modular-forms-weight-1} we use the Eisenstein series $\mathcal{E}(1;\rho^*_\Sigma)$ to show that the dimension of $M_1(\rho^*_\Sigma;\LL_\Sigma)$ equals one if $|\Sigma|\equiv1\pmod{q-1}$. This is one of the very few spaces of non-scalar Drinfeld modular forms that we are able to fully characterize. As a corollary, the series $\mathcal{E}(1;\rho^*_\Sigma)$ are Hecke eigenforms. 
In \S \ref{integrality-eisenstein-coefficients}, Theorem \ref{integrality-results} we describe integrality properties of the $u$-expansions (in the sense of Proposition \ref{descriptionfieldofunif}) of the entries of $\mathcal{E}(m;\rho^*_\Sigma)$. Naturally, these series expansions are much more complicated and less explicit than those obtained by Gekeler in \cite{GEK} for the scalar Eisenstein series.
In \S \ref{Quasimodular-forms-from-Eisenstein} we show how certain results of Petrov \cite{PET} on $A$-expansions can be generalized to show that series such as
$$\sum_{a\in A^+}a^lG_m(u_a)\in K[[u]]$$ with $l,m>0$ such that $l\equiv m\pmod{q-1}$ give rise to $u$-expansions of {\em quasi-modular forms} in the sense of \cite{BOS&PEL0}. These series occur as special values of an entry of the Eisenstein series $\mathcal{E}(m;\rho^*_\Sigma)$ hence confirming a prediction of D. Goss on a link between Petrov's $A$-expansions and Eisenstein series; see Theorem \ref{to-give-Petrov}. In \S \ref{v-adic-modular} we present, succinctly, some applications to $\mathfrak{v}$-adic modular forms.

\subsubsection*{Link between Eisenstein series and Poincar\'e series}
The next lemma provides a connection with Poincar\'e's series.

\begin{Lemma}\label{connection-eisenstein-poincare} $\mathcal{E}(w;\rho^*_\Sigma)=\zeta_A(w;\sigma_\Sigma)\mathcal{P}_{w}^{(0)}(\Phi_{\rho^*_\Sigma})$.
\end{Lemma}
Here $\mathcal{P}_{w}^{(0)}(\Phi_{\rho^*_\Sigma})$ denotes the last column of the matrix valued Poincar\'e series $\mathcal{P}_{w}(\Phi_{\rho^*_\Sigma})$ defined in (\ref{defpoincare}), with $G=\Phi_{\rho^*_\Sigma}$ and: $$\zeta_A(w;\sigma_\Sigma)=\sum_{a\in A^+}\frac{\sigma_\Sigma(a)}{a^w}=\prod_P\left(1-\frac{\sigma_\Sigma(P)}{P^w}\right)^{-1}\in\TT_\Sigma^\times\cap\EE_\Sigma,$$ the product running over the irreducible monic polynomials of $A$.
\begin{proof}[Proof of Lemma \ref{connection-eisenstein-poincare}]
We consider a matrix $\gamma=(\begin{smallmatrix} * & * \\ c & d\end{smallmatrix})\in\Gamma$.
We note that the last column of $\Phi_{\rho^*_\Sigma}(\gamma(z))$ is the last entry of the canonical basis of the vector space $\FF_q^{N\times 1}$. Indeed, $\Phi_{\rho^*_\Sigma}(z)$ itself is a matrix function which is lower triangular 
with $1$ on the diagonal. Moreover, the last column of $\rho^*_\Sigma(\gamma)^{-1}={}^t\rho_\Sigma(\gamma)$ is $\otimes_{i\in\Sigma}\binom{\chi_{t_i}(c)}{\chi_{t_i}(d)}$, which is therefore also equal to the last column of $\rho^*_\Sigma(\gamma)^{-1}\Phi_{\rho^*_\Sigma}(\gamma(z))$ and 
to the last column of $$\rho_\Sigma^*(\gamma)^{-1}\Phi_{\rho^*_\Sigma}(\gamma(z))\rho_\Sigma^*(\begin{smallmatrix}\det(\gamma) & 0\\ 0 & 1\end{smallmatrix}).$$
Therefore, the last column $\mathcal{P}_{w}^{(0)}(\Phi_{\rho^*_\Sigma})$ of $\mathcal{P}_{w}(\Phi_{\rho^*_\Sigma})$ is
$$\sum_{\begin{smallmatrix}\gamma=(\begin{smallmatrix}* & *\\ c & d\end{smallmatrix})\\
c,d\in A\\ \text{ relatively prime}\end{smallmatrix}}(cz+d)^{-w}\bigotimes_{i\in\Sigma}\binom{\chi_{t_i}(c)}{\chi_{t_i}(d)},$$ independent on the choice of the representatives modulo the subgroup $H$ of $\Gamma$.
Observe that the index set of the sum defining the series $\mathcal{E}(w;\rho^*_\Sigma)$,
$A^2\setminus\{(0,0)\}$, is equal to $\mathcal{I}A^+$, where $\mathcal{I}$ is the set of couples $(c,d)\in A^2$ with $c,d$ relatively prime. This means that 
$$\mathcal{E}(w;\rho^*_\Sigma)=\sum_{a\in A^+}\frac{\sigma_\Sigma(a)}{a^w}\sum_{(c,d)\in 
\mathcal{I}}(cz+d)^{-w}\bigotimes_{i\in \Sigma}(\chi_{t_i}(c),\chi_{t_i}(d))=
\zeta_A(w;\sigma_\Sigma)\mathcal{P}_{w}^{(0)}(\Phi_{\rho^*_\Sigma}).$$
\end{proof}

\subsection{The $v$-valuation of Eisenstein series}\label{v-valuation-eisenstein}
\label{expansion-eisenstein} We expand the entries of our vector-valued Eisenstein series along the principles of Theorem \ref{theorem-u-expansions} and we compute their $v$-valuations in certain cases.  

If $|\Sigma|=s>0$ and $N=2^s$, the ordering on $\Sigma$ induces a bijection $\Sigma\xrightarrow{\varepsilon}\{0,\ldots,s-1\}$. This in turn defines a bijection between subsets $J\subset\Sigma$ and integers $0\leq n\leq N-1$.
If $n=n_0+n_12+\cdots+n_{s-1}2^{s-1}$ is the base-$2$ expansion of $n$, the image of $n$ is the subset $J=\{j\in\Sigma:n_j\neq0\}\subset \Sigma$. We can write $|J|_\Sigma:=n$. For example, $|\emptyset|_\Sigma=0$. Then, we can describe in two ways an $N$-tuple of objects parametrized by the subsets of $\{1,\ldots,2^s\}$: $$f=(f^J)_{J\subset \Sigma}=(f_i)_{1\leq i\leq N},$$ by using that the latter is $(f_{|J|_\Sigma+1})_{J\subset\Sigma}$ (note how we distinguish the $\NN^*$-indexing from the $\Sigma$-indexing).
Note that the first entry is $$f_0=f^\emptyset.$$

The Perkins series $\psi(w;\sigma_U)$ defined in (\ref{defperkinsseries}) are elements of $\mathfrak{O}_\Sigma$, if $U\subset\Sigma$. 
We set
$$\psi_a(w;\sigma_\Sigma):=\psi(w;\sigma_\Sigma)(za),$$
functions which also belong to $\mathfrak{O}_\Sigma$.
Their valuations $v$ are positive and we 
we have, for all $a\in A^+$, $$v(\psi_a(w;\sigma_\Sigma))=|a|v(\psi(w;\sigma_\Sigma)).$$ We set
\begin{equation}\label{definition-V}
V(w;\rho^*_\Sigma):=\frac{1}{\widetilde{\pi}^w}\sum_{b\in A}\frac{1}{(z+b)^w}\Big(\sigma_J(b)\Big)_{J\subset\Sigma}.\end{equation} We denote by $V(w;\rho^*_\Sigma)_a$ the function of the variable $z$  obtained by rescaling $z\mapsto az$. We also set
\begin{equation}\label{definition-Z}
\mathcal{Z}(w;\rho^*_\Sigma):=\begin{pmatrix}0 \\ \vdots \\ 0 \\ \zeta_A(w;\sigma_\Sigma)\end{pmatrix}.
\end{equation}
The next Proposition generalizes \cite[Proposition 3.7]{PEL&PER3} to the case of $\rho=\rho^*_\Sigma$.
\begin{Proposition}\label{seriesofEisensteinseries}
If $s=|\Sigma|\equiv w\pmod{q-1}$ and $w>0$, then:
\begin{equation}\label{V-expansion}
\mathcal{E}(w;\rho^*_\Sigma)=-\mathcal{Z}(w;\rho^*_\Sigma)-\widetilde{\pi}^w\sum_{a\in A^+}\rho_\Sigma\Big(\begin{smallmatrix} a & 0 \\ 0 & 1\end{smallmatrix}\Big)V(w;\rho^*_\Sigma)_a.
\end{equation}
Writing 
$\mathcal{E}(w;\rho^*_\Sigma)=(\mathcal{E}^J)_{I\sqcup J=\Sigma}$, we have, more explicitly:
\begin{eqnarray}
\mathcal{E}^J&=&-(-1)^{|J|}\sum_{a\in A^+}\sigma_I(a)\psi_a(w;\sigma_J),\quad J\neq\Sigma,\label{firstformula}\\
\mathcal{E}^\Sigma&=&-\zeta_A(w;\sigma_\Sigma)-(-1)^{|\Sigma|}\sum_{a\in A^+}\psi_a(w;\sigma_\Sigma)\label{secondformula}.
\end{eqnarray}
In particular, if $J=\emptyset\neq\Sigma$, we have that 
\begin{equation}\label{thirdformula}
\mathcal{E}^\emptyset=-\widetilde{\pi}^w\sum_{a\in A^+}\sigma_{\Sigma}(a)G_w(u_a(z))\in\KK_\Sigma[[u]].
\end{equation}
Moreover, if $\Sigma=\emptyset$, we have, for $q-1\mid n$:
\begin{equation}\label{fourthformula}
\mathcal{E}(w;\boldsymbol{1})=-\zeta_A(w)-\widetilde{\pi}^n\sum_{a\in A^+}G_w(u_a(z)).\end{equation}
In all cases, we can identify $\mathcal{E}(w;\rho^*_\Sigma)$ with an element of $\mathfrak{O}_\Sigma^{N\times 1}$.
\end{Proposition}
We deduce, in yet another way, that $\mathcal{E}(w;\rho^*_\Sigma)\in M_{w}(\rho^*_\Sigma;\KK_\Sigma)$. Additionally, we see that it does not belong to $S_{w}(\rho^*_\Sigma;\KK_\Sigma)$ because of the non-vanishing of $\zeta_A(w;\sigma_\Sigma)$ in (\ref{secondformula}). Note that writing $\mathcal{E}(w;\rho^*_\Sigma)={}^t(\mathcal{E}_1,\ldots,\mathcal{E}_{N-1},\mathcal{E}_N)$, we have $v(\mathcal{E}_i)>0$
for $i=1,\ldots,N-1$ and $v(\mathcal{E}_N)=0$. See Lemma \ref{lemmavaluationswhichvanish} for a more general statement.
\begin{proof}[Proof of Proposition \ref{seriesofEisensteinseries}]
The sum defining $\mathcal{E}(w;\rho^*_\Sigma)$ splits in two pieces, a sum over the couples $(a,b)\in A\times A$ with $a\neq0$ and a sum over the couples $(0,b)$ with $b\neq0$. While the second sum is easily seen to be equal to $-\mathcal{Z}(w;\rho^*_\Sigma)$,
for the first sum we have
\begin{eqnarray*}
\lefteqn{\sum_{\begin{smallmatrix}a\in A\setminus\{0\}\\ b\in A\end{smallmatrix}}\frac{1}{(az+b)^w}\bigotimes_{i\in\Sigma}\binom{\chi_{t_i}(a)}{\chi_{t_i}(b)}=}\\
&=&\sum_{a\in A\setminus\{0\}}\rho_\Sigma\Big(\begin{smallmatrix} a & 0 \\ 0 & 1\end{smallmatrix}\Big)\sum_{b\in A}\frac{1}{(az+b)^w}\bigotimes_{i\in\Sigma}\binom{1}{\chi_{t_i}(b)}\\
&=&\sum_{a'\in A^+}\rho_\Sigma\Big(\begin{smallmatrix} a' & 0 \\ 0 & 1\end{smallmatrix}\Big)\left(\sum_{\lambda\in\FF_q^\times}\lambda^{-w}\rho_\Sigma\Big(\begin{smallmatrix} \lambda & 0 \\ 0 & 1\end{smallmatrix}\Big)\cdot\rho_\Sigma\Big(\begin{smallmatrix} 1 & 0 \\ 0 & \lambda\end{smallmatrix}\Big)\right)\sum_{b'\in A}\frac{1}{(a'z+b')^w}\Big(\sigma_J(b')\Big)_{J\subset \Sigma}\\
&=&-\widetilde{\pi}^w\sum_{a'\in A^+}\rho_\Sigma\Big(\begin{smallmatrix} a' & 0 \\ 0 & 1\end{smallmatrix}\Big)V(w;\rho^*_\Sigma)_a,
\end{eqnarray*} 
where we made the change of variables $a=\lambda a'$, $b=\lambda b'$ in the summation, and used that 
$|\Sigma|\equiv w\pmod{q-1}$ because $\sum_{\lambda\in\FF_q^\times}\lambda^{|\Sigma|-w}=-1$. Now note that
$$V(w;\rho^*_\Sigma)=\frac{1}{\widetilde{\pi}^w}\Big((-1)^{|J|}\psi(w;\sigma_J)\Big)_{J\subset\Sigma}.$$
The identity concerning the case $J=\emptyset\neq\Sigma$ is clear, and the last identity, concerning the scalar Eisenstein series, is well known; see, for instance, \cite[(6.3)]{GEK}.
The last assertion of the proposition is a direct consequence of the
fact that $\psi_a(w;\sigma_\Sigma)\in\mathfrak{O}_\Sigma$ for all $a\in A$ and $w\in\NN^*$ and the fact that 
$v(\psi_a(w;\sigma_\Sigma))=|a|v(\psi_a(w;\sigma_\Sigma))\rightarrow\infty$ as $a$ runs in $A^+$.
\end{proof}

Thanks to Theorem \ref{theopsi} we can compute the $v$-valuations of the entries of $\mathcal{E}(1;\rho_\Sigma^*)$ (recall that $\kappa$ has been introduced in (\ref{def-kappa})).
The corresponding problem for $\mathcal{E}(w;\rho_\Sigma^*)$ for general $w$ is at the moment unsolved but the reader can apply Theorem \ref{moregeneralcomputation} to some specific cases.

\begin{Corollary}\label{valuations-eisenstein}
If $|\Sigma|\equiv1\pmod{q-1}$ and $\mathcal{E}(1;\rho_\Sigma^*)=(\mathcal{E}^J)_{J\subset\Sigma}$, we have $v(\mathcal{E}^J)=\kappa(J)$ if $J\subsetneq\Sigma$
and $v(\mathcal{E}^\Sigma)=0$.
\end{Corollary}

\subsection{Application to modular forms of weight one for $\rho^*_\Sigma$}\label{modular-forms-weight-1}

In this subsection we prove Theorem D of the introduction.
We recall that $N=2^s$.
We have:

\begin{Theorem}\label{coro-pel-per}
Assuming that $|\Sigma|\equiv1\pmod{q-1}$, $M_{1}(\rho^*_\Sigma;\LL_\Sigma)$ is of dimension one over $\LL_\Sigma$, generated by 
the Eisenstein series $\mathcal{E}(1;\rho^*_\Sigma)$.
\end{Theorem}
\begin{proof}
We note that in the case $\rho=\rho_\Sigma^*$ we have the following identity for the space $H(\rho;\KK_\Sigma)$ defined in \S\ref{weight-one-first-part}: 
\begin{equation}\label{identity-H-K}
H(\rho;\KK_\Sigma)=\begin{pmatrix}0 \\ \vdots \\ 0 \\ \KK_\Sigma\end{pmatrix}.
\end{equation} 
We claim that if $f=(f_1,\ldots,f_N)$ is a modular form for $\rho^*_\Sigma$, we can identify $f_1,\ldots,f_{N-1}$ with elements of $\mathfrak{M}_\Sigma$ and 
$f_N$ with an element of $\mathfrak{O}_\Sigma$. Indeed
we know already that $f\in\mathfrak{O}^{N\times 1}_\Sigma$. In particular, there exists $\underline{\alpha}\in\KK_\Sigma^{N\times 1}$
such that $f\equiv\underline{\alpha}\pmod{\mathfrak{M}^{N\times 1}}$. But note that for all $a\in A$,
$f(z+a)=\rho^*_\Sigma(T_a) f(z)$ for all $z\in \Omega$ so that $\underline{\alpha}=\rho^*_\Sigma(T_a)\underline{\alpha}$ for all $a\in A$. Identity (\ref{identity-H-K}) allows to deduce the claim.

We conclude by observing that $\mathcal{E}(1;\rho^*_\Sigma)\in M_{1}(\rho^*_\Sigma;\LL_\Sigma)\setminus S_{1}(\rho^*_\Sigma;\LL_\Sigma)$ and applying Theorem \ref{rankoneforweighone} knowing that in this case,
$\delta_\rho=1$.
\end{proof}

This yields a positive answer to \cite[Problem 1.1]{PEL&PER3}. By Theorem 
\ref{theo-hecke-operators}, $\mathcal{E}(1;\rho^*_\Sigma)$ is an eigenform 
for all the Hecke operators defined in \S \ref{section-Hecke-operators}. We deduce:

\begin{Corollary}\label{coro-7-6}
For all $a\in A\setminus\{0\}$ we have $T_a(\mathcal{E}(1;\rho^*_\Sigma))=\mathcal{E}(1;\rho^*_\Sigma)$.
\end{Corollary}

\begin{proof}
By Theorem \ref{coro-pel-per} $M_{1}(\rho^*_\Sigma;\LL_\Sigma)$ is one-dimensional generated by $\mathcal{E}(1;\rho^*_\Sigma)$
and we have $T_a(\mathcal{E}(1;\rho^*_\Sigma))=\lambda_a\mathcal{E}(1;\rho^*_\Sigma)$ for all $a\in A\setminus\{0\}$ for elements $\lambda_a\in\LL_\Sigma$.
It suffices to show that $\lambda_P=1$ for every irreducible element $P\in A$ by using the Hecke operators $T_P$ described in (\ref{TP}). We set
$f=\mathcal{E}(1;\rho^*_\Sigma)$. In (\ref{TP}), 
$g:=\sum_{|b|<|P|}\rho\Big(\begin{smallmatrix} 1 & b \\ 0 & P\end{smallmatrix}\Big)^{-1}f\left(\frac{z+b}{P}\right)\in\mathfrak{M}_\Sigma^{N\times 1}$. Indeed, let $f_1,\ldots,f_N$ be the entries of
$f$. We have $f_1,\ldots,f_{N-1}\in\mathfrak{M}_\Sigma$ and $f_N\in\mathfrak{O}_\Sigma$. This implies that 
the first $N-1$ coefficients of $g$ are in $\mathfrak{M}_\Sigma$ and  
by (\ref{last-entry}) the last coefficient of $g$ is
$$P^{-1}\sigma_\Sigma(P)\sum_{|b|<|P|}f_N\Big(\frac{z+b}{P}\Big)$$ so it is an element of $\Tamecirc{\LL_\Sigma}$ with zero constant term. Hence $\lambda_P$ equals the lower right coefficient of $\rho^*_\Sigma\Big(\begin{smallmatrix}P & 0 \\ 0 & 1\end{smallmatrix}\Big)^{-1}$ which is equal to $1$.
\end{proof}

\subsubsection{Digression: another class of Eisenstein series}
One of the main motivations for the introduction of the Eisenstein series $\mathcal{E}(w;\rho^*_\Sigma)$, for which they have been initially considered in \cite{PEL1}, 
is that the non-zero entry (which is the last one, in the prescribed ordering) tends to $-\zeta_A(w;\sigma_\Sigma)$ (the zeta values defined in (\ref{zeta-value-tate})) as $z\in\Omega$ approaches the cusp infinity or, in other words, it is congruent to $-\zeta_A(w;\sigma_\Sigma)$ modulo $\mathfrak{M}_\Sigma$. These are not the only Eisenstein series which enjoy this property. Another example is discussed in this remark; further investigations will lead to a better understanding of these examples. We consider the $\FF_q$-algebra morphism $\chi:A\rightarrow \FF_q[\underline{t}_\Sigma]^{s\times s}$ (with $s=|\Sigma|$) defined by 
$$\vartheta=\chi(\theta)=\begin{pmatrix} 0 & 1 & \cdots & 0 \\ 0 & 0 & \cdots & 0 \\ \vdots & \vdots &  & \vdots \\ 
0 & 0 & \cdots & 1 \\
-P_0 & -P_1 & \cdots & -P_{s-1}\end{pmatrix},$$ where $P_0,\ldots,P_{s-1}\in\FF_q[\underline{t}_\Sigma]$ are defined by 
$\prod_{i\in\Sigma}(X-t_i)=X^s+P_{s-1}X^{s-1}+\cdots+P_0$. Then, for all $a\in A$,
$\det(\chi(a))=\sigma_\Sigma(a)$ (see \cite[\S 2.1]{PEL2}). We consider the representation of the 
first kind $\varphi_\Sigma^*=\wedge^s\rho^*_\chi$, of dimension $N:=\binom{2s}{s}$. We suppose that $w\equiv s\pmod{q-1}$ and $w>0$.
The last column of the Poincar\'e series $\mathcal{P}_{w}(\Phi_{\varphi^*_\Sigma})$
multiplied by $\zeta_A(w;\sigma_\Sigma)$ equals $$\mathcal{E}(w;\varphi_\Sigma^*):=\sum_{(a,b)\in A\setminus\{(0,0)\}}(az+b)^{-w}\bigwedge^s\binom{\chi(a)}{\chi(b)}.$$

This defines an element of $\operatorname{Hol}(\Omega\rightarrow\EE_\Sigma^{1\times N})$ and a modular form in $M_w(\varphi_\Sigma^*;\KK_\Sigma)\setminus S_w(\varphi_\Sigma^*;\KK_\Sigma)$. Moreover, the only entry $\mathcal{E}_N$ of $\mathcal{E}(w;\varphi_\Sigma^*)$ which does not vanish at infinity, which is the last one, satisfies
$$\mathcal{E}_N\equiv-\zeta_A(w;\sigma_\Sigma)\pmod{\mathfrak{M}_\Sigma}.$$
In other words, $-\zeta_A(w;\sigma_\Sigma)$ is the 'constant term' of the last entry of $\mathcal{E}(w;\varphi^*_\Sigma)$.

\subsection{Integrality properties of coefficients}\label{integrality-eisenstein-coefficients}

We investigate integrality properties of coefficients of Eisenstein series. Our main result in this subsection is Theorem \ref{integrality-results}, in the same vein as classical results of Gekeler \cite[\S 5]{GEK}.

\begin{Definition}\label{integral-forms}{\em An element $f\in M_w^!(\rho;\KK_\Sigma)$ is said to be {\em rationally definable} if there exists a matrix $M\in\GL_N(\KK_\Sigma)$ such that the image of $Mf$ by the embedding $\iota_\Sigma$ of Theorem \ref{theorem-u-expansions}
is an element of $\Tamecirc{K(\underline{t}_\Sigma)}((u))^{N\times 1}$. It is {\em integrally definable} if this image lies in $\Tamecirc{A[\underline{t}_\Sigma]}[u^{-1}][[u]]^{N\times 1}$. If $\mathfrak{v}:K(\underline{t}_\Sigma)\rightarrow\ZZ\cup\{\infty\}$ is a valuation of $K(\underline{t}_\Sigma)$ we say that a rationally defined element $f\in M_w^!(\rho;\KK_\Sigma)$ is {\em $\mathfrak{v}$-integrally definable} if, writing $f_i$ for the $i$-th entry of $Mf$ with $M$ the above mentioned matrix and expanding it as a formal series 
$f_i=\sum_{j\geq j_0}f_{i,j}u^{j}$ with $f_{i,j}\in\Tamecirc{K(\underline{t}_\Sigma)}$, which can be done in a unique way after Proposition \ref{descriptionfieldofunif}, we have $\mathfrak{v}(f_{i,j})\geq 0$ for all $i,j$.
}\end{Definition}
Note that if $N=1$ and $\Sigma=\emptyset$, this coincides, up to multiplication by a proportionality factor, with the scalar modular forms
having $u$-expansions in $K((u))$ and $A[[u]]$, or $\mathfrak{v}$-integral respectively.

We recall from Proposition \ref{seriesofEisensteinseries} the notation $\mathcal{E}^I$ that designates the 
$I$-th entry of $\mathcal{E}=\mathcal{E}(m;\rho^*_\Sigma)$ with $I\subset\Sigma$, $|\Sigma|\equiv m\pmod{q-1}$. Also, we recall that $\omega_I=\prod_{i\in I}\omega(t_i)\in\TT_\Sigma^\times$.
We have:

\begin{Theorem}\label{integrality-results}
For all $I\subset\Sigma$ we have 
$$\omega_I\widetilde{\pi}^{-m}\mathcal{E}_I\in K(\underline{t}_\Sigma)+u\Tamecirc{K(\underline{t}_\Sigma)}[[u]]$$
and $\mathcal{E}(m;\rho^*_\Sigma)$ is $\mathfrak{v}$-integrally definable for the valuations of $K(\underline{t}_\Sigma)$ associated with a non-zero prime ideal $\mathfrak{p}$ of $A$, and this for all but finitely many $\mathfrak{p}$.
\end{Theorem}

To prove this result we introduce another class of matrix-valued functions.
As seen in \S \ref{poincareseries} Poincar\'e series naturally occur as square matrix functions. On the other hand, Eisenstein series, following our constructions in \S \ref{intro-eisenstein-series}, are defined as vector functions. The following matrix function is very useful in studying Eisenstein series for the representation of the first kind $\rho_\Sigma^*$:
\begin{equation}\label{first-expansion-eisenstein}
\boldsymbol{\mathcal{E}}(m;\rho^*_\Sigma):=\sideset{}{'}\sum_{c\in A}\rho_\Sigma\Big(\begin{smallmatrix}c & 0 \\ 0 & 1\end{smallmatrix}\Big)\Psi_m(cz)+E_\Sigma\sideset{}{'}\sum_{d\in A}d^{-m}\rho^*_\Sigma(T_{-d}),
\end{equation}
where $m>0$, $E_\Sigma$ denotes, with $N=2^s$, $s=|\Sigma|$, the $N\times N$-matrix with zero coefficients, except the bottom-right coefficient which is equal to $1$, $\Psi_m(z)=\Psi_m(\rho^*_\Sigma)(z)$ (as defined in \S \ref{section-Psi}) and the sums over $c,d$ run in $A\setminus\{0\}$. We have, as it is easily seen, 
$$\boldsymbol{\mathcal{E}}(m;\rho^*_\Sigma)\in\operatorname{Hol}_{\KK_\Sigma}(\Omega\rightarrow\EE_\Sigma^{N\times N}).$$ There is a bijection between the columns of 
$\boldsymbol{\mathcal{E}}(m;\rho^*_\Sigma)$ and the subsets $I$ of $\Sigma$. We use the ordering described at the beginning of \S \ref{expansion-eisenstein} and we denote by 
$\boldsymbol{\mathcal{E}}_I$ the $I$-th column in such a way that the first column corresponds to $I=\Sigma$ and the last one to $I=\emptyset$. It is easy to show that
$$\boldsymbol{\mathcal{E}}_I=\mathcal{E}(m;\rho^*_I)\otimes\bigotimes_{j\in\Sigma\setminus I}I_2\in M_m\Big(\rho_\Sigma^*\otimes\bigotimes_{j\in\Sigma\setminus I}\boldsymbol{1}_2;\KK_\Sigma\Big),$$
where $\boldsymbol{1}_2$ is the representation (of the first kind) $\gamma\in\Gamma\mapsto I_2=\Big(\begin{smallmatrix}1 & 0 \\ 0 & 1\end{smallmatrix}\Big)$, so that the first column $\boldsymbol{\mathcal{E}}_\Sigma$ equals
$\mathcal{E}(m;\rho^*_\Sigma)$ (compare with (\ref{V-expansion}). 

\begin{proof}[Proof of Theorem \ref{integrality-results}]
By (\ref{def-Gm}) we see ($G_m(\rho)_c:=G_m(\rho)_{z\mapsto cz}$) that 
$$\boldsymbol{\mathcal{E}}(m;\rho^*_\Sigma)=\widetilde{\pi}^m\sideset{}{'}\sum_{c\in A}\rho_\Sigma\Big(\begin{smallmatrix}c & 0 \\ 0 & 1\end{smallmatrix}\Big)G_m(\rho)_c+E_\Sigma\sideset{}{'}\sum_{d\in A}d^{-m}\rho^*_\Sigma(T_{-d}).$$
We have that
$$\widetilde{\pi}^{m}\omega_{\rho^*_\Sigma}\left(\sideset{}{'}\sum_{c\in A}\rho_\Sigma\Big(\begin{smallmatrix}c & 0 \\ 0 & 1\end{smallmatrix}\Big)G_m(\rho)_c\right)\omega_{\rho^*_\Sigma}^{-1}$$ has rational expansion in $\mathfrak{K}_\Sigma^{N\times N}$ and the coefficients are $\mathfrak{v}$-integral for $\mathfrak{v}$ as in the statement of the theorem in virtue of Corollary \ref{v-integrality}. By \cite[Theorem 2]{ANG&PEL},
$$\widetilde{\pi}^{m}\omega_{\rho^*_\Sigma}\left(E_\Sigma\sideset{}{'}\sum_{d\in A}d^{-m}\rho^*_\Sigma(T_{-d})\right)\omega_{\rho^*_\Sigma}^{-1}\in K(\underline{t}_\Sigma)^{N\times N},$$
so that
$$\widetilde{\pi}^{m}\omega_{\rho^*_\Sigma}\boldsymbol{\mathcal{E}}(m;\rho_\Sigma^*)\omega_{\rho^*_\Sigma}^{-1}\in\Big(K(\underline{t}_\Sigma)+u\Tamecirc{K(\underline{t}_\Sigma)}[[u]]\Big)^{N\times N}$$ and the theorem follows.
\end{proof}

\subsection{Some applications to quasi-modular and $\mathfrak{v}$-adic modular forms}\label{Quasimodular-forms-from-Eisenstein}

In this subsection we illustrate how constructions of Drinfeld modular forms defined over $\Omega$ with values in $\CC_\infty$ having `$A$-expansions' as considered by Petrov in \cite{PET} can be naturally carried out as evaluations of our Eisenstein series $\mathcal{E}(m;\rho_\Sigma^*)$
at certain specific points. This also leads to some properties of $\mathfrak{v}$-adic modular forms with $\mathfrak{v}$ a valuation of $K(\underline{t}_\Sigma)$ that will be sketched at the end of the present subsection to illustrate further directions of research.

Consider a finite subset $\Sigma\subset\NN^*$ of cardinality $s$ and, for $i\in\Sigma$, integers $k_i\in\NN$. With $\underline{k}=(k_i)_{i\in\Sigma}\in\NN^\Sigma$, set $\ev=\ev_{\underline{\theta}^{q^{\underline{k}}}}$ the evaluation map that sends an element $f$ of $\EE_\Sigma^{M\times N}$ for integers $M,N$ to $$\ev(f)=(f)_{t_i=\theta^{q^{k_i}},\forall i\in\Sigma}\in\CC_\infty^{M\times N}.$$ the family $\underline{k}\in\NN^\Sigma$ is fixed all along the subsection. 

The Eisenstein series $\mathcal{E}(m;\rho^*_\Sigma)$ defines a non-zero rigid analytic function
$\Omega\rightarrow\EE_{\Sigma}^{N\times 1}$ with $N=2^s$. Hence the evaluation $\ev(\mathcal{E}(m;\rho^*_\Sigma))$ can be viewed as a rigid analytic function $\Omega\rightarrow\CC_\infty^{N\times 1}$. 
We recall, from \cite{PET}, the series 
\begin{equation}\label{A-series-petrov}
f_{k,m}=\sum_{a\in A^+}a^{k-m}G_m(u_a)\in K[[u]],
\end{equation}
where the sum runs over the monic polynomials in $A$.
This series converges in $K[[u]]$ (for the $u$-adic valuation) for every $m>0$ and $k\in\ZZ$.

We show the following result, where we use the notion of {\em Drinfeld quasi-modular form} introduced in \cite[Definition 2.1]{BOS&PEL0}, answering a question of D. Goss on the general nature of the 
$A$-series defined in (\ref{A-series-petrov}). We suppose that $s=|\Sigma|$ and the integer $m>0$ are chosen so that $s\equiv m\pmod{q-1}$. We also set $l=\sum_{i\in\Sigma}q^{k_i}$, so that $l\equiv s\pmod{q-1}$.

\begin{Theorem}\label{to-give-Petrov} 
The first entry of $\ev(\mathcal{E}(m;\rho^*_\Sigma))$ equals 
$-\widetilde{\pi}^m f_{l+m,m}$ and is a non-zero quasi-modular form of weight $l+m$ type $m$ and depth $\leq l$. 
\end{Theorem}

\subsubsection{Preliminaries, Hypothesis H}

We choose a representation of the first kind $\rho:\Gamma\rightarrow\GL_N(\FF_q(\underline{t}_\Sigma))$ satisfying the next:

\medskip

\noindent {\bf Hypothesis H}. {\em We suppose that $\rho$ is constructed starting from the basic representations $\rho_{t_i}$ with $i\in\Sigma$ applying the usual elementary operations $\oplus,\otimes,S^\alpha,\wedge^\beta,(\cdot)^*$.} 

\medskip

Assuming the Hypothesis H amounts to make an initial restriction on the basic representations used to define $\rho$. This condition can be relaxed but is convenient for our exposition because it allows us to refer to existing literature, so we assume it, but many properties described here extend to the general setting of arbitrary representations of the first kind, with appropriate modifications of the statements.

We note that the matrix functions $\Xi_\rho,\Phi_\rho$ introduced in \S \ref{here-defi-phi-rho} and \S \ref{section-Phi-rho} belong to $\GL_N(\Tamecirc{\EE_\Sigma})$ for $N\geq 1$ so that
$$\widetilde{\Phi}_\rho:=\ev(\Phi_\rho),\quad \widetilde{\Xi}_\rho:=\ev(\Xi_\rho)$$
define entire functions $\CC_\infty\rightarrow\CC_\infty^{N\times N}$.

\begin{Lemma}\label{lemmaAhypH}
Assuming the Hypothesis H we have $\widetilde{\Phi}_\rho,\widetilde{\Xi}_\rho\in\GL_N(\CC_\infty[z])$.
\end{Lemma}

\begin{proof}
We begin by proving the property for $\widetilde{\Phi}_\rho$. The Hypothesis H implies that every entry of $\Phi_\rho$ `comes from Perkins' series' in that they are of the type
$$\sum_{a\in A}(z-a)^{-1}\Theta(a)$$ 
where $\Theta:A\rightarrow\FF_q[\underline{t}_\Sigma]$ is a map such that there exists a polynomial 
$P\in\FF_p[X_1,\ldots,X_r]$ (for some $r$) and semi-characters $\sigma_1,\ldots,\sigma_r:A\rightarrow\FF_q[\underline{t}_\Sigma]$ (see Definition \ref{def-semicharacters}) such that 
$\Theta(a)=P(\sigma_1(a),\ldots,\sigma_r(a))$ for all $a\in A$. Hence, to prove the lemma, it suffices to show that, with $f=e_0\psi(1;\sigma_\Sigma)$ ($\psi$ is a Perkins series, see Definition \ref{defperkinsseries}). We have
$$\widetilde{f}:=\ev(f)\in\CC_\infty[z].$$ To justify this we appeal to \cite[Theorem 2]{PEL&PER}. After this result we see that
$$f(z)=\frac{\widetilde{\pi}\prod_{i\in\Sigma}\exp_C\Big(\frac{\widetilde{\pi}z}{\theta-t_i}\Big)}{\prod_{i\in\Sigma}\omega(t_i)}+e_0g(z)$$
where $g:\CC_\infty\rightarrow\EE_\Sigma$ is an entire function which vanishes identically after evaluation
$\ev$ at $\underline{t}_\Sigma=\underline{\theta}^{q^{\underline{k}}}$. Recall that
$$\chi_{t_i}(z)=\omega(t_i)^{-1}\exp_C\Big(\frac{\widetilde{\pi}z}{\theta-t_i}\Big)$$ for all $i\in\Sigma$ and for all $z\in\CC_\infty$.
It is therefore easily seen that $$\ev(\chi_{t_i}(z))=z^{q^{k_i}}.$$ Hence the claimed property of $\widetilde{f}$ follows, and together with it, that of $\widetilde{\Phi}_\rho$.

To show that $\widetilde{\Xi}_\rho\in\GL_N(\CC_\infty[z])$ it suffices to verify it for $\rho=\rho_{t_i}$ with $i\in\Sigma$ so we assume now $\Sigma=\{i\}$ and $\underline{k}=k\geq 0$. In this case however, $\widetilde{\Phi}_\rho=\widetilde{\Xi}_\rho=(\begin{smallmatrix}1 & z^{q^k} \\ 0 & 1\end{smallmatrix})$,
and thanks to the Hypothesis H, 
\begin{equation}\label{link-with-Ttheta}
\widetilde{\Xi}_\rho=\ev\Big(\rho(T_\theta)\Big)_{\theta\mapsto z}\in\GL_N(\FF_q[z]).
\end{equation}

The proof of the lemma is complete.
\end{proof}

We can now prove:

\begin{Lemma}\label{lemmaBhypH}
Under the Hypothesis H we have $\widetilde{\Phi}_\rho=\widetilde{\Xi}_\rho$.
\end{Lemma}

\begin{proof}
By Proposition \ref{propgeneralitiestame} (c) we have $\Xi_\rho=\Phi_\rho(I_N+\mathcal{N}_1)$ with 
$\mathcal{N}_1$ a function belonging to $e_0\EE_\Sigma[e_0]^{N\times N}$. evaluating we get
$$\widetilde{\Xi}_\rho=\widetilde{\Phi}_\rho(I_N+\mathcal{N}_2)$$ for $\mathcal{N}_2\in e_0\CC_\infty[e_0]^{N\times N}$. By Lemma \ref{lemmaAhypH} we see that $I_N+\mathcal{N}_2\in\GL_N(\CC_\infty[z])$
and this shows that $\mathcal{N}_2=0_N$ because the functions $z\mapsto z$ and $z\mapsto e_0(z)$
are algebraically independent over $\CC_\infty$, as is easy to see.
\end{proof}

We now choose an integer $n>0$ and we study $\ev(G_m(\rho))$ where $G_m(\rho)$ has been defined 
in (\ref{def-Gm}). We recall that $G_1(\rho)=\widetilde{\pi}^{-1}\Psi_1(\rho)=u\Phi_\rho$. It is easy to see (we leave the verification to the reader) that for all $\rho$ satisfying the Hypothesis H, $\Phi_\rho$ can be expanded into an $N\times N$ matrix of entire functions of the variables $z$ and $\underline{t}_\Sigma$
($|\Sigma|+1$ variables). It follows that for all $m\geq 1$, 
$$D_{m-1}(\widetilde{\Xi}_\rho)=D_{m-1}(\widetilde{\Phi}_\rho)=\ev(D_{m-1}(\Phi_\rho)),$$
so we have:
\begin{equation}\label{link-G-xi}
\ev\Big(G_m(\rho)\Big)=D_{m-1}\Big(u\widetilde{\Xi}_\rho\Big).
\end{equation}

\subsubsection{Matrix functions and proof of Theorem \ref{to-give-Petrov}}
From now on we suppose that $\rho=\rho^*_\Sigma$ and that $|\Sigma|\equiv m\pmod{q-1}$ with $m>0$. 
Recalling the matrix functions $\boldsymbol{\mathcal{E}}$ of \S \ref{integrality-eisenstein-coefficients},
from (\ref{link-G-xi}) we obtain the series expansion:
\begin{equation}\label{starting-identity}
\widetilde{\pi}^{-m}\ev\Big(\boldsymbol{\mathcal{E}}(m;\rho^*_\Sigma)\Big)=\sideset{}{'}\sum_{c\in A}\ev\Big(\rho_\Sigma(\begin{smallmatrix}c & 0 \\ 0 & 1\end{smallmatrix})\Big)D_{m-1}\Big(u\widetilde{\Xi}_{\rho^*_\Sigma}\Big)_c+\widetilde{\pi}^{-m}E_\Sigma\ev\Big(\sideset{}{'}\sum_{d\in A}d^{-m}\rho^*_\Sigma(T_{-d})\Big),\end{equation}
where $(\cdot)_c$ indicates that we have applied the substitution $z\mapsto cz$. 
We rewrite the identity (\ref{starting-identity}) at the level of the first columns in a more convenient way. Our next task is to show the subsequent identities (\ref{constant-term}) and (\ref{first-term-in-E}).
We note that
the first column of
$$E_\Sigma\sideset{}{'}\sum_{d\in A}d^{-m}\rho^*_\Sigma(T_{-d})$$
equals $$-\begin{pmatrix}0 \\ \vdots \\ 0 \\ \mathcal{Z}\end{pmatrix}$$
where $\mathcal{Z}=\zeta_A(m;\sigma_\Sigma)$ is the $\zeta$-value (\ref{zeta-value-tate}), and we get $$\ev(\mathcal{Z})=\zeta_A(m-l)=\sum_{d\geq 0}\sum_{a\in A^+(d)}a^{l-m}$$ (sum over the polynomials of $A$ which are monic of degree $n$), a special value of the Goss zeta function. We resume the computation as follows (the index $1$ indicates that we are extracting the first column):
\begin{equation}\label{constant-term}
\widetilde{\pi}^{-m}\left(E_\Sigma\sideset{}{'}\sum_{d\in A}d^{-m}\rho^*_\Sigma(T_{-d})\right)_1=-\begin{pmatrix}0 \\ \vdots \\ 0 \\ \widetilde{\pi}^{-m}\zeta_A(m-l)\end{pmatrix}.
\end{equation}

We now compute $D_{m-1}(u\widetilde{\Xi}_{\rho^*_\Sigma})$. For this, set, with $I\subset\Sigma$,
$l_I=\sum_{i\in I}q^{k_i}$ (note that $I\subset J$ implies $l_I\leq l_J$ and if $I\sqcup J=\Sigma$, $l_I+l_J=l$). In place of (\ref{link-with-Ttheta}) we have the explicit formula
$$\widetilde{\Xi}_{\rho^*_\Sigma}=\bigotimes_{i\in\Sigma}\Big(\begin{smallmatrix} 1 & 0\\  -z^{q^{k_i}} & 1\end{smallmatrix}\Big),$$
and $D_j(\widetilde{\Xi}_{\rho^*_\Sigma})=\sum\otimes_{k\in\Sigma} D_{i_k}(\begin{smallmatrix} 1 & 0\\  -z^{q^{k_i}} & 1\end{smallmatrix})$ where the sum runs over the families $(i_k:k\in\Sigma)\subset\NN$ such that $\sum_{k}i_k=j$. Since $D_i(z^{q^k})=1,(-\widetilde{\pi})^{q^k}z^{q^k},0$
depending on whether $i=q^k,0$ or another value distinct from the previous ones, we see that
$$D_j\Big(\widetilde{\Xi}_{\rho^*_\Sigma}\Big)=\sum_{\begin{smallmatrix} I\sqcup J=\Sigma\\ l_I=j\end{smallmatrix}}\bigotimes_{i\in J}\Big(\begin{smallmatrix} 1 & 0\\  -z^{q^{k_i}} & 1\end{smallmatrix}\Big)\otimes\bigotimes_{h\in I}\Big(\begin{smallmatrix} 0 & 0\\  \widetilde{\pi}^{-q^{k_h}} & 0\end{smallmatrix}\Big).$$

By $D_{m-1-l_I}(u)=G_{m-l_I}(u)$ we deduce the formula
\begin{equation}\label{higher-der-m-1}
D_{m-1}\Big(u\widetilde{\Xi}_{\rho^*_\Sigma}\Big)=\sum_{\begin{smallmatrix} I\sqcup J=\Sigma\\ l_I<m\end{smallmatrix}}G_{m-l_I}(u)\bigotimes_{i\in J}\Big(\begin{smallmatrix} 1 & 0\\  -z^{q^{k_i}} & 1\end{smallmatrix}\Big)\otimes\bigotimes_{h\in I}\Big(\begin{smallmatrix} 0 & 0\\  \widetilde{\pi}^{-q^{k_h}} & 0\end{smallmatrix}\Big).
\end{equation}

Note that, with $c\in A$,
\begin{multline*}
\ev\Big(\rho_\Sigma(\begin{smallmatrix}c & 0 \\ 0 & 1\end{smallmatrix})\Big)\left(\bigotimes_{i\in J}\Big(\begin{smallmatrix} 1 & 0\\  -z^{q^{k_i}} & 1\end{smallmatrix}\Big)\otimes\bigotimes_{h\in I}\Big(\begin{smallmatrix} 0 & 0\\  \widetilde{\pi}^{-q^{k_h}} & 0\end{smallmatrix}\Big)\right)=\\ =
\bigotimes_{i\in J}\left(\Big(\begin{smallmatrix} c^{q^{k_i}} & 0\\  0 & 1\end{smallmatrix}\Big)\Big(\begin{smallmatrix} 1 & 0\\  -z^{q^{k_i}} & 1\end{smallmatrix}\Big)\right)\otimes\bigotimes_{h\in I}\left(\Big(\begin{smallmatrix} c^{q^{k_h}} & 0\\  0 & 1\end{smallmatrix}\Big)\Big(\begin{smallmatrix} 0 & 0\\  \widetilde{\pi}^{-q^{k_h}} & 0\end{smallmatrix}\Big)\right)=\\
=\left(\bigotimes_{i\in J}\Big(\begin{smallmatrix} c^{q^{k_i}} & 0\\  -(cz)^{q^{k_i}} & 1\end{smallmatrix}\Big)\right)\otimes\bigotimes_{h\in I}\Big(\begin{smallmatrix} 0 & 0\\  \widetilde{\pi}^{-q^{k_h}} & 0\end{smallmatrix}\Big).\end{multline*}

Considering (\ref{higher-der-m-1}) we get (remember that the index $(\cdot)_1$ means that we are extracting the first column):
\begin{multline*}
\left(\sideset{}{'}\sum_{c\in A}\ev\Big(\rho_\Sigma(\begin{smallmatrix}c & 0 \\ 0 & 1\end{smallmatrix})\Big)D_{m-1}\Big(u\widetilde{\Xi}_{\rho^*_\Sigma}\Big)_c\right)_1=\\ =\sum_{\begin{smallmatrix} I\sqcup J=\Sigma\\ l_I<m\end{smallmatrix}}\left(\sideset{}{'}\sum_{c\in A}c^{l_J}G_{m-l_I}(u_c)\right)
\left(\bigotimes_{i\in J}\Big(\begin{smallmatrix} 1 & 0\\  -z^{q^{k_i}} & 1\end{smallmatrix}\Big)\otimes\bigotimes_{h\in I}\Big(\begin{smallmatrix} 0 & 0\\  \widetilde{\pi}^{-q^{k_h}} & 0\end{smallmatrix}\Big)\right)_1.
\end{multline*}

Note that $\sum_{c\in A\setminus\{0\}}c^{l_J}G_{m-l_I}(u_c)=-f_{m-l_I+l_J,m-l_I}$ (see (\ref{A-series-petrov})) if $l=l_\Sigma\equiv m\pmod{q-1}$, and it equals zero otherwise (because $l=l_I+l_J$). But $|\Sigma|\equiv l\pmod{q-1}$. Hence, writing $F_I$ for $f_{m-l_I+l_J,m-l_I}$ for simplicity:
\begin{equation}\label{first-term-in-E}
\left(\sideset{}{'}\sum_{c\in A}\ev\Big(\rho_\Sigma(\begin{smallmatrix}c & 0 \\ 0 & 1\end{smallmatrix})\Big)D_{m-1}\Big(u\widetilde{\Xi}_{\rho^*_\Sigma}\Big)_c\right)_1=-\sum_{\begin{smallmatrix} I\sqcup J=\Sigma\\ l_I<m\end{smallmatrix}}F_I
\left(\bigotimes_{i\in J}\Big(\begin{smallmatrix} 1\\  -z^{q^{k_i}}\end{smallmatrix}\Big)\otimes\bigotimes_{h\in I}\Big(\begin{smallmatrix} 0 \\  \widetilde{\pi}^{-q^{k_h}} \end{smallmatrix}\Big)\right).
\end{equation}

\begin{proof}[Proof of Theorem \ref{to-give-Petrov}] We study the first column of 
$\widetilde{\pi}^{-m}\ev(\boldsymbol{\mathcal{E}}(m;\rho^*_\Sigma))$. Gathering together (\ref{constant-term}) and (\ref{first-term-in-E}) we find
\begin{equation}\label{final-identity-for-modularity}
\widetilde{\mathcal{E}}:=\ev\Big(\mathcal{E}(m;\rho^*_\Sigma)\Big)=-\widetilde{\pi}^{m}\sum_{\begin{smallmatrix} I\sqcup J=\Sigma\\ l_I<m\end{smallmatrix}}F_I\left(\bigotimes_{i\in J}\Big(\begin{smallmatrix} 1\\  -z^{q^{k_i}}\end{smallmatrix}\Big)\otimes\bigotimes_{h\in I}\Big(\begin{smallmatrix} 0 \\  \widetilde{\pi}^{-q^{k_h}} \end{smallmatrix}\Big)\right)-
\begin{pmatrix}0 \\ \vdots \\ 0 \\ \zeta_A(m-l)\end{pmatrix}.
\end{equation}

Observe, from the modularity of $\mathcal{E}:=\mathcal{E}(m;\rho^*_\Sigma)$, the following identity, where $(v)_1$ now denotes the first entry of an element $v\in R^{N\times 1}$ for some ring $R$, and where $\gamma=(\begin{smallmatrix}* & * \\ c & d\end{smallmatrix})\in\Gamma$:
\begin{equation*}
\Big(\mathcal{E}(\gamma(z))\Big)_1=J_\gamma(z)^m\Big(\rho_\Sigma^*(\gamma)\mathcal{E}(z)\Big)_1=\det(\gamma)^{-m}J_\gamma(z)^m\bigotimes_{i\in\Sigma}\Big(\chi_{t_i}(d),-\chi_{t_i}(c)\Big)\mathcal{E}(z).\end{equation*}
This is obtained by noticing that $\det(\gamma)^{-|\Sigma|}\otimes_{i\in\Sigma}(\chi_{t_i}(d),-\chi_{t_i}(c))$
is the first row of $\rho_\Sigma^*(\gamma)$, and $|\Sigma|\equiv m\pmod{q-1}$. Evaluating at $t_i=\theta^{q^{k_i}}$ for all $i\in\Sigma$ this becomes $\det(\gamma)^{-m}\otimes_{i\in\Sigma}(d^{q^{k_i}},-c^{q^{k_i}})$.
Observe that, for any $x\in\CC_\infty^\times$,
$$\bigotimes_{i\in\Sigma}(d^{q^{k_i}},-c^{q^{k_i}})\left(\bigotimes_{i\in J}\Big(\begin{smallmatrix} 1\\  -z^{q^{k_i}}\end{smallmatrix}\Big)\otimes\bigotimes_{h\in I}\Big(\begin{smallmatrix} 0 \\ x^{-q^{k_h}} \end{smallmatrix}\Big)\right)=J_\gamma(z)^{l_J}(-cx)^{l_I}.$$ Moreover,
$$\bigotimes_{i\in\Sigma}(d^{q^{k_i}},-c^{q^{k_i}})\begin{pmatrix}0 \\ \vdots \\ 0 \\ \zeta_A(m-l)\end{pmatrix}=(-c)^{l}\zeta_A(m-l).$$
Using (\ref{final-identity-for-modularity}) yields the identity for $f:=(\widetilde{\mathcal{E}})_1$ (first entry):
 \begin{multline}\label{second-step-cal-E}
f(\gamma(z))=\det(\gamma)^{-m}J_\gamma(z)^m\bigotimes_{i\in\Sigma}\Big(d^{q^{k_i}},-c^{q^{k_i}}\Big)\cdot\\ \cdot \left(-\widetilde{\pi}^{m}\sum_{\begin{smallmatrix} I\sqcup J=\Sigma\\ l_I<m\end{smallmatrix}}F_I\left(\bigotimes_{i\in J}\Big(\begin{smallmatrix} 1\\  -z^{q^{k_i}}\end{smallmatrix}\Big)\otimes\bigotimes_{h\in I}\Big(\begin{smallmatrix} 0 \\  \widetilde{\pi}^{-q^{k_h}} \end{smallmatrix}\Big)\right)-
\begin{pmatrix}0 \\ \vdots \\ 0 \\ \zeta_A(m-l)\end{pmatrix}\right)=\\
=\det(\gamma)^{-m}J_\gamma(z)^{m+l}\left(-\sum_{\begin{smallmatrix} I\sqcup J=\Sigma\\ l_I<m\end{smallmatrix}}\widetilde{\pi}^{m-l_I}F_IL_\gamma(z)^{l_I}-\zeta_A(m-l)L_\gamma(z)^l\right),\end{multline} 
where $L_\gamma(z)=-\frac{c}{cz+d}$. This implies that $f$ is a {\em Drinfeld quasi-modular form of weight
$l+m$ type $m$ and depth $\leq l$} in the sense of \cite[Definition 2.1]{BOS&PEL0}. Basic properties of 
quasi-modular forms imply that $f=-\widetilde{\pi}^mF_\emptyset$ and the proof of our theorem is complete.\end{proof}

These results overlap, at least partially, with Petrov's work \cite{PET}. In his Theorem 1.3 Petrov shows that if in addition to the necessary conditions $l>m$ and $l\equiv m\pmod{q-1}$ we also impose $m\leq p^{-v_p(l)}$ where $v_p$ is the $p$-adic valuation of $\ZZ$, then 
$f_{l+m,m}$ is the $u$-expansion of a Drinfeld cusp form in $S_{l}(\det^{-m};\CC_\infty)$, a Drinfeld cusp form of weight $l+m$ and type $m$ in the terminology of \cite{GEK} and therefore a quasi-modular form of depth zero. The reader can easily deduce the following result which is however slightly weaker than Petrov's
(note that $p^{v_p(l)}=q^{v_q(l)}p^{v_p(\ell_q(l))}$, with $v_q$ denoting the order of divisibility by $q$ and $\ell_q$ denoting the sum of the digits in the $q$-ary expansion).

\begin{Corollary}\label{corollary-a-la-petrov}
If $l>m$ with $l\equiv m\pmod{q-1}$ and $m\leq q^{v_q(l)}$ then $f_{l+m,m}$ is the $u$-expansion of a
modular form in $S_{l}(\det^{-m};\CC_\infty)$. 
\end{Corollary}

\begin{proof} Indeed with this hypothesis on the order of divisibility by $q$ in the sums in (\ref{second-step-cal-E}) there is no $I$ such that $l_I<m$, unless $I=\emptyset$. Moreover, $\zeta_A(m-l)=0$ (trivial zero) and the depth of $f$ is zero. \end{proof}

\subsubsection{An example of Hecke eigenform} Consider $\Sigma$ such that $s=|\Sigma|\equiv1\pmod{q-1}$ and 
set $m=1$. Both Corollary \ref{corollary-a-la-petrov} and Petrov's \cite[Theorem 3.1]{PET} imply that 
$f:=f_{l+1,1}$ is the $u$-expansion of an element of $S_{l+1}(\det^{-1};\CC_\infty)\setminus\{0\}$. It is proportional to an entry of $\ev(\mathcal{E}(1;\rho^*_\Sigma))$. It is easy to see that this cusp form is not doubly cuspidal.
It is also well known that $f$ is the $u$-expansion of an Hecke eigenform. We can deduce this property from 
the fact that $\mathcal{E}:=\mathcal{E}(1;\rho^*_\Sigma)$ is a Hecke eigenform. We come back to (\ref{Hecke-first-entry}). We have, for all $P\in A^+$ irreducible, that
$$\Big(T_P(\mathcal{E})\Big)_1=\sigma_\Sigma(P)\Big(\mathcal{E}(Pz)\Big)_1+P^{-1}\sum_{|b|<|P|}\Big(\mathcal{E}\Big(\frac{z+b}{P}\Big)\Big)_1$$ and this equals $(\mathcal{E})_1$ by Corollary \ref{coro-7-6}.
Evaluating at $t_i=\theta^{q^{k_i}}$ for all $i\in\Sigma$ implies the identity
$$f(Pz)+P^{-1-l}\sum_{|b|<|P|}f\Big(\frac{z+b}{P}\Big)=P^{-l}f(z)$$
which tells us that $f$ is a Hecke eigenform for all the Hecke operators $\mathcal{T}_P$, with eigenvalue $P\in A^+$ irreducible (the operators $\mathcal{T}_P$ are those of \cite{GEK}, we use the normalisation of \cite{GEK} to allow an easier comparison with existing results).

\subsubsection{Examples of quasi-modular forms}\label{few-examples-quasimodular}

The content of this subsection is also related to the sequence of {\em extremal quasi-modular forms} $(x_k)_{k\geq 0}$ introduced in \cite{BOS&PEL}, where the initial explicit elements are $x_0=-E$, $x_1=-Eg-h$, in the notations of \cite{GEK}, and where $E$ is the normalized false Eisenstein series of weight $2$ already used in \S \ref{serre-derivatives}, which is a quasi-modular form of weight $2$, type $1$ and depth $1$ in the sense of \cite{BOS&PEL0}. From Theorem \ref{to-give-Petrov} we deduce that 
$\mathcal{E}(q^n;\rho^*_t)_{t=\theta}=-\widetilde{\pi}^{q^n}f_{q^n+1,q^n}$ for all $n\geq0$ and $x_n=-f_{q^n+1,q^n}$. If $n=0$, we deduce Gekeler's series expansion \cite[p. 686]{GEK}: 
 \begin{equation}\label{seriesfalse}E=\sum_{a\in A^+}au_a.\end{equation} 
Taking $\mathcal{E}(1;\rho^*_t)_{t=\theta^{q^n}}$ for $n\geq 1$ we get, up to a proportionality factor, Petrov's sequence of Hecke eigenforms $$F_{n}=\sum_{a\in A^+}a^{q^n}u_a$$ of weight $q^n+1$ and type $1$, notably the initial values $F_1=h$ and $F_2=hg^q$ (see \cite[\S 3.2]{PET} and the proof of Theorem 3.6 ibid.).

\subsubsection{$\mathfrak{v}$-adic modular forms from Eisenstein series}\label{v-adic-modular}

In this short subsection we quickly introduce further desirable directions of investigation, with few details to preserve the flow of the main topics of the present work.  Consider an element $f\in K(\underline{t}_\Sigma)+u\Tamecirc{K(\underline{t}_\Sigma)}[[u]]$. We say that $f$ is {\em an entry of a rational Drinfeld modular form} if there exist $w\in\ZZ$, $\rho:\Gamma\rightarrow\GL_N(K(\underline{t}_\Sigma))$ a representation of the first kind, $F\in M_w(\rho:\KK_\Sigma)$ and a linear map $\lambda:\KK_\Sigma^N\rightarrow\KK_\Sigma$ such that $f=\lambda(F)$. We denote by $\mathcal{X}$ the set of all entries of rational Drinfeld modular forms.  

For $f\in K(\underline{t}_\Sigma)+u\Tamecirc{K(\underline{t}_\Sigma)}s[[u]]$ we write $f=f_0+\sum_{i>0}f_iu^i$
with $f_0\in K(\underline{t}_\Sigma)$ and $f_i\in\Tamecirc{K(\underline{t}_\Sigma)}$ for $i>0$. This expansion exists and is unique (see Proposition \ref{descriptionfieldofunif}). Let $\mathfrak{v}:K(\underline{t}_\Sigma)\rightarrow\ZZ\cup\{\infty\}$ be an additive valuation. We say that $f$ is {\em $\mathfrak{v}$-integral} if 
$f_i\in\Tamecirc{\mathcal{O}_\mathfrak{v}}$, where $\mathcal{O}_\mathfrak{v}$ is the subring of 
$K(\underline{t}_\Sigma)$ of elements with non-negative $\mathfrak{v}$-valuation, i. e. 
$f\in\mathcal{O}_\mathfrak{v}+u\Tamecirc{\mathcal{O}_\mathfrak{v}}[[u]]$. Over the ring $\mathcal{O}_\mathfrak{v}+u\Tamecirc{\mathcal{O}_\mathfrak{v}}[[u]]$ of $\mathfrak{v}$-integral series we have the infimum $\mathfrak{v}$-valuation (relative to the series expansion $f=\sum_if_iu^i$) and we denote by 
$\mathcal{X}_{\mathfrak{v}}$ the matric space of all entries of rational Drinfeld modular forms which are 
$\mathfrak{v}$-integral.

\begin{Definition}\label{v-adic-modular}
A $\mathfrak{v}$-adic Drinfeld modular form is an element of the completed space $\widehat{\mathcal{X}}_{\mathfrak{v}}$.
\end{Definition}

Following the ideas of Goss in \cite{GOS4} the reader can check the following explicit example. Consider 
$\Sigma=\Sigma'\sqcup\{1\}$ with $s'=|\Sigma'|$ and set $\mathfrak{v}$ to be the $\chi_{t_1}(\mathfrak{p})$-adic valuation of $K(\underline{t}_\Sigma)$ with $\mathfrak{p}=(P)$ a prime ideal of $A$ of degree $d$ (and $P$ monic). We choose $m>0$.
We consider a sequence of positive integers $(k_i)_{i\geq 0}$ with $k_i=r+\alpha_i(q^d-1)$, with $r\in\{0,\ldots,q^d-2\}$ with $k_i\rightarrow\infty$ as $i\rightarrow\infty$ and with $\alpha_i$ converging $p$-adically. We also suppose that for all $i$, $k_i+s'\equiv m\pmod{q-1}$. Then, as $i\rightarrow\infty$,
the sequence of series
$$\sum_{\begin{smallmatrix}a\in A^+\\ P\nmid a\end{smallmatrix}}\chi_{t_1}(a)^{k_i}\sigma_{\Sigma'}(a)G_m(u_a)\in K(\underline{t}_\Sigma)[[u]],$$
all $\mathfrak{v}$-integral, defines a $\mathfrak{v}$-adic Drinfeld modular form which is non-zero. Of course, it is related to an Eisenstein series $\mathcal{E}(m;\rho^*_{\Sigma''})$, for a suitable $\Sigma''$, after an appropriate evaluation.

\subsubsection*{A remark} It is an interesting problem to determine an appropriate complete topological group of weights for $\mathfrak{v}$-adic modular forms in the sense of our Definition \ref{v-adic-modular}. We note indeed that the union $$\bigcup_{w,\Sigma,\rho}M_{w}(\rho;\KK_\Sigma),\quad w>0,\quad \Sigma\subset\NN^*,\quad \rho\text{ of the first kind},$$ $\Sigma$ being finite,
generates an algebra over $\cup_\Sigma\KK_\Sigma$ with multiplication $\otimes$. It is not difficult to show that this algebra is graded over 
the monoid $(\ZZ,+)\oplus(\{\rho:\text{ of the first kind}\},\otimes)$. To define his $\infty$-adic and $\mathfrak{v}$-adic zeta and $L$-functions, Goss introduced several complete topological spaces containing a copy of $\ZZ$, see \cite[Chapter 8]{GOS}. For instance, the complete topological group $\mathbb{S}$ projective limit of the groups
$\ZZ/((q^d-1)p^n)\ZZ$ as $n\rightarrow\infty$ with $d=\deg_\theta(P)$ and $q=p^e$, isomorphic to $\ZZ/(q^d-1)\ZZ\times \ZZ_p$, contains the weights of the $\mathfrak{p}$-adic modular forms of \cite{GOS5},
with $\mathfrak{p}$ the ideal of $A$ generated by $P$ irreducible. The same question arises when one wants to define a topological space over which interpolate the $L$-series of \cite{PEL0}, see \cite{GOS5}. At the time being, there is no complete topological group containing $(\{\rho:\text{ of the first kind}\},\otimes)$ behaving as nicely as $\mathbb{S}$, allowing to give rise to a nice space of weights for our $\mathfrak{v}$-adic modular forms. A similar question has been addressed in connection with multiple zeta values in Tate algebras, see \cite[Remark 3.1.2]{GEZ&PEL}.

\section{Modular forms for the representations $\rho_\Sigma^*$}\label{modular-certain-repr}

In this section we consider modular forms associated to representations of the first kind, with values in vector spaces over $\KK_\Sigma$ rather than vector spaces over $\LL_\Sigma$. To classify them we cannot use the techniques of specialization at roots of unity of \S \ref{evaluation-banach}. We are therefore led to introduce other techniques which, however, are harder to apply in the general setting of all the representations of the first kind. At least, they lead to proofs of Theorems E, F in the introduction.
We will focus on the representations $\rho=\rho^*_\Sigma$ only, as they seem to have a larger spectrum of applications. We are going to determine the complete structure of the spaces $M_{w}^!(\rho;\KK_\Sigma)$ in Theorem \ref{weakstructureresult}.
An important tool introduced in this section (see \S \ref{stronglyregularsection} is the notion of {\em strongly regular modular form}. 
the $v$-valuations of the entries of a strongly regular modular form are submitted to certain strong lower bounds making them into a module over the scalar modular forms, the structure of which can be easily computed, see Theorem \ref{theorem2}. If $|\Sigma|\leq q-1$, then the notions of modular form and strongly regular modular form agree (Corollary \ref{regular-equals-modular}). If $|\Sigma| \geq q$, this is no longer true but in Theorem \ref{regularitystrongregularity} we show that twisting an element of $M_w(\rho^*_\Sigma;\KK_\Sigma)$
by a large enough power of the operator $\tau$ defined in the corresponding section (the exponent depending on $w$ and $\Sigma$) yields a strongly regular modular form. Besides these properties, the precise structure of the $\KK_\Sigma$-vector spaces $M_w(\rho^*_\Sigma;\KK_\Sigma)$
for a general choice of $\Sigma$ subset of $\NN^*$ remains presently unknown.

\subsection{Structure of weak modular forms}

We consider a finite non-empty subset $\Sigma\subset\NN^*$. The structure of the $\KK_\Sigma$-vector space $M_{w}^!(\rho;\KK_\Sigma)$ is quite simple to describe. 
The main result of this subsection is the following.
\begin{Theorem}\label{weakstructureresult}
Assuming that $\rho=\rho^*_\Sigma\det\!^{-m}$, we have:
$$M_{w}^!(\rho;\KK_\Sigma)=M_{w-1}^!(\rho;\KK_\Sigma)\otimes\mathcal{E}(1;\rho^*_{t_k})+M_{w-q}^!(\rho;\KK_\Sigma)\otimes\mathcal{E}(q;\rho^*_{t_k}).$$
\end{Theorem}

We choose $k\in\Sigma$.
We set $\Sigma'=\Sigma\setminus\{k\}$. We denote by $\rho^*_{\Sigma'}$ the Kronecker factor of the representation $\rho^*_\Sigma$. Hence:
\begin{equation}\label{writingordering}
\rho^*_{\Sigma}=\rho^*_{\Sigma'}\otimes\rho^*_{t_k}.\end{equation} We can suppose, without loss of generality, that $k=\min(\Sigma)$.
The natural ordering of $\Sigma\subset\NN^*$ has to be considered to write the Kronecker product.
We set $\rho=\rho^*_\Sigma\det\!^{-m}$.

\begin{proof}[Proof of Theorem \ref{weakstructureresult}]
We consider the Eisenstein series $\mathcal{E}(1;\rho^*_t)$ of weight $1$ associated with the representation 
$\rho^*_t$.
Explicitly, this is the series: 
$$\mathcal{E}(1;\rho^*_t)=\sideset{}{'}\sum_{a,b\in A}(az+b)^{-1}\binom{\chi_t(a)}{\chi_t(b)},\quad n>0.$$
We denote by $\mathcal{E}$ the transposition (row function) of $\mathcal{E}(1;\rho^*_t)$.
We also set $$\mathfrak{E}=\binom{\mathcal{E}}{\tau(\mathcal{E})}\in\operatorname{Hol}(\Omega\rightarrow\KK_\Sigma^{2\times 2}).$$
Note that $\tau(\mathcal{E})={}^t\mathcal{E}(q;\rho^*_t)$. Let $h=-u+o(u)$ be as in \S \ref{Gekeler-Poincare}. By \cite[Theorem 3.9]{PEL&PER3}:
$$\det(\mathfrak{E})=-\widetilde{\pi}\zeta_A(q;\chi_t)h(z),$$
which is also equal to $$-\frac{\widetilde{\pi}^{q+1}h(z)}{(\theta^q-t)(\theta-t)\omega(t)}$$ by the formula 
\begin{equation}\label{first-formula-pel}
\zeta_A(1;\chi_t)=\frac{\widetilde{\pi}}{(\theta-t)\omega(t)}
\end{equation} which holds in $\TT$ and  can be found in \cite{PEL1}, after application of $\tau$. 
The function $h$ does not vanish on $\Omega$ and $v(h)=1$.
Since the function $\det(\mathfrak{E})$ can vanish identically for certain values of $t$ with $|t|>1$, 
the matrix function $\mathfrak{E}(z)^{-1}$ belongs to 
$\operatorname{Hol}(\Omega\rightarrow \TT^{2\times 2})$ but not to $\operatorname{Hol}(\Omega\rightarrow \EE^{2\times 2})$. Note that $\tau^2(\omega)^{-1}\mathfrak{E}(z)^{-1}$ defines 
a function of $\operatorname{Hol}(\Omega\rightarrow \EE^{2\times 2})$.
We are going to generalize some aspects of the proof of \cite[Theorem 3.9]{PEL&PER3}. 
Let $\mathcal{G}$ be an element of $M_{w}^!(\rho;\KK_\Sigma)$. Then by definition for all $\gamma\in\Gamma$ and $z\in\Omega$, we have 
$$\mathcal{G}(\gamma(z))=J_\gamma^w\det(\gamma)^{-m}\rho^*_\Sigma(\gamma)\mathcal{G}(z).$$
We now set $\boldsymbol{E}=\tau^2(\omega(t))^{-1}\mathfrak{E}^*$, $\boldsymbol{E}_{t_k}$ the same function in the variable $t_k$ instead of $t$, and
\begin{equation}\label{defF}
\boldsymbol{F}:=I_{N'}\otimes\boldsymbol{E}_{t_k}\in\operatorname{Hol}(\Omega\rightarrow\EE_{\{k\}}^{N\times N}),\end{equation} with $N=2^s$, $s'=s-1$, and $N'=2^{s'}$.
We have:
$$\boldsymbol{F}(\gamma(z))=(\boldsymbol{1}_{N'}\otimes\rho_{t_k}(\gamma))\boldsymbol{F}(z)\left(\boldsymbol{1}_{N'}\otimes\begin{pmatrix} J_\gamma(z)^{-1} & 0 \\ 0 & J_\gamma(z)^{-q}\end{pmatrix}\right).$$
Now setting $\boldsymbol{G}={}^t\mathcal{G}$ and, denoting with 
$\boldsymbol{H}$ the row function $\boldsymbol{GF}$, with values in $\KK_\Sigma^{1\times N}$, we have:
\begin{eqnarray*}
\lefteqn{\boldsymbol{H}(\gamma(z))=}\\ &=&J_\gamma(z)^w\det(\gamma)^{-m}\boldsymbol{G}(z)\rho^{-1}_{\Sigma}(\gamma)(\boldsymbol{1}_{N'}\otimes\rho_{t_k}(\gamma))(\boldsymbol{1}_{N'}\otimes\boldsymbol{E}_{t_k}(z)^{-1})\times \\ & & \times\left(\boldsymbol{1}_{N'}\otimes\begin{pmatrix} J_\gamma(z)^{-1} & 0 \\ 0 & J_\gamma(z)^{-q}\end{pmatrix}\right)\\
&=&\det(\gamma)^{-m}\boldsymbol{G}(z)(\rho^{-1}_{\Sigma'}(\gamma)\otimes\boldsymbol{1}_2)(\boldsymbol{1}_{N'}\otimes\boldsymbol{E}_{t_k}(z)^{-1})\left(\boldsymbol{1}_{N'}\otimes\begin{pmatrix} J_\gamma(z)^{w-1} & 0 \\ 0 & J_\gamma(z)^{w-q}\end{pmatrix}\right)\\
&=&\det(\gamma)^{-m}\boldsymbol{G}(z)(\boldsymbol{1}_{N'}\otimes\boldsymbol{E}_{t_k}(z)^{-1})(\rho^{-1}_{\Sigma'}(\gamma)\otimes\boldsymbol{1}_2)\left(\boldsymbol{1}_{N'}\otimes\begin{pmatrix} J_\gamma(z)^{w-1} & 0 \\ 0 & J_\gamma(z)^{w-q}\end{pmatrix}\right)\\
&=&\det(\gamma)^{-m}\boldsymbol{H}(z)(\rho^{-1}_{\Sigma'}(\gamma)\otimes\boldsymbol{1}_2)\left(\boldsymbol{1}_{N'}\otimes\begin{pmatrix} J_\gamma(z)^{w-1} & 0 \\ 0 & J_\gamma(z)^{w-q}\end{pmatrix}\right).
\end{eqnarray*}
In the above computation, we have observed the distributive property of the mixed product $(A\otimes B)(C\otimes D)=(AC)\otimes(BD)$ (for matrices $A,B,C,D$). This identity that we have found,
$$\boldsymbol{H}(\gamma(z))=\det(\gamma)^{-m}\boldsymbol{H}(z)(\rho^{-1}_{\Sigma'}(\gamma)\otimes\boldsymbol{1}_2)\left(\boldsymbol{1}_{N'}\otimes\begin{pmatrix} J_\gamma(z)^{w-1} & 0 \\ 0 & J_\gamma(z)^{w-q}\end{pmatrix}\right)$$
means the following. The column holomorphic function $\mathcal{H}:={}^t\boldsymbol{H}$, with values in $\KK_\Sigma^{N\times 1}$ can be written as
$\mathcal{H}=\mathcal{H}_1\odot\mathcal{H}_2$ with both $\mathcal{H}_1$ and $\mathcal{H}_2$ columns of size
$N'=2^{\Sigma'}$, where the symbol $\odot$ is defined, if $a={}^t(a_1,\ldots,a_{N'})$
and $b={}^t(b_1,\ldots,b_{N'})$, by $a\odot b=(a_1,b_1,a_2,b_2,\ldots,a_{N'},b_{N'})$. 
Then, both $\mathcal{H}_1,\mathcal{H}_2$ are separately weak modular forms for  $\rho^*_{\Sigma'}\det^{-m}$, with values in $\KK_\Sigma$ and the weights are respectively $w-1$ and $w-q$.
\end{proof}
We have:
\begin{Theorem}\label{completeexpression}
The following equality of $\KK_\Sigma$-vector spaces holds, for any $w\in\ZZ$, $m\in\ZZ/(q-1)\ZZ$ and finite $\Sigma\subset\NN^*$:
\begin{equation}\label{decompositionshrek}
M_{w}^!(\rho^*_\Sigma\det\!{}^{-m};\KK_\Sigma)=\bigoplus_{I\sqcup J=\Sigma}\left(\bigotimes_{i\in I}\mathcal{E}(1;\rho^*_{t_i})\right)\otimes\left(\bigotimes_{j\in J}\mathcal{E}(q;\rho^*_{t_i})\right)M_{w-|I|-q|J|}^!(\det\!{}^{-m};\KK_\Sigma).\end{equation}
\end{Theorem}
Denoting by $M^!(\det^\bullet;\KK_\Sigma)$ the $\ZZ\times\ZZ/(q-1)\ZZ$-graded $B$-algebra of 
scalar weak $\KK_\Sigma$-valued Drinfeld modular forms for $\Gamma$ of any weight and type, and setting $M^!(\rho^*_\Sigma\det^\bullet;\KK_\Sigma)=\oplus_{w,m}
M_{w}^!(\rho^*_\Sigma\det^{-m};\KK_\Sigma)$, which is a graded module over $M^!(\rho_\emptyset;\KK_\Sigma)$, we obtain:
\begin{Corollary}\label{coroweakstructure}
The $\KK_\Sigma$-vector space $M^!(\rho^*_\Sigma\det^\bullet;\KK_\Sigma)$ is a graded free $M^!(\det^\bullet;\KK_\Sigma)$-module of rank $N=2^{s}$.
\end{Corollary}
Observe that further, the generators of this module are explicitly described in Theorem \ref{completeexpression}.
Denoting by $M^!(\rho^*_\Sigma;\KK_\Sigma)=\oplus_{w\in\ZZ}M_{w}^!(\rho^*_\Sigma;\KK_\Sigma)$ the sub-module of $M^!(\det^\bullet;\KK_\Sigma)$ of weak modular forms for $\rho^*_\Sigma$ and setting $M^!(\rho_\emptyset;\KK_\Sigma)=\oplus_wM^!_w(\rho_\emptyset;\KK_\Sigma)$,
We also deduce the following corollary:
\begin{Corollary}\label{coroweakstructuretypezero}
The $\KK_\Sigma$-vector space $M^!(\rho^*_\Sigma;\KK_\Sigma)$ is a graded free $M^!(\rho_\emptyset;\KK_\Sigma)$-module of rank $N$.
\end{Corollary}

\begin{proof}[Proof of Theorem \ref{completeexpression}]
 We deduce from Theorem \ref{weakstructureresult}, by induction on $|\Sigma|$, that a weaker version of (\ref{decompositionshrek}) holds, with $\sum$ in place of $\bigoplus$.
It remains to show that the sum is a direct sum. For this, it suffices to show that the $N=2^s$ functions
$\otimes_{i\in I}\mathcal{E}(1;\rho^*_{t_i})\otimes\bigotimes_{j\in J}\mathcal{E}(q;\rho^*_{t_j})$, for $I\sqcup J=\Sigma$, which define elements of $\mathfrak{O}_\Sigma^{N\times 1}$ are linearly independent over the field $\KK_\Sigma((u))$. Note indeed that $M^!_{w-|I|-q|J|}(\det^{-m};\KK_\Sigma)\hookrightarrow \KK_\Sigma((u))$ because all the elements of the space on the left are $A$-periodic and tempered.

Let $a,b$ be two elements of 
$\mathfrak{K}_\Sigma$. We write $a\approx b$ if $v(a)=v(b)$  (note that if $a=0$ and $a\approx b$ then $b=0$) and we extend the definition to vectors and matrices whose entries are all in $\mathfrak{K}^\times$ by saying that $(a_{i,j})\approx(b_{i,j})$ if for all $i,j$, $v(a_{i,j})=v(b_{i,j})$. Then by Proposition \ref{seriesofEisensteinseries}, we have 
$\mathcal{E}(1;\rho^*_{t_i})\approx\binom{u}{1}$ and $\mathcal{E}(q;\rho^*_{t_i})\approx\binom{u^q}{1}$.
Hence, up to permutation of rows and columns, we have the $\approx$-equivalence of $N\times N$-matrices in $\mathfrak{O}_\Sigma^{N\times N}$:
$$\mathcal{N}:=\left(\bigotimes_{i\in I}\mathcal{E}(1;\rho^*_{t_i})\right)\otimes\left(\bigotimes_{j\in J}\mathcal{E}(q;\rho^*_{t_j})\right)_{I\sqcup J=\Sigma}\approx\begin{pmatrix}u^q & u \\ 1 & 1 \end{pmatrix}^{\otimes s}.$$ The anti-diagonal of the matrix on the right is 
equal to $(1,u)^{\otimes s}$ (up to reordering). This corresponds to a unique  monomial which minimises the $v$-valuation in the series expansion of the determinant of $\mathcal{N}$. We deduce that 
$\det(\mathcal{N})\approx u^{a_s}$, where $(a_s)_{s\geq 1}$ is the sequence defined, inductively, by $a_1=1$ and $a_s=2a_{s-1}+2^{s-1}$ for $s>1$. The matrix $\mathcal{N}$ is therefore non-singular, and the functions $\otimes_{i\in I}\mathcal{E}(1;\rho^*_{t_i})\otimes\bigotimes_{j\in J}\mathcal{E}(q;\rho^*_{t_j})$ for $I\sqcup J=\Sigma$ are linearly independent over $\KK_\Sigma((u))$, from which the result follows.
 \end{proof}
 
 \subsection{Strongly regular modular forms}\label{stronglyregularsection}

We keep considering a finite non-empty subset $\Sigma\subset\NN^*$ of cardinality $s$, the representation $\rho=\rho^*_\Sigma$, $k:=\max(\Sigma)$. We discuss quite a restricted but useful class of modular forms which have a particularly simple behaviour at infinity.
\begin{Definition}\label{defstronglyholomorphic}
{\em A tempered $\rho^*_\Sigma$-quasi-periodic holomorphic function
$$\mathcal{G}:\Omega\rightarrow \KK_\Sigma^{N\times 1}$$
is called {\em strongly regular at infinity} if
$$\begin{pmatrix} u^{-1} & 0 \\ 0 & 1\end{pmatrix}^{\otimes s}\mathcal{G}(z)\in\mathfrak{O}_\Sigma^{N\times 1}.$$}
\end{Definition}
Note, with $\operatorname{Diag}$ denoting a diagonal matrix, that
\begin{eqnarray*}
\begin{pmatrix} u^{-1} & 0 \\ 0 & 1\end{pmatrix}^{\otimes 2}&=&\operatorname{Diag}(u^{-2},u^{-1},u^{-1},1)\\
\begin{pmatrix} u^{-1} & 0 \\ 0 & 1\end{pmatrix}^{\otimes 3}&=&\operatorname{Diag}(u^{-3},u^{-2},u^{-2},u^{-1},u^{-2},u^{-1},u^{-1},1).
\end{eqnarray*}

Note also that writing
\begin{equation}\label{thesequenceni}
\begin{pmatrix} u^{-1} & 0 \\ 0 & 1\end{pmatrix}^{\otimes s}=\operatorname{Diag}(u^{-s},\ldots,u^{-n_1},u^{-n_0}),
\end{equation} and letting $s$ tend to infinity, an integer sequence $(n_i)_{i\geq 0}$ is defined and coincides with the one's-counting sequence (compare with the sequence $(a_i)_i$ in the proof of Theorem \ref{completeexpression}).
We need the next Lemma, where we use the sequence introduced in (\ref{thesequenceni}) and 
the notation $\odot$ introduced in the course of the proof of Theorem \ref{weakstructureresult}.
\begin{Lemma}\label{veryelementary}
We have $(n_{i})_{i\geq 0}=(n_{2i})_{i\geq 0}\odot(n_{2i+1}+1)_{i\geq0}$.
\end{Lemma}
\begin{proof}
Straightforward computation of the carry over in binary addition when we add one to an integer.
\end{proof}

The above serves to make the next definition.
\begin{Definition}\label{defstronglyholomorphicvect}
{\em A weak modular form $\mathcal{G}\in M_{w}^!(\rho^*_\Sigma\det^{-m};\KK_\Sigma)$
is said {\em strongly regular} (of weight $w$) if
it is strongly regular at infinity after definition \ref{defstronglyholomorphic}.}
\end{Definition}
The $\KK_\Sigma$-vector spaces of strongly regular modular forms have quite a simple structure which can be described essentially by adapting the proof of Theorem \ref{weakstructureresult}; see Theorem \ref{theorem2}.
Also, regarding the Definition \ref{defstronglyholomorphic} of strongly regular functions, if we want to use the indexation of the components of $\mathcal{G}$, $\mathcal{G}=(\mathcal{G}^J)_{I\sqcup J=\Sigma}$ (so that the first entry $\mathcal{G}^\emptyset$ has a $u$-expansion) 
we then get that the above condition is equivalent to
\begin{equation}\label{equivalentasymptotic}
\mathcal{G}^J(z)u^{-|I|}\in\mathfrak{O}_\Sigma,\quad \forall I,J\text{ such that }I\sqcup J=\Sigma.
\end{equation}
We denote by $M_{w}^\dag(\rho^*_{\Sigma}\det^{-m};\KK_\Sigma)$ 
the $\KK_\Sigma$-sub-vector space of $M_{w}^!(\rho^*_{\Sigma}\det^{-m};\KK_\Sigma)$ generated by the strongly regular  modular forms of weight $w$ for $\rho^*_{\Sigma}\det^{-m}$ (with values in $\KK_\Sigma$).

\subsubsection*{Examples of strongly regular modular forms}
Any scalar Drinfeld modular form 
is strongly regular. In fact, we have $M_{w}^\dag(\det^{-m};\KK_\Sigma)=M_{w}(\det^{-m};\KK_\Sigma)$
for all $w,m$. From Proposition \ref{seriesofEisensteinseries} we immediately see that 
$\mathcal{E}(1;\rho^*_t)\in M_{1}^\dag(\rho^*_t;\KK)$ and $\mathcal{E}(q;\rho^*_t)\in M_{q}^\dag(\rho^*_t;\KK)$. In particular, after Theorem \ref{completeexpression} and Corollary \ref{coroweakstructure}, the generators of the module $M^!(\rho^*_\Sigma\det^\bullet;\KK_\Sigma)$ described in the statements are all strongly regular modular forms.

\subsubsection{Structure of strongly regular modular forms} We shall prove:
\begin{Theorem}\label{theorem2}
The following equality of $\KK_\Sigma$-vector spaces holds, for any $w\in\ZZ$, $m\in\ZZ/(q-1)\ZZ$, finite $\Sigma\subset\NN^*$:
\begin{equation}\label{theodag}
M_{w}^\dag(\rho^*_{\Sigma}\det\!{}^{-m};\KK_\Sigma)=\bigoplus_{I\sqcup J=\Sigma}\left(\bigotimes_{i\in I}\mathcal{E}(1;\rho^*_{t_i})\right)\otimes\left(\bigotimes_{j\in J}\mathcal{E}(q;\rho^*_{t_j})\right)M_{w-i-qj}(\det\!{}^{-m};\KK_\Sigma).\end{equation}
\end{Theorem}
The direct sum $M^\dag(\rho^*_{\Sigma}\det\!{}^{\bullet};\KK_\Sigma):=\oplus_{w,m}M_{w}^\dag(\rho^*_{\Sigma}\det\!{}^{-m};\KK_\Sigma)$ is a graded module over 
the graded algebra $M(\det^\bullet;\KK_\Sigma)$ of scalar Drinfeld modular forms $\Omega\rightarrow\KK_\Sigma$ for any power of the determinant character.
We immediately deduce:
\begin{Corollary}\label{explicitgeneratorsforstrongregularity}
the $M(\det^\bullet;\KK_\Sigma)$-module $M^\dag(\rho^*_{\Sigma}\det\!{}^{\bullet};\KK_\Sigma)$ is free of rank $N$ generated by the functions $(\otimes_{i\in I}\mathcal{E}(1;\rho^*_{t_i}))\otimes(\otimes_{j\in J}\mathcal{E}(q;\rho^*_{t_j}))$, for $I,J\subset\Sigma$ such that $I\sqcup J=\Sigma$.
\end{Corollary}
After the work of Marks and Mason \cite{MAR&MAS} and Bantay and Gannon \cite{BAN&GAN} in the setting of complex vector-valued modular forms, this is expected. These authors prove that vector spaces of vector valued modular forms for $\operatorname{SL}_2(\ZZ)$ associated to an indecomposable finite dimensional complex representation of this group (and satisfying some additional mild technical conditions) all are free of  dimension that of the representation. 

Similarly, we have, writing $M(\boldsymbol{1};\KK_\Sigma)$ for the graded algebra of scalar Drinfeld modular forms for $\Gamma$ (it is equal to the graded algebra $\KK_\Sigma[g,\Delta]$ see \cite[Corollary (6.5)]{GEK}) and $M^\dag(\rho^*_\Sigma;\KK_\Sigma)$ the $M(\boldsymbol{1};\KK_\Sigma)$-module of strongly regular modular forms for $\rho^*_\Sigma$:

\begin{Corollary}\label{corostrongstructuretypezero}
The graded $M(\boldsymbol{1};\KK_\Sigma)$-module $M^\dag(\rho^*_\Sigma;\KK_\Sigma)$ is free of rank $N=2^{s}$.
\end{Corollary}
We can take  in the above result the generators of Corollary \ref{explicitgeneratorsforstrongregularity}.

\begin{proof}[Proof of Theorem \ref{theorem2}]
It is easily seen that the left-hand side of (\ref{theodag}) is contained in the right-hand side and we have to prove the reverse inclusion. Corollary \ref{completeexpression} ensures the equality of the corresponding $\KK_\Sigma$-vector spaces 
of weak modular forms (``when $\dag$ is replaced with $!$"). This means that if $\mathcal{G}\in 
M_{w}^\dag(\rho^*_{\Sigma}\det\!{}^{-m};\KK_\Sigma)$, then
$$\mathcal{G}\in M_{w}^!(\rho^*_{\Sigma}\det\!{}^{-m};\KK_\Sigma)=\bigoplus_{I\sqcup J=\Sigma}\left(\bigotimes_{i\in I}\mathcal{E}(1;\rho^*_{t_i})\right)\otimes\left(\bigotimes_{j\in J}\mathcal{E}(q;\rho^*_{t_j})\right)M_{w-i-qj}^!(\det\!{}^{-m};\KK_\Sigma).$$
All we need to prove is that the coefficients occurring in the various spaces of scalar weak modular forms $M_{w-i-qj}^!(\det\!{}^{-m};\KK_\Sigma)$
are in fact Drinfeld modular forms (regular at infinity). To see this it suffices to show that 
$$\mathcal{G}\in M_{w-1}(\rho^*_{\Sigma'}\det\!{}^{-m};\KK_\Sigma)\otimes\mathcal{E}(1;\rho^*_{t_k})+M_{w-q}(\rho^*_{\Sigma'}\det\!{}^{-m};\KK_\Sigma)\otimes\mathcal{E}(q;\rho^*_{t_k}),$$
where $k$ is an integer such that $k<\min(\Sigma')$ with $\Sigma=\Sigma'\sqcup\{k\}$. A simple induction will then allow to complete the proof.

Lemma \ref{veryelementary} implies that for all $s\geq 1$, writing
$$\begin{pmatrix}u^{-1} & 0 \\ 0 & 1\end{pmatrix}^{\otimes s}=\operatorname{Diag}(U_s),$$
then
\begin{equation}\label{Us}
U_s=u^{-1}U_{s-1}\odot U_{s-1}.\end{equation} Now, we set $\boldsymbol{G}=\boldsymbol{G}_1\odot\boldsymbol{G}_2$ with $\boldsymbol{G}={}^t\mathcal{G}$ an element of 
$M_{w}^\dag(\rho^*_{\Sigma}\det^{-m};\KK_\Sigma).$ We know by the proof of Theorem \ref{weakstructureresult} that 
$$\boldsymbol{H}=\boldsymbol{H}_1\odot\boldsymbol{H}_2=\boldsymbol{GF}$$
(with $\boldsymbol{F}$ as in (\ref{defF})) is such that 
$$\mathcal{H}_1={}^t\boldsymbol{H}_1\in M_{w-1}^!(\rho^*_{\Sigma'}\det\!{}^{-m};\KK_\Sigma),\quad \text{ and }\mathcal{H}_2={}^t\boldsymbol{H}_2\in M_{w-q}^!(\rho^*_{\Sigma'}\det\!{}^{-m};\KK_\Sigma).$$
It remains to prove that $\mathcal{H}_1$ and $\mathcal{H}_2$ are both strongly regular. We have to show that 
$$\boldsymbol{H}_j(z)\operatorname{Diag}(U_{s-1})\in\mathfrak{O}_\Sigma^{1\times N'},\quad j=1,2.$$
By hypothesis, we know that the entries of $\boldsymbol{G}(z)\operatorname{Diag}(U_s)$ 
are in $\mathfrak{O}_\Sigma$.
Explicitly, the entries of $u(z)^{-1}\boldsymbol{G}_1(z)\operatorname{Diag}(U_{s-1})$ and 
of $\boldsymbol{G}_2(z)\operatorname{Diag}(U_{s-1})$ are in $\mathfrak{O}_\Sigma$.
We recall the relation
$a\approx b$, for elements of $\mathfrak{K}_\Sigma^\times$, and its extension to matrices with non-zero entries.
We note that $\boldsymbol{H}_1,\boldsymbol{H}_2$ are given, explicitly, by the formulas:
$$\boldsymbol{H}_1=\frac{-\boldsymbol{G}_1\tau(e_2)+\boldsymbol{G}_2\tau(e_1)}{\widetilde{\pi}\zeta_A(q;\chi_{t_k})h},\quad \boldsymbol{H}_2=\frac{\boldsymbol{G}_1e_2-\boldsymbol{G}_2e_1}{\widetilde{\pi}\zeta_A(q;\chi_{t_k})h},$$
where $\mathcal{E}=(e_1,e_2)$ (\footnote{The reader will not mix these functions with the functions $e_i$ of \S \ref{tameseriestheory}.}).
By the well-known $u$-expansion $h=-u+o(u)$ (which tells us that $v(h)=1$ and $h\approx u$), we thus have
$$u\boldsymbol{H}_1\approx-\boldsymbol{G}_1\tau(e_2)+\boldsymbol{G}_2\tau(e_1),\quad u\boldsymbol{H}_2\approx\boldsymbol{G}_1e_2-\boldsymbol{G}_2e_1.$$ We first study $\mathcal{H}_1$.
We have:
\begin{eqnarray*}
\boldsymbol{H}_1\operatorname{Diag}(U_{s-1})&\approx&u^{-1}(-\boldsymbol{G}_1\tau(e_2)+\boldsymbol{G}_2\tau(e_1))\operatorname{Diag}(U_{s-1})\\
&\approx&-u^{-1}\boldsymbol{G}_1\operatorname{Diag}(U_{s-1})\tau(e_2)+\boldsymbol{G}_2\operatorname{Diag}(U_{s-1})u^{-1}\tau(e_1).
\end{eqnarray*}
Now, by hypothesis $u^{-1}\boldsymbol{G}_1\operatorname{Diag}(U_{s-1})\in
\mathfrak{O}_\Sigma^{1\times N'}$, while $v(\tau(e_2))=0$, from which we deduce that $u^{-1}\boldsymbol{G}_1\operatorname{Diag}(U_{s-1})\tau(e_2)\in \mathfrak{O}_\Sigma^{1\times N'}$. On the other hand, we have that $\tau(e_1)\approx u^q$. hence, we have that
$\boldsymbol{G}_2\operatorname{Diag}(U_{s-1})u^{-1}\tau(e_1)\approx\boldsymbol{G}_2\operatorname{Diag}(U_{s-1})u^{q-1}\in 
\mathfrak{M}_\Sigma^{1\times N'}\subset 
\mathfrak{O}_\Sigma^{1\times N'}$. Therefore all entries of $\boldsymbol{H}_1\operatorname{Diag}(U_{s-1})$ are in $\mathfrak{O}_\Sigma$ and $\mathcal{H}_1$ is strongly regular.

Let us now deal with $\mathcal{H}_2$. Similarly, we have that
\begin{eqnarray*}
\boldsymbol{H}_2\operatorname{Diag}(U_{s-1})&\approx&u^{-1}(\boldsymbol{G}_1e_2-\boldsymbol{G}_2e_1)\operatorname{Diag}(U_{s-1})\\
&\approx&u^{-1}\boldsymbol{G}_1\operatorname{Diag}(U_{s-1})e_2-\boldsymbol{G}_2\operatorname{Diag}(U_{s-1})u^{-1}e_1.
\end{eqnarray*}
Since $v(e_2)=0$, we have that the term $u^{-1}\boldsymbol{G}_1\operatorname{Diag}(U_{s-1})e_2$ has all the entries in $\mathfrak{O}_\Sigma$. Moreover, $e_1\approx u$ so that all the entries of
$\boldsymbol{G}_2\operatorname{Diag}(U_{s-1})u^{-1}e_1$ are in $\mathfrak{O}_\Sigma$ by the hypothesis on $\mathcal{G}_2$. Hence, $\boldsymbol{H}_2\operatorname{Diag}(U_{s-1})\in\mathfrak{O}_\Sigma^{1\times N'}$ and $\mathcal{H}_2$ is strongly regular. This completes the proof of the Theorem.
\end{proof}

\subsection{More structure properties} In contrast with that of strongly regular modular forms, the structure of the vector spaces $M_w(\rho^*_\Sigma\det^{-m};\KK_\Sigma)$ is more difficult to describe. In this subsection, we give some properties of them. 
Let $r\geq 0$ be the unique integer such that $r(q-1)+1\leq s\leq (r+1)(q-1)$. We want to show:

\begin{Theorem}\label{regularitystrongregularity} Let $f\in M_w(\rho^*_\Sigma\det^{-m};\KK_\Sigma)$.
Then, $\tau^{r}(f)\in M_{wq^r}^\dag(\rho^*_\Sigma\det^{-m};\KK_\Sigma)$.
\end{Theorem}

To prove this result we need some preliminary results with some tools to handle the representations $\rho_\Sigma$ and $\rho_\Sigma^*$.
We order, for $\gamma\in\gamma$, the columns of $\rho_\Sigma(\gamma)$ from $\emptyset$ to $\Sigma$ along the total order described in \S \ref{expansion-eisenstein}, and we order the rows from $\Sigma$ to $\emptyset$ along the opposite of this order. Let $M=(M_{I,J})_{I,J\subset\Sigma}\in B^{N\times N}$ be a matrix with entries in some ring $B$, with rows and columns indexed as above (the first index always indicates rows).
Since the opposite order of the inclusion order on the subsets of $\Sigma$ is obtained by computing complementaries $I\mapsto I^c:=\Sigma\setminus I$, we have the following transposition rule:
\begin{equation}\label{transpose}
{}^tM=(M_{J^c,I^c})_{I,J\subset\Sigma}\in B^{N\times N}.
\end{equation}
Now we write with $a\in A$:
$$\rho_\Sigma(T_a)=(\rho_{I,J}(T_a))_{I,J\subset\Sigma}\in\FF_q(\underline{t}_\Sigma)^{N\times N},$$ and we do similarly for $\rho^*_\Sigma(T_a)=(\rho^*_{I,J}(T_a))_{I,J\subset\Sigma}$. For $U\subset\Sigma$, we recall the semi-character $\sigma_U=\prod_{i\in U}\chi_{t_i}$.
An elementary computation, the fact that
the inverse of $\rho_{t_i}(T_a)$ is $\rho_{t_i}(T_{-a})$, and an application of (\ref{transpose}), lead to:
\begin{Lemma}\label{computationcoefficientsvarphi}
For $I,J\subset\Sigma$, we have:
$$\rho_{I,J}(T_a)=\left\{\begin{matrix}0 \text{ if }I\cup J\subsetneq \Sigma\\ \sigma_{I\cap J}(a)\text{ if }I\cup J=\Sigma\end{matrix}\right.,\quad 
\rho_{I,J}^*(T_a)=\left\{\begin{matrix}0 \text{ if }J\cap I\neq\emptyset\\ (-1)^{|(J\cup I)^c|}\sigma_{(J\cup I)^c}(a)\text{ if }J\cap I=\emptyset\end{matrix}\right..$$
\end{Lemma}
Note that $\rho_\Sigma(T_a)$ is symmetric with respect to the anti-diagonal (we can switch $I,J$) and that the entries in the diagonal are all equal to 
$1$ because these are the entries indexed by $I,J$ with $I\sqcup J=\Sigma$. The coefficient of $\rho_\Sigma(T_a)$ in the upper-right corner
is equal to $\sigma_\Sigma(a)=\prod_{i\in\Sigma}\chi_{t_i}(a)$. We deduce the explicit expression of the coefficients of $\Phi_{\rho^*_\Sigma}=(\Phi_{I,J})_{I,J}$ (defined in \S \ref{section-Phi-rho}) in term of Perkins' series. In particular, since the function $\kappa$ in (\ref{def-kappa}) is strictly decreasing, we deduce from Theorem \ref{theopsi} the following property. If 
$I,J\subset\Sigma$ with $I\cap J=\emptyset$ and $I\cup J\neq\Sigma$ (not corresponding to a diagonal coefficient), then
\begin{equation}\label{valuation-Phi}
v(\Phi_{I,J})\geq \kappa(I)-1.
\end{equation}
We set $\rho=\rho^*_\Sigma\det^{-m}$.
The above properties can be used to prove:
\begin{Lemma}\label{valuations-rho-star}
Let $f=(f^I)_I$ be a $\rho$-quasi-periodic function with $\iota_\Sigma(f)\in\mathfrak{O}_\Sigma$. Then, if $I\subsetneq\Sigma$, $v(f^I)\geq\kappa(I)$.
\end{Lemma}
\begin{proof}
By the proof of Proposition \ref{quasiperiodictempered}, we have 
$$f=\Phi_\rho g$$
where $g=(g^I)_I\in\KK_\Sigma[[u]]^{N\times 1}$. Since the entries of $\Phi_{\rho}$
are in $\Tamecirc{\KK_\Sigma}$ (valuations in $]-1,0]\cup\{\infty\}$) we see, inductively,
that $g^I\in u\KK_\Sigma[[u]]$ if $I\subsetneq\Sigma$ (while $g^\Sigma\in\KK_\Sigma[[u]]$) and (\ref{valuation-Phi}) allows to conclude.
\end{proof}
This generalizes Corollary \ref{valuations-eisenstein}. Theorem 
\ref{regularitystrongregularity} now follows easily.
Thanks to the alternative condition for strong regularity (\ref{equivalentasymptotic}) and 
Lemma \ref{valuations-rho-star}, the property of the Theorem is verified taking into account that if $I\subsetneq\Sigma$ then $q^r\kappa(I)\geq |I|$, which is easily seen.

Note that if $s=1$, every Drinfeld modular form for $\rho^*_t$ is strongly regular, which is a restatement of Theorem 3.9 of \cite{PEL&PER3}. We have 
$$M_{w}^\dag(\rho^*_\Sigma\det{}^{-m};\KK_\Sigma)\subset M_{w}(\rho^*_\Sigma\det{}^{-m};\KK_\Sigma)\subset M_{w}^!(\rho^*_\Sigma\det{}^{-m};\KK_\Sigma),$$ and the inclusions are in general strict. However, as an immediate consequence of Theorem \ref{regularitystrongregularity}, we have:
\begin{Corollary}\label{regular-equals-modular}
If $s=|\Sigma|<q$, then $M_{w}^\dag(\rho^*_\Sigma\det^{-m};\KK_\Sigma)=M_{w}(\rho^*_\Sigma\det^{-m};\KK_\Sigma)$. For any $s$, $M_{w}(\rho^*_\Sigma\det^{-m};\KK_\Sigma)$ is of finite dimension over $\KK_\Sigma$.
\end{Corollary}
In particular, one can easily check that, in the above hypotheses,
\begin{equation}\label{eisenstein-small-sigma}
\mathcal{E}(s;\rho^*_\Sigma)=(-1)^{s}\bigotimes_{i\in\Sigma}\mathcal{E}(1;\rho^*_{t_i}).\end{equation}
In fact, the formula \ref{eisenstein-small-sigma} can be proved also for $s=q$ by using the methods of \S \ref{shufflerelations}. This implies and generalizes \cite[Theorem 4.4]{GUI&PET} (see the identity at the level of the first coefficients).

We also deduce the next result which asserts, in particular, that there are no non-zero modular forms of negative weight:
\begin{Corollary}
We have $M_{w}(\rho^*_\Sigma\det^{-m};\KK_\Sigma)=\{0\}$ for $w<0$, for $w=0$ and $m\neq0$, or
for $w=0$ and $\Sigma\neq\emptyset$. 
\end{Corollary}
\begin{proof}
Note that $M_{w}^\dag(\rho^*_\Sigma\det^{-m};\KK_\Sigma)=\{0\}$ if $w<0$. Hence we obtain the first assertion, combining with Theorem \ref{regularitystrongregularity}. The other properties are easy.
\end{proof}

\section{Harmonic product and Eisenstein series}\label{shufflerelations}

In this section we study another aspect of the Eisenstein series of \S \ref{eisensteinseries} associated to representations of the form $\rho^*_\Sigma$ with $\Sigma$ a finite subset of $\NN^*$. The first entries of these Eisenstein series are proportional, by Proposition \ref{seriesofEisensteinseries}, to combinations of series such as
\begin{equation}\label{prototype-first-entry-eisenstein}
\sum_{a\in A^+}\sigma_\Sigma(a)G_m(u_a)\in A[\underline{t}_\Sigma][[u]]\end{equation}
where $\sigma_\Sigma$ is the semi-character $a\mapsto \prod_{i\in\Sigma}\chi_{t_i}(a)$ and $G_m$ the $m$-th Goss polynomial associated to the lattice $\widetilde{\pi}A\subset\CC_\infty$. In \cite{PEL3,GEZ&PEL} an $\FF_p$-algebra structure is described, over the set of {\em multiple zeta series in the Tate algebras} $\TT_\Sigma$ (or more precisely, in $\EE_\Sigma\subset\TT_\Sigma$) generalizing Thakur's multiple zeta values (see for example \cite{AND&THA,THA2}). We will see, in this section, that this algebra structure determines a multiplication rule for the series \ref{prototype-first-entry-eisenstein} and can be viewed as a  source of explicit relations connecting Eisenstein series.

The results of the present section cover various aspects of an {\em harmonic product formula} (Theorem \ref{anewsumshuffle} and complements) generalizing \cite[Theorems 2.3, 3.1]{PEL3}. We present now the basic tools.

We recall that, as usual in this text, $\Sigma$ denotes a finite subset of $\NN^*$ of cardinality $s$ (the empty set is allowed).
Let $L/\FF_q$ be a field extension. 

\begin{MainData}\label{MaindataA}
Let us suppose we are given with:
{\em \begin{enumerate}
\item Injective $\FF_q$-linear maps $\delta_i:A\rightarrow L$, for $i\in\Sigma$.
\item For $\alpha_{i,j}\in\NN$ ($i\in\Sigma$ and $j=1,\ldots,r$), maps $\sigma_j:A\rightarrow L$ defined by $\sigma_j(a):=\prod_{i\in\Sigma}\delta_i(a)^{\alpha_{i,j}}$. We call {\em semi-characters} such maps  $A\rightarrow L$ (\footnote{Note that they generalize the semi-characters that we have discussed so far.}).
\item Injective $\FF_q$-linear map $\gamma:A\rightarrow L$ (we adopt the notation $\gamma_a$ for the evaluation of $\gamma$ in $a\in A$).
\end{enumerate}
}
\end{MainData}

We consider a semi-character $\sigma=\prod_{i\in\Sigma}\delta_i^{\alpha_i}$ with linear maps $\delta_i$ as above, $i\in\Sigma$ (empty products are allowed). 
The map $\boldsymbol{1}$ sending $A$ to $1\in L$ is the trivial semi-character.

Together with the objects that we have introduced so far, we consider, for integers $n_i\in\NN^*$ with $i=1,\ldots,r$ {\em composition arrays}
\begin{equation}\label{compositionarray}
\mathcal{C}:=\begin{pmatrix} \sigma_1 & \cdots & \sigma_r \\ n_1 & \cdots & n_r\end{pmatrix}.
\end{equation}
When $r=1$, we may sometimes write $(n;\sigma)$ instead of $\binom{\sigma}{n}$. 
If $\mathcal{C}=(\begin{smallmatrix} \boldsymbol{1} & \cdots &\boldsymbol{1} \\ n_1 & \cdots & n_r\end{smallmatrix})$ we simplify it to $\mathcal{C}=(n_1,\ldots,n_r)$. 
The {\em degree} of $\mathcal{C}$ is $\binom{\sigma}{n}$ where $\sigma=\sigma_1\cdots\sigma_r$ and 
$n=\sum_in_i$. The {\em weight} is $n$ and the {\em type} is $\sigma$. If $\sigma=\boldsymbol{1}$ we say that the type is {\em trivial}. For a composition array as in (\ref{compositionarray}), we introduce the {\em twisted power sum}
$$S_d(\mathcal{C}):=
\sum_{\begin{smallmatrix}d_1>\cdots>d_r\geq 0\\
a_1,\ldots,a_r\in A^+\\
\deg_\theta(a_i)=d_i,\forall i=1,\ldots,r\end{smallmatrix}}
\frac{\sigma_1(a_1)\cdots\sigma_r(a_r)}{\gamma_{a_1}^{n_1}\cdots\gamma_{a_r}^{n_r}}\in L.$$
These twisted power sums generalize the classical power sums of Thakur in \cite{THA}, as well as the twisted power sums of \cite{PEL&PER2}. 
We shall show the following generalization of \cite[Theorem 3.1]{PEL3}:
\begin{Theorem}\label{simplesumshuffle}
Let $\sigma,\psi$ be two semi-characters and $m,n$ two positive integers. For any $\alpha,\beta$
semi-characters and $i,j\in \NN^*$ there is an element $f_{\alpha,\beta,i,j}\in\FF_p$ such that, for all $d\geq 0$,
$$S_d\left(\begin{matrix} \sigma \\ m \end{matrix}\right)
S_d\left(\begin{matrix} \psi \\ n \end{matrix}\right)-S_d\left(\begin{matrix} \sigma\psi \\ m+n \end{matrix}\right)=\sum_{\begin{smallmatrix}\alpha\beta=\sigma\psi\\ i+j=m+n\end{smallmatrix}}f_{\alpha,\beta,i,j}S_d\left(\begin{matrix} \alpha & 
\beta\\ i & j \end{matrix}\right).$$
\end{Theorem}
In the theorem, the sum is on the couples of semi-characters $(\alpha,\beta)$ such that $\alpha\beta=\sigma\psi$, and over the decompositions $n+m=i+j$, so there are only finitely many terms in it.

\begin{MainData}\label{MaindataB}
{\em Let us assume that: 
\begin{enumerate}
\item $L$ is endowed with a valuation $\nu:L\rightarrow\QQ\cup\{\infty\}$, it is complete for this valuation
\item $\nu(\delta_{i,j}(a))\in\{0,\infty\}$ for all $i,j$ and $a\in A$  
\item (3) $\gamma_a^{-1}\rightarrow0$ as $a$ runs in $A$ (for the valuation). 
\end{enumerate}
}
\end{MainData}

Then, the series
\begin{equation}\label{definition-sums-fa}
f_A(\mathcal{C}):=\sum_{d\geq 0}S_d(\mathcal{C})
\end{equation}
converges in $L$ for any composition array $\mathcal{C}$ as in (\ref{compositionarray}).
Let $n$ be a positive integer, and let $\sigma:A\rightarrow L$ be a semi-character such that $\nu$ is trivial over its image.
We denote by $\mathcal{F}_n^\sigma$ the $\FF_p$-sub-vector space of $L$ generated by the elements $f_A(\begin{smallmatrix}\sigma_1 & \cdots & \sigma_r \\ n_1 & \cdots & n_r\end{smallmatrix})$ with 
$r>0$, $\prod_i\sigma_i=\sigma$, $\sum_in_i=n$ (with $n_i>0$ for all $i$). We also set $\mathcal{F}^{\boldsymbol{1}}_0:=\FF_p$ and $\mathcal{F}^{\sigma}_0:=(0)$ if $\sigma\neq\boldsymbol{1}$.
We consider the sum 
$\mathcal{F}:=\sum_{n,\sigma}\mathcal{F}_n^\sigma$.
The above result can be used, in a lengthy but straightforward way very similar to that of \cite{PEL3}, to prove the next result.
\begin{Theorem}\label{anewsumshuffle} For all $m,n>0$ and $\sigma,\psi$ semi-characters,
We have that $\mathcal{F}_m^\sigma\mathcal{F}_n^\psi\subset\mathcal{F}^{\sigma\psi}_{m+n}$, and the $\FF_p$-vector space $\mathcal{F}$ is an $\FF_p$-algebra. 
\end{Theorem}

\subsection{Existence of the harmonic product}
We prove Theorem \ref{simplesumshuffle}.
We will use the methods of \cite[\S 3.1.2 and \S 3.1.3]{PEL3} which deeply borrow from Thakur in \cite{THA2}. The following result can be found in \cite{PEL3}. 
\begin{Proposition} Let $\Sigma$ be a finite subset of $\NN^*$.
Consider $U,V$ such that $U\sqcup V=\Sigma$. Let $L/\FF_q$ be a field extension and let us suppose that
$x_i$ ($i\in \Sigma$) are elements of $L$ and let $z$ be an element of $L\setminus\FF_q$.
Then, the following formula holds:
$$\sum_{\mu,\nu\in\FF_q^2\setminus\Delta}\frac{\prod_{i\in U}(x_i+\mu)\prod_{j\in V}(x_j+\nu)}{(z+\mu)(z+\nu)}=-\sum_{\begin{smallmatrix}I\sqcup J=\Sigma
\\
|J|\equiv1\pmod{q-1}\\
J\subset U\text{ or }J\subset V\end{smallmatrix}}\sum_{\mu\in\FF_q}\frac{\prod_{k\in I}(x_k+\mu)}{(z+\mu)}.$$
\end{Proposition}
With appropriate choices of the set $\Sigma$, of the subsets $U,V$, of the elements $x_i$ and $z$ and applying 
a power of an endomorphism of $L$ which is $\FF_q(x_i:i\in \Sigma)$-linear and which sends $z$ to $z^q$, and specialization of some 
$x_i$ to $z$, we deduce:
\begin{Corollary}\label{thecorollary}
Considering a finite set $\Sigma\subset\NN^*$, a partition $\Sigma=U\sqcup V$, a positive integer $N$
and two integers $\alpha,\beta$ such that $N=\alpha+\beta$, for all $1\leq k\leq N$ and $I\subset \Sigma$, there exists $c_{I,k}\in\FF_p$ such that 
$$\sum_{\mu,\nu\in\FF_q^2\setminus\Delta}\frac{\prod_{i\in U}(x_i+\mu)\prod_{j\in V}(x_j+\nu)}{(z+\mu)^\alpha(z+\nu)^\beta}=\sum_{\begin{smallmatrix} k=1,\ldots,N\\ I\subset\Sigma\end{smallmatrix}}c_{I,k}\sum_{\mu\in \FF_q}\frac{\prod_{i\in I}(x_i+\mu)}{(z+\mu)^k}.$$
\end{Corollary} 
In the above formula, $\Delta$ denotes the diagonal subset.
We can now prove Theorem \ref{simplesumshuffle}. 
We recall that we have denoted by $A^+(d)$ the set of monic polynomials of degree $d$ in $A$. We also denote by $A^+(<d)$
the set of monic polynomials of $A$ which have degree $<d$. For $n\in A^+(d)$ and
$m\in A^+(<d)$, we write
$$S_{m,n}=\{(n+\mu m,n+\nu m);\mu,\nu\in\FF_q,\mu\neq\nu\}\subset \Big(A^+(d)\times A^+(d)\Big)\setminus\Delta,$$
where $\Delta$ is the diagonal 
of $A^+(d)\times A^+(d)$.
Similarly, we define for $n\in A^+(d)$ and $m\in A^+(<d)$:
$$S'_{m,n}=\{(n+\mu m,m);\mu\in\FF_q\}\subset A^+(d)\times A^+(<d).$$
From \cite[Lemmas 3.10 and 3.11]{PEL3} and following the original ideas of Thakur in \cite{THA2},
we deduce that 
the sets $S_{m,n}$ determine a partition of $A^+(d)\times A^+(d)\setminus\Delta$ and the sets $S'_{m,n}$ determine a partition of $A^+(d)\times A^+(<d)$. Moreover,
$S'_{m,n}=S'_{m',n'}$ if and only if $S_{m,n}=S_{m',n'}$.

Now, let us choose $d>0$. We write $\sigma\psi=\prod_{i\in\Sigma}\delta_i$ with $\delta_i$ an injective
$\FF_q$-linear map $A\rightarrow L$ for all $i\in\Sigma$ (there can be repetitions), and $\sigma=\prod_{i\in U}\delta_i$,
$\psi=\prod_{i\in V}\delta_i$ with $U\sqcup V=\Sigma$.
We have, with $\mathcal{U}$ a set of representatives of the above-mentioned partition:
\begin{multline*}
S_d\left(\begin{matrix} \sigma \\ \alpha \end{matrix}\right)S_d\left(\begin{matrix} \psi \\ \beta \end{matrix}\right)-S_d\left(\begin{matrix} \sigma\psi \\ N \end{matrix}\right)=
\sum_{(a,b)\in A^+(d)\times A^+(d)\setminus\Delta}\frac{\sigma(a)\psi(b)}{\gamma_a^\alpha \gamma_b^\beta}=\\ =
\sum_{(m,n)\in \mathcal{U}}\sum_{(a,b)\in S_{m,n}}\frac{\sigma(a)\psi(b)}{\gamma_a^\alpha \gamma_b^\beta}.
\end{multline*}
We focus on the sub-sum corresponding to the choice of a set $S_{m,n}$. We want now to compute:
\begin{eqnarray*}
\lefteqn{\sum_{(a,b)\in S_{m,n}}\frac{\sigma(a)\psi(b)}{\gamma_a^\alpha \gamma_b^\beta}=}\\
&=&\sum_{(\mu,\nu)\in\FF_q^2\setminus\Delta}\frac{\sigma(n+\mu m)\psi(n+\nu m)}{\gamma_{n+\mu m}^\alpha \gamma_{n+\nu m}^\beta}\\
&=&\sum_{(\mu,\nu)\in\FF_q^2\setminus\Delta}\frac{\prod_{i\in U}\delta_i(n+\mu m)\prod_{j\in V}\delta_i(n+\nu m)}{(\gamma_n+\mu \gamma_m)^\alpha (\gamma_n+\nu \gamma_m)^\beta}\\
&=&\frac{\sigma(m)\psi(m)}{\gamma_m^N}\sum_{(\mu,\nu)\in\FF_q^2\setminus\Delta}\frac{\prod_{i\in U}\left(\frac{\delta_i(n)}{\delta_i(m)}+\mu\right)\prod_{j\in V}\left(\frac{\delta_j(n)}{\delta_j(m)}+\nu\right)}{\left(\frac{\gamma_n}{\gamma_m}+\mu\right)^\alpha \left(\frac{\gamma_n}{\gamma_m}+\nu\right)^\beta}.
\end{eqnarray*}
Note that we have used the $\FF_q$-linearity of $\delta_i$ for all $i\in\Sigma$ so that
$\delta_i(n+\mu m)=\delta_i(n)+\mu\delta_i(m)$ and the hypothesis of injectivity, to divide by $\delta_i(m)$ which needs to be non-zero. Similarly, we have used the $\FF_q$-linearity of the map $a\mapsto \gamma_a$ and the fact that $\gamma_n+\lambda\gamma_m$ does not vanish, because $n,m$, in the above computation, have distinct degrees.
Applying Corollary 
\ref{thecorollary} with $x_i=\frac{\delta_i(n)}{\delta_i(m)}$ for $i\in\Sigma$ and $z=\frac{\gamma_n}{\gamma_m}$
which does not belong to $\FF_q$, we obtain the identity:
\begin{multline}\label{crucial-step}
\sum_{(a,b)\in S_{m,n}}\frac{\sigma(a)\psi(b)}{\gamma_a^\alpha \gamma_b^\beta}=\\ =
\sigma(m)\psi(m)\gamma_m^{-N}\sum_{\begin{smallmatrix}I\subset\Sigma\\ k=1,\ldots,N\end{smallmatrix}}c_{I,k}\sum_{\mu\in\FF_q}\frac{\prod_{i\in I}\left(\frac{\delta_i(n)}{\delta_i(m)}+\mu\right)}{\left(\frac{\gamma_n}{\gamma_m}+\mu\right)^k}=\\ =\sum_{\begin{smallmatrix}I\sqcup J=\Sigma\\ k=1,\ldots,N\end{smallmatrix}}c_{I,k}\sum_{\mu\in\FF_q}
\frac{\prod_{i\in I}\delta_i(n+\mu m)\prod_{j\in J}\delta_j(m)}{\gamma_{n+\mu m}^k\gamma_m^{N-k}}.
\end{multline}
The latter is a sum over $S'_{m,n}$.
In view of our previous observations, this concludes the proof of our Theorem.
The deduction of Theorem \ref{anewsumshuffle} from Theorem \ref{simplesumshuffle} is standard and we omit it. If we choose $\delta_i=\chi_{t_i}$ for $i\in\Sigma$ and $\gamma_a=e_C(az)$, and we follow closely the above proof of Theorem \ref{anewsumshuffle} in conjonction with \cite[Theorem 3.1]{PEL3}, we deduce the following explicit result that will be used later, with $\sigma_\Sigma=\prod_{i\in\Sigma}\chi_{t_i}$ and $\gamma_a=e_C(az)$ for $a\in A\setminus\{0\}$.

\begin{Theorem}\label{corsumshuffle}
The following formula holds, for all $\Sigma\subset\NN^*$ and $U\sqcup V=\Sigma$:
\begin{multline*}
f_A\left(\begin{matrix} \sigma_U \\
1\end{matrix}\right)f_A\left(\begin{matrix} \sigma_V \\
1\end{matrix}\right)-f_A\left(\begin{matrix}\sigma_\Sigma \\
2\end{matrix}\right)= \\ f_A\left(\begin{matrix}\sigma_U & \sigma_V \\
1 & 1\end{matrix}\right)+f_A\left(\begin{matrix}\sigma_V & \sigma_U \\
1 & 1\end{matrix}\right)-\sum_{\begin{smallmatrix}I\sqcup J= \Sigma \\
|J|\equiv1\pmod{q-1}\\
J\subset U\text{ or }J\subset V\end{smallmatrix}}f_A\left(\begin{matrix}\sigma_{I} & \sigma_{J} \\
1 & 1\end{matrix}\right).\end{multline*}
\end{Theorem}

In the next three short subsections we give the three main sets of Data \ref{MaindataA} that are considered in this paper (we will mainly consider the second one, described in \S \ref{exa2}).

\subsubsection{Multiple zeta values}\label{exa1} To choose the Data \ref{MaindataA} we consider variables $\underline{t}_\Sigma=\{t_i:i\in\Sigma\}$ and 
the field $L=\KK_\Sigma:=\widehat{K(\underline{t}_\Sigma)}_{v_\infty}$ obtained by completing 
$K(\underline{t}_\Sigma)$ with respect to the Gauss' valuation $\nu$ extending the valuation $v_\infty$ of $K$. We consider further the injective $\FF_q$-algebra morphisms $\delta_i(a):=\chi_{t_i}(a)$ for all $i\in\Sigma$ to build our semi-characters. As we did previously, we write, for $U$ a finite subset of $\NN^*$, 
$\sigma_U(a):=\prod_{i\in U}\chi_{t_i}(a)$. More generally, we can also consider elements in the {\em monoid of degrees} of \cite[\S 2.1]{GEZ&PEL} in place of $U$; this amounts in considering semi-characters $\sigma$ defined by 
\begin{equation}\label{semi-char-monoid}
\sigma(a)=\prod_{i\in\Sigma}\chi_{t_i}(a)^{n_i}
\end{equation}
with $\Sigma\subset\NN^*$ finite and $n_i\geq 0$. 
Finally, we choose $\gamma$ the identity map, so that for all $a\in A$, $\gamma_a=a\in L$. Then
we also have the Data \ref{MaindataB} and we are in the settings of \cite{PEL3}. In the notations of ibid., we have 
$\zeta_A(\mathcal{C})=f_A(\mathcal{C})$ for any $\mathcal{C}$ as in (\ref{compositionarray}) and we can speak about degree, weight and type of $\zeta_A(\mathcal{C})$. One proves (see \cite[Corollary 3.3]{GEZ&PEL}) that the $K[\underline{t}_\NN]$-algebra they generate is graded by the degrees. Note also that  for any such element there exists a finite subset $\Sigma$ of $\NN^*$ such that it belongs to $\EE_\Sigma\subset\KK_\Sigma$ . If we consider the particular case of 
composition arrays $\mathcal{C}$ as in (\ref{compositionarray}) such that the semi-characters $\sigma_i$ are all equal 
to the trivial semi-character $\boldsymbol{1}$ (trivial type), then it is easy to see that the series $\zeta_A(\mathcal{C})\in K_\infty$ are the {\em multiple zeta values} of Thakur (the reader can find more in the papers \cite{AND&THA,THA2} and the survey \cite{THA3} also provides a wider set of references).

\subsubsection{$A$-periodic multiple sums}\label{exa2} 

These are related to first entries of Eisenstein series for $\rho^*_\Sigma$.
We choose, for the Data \ref{MaindataA}:
$$\gamma_a:=e_C(az),\quad a\in A\setminus\{0\}.$$
This choice leads us to work with the same semi-characters as in \S \ref{exa1}, and in the field 
$L=K(\underline{t}_\Sigma)((u))$ which is complete for the valuation $\nu=v$, giving the order at $u=0$ of a formal power series of $u$. We also have the Data \ref{MaindataB}.
In this case, for $\mathcal{C}$ as in (\ref{compositionarray}), we set $\varphi_A(\mathcal{C})=f_A(\mathcal{C})$ and we can continue to speak about degree, weight and type of such a sum. Explicitly:
$$\varphi_A(\mathcal{C})=\sum_{\begin{smallmatrix}d_1>\cdots>d_r\geq 0\\
a_1,\ldots,a_r\in A^+\\
\deg_\theta(a_i)=d_i,\\ \forall i=1,\ldots,r\end{smallmatrix}}
\sigma_1(a_1)\cdots\sigma_r(a_r)u_{a_1}^{n_1}\cdots u_{a_r}^{n_r}\in L,$$ (with $u_a=e_C(az)^{-1}$). These series define formal series of $K(\underline{t}_\Sigma)[[u]]$ and each of them is also converging for $u$ in a non-empty disk of $\CC_\infty$ of radius $\leq c$ for some $c\in|\CC_\infty|\cap]0,1[$, containing $0$. From Theorem \ref{anewsumshuffle} we deduce:

\begin{Corollary}\label{coro-varphi-A}
The $\FF_p$-vector space spanned by $1$ and the series $\varphi_A(\mathcal{C})$ with $\mathcal{C}$ as in (\ref{compositionarray}) is an $\FF_p$-algebra. The multiplication rule is compatible with the filtration induced by the semigroup of the elements $(w,\sigma)$ with $w\in\ZZ$ and $\sigma$ semi-characters
as in \S \ref{exa1}.
\end{Corollary}

Again with $\mathcal{C}$ as in (\ref{compositionarray}), we consider a variant of the above sums based on Goss' polynomials:
\begin{equation}\label{variant-Goss-sums}
\phi_A(\mathcal{C})=\sum_{|a_1|>\cdots>|a_r|>0}\sigma_1(a_1)\cdots\sigma_r(a_r)G_{n_1}(u_{a_1})\cdots G_{n_r}(u_{a_r}),
\end{equation} with the sum running over elements $a_1,\ldots,a_r\in A^+$. These sums are more closely related to the first entries of our Eisenstein series. We have the next result.

\begin{Corollary}\label{coro-Phi-A}
The $K$-vector space spanned by $1$ and the series $\phi_A(\mathcal{C})$ with $\mathcal{C}$ as in (\ref{compositionarray}) is a $K$-algebra and equals the $K$-vector space spanned by the series
$\varphi_A(\mathcal{C})$.
\end{Corollary}

\begin{proof}
We claim that the family $(G_m(X))_{m>0}$ is a $K$-basis of $XK[X]$. 
First of all, these polynomials are 
linearly independent over $K$ because the functions $z\mapsto G_m(u(z))$, meromorphic over $\CC_\infty$,
have poles of distinct orders at the elements $a\in A\subset\CC_\infty$. To show that these polynomials span $XK[X]$ it suffices to prove that for $k>0$, $u^k$ belongs to the $K$-span $V$ of the polynomials $G_m(u)$ with $m>0$. This is clear for $k=1$. Now assuming that $u^{k-1}$ belongs to $V$, by the fact that $u^k=uu^{k-1}$, it suffices to show that $uG_m(u)\in V$ for all $m$, but this easily follows from \cite[Proposition (3.4) (ii)]{GEK} and induction on $m$ hence proving the claim. The result now follows from Corollary \ref{coro-varphi-A}. \end{proof}

\subsubsection*{Remark} 
The product rule of Corollary \ref{coro-Phi-A} does not seem to be compatible with a filtration involving the composition arrays in a simple way, unlike Corollary \ref{coro-varphi-A}. Note however the following formula, which is homogeneous in the orders of the Goss' polynomials:
$$\sum_{m+n=k}G_m(X)G_n(X)=\left(\binom{k}{1}-1\right)G_k(X),\quad k\geq 0.$$
To prove this formula we use (\ref{G-rho}) and Lemma \ref{generalisation-gekeler}, and 
$$\boldsymbol{G}(\boldsymbol{1})=x\Exp{D}{x}(G_1(u))=\frac{ux}{1-u\exp_C(x)}.$$
Hence we obtain the next Riccati-like differential equation from which the above identities can be derived:
$$\boldsymbol{G}(\boldsymbol{1})^2=x\frac{\partial}{\partial x}\Big(\boldsymbol{G}(\boldsymbol{1})\Big)-\boldsymbol{G}(\boldsymbol{1}).$$

\subsubsection{Multiple sums in $\mathfrak{K}_\Sigma$}\label{exa3} 

There is a third important type of multiple sums that is determined by making the following choice of Data \ref{MaindataA}, it will be only used in \S \ref{Multiple-Eisenstein-series}. We consider $L=\mathfrak{K}$
the field of uniformizers with the valuation $\nu=v$. 
As in \S \ref{exa2} we use $\gamma_a:=e_C(az)$ for $a\in A\setminus\{0\}$. Instead of the semi-characters of \S \ref{exa1},
we use, for $i\in\NN^*$, $\delta_i:A\rightarrow L$ defined by $$\delta_i(a)=\chi_{t_i}(az)=\frac{\exp_C\left(\frac{\widetilde{\pi}z}{\theta-t}\right)}{\omega(t)},$$
seen as a tame series. These maps are clearly $\FF_q$-linear and injective, and they give rise to semi-characters $$\widetilde{\sigma}_U(a):=\prod_{i\in U}\chi_{t_i}(az)$$ with $U$ a finite subset of $\NN^*$ (\footnote{Or more generally, an element of the monoid of degrees as in \S \ref{exa1}.}). With them we can construct the formal series
\begin{equation}\label{third-kind-multiple-sum}
\widetilde{\varphi}_A\Big(\begin{smallmatrix}\widetilde{\sigma}_1 & \cdots & \widetilde{\sigma}_r \\ n_1 & \cdots & n_r\end{smallmatrix}\Big)=\sum_{|a_1|>\cdots>|a_r|>0}\widetilde{\sigma}_1(a_1)\cdots\widetilde{\sigma}_r(a_r)u_{a_1}^{n_1}\cdots u_{a_r}^{n_r},
\end{equation} where the semi-characters $\widetilde{\sigma}_i$ are of the above form, where $n_1,\ldots,n_r$ are positive integers, and with the sum running over elements $a_1,\ldots,a_r\in A^+$.
This time however, we do not have a consistent set of Data \ref{MaindataB}. The condition (2) does not hold in general. We cannot guarantee the convergence of the series in (\ref{third-kind-multiple-sum}) for the $v$-valuation.
However, when these series converge for the $v$-valuation (this can happen), they give rise to well defined elements of 
$L$.

We have:

\begin{Corollary}\label{coro-varphi-A-tilde}
There is a multiplication rule on the series (\ref{third-kind-multiple-sum}) that are convergent for the $v$-valuation. Choosing the correspondence $\chi_{t_i}\leftrightarrow\delta_i$ identifies, if all the terms are well defined, the multiplication 
rule with that of Corollary \ref{coro-varphi-A} and that of \S \ref{exa1}.
\end{Corollary}

\subsubsection*{Example} We have the following formulas expressing the same harmonic product rule in the three different settings of \S \ref{exa1}, \ref{exa2} and \ref{exa3}. We use $\delta$ the semi-character defined by $\delta(a)=\chi_t(az)$ (\footnote{Note that for coherence with other references, we render in different ways the multiple series of `depth' $r=1$ or with trivial semi-character (scalar). In particular, we sometimes write, for the arguments of multiple sums,
$(n;\sigma)$ instead of $\binom{\sigma}{n}$, and $(n_1,\ldots,n_r)$ instead of $(\begin{smallmatrix}\boldsymbol{1} & \cdots & \boldsymbol{1} \\ n_1 & \cdots & n_r\end{smallmatrix})$.}):
\begin{multline}\label{three-special-formulas} \zeta_A(1;\chi_t)\zeta_A(q-1)=
\zeta_A(q;\chi_t)+\zeta_A\Big(\begin{smallmatrix}\chi_t & \boldsymbol{1} \\
1 & q-1\end{smallmatrix}\Big),\\ \varphi_A(1;\chi_t)\varphi_A(q-1)=
\varphi_A(q;\chi_t)+\varphi_A\Big(\begin{smallmatrix}\chi_t & \boldsymbol{1} \\
1 & q-1\end{smallmatrix}\Big),\\ \widetilde{\varphi}_A(1;\delta)\widetilde{\varphi}_A(q-1)=
\widetilde{\varphi}_A(q;\delta)+\widetilde{\varphi}_A\Big(\begin{smallmatrix}\delta & \boldsymbol{1} \\
1 & q-1\end{smallmatrix}\Big).\end{multline}
It is not difficult to verify that all the multiple series involved in the third formula converge for the $v$-valuation.
To prove the identities (\ref{three-special-formulas}) one observes that the first identity follows from identities on multiple power sums (see \cite[\S 7.2]{GEZ&PEL}), then uses that the product rules of \S \ref{exa1}, \ref{exa2}, \ref{exa3} are the same upon choice of the appropriate correspondence between the semi-characters.

\subsection{Examples of formulas}

The harmonic product can be applied to obtain identities for certain modular forms, notably Eisenstein series.
We give three examples. In \S \ref{example-1-eisenstein} an identity for Eisenstein series of weight $q+1$ for $\rho^*_{\{1,2\}}$, in \S \ref{example-2-eisenstein} an identity for Eisenstein series of weight $2$ for $\rho^*_\Sigma$ with $|\Sigma|\equiv2\pmod{q-1}$ and in \S \ref{example-3-eisenstein} we present a question on Serre's derivatives of Eisenstein series of weight $1$ and their possible relation with Poincar\'e series of weight $3$. 

We begin with two lemmas.

\begin{Lemma}\label{condition-for-vanishing}
Let $\rho:\Gamma\rightarrow\GL_N(\FF_q(\underline{t}_\Sigma))$ be a representation of the first kind. Assume that $\rho$ is irreducible
and let $f$ be an element in $M^!_w(\rho;\KK_\Sigma)$. If the entries of $f$ are linearly dependent over $\KK_\Sigma$, then $
f$ vanishes identically.
\end{Lemma}

\begin{proof}
This is straightforward but we prefer to give the details. Let $V$ be the $\KK_\Sigma$-subspace of $\KK_\Sigma^{1\times N}$ the elements of which are the $v$'s such that $vf=0$. Assume that $V\neq\{0\}$ and let us consider $\gamma\in\Gamma$. Then
$$0=vf(\gamma^{-1}(z))=vJ_{\gamma^{-1}}(z)^w\rho(\gamma^{-1})f(z).$$
Hence $v\rho(\gamma^{-1})\in V$ and this, for all $\gamma\in \Gamma$. This means that $\rho^*$ has the space $W={}^tV$
which is invariant that is, for all $\gamma\in\Gamma$, $\rho^*(\gamma)W\subset W$ with $W\neq\{0\}$. But $\rho$ is irreducible if and only if $\rho^*$ is irreducible.
\end{proof}

\begin{Lemma}\label{irreducibility}
For all $\Sigma$ finite subset of $\NN^*$ the representations $$\rho_\Sigma,\rho_\Sigma^*:\Gamma\rightarrow\GL_N(\KK_\Sigma)$$ are irreducible.
\end{Lemma}
\begin{proof}
Since $\FF_q(\underline{t}_\Sigma)$ is contained in the residual field $\FF_q^{ac}(\underline{t}_\Sigma)$, if the statement of the lemma were false there would exist a non-trivial subvector space $\{0\}\subsetneq U\subsetneq \FF_q^{ac}(\underline{t}_\Sigma)^{N\times 1}$ such that $\rho_\Sigma(\gamma)U\subset U$ for all $\gamma\in \Gamma$. This would be, however, in contradiction with  \cite[Theorem 14]{PEL2}.
\end{proof}
In particular, the representations $\rho^*_\Sigma$ are irreducible for $\Sigma$ a finite subset of $\NN^*$. We are going to use Lemma \ref{condition-for-vanishing} by means of the following consequence: if an element 
of $M_w(\rho^*_\Sigma;\KK_\Sigma)$ has vanishing first entry, then it vanishes identically. This can be applied to prove (\ref{eisenstein-small-sigma}) for $s\leq q$. To see this we choose $k\in\Sigma$ and we write $\Sigma':=\Sigma\setminus\{k\}$, with $\Sigma$ non-empty finite subset of $\NN^*$ of cardinality $s\leq q$. 
The harmonic product formula of Theorem \ref{corsumshuffle} yields inductively $$\varphi_A(s-1,\sigma_{\Sigma'})\varphi_A(1,\chi_{t_k})=\varphi_A(s,\sigma_\Sigma).$$ 
This formula can also be written more explicitly in the following way:
$$\prod_{i\in\Sigma}\left(\sum_{a\in A^+}\chi_{t_i}(a)u_a\right)=\sum_{a\in A^+}\sigma_\Sigma(a)u_a^s.$$
This implies (\ref{eisenstein-small-sigma}); we leave the details to the reader.

\subsubsection{An identity for Eisenstein series of weight $q+1$}\label{example-1-eisenstein} We use $\Sigma=\{1,2\}$ and we suppose that $q>2$. We denote by $g$ the (scalar) normalized Eisenstein series of weight $q-1$ for $\boldsymbol{1}$ (following Gekeler's notations in \cite{GEK}).
\begin{Proposition}\label{newformula}
The following identity holds when $q>2$:
$$-\mathcal{E}(q+1;\rho^*_\Sigma)=\mathcal{E}(1;\rho^*_{t_1})\otimes\mathcal{E}(q;\rho^*_{t_2})+
\mathcal{E}(q;\rho^*_{t_1})\otimes\mathcal{E}(1;\rho^*_{t_2})+(\theta^q-\theta)^{-1}g\mathcal{E}(1,\rho_{t_1}^*)\otimes\mathcal{E}(1,\rho_{t_2}^*).$$
\end{Proposition}

To prove it, we use the next Lemma in the settings of Theorem \ref{corsumshuffle}.
\begin{Lemma}\label{formula-harmonic-1} The following formula holds:
\begin{multline}\label{formula-qu-plus-one}
f_A(q+1,\sigma_\Sigma)=\\ =f_A(1,\chi_{t_1})f_A(q,\chi_{t_2})+f_A(q,\chi_{t_1})f_A(1,\chi_{t_2})-
f_A(q-1)f_A(1,\chi_{t_1})f_A(1,\chi_{t_2}).\end{multline}
\end{Lemma}
\begin{proof}
We have the following formulas where we also observe, with $\Sigma=\{1,2\}$, the formula $f(1;\chi_{t_1})f_A(1;\chi_{t_2})=f_A(2;\sigma_\Sigma)$:
\begin{eqnarray*}
f_A(1,\chi_{t_1})f_A(q,\chi_{t_2})&=&f_A\binom{\sigma_\Sigma}{q+1}+f_A\begin{pmatrix}\sigma_\Sigma & \boldsymbol{1} \\ 2 & q-1\end{pmatrix}+f_A\begin{pmatrix}\chi_{t_2} & \chi_{t_1} \\ 2 & q-1\end{pmatrix}\label{eq11}\\
f_A(1,\chi_{t_2})f_A(q,\chi_{t_1})&=&f_A\binom{\sigma_\Sigma}{q+1}+f_A\begin{pmatrix}\sigma_\Sigma & \boldsymbol{1} \\ 2 & q-1\end{pmatrix}+f_A\begin{pmatrix}\chi_{t_1} & \chi_{t_2} \\ 2 & q-1\end{pmatrix}\label{eq12}\\
f_A(q-1)f_A(2,\sigma_\Sigma)&=&f_A\binom{\sigma_\Sigma}{q+1}+2f_A\begin{pmatrix}\sigma_\Sigma & \boldsymbol{1} \\ 2 & q-1\end{pmatrix}-f_A\begin{pmatrix}\chi_{t_2} & \chi_{t_1} \\ 2 & q-1\end{pmatrix}-\\ & & -f_A\begin{pmatrix}\chi_{t_1} & \chi_{t_2} \\ 2 & q-1\end{pmatrix}\label{eq13}.
\end{eqnarray*} The formula (\ref{formula-qu-plus-one}) follows easily; it also holds for $q=2$.
\end{proof}

\begin{proof}[Proof of Proposition \ref{newformula}] We note that since $q>2$, $\mathcal{E}(2,\rho_\Sigma^*)=\mathcal{E}(1,\rho_{t_1}^*)\otimes\mathcal{E}(1,\rho_{t_2}^*)$ by (\ref{eisenstein-small-sigma}).
The first coordinates of the modular forms
$\mathcal{E}(q+1;\varphi_\Sigma),\mathcal{E}(1;\varphi_{t_1})\otimes\mathcal{E}(q;\varphi_{t_2}),
\mathcal{E}(q;\varphi_{t_1})\otimes\mathcal{E}(1;\varphi_{t_2}),g\mathcal{E}(2,\varphi_\Sigma)$
 are proportional to the following $A$-expansions (where we recall once again that $G_n(X)$ denotes the $n$-th Goss polynomial \cite[\S (3.4)]{GEK}):
\begin{eqnarray*}
\mathcal{X}&:=&\sum_{a\in A^+}\sigma_\Sigma(a)G_{q+1}(u_a),\\
\mathcal{Y}_1&:=&\left(\sum_{a\in A^+}\chi_{t_1}(a)u_a\right)\left(\sum_{b\in A^+}\chi_{t_2}(b)u_b^q\right)=\varphi_A(1,\chi_{t_1})\varphi_A(q,\chi_{t_2}),\\
\mathcal{Y}_2&:=&\left(\sum_{a\in A^+}\chi_{t_2}(a)u_a\right)\left(\sum_{b\in A^+}\chi_{t_1}(b)u_b^q\right)=\varphi_A(q,\chi_{t_1})\varphi_A(q,\chi_{t_1}),\\
\mathcal{Z}&:=&\left(1-(\theta^q-\theta)\sum_{a\in A^+}u_a^{q-1}\right)\left(\sum_{a\in A^+}\sigma_\Sigma(a)u_a^2\right)=(1-(\theta^q-\theta)\varphi_A(q-1))\varphi_A(2,\sigma_\Sigma).
\end{eqnarray*}
Note that $\mathcal{Y}_1,\mathcal{Y}_2\in\mathcal{F}_{q+1}^{\sigma_\Sigma}$. A simple computation yields $G_{q+1}(X)=X^{q+1}+(\theta^q-\theta)^{-1}X^2$. Hence
$$\mathcal{X}=(\theta^q-\theta)^{-1}\varphi_A(2;\sigma_\Sigma)+\varphi_A(q+1;\sigma_\Sigma).$$
By using Lemma \ref{formula-harmonic-1} with $f_A=\varphi_A$, the first entry of the modular form given by the difference of both sides of the identity of our statements vanishes identically so this modular form vanishes identically by Lemmas \ref{condition-for-vanishing} and \ref{irreducibility}.
\end{proof}

\subsubsection{An identity for Eisenstein series of weight $2$}\label{example-2-eisenstein} We prove here a  more complicate identity involving Eisenstein series of weights $1$ and $2$ in the case of $q$ odd. We suppose that 
$|\Sigma|\equiv2\pmod{q-1}$ and we write $s=|\Sigma|=\alpha(q-1)+2$, $\alpha\in\NN$.
We have:
\begin{Proposition}\label{abeautifulidentity}
If $q$ is odd the following formula holds:
$$\sum_{\begin{smallmatrix}U\sqcup V=\Sigma\\ 
|U|\equiv1\pmod{q-1}\\
|V|\equiv1\pmod{q-1}
\end{smallmatrix}}\mathcal{E}(1;\varphi_U)\otimes\mathcal{E}(1;\varphi_V)=-\mathcal{E}(2;\varphi_\Sigma).$$
\end{Proposition}
\begin{proof}
This is a simple combination of Lemmas \ref{condition-for-vanishing} and \ref{irreducibility} and the next Lemma \ref{lemma-identitywithfA}.
\end{proof}
\begin{Lemma}\label{lemma-identitywithfA} The following formula holds:
\begin{equation}\label{identitywithfA}
\sum_{\begin{smallmatrix}  (U,V)\text{ such that }U\sqcup V=\Sigma\\ 
|U|\equiv1\pmod{q-1}\\
|V|\equiv1\pmod{q-1}
\end{smallmatrix}}f_A\begin{pmatrix} \sigma_U \\ 1\end{pmatrix}f_A\begin{pmatrix} \sigma_V \\ 1\end{pmatrix}=2f_A\begin{pmatrix} \sigma_\Sigma \\ 2\end{pmatrix}.
\end{equation}
\end{Lemma}

\begin{proof}
We set $m=\alpha(q-1)+2$  and $n=\alpha(q-1)+1$, for $\alpha\geq 0$.
We claim that
\begin{eqnarray}
\sum_{\begin{smallmatrix} k\equiv0\pmod{q-1}\\ 0<k\leq \alpha\end{smallmatrix}}\binom{n}{k}\equiv0\pmod{p}\label{congruenceone},\\
\sum_{\begin{smallmatrix} k\equiv1\pmod{q-1}\\ 0\leq k\leq \alpha\end{smallmatrix}}\binom{m}{k}\equiv2\pmod{p}\label{congruencetwo}.
\end{eqnarray}
To see this we consider more generally $N\in\NN$ and we write $N=\alpha(q-1)+l$ with $\alpha\geq 0$ and $0\leq l\leq q-2$.
Let $\lambda,\mu$ be in $\FF_q$. Then,
$$(\lambda+\mu)^l=(\lambda+\mu)^N=\sum_{r=0}^N\binom{N}{r}\lambda^r\mu^{N-r}=\sum_{r_0=0}^{q-2}\lambda^{r_0}\mu^{\nu(r_0)}\underbrace{\sum_{\begin{smallmatrix}r\equiv r_0\pmod{q-1}\\ 0\leq r\leq N\end{smallmatrix}}\binom{N}{r}}_{=:\beta_{r_0}},$$ where $\nu(r_0)$ is the unique integer in $\{0,\ldots,q-2\}$ such that $l-r_0\equiv \nu(r_0)\pmod{q-1}$. Setting further $\lambda=1$, 
we have the polynomial 
$$P(X)=(X+1)^l-\sum_{r_0=0}^{q-2}\beta_{r_0}X^{\nu(r_0)}\in\FF_p[X],$$
which vanishes identically over $\FF_q$, and has degree $\leq q-2$. This implies that it is identically zero. Taking $N=m=\alpha(q-1)+2$ we have $l=2$ and computing the coefficient of $X$ in $P$, we deduce (\ref{congruenceone}). Taking $N=n=\alpha(q-1)+1$ and computing the constant term of $P$, we deduce  (\ref{congruencetwo}). This shows the claim.
We can complete the proof of formula \ref{identitywithfA}.
We use Theorem \ref{corsumshuffle}, which tells us that if $U\sqcup V=\Sigma$ with $|U|\equiv|V|\equiv1\pmod{q-1}$, 
\begin{multline*}
f_A\left(\begin{matrix} \sigma_U \\
1\end{matrix}\right)f_A\left(\begin{matrix} \sigma_V \\
1\end{matrix}\right)-f_A\left(\begin{matrix}\sigma_\Sigma \\
2\end{matrix}\right)= \\ f_A\left(\begin{matrix}\sigma_U & \sigma_V \\
1 & 1\end{matrix}\right)+f_A\left(\begin{matrix}\sigma_V & \sigma_U \\
1 & 1\end{matrix}\right)-\sum_{\begin{smallmatrix}I\sqcup J= \Sigma \\
|J|\equiv1\pmod{q-1}\\
J\subset U\text{ or }J\subset V\end{smallmatrix}}f_A\left(\begin{matrix}\sigma_{I} & \sigma_{J} \\
1 & 1\end{matrix}\right).\end{multline*}
We sum these identities over all such partitions $\Sigma=U\sqcup V$. First of all, the number of such partitions is equal to 
$$\sum_{k=0}^\alpha\binom{s}{k(q-1)+1}$$ which is congruent to $2$ modulo $p$ by (\ref{congruenceone}). Let 
$$f:\mathcal{P}(\Sigma)^2\rightarrow L$$ be any map with values in a field $L$ of characteristic $p$, where $\mathcal{P}(\Sigma)$ is the set of subsets of $\Sigma$.
Then,
\begin{multline*}
\sum_{\begin{smallmatrix} U\sqcup V=\Sigma\\ 
|U|\equiv1\pmod{q-1}
\end{smallmatrix}}\left(\sum_{\begin{smallmatrix} I\sqcup J=\Sigma\\ 
|J|\equiv1\pmod{q-1}\\ J\subset U\text{ or }J\subset V\end{smallmatrix}}f(I,J)-f(U,V)-f(V,U)\right)=\\ 
=\sum_{\begin{smallmatrix} U\sqcup V=\Sigma\\ 
|U|\equiv1\pmod{q-1}\end{smallmatrix}}\sum_{\begin{smallmatrix} I\sqcup J=\Sigma\\ 
|J|\equiv1\pmod{q-1}\\ J\subsetneq U\text{ or }J\subsetneq V\end{smallmatrix}}f(I,J)=\\
=\sum_{\begin{smallmatrix} I\sqcup J=\Sigma\\ 
|J|\equiv1\pmod{q-1}\end{smallmatrix}}f(I,J)\sum_{\begin{smallmatrix} U\sqcup V=\Sigma\\ 
|U|\equiv1\pmod{q-1}\\ U\supsetneq J\text{ or }V\supsetneq J\end{smallmatrix}}1,
\end{multline*}
which vanishes by (\ref{congruencetwo}).
Observing that we can choose $f(I,J)=f_A(\begin{smallmatrix} \sigma_I & \sigma_J \\ 1 & 1\end{smallmatrix})$ terminates the proof.
\end{proof}

As a complement of Proposition \ref{abeautifulidentity} we propose the following question, to be compared with Cornelissen, \cite[Proposition (1.15)]{COR}. We assume that $|\Sigma|\equiv2\pmod{q-1}$.

\begin{Question}\label{conjectureweight2}
Do the forms $\mathcal{E}(1;\rho^*_U)\otimes\mathcal{E}(1;\rho^*_V)$, for $U\sqcup V=\Sigma$
and $|U|\equiv|V|\equiv1\pmod{q-1}$ generate 
the module $M_{2}(\rho^*_\Sigma;\KK_\Sigma)$?
\end{Question}

\subsubsection{Serre's derivatives of Eisenstein series}\label{example-3-eisenstein}

We return to the operators $\partial_n^{(w)}(f)$ introduced in \S \ref{serre-derivatives}. We suppose that $\Sigma\subset\NN^*$ is such that $s=|\Sigma|\equiv1\pmod{q-1}$ and we study the $u$-expansion of the first entry (indexed by $\emptyset$) of $$\partial_1^{(1)}(\mathcal{E}(1;\rho^*_\Sigma))\in S_{3}(\rho^*_\Sigma\det{}^{-1};\KK_\Sigma).$$ By Proposition \ref{seriesofEisensteinseries}, the first entry of $\mathcal{E}(1;\rho^*_\Sigma)$ is equal to $-\widetilde{\pi}\varphi_A(1;\sigma_\Sigma)$.
We compute, by setting $\Sigma'=\Sigma\sqcup\{0\}$:
\begin{eqnarray*}
\lefteqn{\partial^{(1)}_1(f_A(1;\sigma_\Sigma))=}\\
&=&\sum_{a\in A^+}\sigma_\Sigma(a)au_a^2-\sum_{a\in A^+}au_a\sum_{b\in A^+}\sigma_\Sigma(b)u_b\\
&=&-[\varphi_A(1;\chi_{t_0})\varphi_A(1;\sigma_\Sigma)-\varphi_A(2;\sigma_{\Sigma'})]_{t_0=\theta}.
\end{eqnarray*}
Hence:
\begin{Lemma}
We have the formula:
$$\partial^{(1)}_1(f_A(1;\sigma_\Sigma))=\left[\sum_{\begin{smallmatrix}I\sqcup J=\Sigma'\\ |J|\equiv1\pmod{q-1}\\J=\{0\}\text{ or }J\subset\Sigma'\end{smallmatrix}}\varphi_A\begin{pmatrix}\sigma_I & \sigma_J\\ 1 & 1\end{pmatrix}-
\varphi_A\begin{pmatrix}\chi_{t_0} & \sigma_\Sigma\\ 1 & 1\end{pmatrix}-
\varphi_A\begin{pmatrix} \sigma_\Sigma& \chi_{t_0}\\ 1 & 1\end{pmatrix}\right]_{t_0=\theta}.$$
\end{Lemma}
\begin{proof}
This follows directly from Theorem \ref{corsumshuffle} (we interpret Serre's derivatives
in terms of specializations of the harmonic relations of \S \ref{shufflerelations}).
\end{proof}
In particular we have, applying again Lemmas \ref{condition-for-vanishing} and \ref{irreducibility}:
\begin{Lemma}
If $s=|\Sigma|\leq q-1$, then $\partial^{(s)}_1(\mathcal{E}(s;\rho^*_\Sigma))=0$.
\end{Lemma}
We propose the next question if $s=|\Sigma|\equiv1\pmod{q-1}$:
\begin{Question}
Is the form $\partial^{(1)}_1(\mathcal{E}(1;\sigma_\Sigma))$ and the last column of the Poincar\'e series $\mathcal{P}_{3}(G)$ proportional with a proportionality factor in $\LL_\Sigma^\times$?
\end{Question}
In the above question, $G$ is as in Proposition \ref{proposition-poincare-bis} with $m=1$.
This is suggested by the fact that the scalar cusp forms $\partial^{(q-1)}_1(g)$ and $h$ are one proportional to the other (notations of \cite{GEK}). We do not know if, in the case $s=q$, $\partial^{(q-1)}_n(\mathcal{E}(1;\sigma_\Sigma))\neq0$
for $n=1,\ldots,q-2$. 

\subsection{A conjecture on multiple sums}

We write
$\mathcal{Z}_\zeta$ for the $\FF_p$-algebra 
$\mathcal{F}=\sum_{n,\sigma}\mathcal{F}_{n,\sigma}$ where $\mathcal{F}_{n,\sigma}$ is the 
$\FF_p$-subvector space of $\FF_p[t_i:i\in\NN][[\frac{1}{\theta}]]$ (with the Gauss norm $\|\cdot\|$ extending $|\cdot|$) generated by the sums $f_A(\mathcal{C})$ of (\ref{definition-sums-fa}) in the settings of \S \ref{exa1} so that the semi-characters
$\sigma$ involved in the compositions arrays (\ref{compositionarray}) are maps from $A$ to $\FF_q[t_i:i\in\NN]$ defined by 
\begin{equation}\label{first-sample-semicharacters}
\sigma(a)=\prod_{i\in\NN}\chi_{t_i}(a)^{n_i},\quad a\in A,
\end{equation}
with $n_i\in\NN$ and $n_i=0$ for all but finitely many $i\in\NN$ (so a variable $t_0$ is allowed). In this case 
we prefer to write $\zeta_A(\mathcal{C})$ instead of $f_A(\mathcal{C})$.
The algebra $\mathcal{Z}_\zeta$ is the $\FF_p$-algebra of the {\em multiple zeta values} (in Tate algebras). 

Similarly, we write $\mathcal{Z}_\varphi$ for the $\FF_p$-algebra $\mathcal{F}=\sum_{n,\sigma}\mathcal{F}_{n,\sigma}$ where $\mathcal{F}_{n,\sigma}$ is this time the 
$\FF_p$-subvector space of $\FF_p[\theta][t_i:i\in\NN^*][[u]]$ (with the $v$-valuation) generated by the 
sums $\varphi_A(\mathcal{C})$ of (\ref{exa2}). Theorem \ref{anewsumshuffle} implies that $\mathcal{Z}_\zeta$ and $\mathcal{Z}_\varphi$ are $\FF_p$-algebras. 
However,  we do not know if $\mathcal{Z}_\varphi$ is graded by the degrees like 
$\mathcal{Z}_\zeta$. The algebra $\mathcal{Z}_\varphi$ is the algebra of {\em $A$-periodic multiple sums}.
We propose:

\begin{Conjecture}\label{conj-zetavarphi}
The correspondence $\zeta_A(\mathcal{C})\leftrightarrow\varphi_A(\mathcal{C})$ induces an isomorphism of $\FF_p$-algebras $\mathcal{Z}_\zeta\cong\mathcal{Z}_\varphi$.
\end{Conjecture}

Conjecture \ref{conj-zetavarphi}
implies that $\mathcal{Z}_\varphi$ is graded by the degrees.
Moreover, all the identities for multiple zeta values in $\mathcal{Z}_\zeta$ correspond to identities for multiple $A$-periodic sums, many of which can be proved directly 
(e.g. Lemmas \ref{formula-harmonic-1} and \ref{lemma-identitywithfA}).
For example, assuming this conjecture, note that in the proof of Proposition \ref{newformula}, $\mathcal{X}$ and $\mathcal{Z}$ are not homogeneous for the degrees. By Conjecture \ref{conj-zetavarphi}, any linear dependence relation among $\mathcal{X}, \mathcal{Y}_1,\mathcal{Y}_2$ and $\mathcal{Z}$ must come from two homogeneous ones, one in $\mathcal{F}_{q+1}^{\sigma_\Sigma}$ and another one in $\mathcal{F}_{2}^{\sigma_\Sigma}$, both defined over $\FF_p$.
Through Conjecture \ref{conj-zetavarphi} we see that these relations are indeed derived from (\ref{formula-qu-plus-one}) and the identity $\varphi_A(2;\sigma_\Sigma)=\varphi_A(1;\chi_{t_1})\varphi_A(1;\chi_{t_2})$. 

\section{Perspectives on algebraic properties of Eisenstein series}\label{Some-conjectures}

We give here some conjectures which allow to produce examples of relations which can be in certain cases  verified by explicit computations. 
This section provides perspectives suggested by experimental investigations we did for modular forms associated to the representations $\rho^*_\Sigma$. Conjecture \ref{conj-mult-eisentein} using the notion of multiple Eisenstein series, and Conjectures \ref{conjecture-pellarin}, \ref{generalisation-thakur} and \ref{conj-zetavarphi} together provide a collection of identities between our Eisenstein series, introduced in \S \ref{eisensteinseries}. Some special cases can be verified by explicit computation.

\subsection{Multiple Eisenstein series}\label{Multiple-Eisenstein-series}

In \cite{CHE}, Chen introduces a function field variant of {\em Eisenstein} and {\em double Eisenstein series}
as initially defined by Gangl, Kaneko and Zagier in \cite{GAN&KAN&ZAG}. We propose here a generalization of her viewpoint. We begin with a description of the required settings, introducing a vector-valued generalization of multiple Eisenstein series. We state Conjecture \ref{conj-mult-eisentein} suggesting natural correspondences between multiple zeta values and multiple Eisenstein series. 

We consider $\rho_1,\ldots,\rho_r$ representations of the first kind which are constructed starting from basic representations by using the operations $\oplus,\otimes,\wedge^\alpha,S^\beta$ as well as the `comatrix operation' $\operatorname{Co}$, defined through the comatrix map. All these representations extend to monoid maps defined over $A^{2\times 2}$, with its standard matrix product. Before going on we need some notation: we need to work with composition arrays having the first line composed by representations of the first kind.

We consider positive integers $n_1,\ldots,n_r$ and composition arrays (with $(\cdot)^*$ contragredient)
$$\widehat{\mathcal{C}}=\begin{pmatrix} \rho_1 & \cdots & \rho_r\\ n_1 & \cdots & n_r\end{pmatrix},\quad \widehat{\mathcal{C}}^*=\begin{pmatrix} \rho_1^* & \cdots & \rho_r^*\\ n_1 & \cdots & n_r\end{pmatrix}.$$
We also set, for $j\in\{0,\ldots,r\}$
$$\widehat{\mathcal{C}}_{\leq j}=\begin{pmatrix} \rho_1 & \cdots & \rho_j\\ n_1 & \cdots & n_j\end{pmatrix},\quad \widehat{\mathcal{C}}_{>j}=\begin{pmatrix} \rho_{j+1} & \cdots & \rho_r\\ n_{j+1} & \cdots & n_r\end{pmatrix},$$
so that $\widehat{\mathcal{C}}_{\leq r}=\widehat{\mathcal{C}}$ and we set $\widehat{\mathcal{C}}_{<1}=\emptyset$. We now define:
$$\widetilde{\boldsymbol{\Phi}}(\widehat{\mathcal{C}}^*)=\sum_{|a_1|>\cdots>|a_r|>0}\rho_1\Big(\begin{smallmatrix} a_1 & 0 \\ 0 & 1\end{smallmatrix}\Big)\otimes\cdots\otimes\rho_r\Big(\begin{smallmatrix} a_r & 0 \\ 0 & 1\end{smallmatrix}\Big)\cdot
\Psi_{n_1}(\rho_1^*)_{a_1}\otimes\cdots\otimes\Psi_{n_r}(\rho_r^*)_{a_r},$$ with the sum running over elements $a_1,\ldots,a_r\in A^+$. The dot $\cdot$ is the usual matrix product, and the index $(\cdot)_{a}$ with $a\in A$ designates the substitution $z\mapsto az$. The matrices $(\begin{smallmatrix} a_j & 0 \\ 0 & 1\end{smallmatrix})$ do not belong to $\Gamma$ but all the terms of the series are well defined thanks to the hypothesis on $\rho_1,\ldots,\rho_r$. This series converges to a rigid analytic map
$\Omega\rightarrow\KK_\Sigma^{N\times N}$ for appropriate $\Sigma\subset\NN^*$ and $N>0$ and to 
an element of $\mathfrak{M}_\Sigma^{N\times N}$.

In the case $\rho_i=\rho_{U_i}$ with, for $i=1,\ldots,r$, $U_i$ finite subsets of $\NN^*$, the case that interests us the most in the present paper, we have:
$$\left[\widetilde{\boldsymbol{\Phi}}(\widehat{\mathcal{C}}^*)\right]_1=\widetilde{\pi}^{\sum_in_i}\sum_{\begin{smallmatrix}a_1,\ldots,a_r\in A^+\\ |a_1|>\cdots>|a_r|\end{smallmatrix}}\rho_{U_1}\Big(
\begin{smallmatrix} a_1 & 0 \\ 0 & 1\end{smallmatrix}\Big)\otimes\cdots\otimes\rho_{U_r}\Big(
\begin{smallmatrix} a_r & 0 \\ 0 & 1\end{smallmatrix}\Big)\cdot V(n_1;\rho^*_{U_1})_{a_1}\otimes\cdots\otimes
V(n_r;\rho^*_{U_r})_{a_r}$$
where $[\cdot]_1$ denotes the first column of a matrix and $V(n;\rho^*_U)$ is defined in 
(\ref{definition-V}).

We also set
$$\boldsymbol{\mathcal{Z}}(\widehat{\mathcal{C}}^*):=(\rho_1\otimes\cdots\otimes\rho_r)\Big(\begin{smallmatrix} 0 & 0 \\ 0 & 1\end{smallmatrix}\Big)\cdot\sum_{\begin{smallmatrix}
a_1,\ldots,a_r\in A^+\\
|a_1|>\cdots>|a_r|>0\end{smallmatrix}}a_1^{-n_1}\cdots a_r^{-n_r}\rho_1^*\Big(T_{-a_1}\Big)\otimes\cdots\otimes\rho_r^*\Big(T_{-a_r}\Big),$$ a series which converges in $\FF_q(\underline{t}_\Sigma)((\theta^{-1}))^{N\times N}$ (we follow the same conventions used in the definition of $\widetilde{\boldsymbol{\Phi}}(\widehat{\mathcal{C}}^*)$). In the case $\rho_i=\rho_{U_i}$ for all $i$, we have
$$\left[\boldsymbol{\mathcal{Z}}(\widehat{\mathcal{C}}^*)\right]_1=\begin{pmatrix} 0 \\ \vdots \\ 0 \\ \zeta_A\Big(\begin{smallmatrix}\sigma_{U_1} & \cdots & \sigma_{U_r} \\ n_1 & \cdots & n_r\end{smallmatrix}\Big)\end{pmatrix}\in\TT_{\Sigma}^{N\times 1}$$ where $\Sigma=\cup_iU_i$ and $N$ is the product of the dimensions of the representations $\rho_i$, agreeing with (\ref{definition-Z}).

\begin{Definition}
{\em The {\em multiple Eisenstein series} associated with the composition array $\mathcal{\widehat{\mathcal{C}}^*}$ is the series
$$\mathcal{E}_A(\widehat{\mathcal{C}}^*)=\left[\widetilde{\boldsymbol{\Phi}}(\widehat{\mathcal{C}}^*)+\widetilde{\boldsymbol{\Phi}}(\widehat{\mathcal{C}}_{\leq r-1}^*)\otimes\boldsymbol{\mathcal{Z}}(\widehat{\mathcal{C}}_{>r-1}^*)+\cdots+\boldsymbol{\mathcal{Z}}(\widehat{\mathcal{C}}^*)\right]_1\in\mathfrak{O}_\Sigma^{N\times 1}.$$ 
We say that $\mathcal{E}_A(\widehat{\mathcal{C}}^*)$ is of {\em degree} $\binom{\rho^*_1\otimes\cdots\otimes\rho^*_r}{n_1+\cdots+n_r}$.}
\end{Definition}

It is easy to verify that, if the representations $\rho_i$ are all equal to $\boldsymbol{1}$ (case in which $N=1$ and $\Sigma=\emptyset$) and $r=1,2$, this coincides with \cite[Definition 3.2]{CHE}, namely, the function
$E_k(z)$ defined in ibid. coincides with our $\mathcal{E}_A(k)$ (for $k>0$) and similarly, 
$E_{r',s'}(z)$ of ibid. coincides with our $\mathcal{E}_A(r',s')$. The case of depth $r=1$ can be resumed in the next formula which follows easily from (\ref{first-expansion-eisenstein}) and Proposition \ref{seriesofEisensteinseries}, where $m>0$ and $\Sigma\subset\NN^*$ a finite subset such that
$|\Sigma|\equiv m\pmod{q-1}$:
\begin{equation}\label{lemma-depth-one-eisenstein}
\mathcal{E}_A\binom{\rho^*_\Sigma}{m}=\left[\widetilde{\boldsymbol{\Phi}}\binom{\rho^*_\Sigma}{m}+\boldsymbol{\mathcal{Z}}\binom{\rho^*_\Sigma}{m}\right]_1=-\mathcal{E}(m;\rho_\Sigma^*).
\end{equation}
It is also easy to verify the next lemma, where $\rho_1^*,\ldots,\rho_r^*$ are representations of the form $\rho_{U_j}^*$ with $U_j\subset \Sigma$  for some finite subset $\Sigma$ of $\NN^*$, and where $\sigma_1,\ldots,\sigma_r$ denote the projections of 
$\rho_1,\ldots,\rho_r$ on their upper-right coefficients (these are semi-characters). We recall the multiple sums $\phi_A$ defined in (\ref{variant-Goss-sums}).
\begin{Lemma}\label{first-last-entries} Writing
$$\widehat{\mathcal{C}}^*=\begin{pmatrix}\rho_1^* & \cdots & \rho_r^* & \boldsymbol{1} & \cdots & \boldsymbol{1} \\ n_1 & \cdots & n_r & m_1 & \cdots & m_s\end{pmatrix},\quad \mathcal{C}=\begin{pmatrix}\sigma_1 & \cdots & \sigma_r & \boldsymbol{1} & \cdots & \boldsymbol{1} \\ n_1 & \cdots & n_r & m_1 & \cdots & m_s\end{pmatrix},$$
for $r>0,s\geq 0$, the first entry $\mathcal{E}_1$ of $\mathcal{E}=\mathcal{E}_A(\widehat{\mathcal{C}}^*)$ satisfies, with $n=\sum_in_i$,
\begin{multline*}
\mathcal{E}_1=\widetilde{\pi}^n\Big(\widetilde{\pi}^{\sum_{j\leq s}m_j}\phi_A\Big(\begin{smallmatrix}\sigma_1 & \cdots & \sigma_r & \boldsymbol{1} & \cdots & \boldsymbol{1} \\ n_1 & \cdots & n_r & m_1 & \cdots & m_s\end{smallmatrix}\Big)+\widetilde{\pi}^{\sum_{j\leq s-1}m_j}\phi_A\Big(\begin{smallmatrix}\sigma_1 & \cdots & \sigma_r & \boldsymbol{1} & \cdots & \boldsymbol{1} \\ n_1 & \cdots & n_r & m_1 & \cdots & m_{s-1}\end{smallmatrix}\Big)\zeta_A(m_s)+\cdots\\ \cdots+\widetilde{\pi}^{m_1}\phi_A\Big(\begin{smallmatrix}\sigma_1 & \cdots & \sigma_r & \boldsymbol{1} \\ n_1 & \cdots & n_r & m_1 \end{smallmatrix}\Big)\zeta_A(m_2,\ldots,m_s)+\phi_A\Big(\begin{smallmatrix}\sigma_1 & \cdots & \sigma_r \\ n_1 & \cdots & n_r \end{smallmatrix}\Big)\zeta_A(m_1,\ldots,m_s)\Big)\end{multline*}
and the last entry $\mathcal{E}_N\in\mathfrak{O}_\Sigma$ of $\mathcal{E}=\mathcal{E}_A(\widehat{\mathcal{C}}^*)$ satisfies
$$\mathcal{E}_N-\zeta_A(\mathcal{C})\in \mathfrak{M}_\Sigma.$$
\end{Lemma}

\subsubsection{Eulerian multiple zeta values}

We consider semi-characters $\sigma_1,\ldots,\sigma_r$ defined as in (\ref{semi-char-monoid}) and positive integers $n_1,\ldots,n_r$. We write $\sigma=\prod_i\sigma_i=\prod_j\chi_{t_j}^{\nu_j}$ for the type and $n=\sum_in_i$ for the weight of 
the multiple zeta value $\zeta_A(\begin{smallmatrix} \sigma_1 & \cdots & \sigma_r\\ n_1 & \cdots & n_r\end{smallmatrix})$. In this subsection we return to the settings of \S \ref{exa1} to make the following definition.

\begin{Definition}\label{eulerianity}
{\em Let $Z$ be a $K$-linear combination of multiple zeta values of degree $\binom{\sigma}{n}$. We say that $Z$ is {\em Eulerian} if $$Z\in K(\underline{t}_\Sigma)\frac{\widetilde{\pi}^n}{\prod_j\omega(t_j)^{\nu_j}}.$$}
\end{Definition}

This agrees with the notion of eulerian multiple zeta value of Thakur as in \cite[Definition 5.10.8]{THA0} because in the case of trivial type the product involving the Anderson-Thakur function is equal to one. See also \cite{CHA&PAP&YU}. Examples of Eulerian combinations of multiple zeta values in our settings are given by the elements $\zeta_A(n;\sigma_\Sigma)$ with $|\Sigma|\equiv n\pmod{q-1}$. By using \cite[(39)]{GEZ&PEL} we see that 
the elements $\zeta_A(\begin{smallmatrix} \chi_t & \boldsymbol{1} & \cdots &\boldsymbol{1} \\ 1 & q-1 & \cdots & q^{r}(q-1)\end{smallmatrix})$ are eulerian for all $r\geq 0$.

\subsubsection{A conjecture for multiple Eisenstein series}

We denote by $\mathcal{W}^{\rho^*}_n$ the $\FF_p$-vector space of multiple Eisenstein series of
degree $\binom{\rho^*}{n}$, with $n>0$ where $\rho^*$ a product of representations of the type $\rho^*_{U_i}$.
Writing $\rho^*=\otimes_j(\rho_{t_j}^*)^{\otimes \nu_j}$, we set $\sigma=\prod_j\chi_{t_j}^{\nu_j}$.
We consider  $\mathcal{C},\widehat{\mathcal{C}}^*$ as in Lemma \ref{first-last-entries}. We address the following:
\begin{Conjecture}\label{conj-mult-eisentein}
The following properties hold:
\begin{enumerate}
\item We have inclusions $\mathcal{W}^{\rho^*}_m\otimes\mathcal{W}^{\psi^*}_n\subset \mathcal{W}^{\rho^*\otimes\psi^*}_{m+n}$.
\item The correspondence $\zeta_A(\mathcal{C})\mapsto\mathcal{E}_A(\widehat{\mathcal{C}}^*)$ defines
an isomorphism $\eta$ of $\FF_p$-vector spaces between the space $\mathcal{Z}^{\sigma}_n$ of multiple zeta values of degree $\binom{\sigma}{n}$ and $\mathcal{W}^{\rho^*}_n$ which is compatible with the multiplication rules in such a way that the sum
$\mathcal{W}:=\sum_{n,\rho^*}\mathcal{W}^{\rho^*}_n$ is graded, and endowed with a structure of $\FF_p$-algebra with multiplication $\otimes$, isomorphic to the algebra $\bigoplus_{n,\sigma}\mathcal{Z}^\sigma_n$.
\item An element $f\in\mathcal{Z}^\sigma_n$ is eulerian if and only if $\eta(f)$ is a modular form in $M_n(\rho^*;\LL_\Sigma)$.
\end{enumerate}
\end{Conjecture}

The next result describes a depth two identity which illustrates the pertinence of the above conjecture in a special case, interesting because lying outside the case of Eisenstein series. The reader will notice that the proof given is quite ad hoc and not easily generalizable. While the first item of the conjecture is likely to be at reach by an appropriate generalization of the harmonic product of \S \ref{shufflerelations}, the equivalence between eulerianity of constant terms and modularity of vector functions may require deeper arithmetic tools.

\begin{Proposition}\label{a-very-special-identity}
The following identity holds:
$$\mathcal{E}(1;\rho^*_t)\otimes\mathcal{E}(q-1;\boldsymbol{1})+\mathcal{E}(q;\rho^*_t)=\mathcal{E}_A\Big(\begin{smallmatrix}\rho^*_t & \boldsymbol{1}\\ 1 & q-1\end{smallmatrix}\Big).$$
\end{Proposition}

\begin{proof}
We claim that 
\begin{equation}\label{e-weight1}
\mathcal{E}(1;\rho^*_1)=\widetilde{\pi}\begin{pmatrix} -\varphi_A\binom{\chi_t}{1} \\
\widetilde{\varphi}_A\binom{\delta}{1}\end{pmatrix}u_a-\binom{0}{\zeta_A(1;\chi_t)},\end{equation}
where $\delta$ is the semi-character $a\mapsto \chi_t(az)$ and $\widetilde{\varphi}_A$ has been introduced in \S \ref{exa3}. This follows easily from (\ref{lemma-depth-one-eisenstein}), Perkins' identity (\ref{Perkins-formula}) and Proposition \ref{seriesofEisensteinseries}.
In a similar way, we can easily prove the identity
\begin{equation}\label{e-weightq}
\mathcal{E}(q;\rho^*_1)=\widetilde{\pi}^q\begin{pmatrix} -\varphi_A\binom{\chi_t}{q} \\
\widetilde{\varphi}_A\binom{\delta}{q}\end{pmatrix}+\frac{\widetilde{\pi}^q}{(t-\theta)\omega(t)}\binom{0}{\widetilde{\varphi}_A(q-1)}-\binom{0}{\zeta_A(q;\chi_t)}.\end{equation}
To see this, note that $\tau(\mathcal{E}(1;\rho^*_1))=\mathcal{E}(q;\rho^*_1)$ and use (\ref{taudifferechit}); all the series involved in these formulas are convergent for the $v$-valuation. We deduce from (\ref{e-weight1}) and (\ref{e-weightq}) that
\begin{multline*}
\mathcal{E}(1;\rho^*_t)\otimes\mathcal{E}(q-1;\boldsymbol{1})+\mathcal{E}(q;\rho^*_t)=\\ =\widetilde{\pi}^q\begin{pmatrix}\varphi_A\Big(\begin{smallmatrix}\chi_t & \boldsymbol{1} \\ 1 & q-1\end{smallmatrix}\Big)
\\ -\widetilde{\varphi}_A\Big(\begin{smallmatrix}\delta & \boldsymbol{1} \\ 1 & q-1\end{smallmatrix}\Big)\end{pmatrix}+\widetilde{\pi}\begin{pmatrix}\varphi_A\Big(\begin{smallmatrix}\chi_t  \\ 1 \end{smallmatrix}\Big)
\\ -\widetilde{\varphi}_A\Big(\begin{smallmatrix}\delta \\ 1 \end{smallmatrix}\Big)\end{pmatrix}\zeta_A(q-1)+
\begin{pmatrix}0\\ \zeta_A\Big(\begin{smallmatrix}\chi_t & \boldsymbol{1} \\ 1 & q-1\end{smallmatrix}\Big)\end{pmatrix}.\end{multline*}
This identity is reached applying the second and the third identities in (\ref{three-special-formulas}) and 
the formula (\ref{first-formula-pel}). But a direct computation
shows that
$$\mathcal{E}_A\begin{pmatrix} \rho^*_t & \boldsymbol{1} \\ 1 & q-1\end{pmatrix}=\left[
\widetilde{\boldsymbol{\Phi}}\begin{pmatrix} \rho^*_t & \boldsymbol{1} \\ 1 & q-1\end{pmatrix}+
\widetilde{\boldsymbol{\Phi}}\begin{pmatrix} \rho^*_t  \\ 1 \end{pmatrix}\otimes\boldsymbol{\mathcal{Z}}\begin{pmatrix} \boldsymbol{1} \\ q-1\end{pmatrix}
+\boldsymbol{\mathcal{Z}}\begin{pmatrix} \rho^*_t & \boldsymbol{1} \\ 1 & q-1\end{pmatrix}
\right]_1$$ equals the right-hand side of the above identity.
\end{proof}
We deduce that $\mathcal{E}_A(\begin{smallmatrix} \rho^*_t & \boldsymbol{1} \\ 1 & q-1\end{smallmatrix})$
is in $M_q(\rho^*_t;\LL)$. One further proves that it is non-zero and is not a cusp form. In fact we have that 
$$\mathcal{E}_A\Big(\begin{smallmatrix} \rho^*_t & \boldsymbol{1} \\ 1 & q-1\end{smallmatrix}\Big)-\binom{0}{\zeta_A\Big(\begin{smallmatrix}\chi_t & \boldsymbol{1} \\ 1 & q-1\end{smallmatrix}\Big)}\in\mathfrak{M}_\Sigma^{2\times 1},$$ with $\mathfrak{M}_\Sigma$ the maximal ideal of the valuation $v$, and $\Sigma$ a singleton, so that $\eta(\zeta_A(\begin{smallmatrix} \chi_t & \boldsymbol{1} \\ 1 & q-1\end{smallmatrix}))=
\mathcal{E}_A(\begin{smallmatrix} \rho^*_t & \boldsymbol{1} \\ 1 & q-1\end{smallmatrix})$ and 
we see, by \cite[Lemma 6.12]{GEZ&PEL} that $\zeta_A(\begin{smallmatrix} \chi_t & \boldsymbol{1} \\ 1 & q-1\end{smallmatrix})$ is Eulerian. One proves easily that 
$\zeta_A\Big(\begin{smallmatrix}\chi_t & \boldsymbol{1} \\ 1 & q-1\end{smallmatrix}\Big)=\frac{\theta-t}{\theta^q-\theta}\zeta_A\Big(\begin{smallmatrix}\chi_t \\ q \end{smallmatrix}\Big)$. However, the cusp form
$\mathcal{E}_A\Big(\begin{smallmatrix}\chi_t & \boldsymbol{1} \\ 1 & q-1\end{smallmatrix}\Big)-
\frac{\theta-t}{\theta^q-\theta}\mathcal{E}_A\Big(\begin{smallmatrix}\chi_t \\ q \end{smallmatrix}\Big)
$ does not vanish identically by Corollary \ref{regular-equals-modular}. Hence the item (3) of Conjecture \ref{conj-mult-eisentein} does not extend to $K$-linear combinations of multiple Eisenstein series.

\subsection{A conjecture for zeta values in Tate algebras}\label{conjecture-tate-algebras}
We now focus on zeta values in Tate algebras (\ref{zeta-value-tate}). Recall from the introduction that $q=p^e$ with $e>0$. Hence
$\tau=\mu^e$ where $\mu$ is the $\FF_p$-linear automorphism of $\CC_\infty$ given by $c\mapsto c^p$ for 
$c\in\CC_\infty$, which can be extended $\FF_p(\underline{t}_\Sigma)$-linearly to $\KK_\Sigma$ for any finite set $\Sigma$.
We introduce the following $\FF_p$-algebra
$$\mathbb{I}:=\FF_p\left[\mu^m(\zeta_A(1,\chi_{t_i})):\begin{matrix}i\in\NN^*\\ m\in\ZZ\end{matrix}\right]\subset\bigcup_{k\geq 0}\FF_p[t_i:i\in\NN^*][[\theta^{-\frac{1}{p^k}}]].$$ We set $\zeta_A(0):=1$.
The $\FF_p$-algebra $\mathbb{I}$ is thus generated by all the $\mu$-twists (negative or positive) of the functions $\zeta_A(1,\chi_{t_i})$ for $i\in\Sigma$. It is very important to allow negative values for $m$, and for this reason this $\FF_p$-algebra carries a structure of inversive $\mu$-difference algebra. We address the following
\begin{Conjecture}\label{conjecture-pellarin}
For all $n\in\NN^*$ and $\Sigma\subset\NN^*$ such that $|\Sigma|\equiv n\pmod{q-1}$ we have a unique expansion
\begin{equation}\label{identity-conjecture}
\zeta_A(n;\sigma_\Sigma)=\sum_{\begin{smallmatrix}0\leq k\leq n\\ k\equiv0\pmod{q-1}\end{smallmatrix}}\zeta_A(k)\eta_k,\quad \eta_k\in\mathbb{I}.\end{equation}
\end{Conjecture}
Recall that in our conventions, $\zeta_A(k)=\zeta_A(k;\boldsymbol{1})$ are the usual Carlitz zeta values.
We are going to give some examples of relations along the predictions of this conjecture. Note that the factors $\eta_k$ need not to lie in $\FF_p[t_i:i\in\NN^*]((\frac{1}{\theta}))$. However, there exists $l\in\NN$ such that
$\mu^l(\eta_k)\in\FF_p[t_i:i\in\NN^*]((\frac{1}{\theta}))$ for all $k\equiv 0\pmod{q-1}$ in the range $0\leq k\leq n$ and all the terms involved are products of zeta values. Since $\mu^l(\zeta_A(k;\sigma_\Sigma))=
\zeta_A(kp^l;\sigma_\Sigma)$, the identity (\ref{identity-conjecture}) is equivalent to an algebraic identity of zeta values as in (\ref{zeta-value-tate}) defined over $\FF_p$. We recall from Thakur conjectures in \cite[\S 5.3]{THA2} that the only $\FF_p$-relations
among his multiple zeta values in $K_\infty$ are those which come from the harmonic product.

\begin{Conjecture}\label{generalisation-thakur}
The only $\FF_p$-algebraic relations in $\mathbb{I}$ are those coming from the harmonic product.
\end{Conjecture}
After Conjecture \ref{generalisation-thakur}, all the algebraic relations defined over $\FF_p$ between 
the elements $\zeta_A(n;\sigma_\Sigma)$ with $n\equiv|\Sigma|\pmod{q-1}$ can be derived from the harmonic product and for each zeta value $\zeta_A(n;\sigma_\Sigma)$ it should be possible to derive explicit formulas like in (\ref{identity-conjecture}) by using the harmonic product of Theorem \ref{corsumshuffle} (or in \cite{PEL3}). However, carrying this program might be very difficult in practice due to the combinatorial computations involved. The challenge is to introduce other techniques to tackle it. This was accomplished by 
Hung Le and Ngo Dac in \cite{HUN&NGO}, where they proved a particular case of this conjecture hence proving a conjectural formula of the author of the present text. Their result is reviewed in the following \S \ref{Some-evidences}.

\subsubsection{Some evidences}\label{Some-evidences}

We focus on the case $n=1$ in Conjecture \ref{generalisation-thakur} so that we can now suppose that 
$|\Sigma|=m(q-1)+1$ with $m\geq 0$. We know from \cite{ANG&PEL,APTR} that
\begin{equation}\label{formula-a-la-euler}
\zeta_A(1;\sigma_\Sigma)=\frac{(-1)^m\widetilde{\pi}\mathbb{B}_\Sigma}{\omega_\Sigma},\quad |\Sigma|\equiv1\pmod{q-1},\quad |\Sigma|>1,
\end{equation}
where $\mathbb{B}_\Sigma\in A[\underline{t}_\Sigma]$ (\footnote{$\mathbb{B}$ stands for 'Bernoulli'.}) is a monic polynomial in $\theta$ of degree $m-1$
when $m\geq 1$ and $\omega_\Sigma=\prod_{i\in\Sigma}\omega(t_i)\in\TT_\Sigma^\times$.
If $m=0$, the conjecture is clearly verified thanks to the formula (\ref{first-formula-pel}). If $m=1$ then 
$\mathbb{B}_\Sigma=1$ by \cite[Corollary 7.3]{APTR} so that $$\zeta_A(1,\sigma_\Sigma)=\tau^{-1}\left(\prod_{i\in\Sigma}\zeta_A(1,\chi_{t_i})\right)\in\mathbb{I}$$
confirming Conjecture \ref{conjecture-pellarin} also in this case.

To describe the case $m=2$ (so that $|\Sigma|=2q-1$) we shall introduce the notation
$$\mathcal{L}_U^{(m)}:=\tau^m\left(\prod_{i\in U}\zeta_A(1,\chi_{t_i})\right),$$ for $U\subset\Sigma$.
Then it is possible to show the following explicit formula:
$$\zeta_A(1,\sigma_\Sigma)=\sum_{\begin{smallmatrix}\Sigma=U_1\sqcup U_2\\ |U_1|=q-1\\ |U_2|=q\end{smallmatrix}}\mathcal{L}_{U_1}^{(-1)}\mathcal{L}_{U_2}^{(-2)},$$ where $\sqcup$ denotes disjoint union.
Now, recall that the right-hand side is equal to 
$\frac{\widetilde{\pi}\mathbb{B}^*_\Sigma}{\omega_\Sigma},$
with $$\mathbb{B}^*_\Sigma=-\sum_{\begin{smallmatrix}U_2\subset \Sigma\\ |U_2|=q\end{smallmatrix}}\prod_{i\in U_2}\Big(t_i-\theta^{\frac{1}{q}}\Big),$$
while the left-hand side is easily seen to be equal to 
$\frac{\widetilde{\pi}\mathbb{B}_\Sigma}{\omega_\Sigma},$ with $$\mathbb{B}_\Sigma=\theta-\sum_{\begin{smallmatrix}V\subset\Sigma\\ |V|=q\end{smallmatrix}}\prod_{i\in V}t_i=-e_q\Big(t_i-\theta^{\frac{1}{q}}:i\in\Sigma\Big)$$ (with $e_n$ denoting here the $n$-th elementary symmetric polynomial), and it is easy to see that $\mathbb{B}_\Sigma=\mathbb{B}_\Sigma^*$ (all the terms defined over $\FF_p[\theta^{\frac{1}{q}}]$ but not over $\FF_p[\theta]$ cancel. More generally we have the next result (see \cite[Theorem 1.3]{HUN&NGO}):
\begin{Theorem}[Hung Le and Ngo Dac]\label{conj2}
For all $m\geq 0$ and for all $q>m$ the following formula holds:
\begin{equation}\label{formulaconj}
\zeta_A(1;\sigma_\Sigma)=\sum_{\begin{smallmatrix}U_1\sqcup\cdots\sqcup U_m=\Sigma\\
q^{-1}|U_1|+\cdots+q^{-m}|U_m|=1\end{smallmatrix}}\mathcal{L}_{U_1}^{(-1)}\cdots\mathcal{L}_{U_m}^{(-m)}.\end{equation}
\end{Theorem}
The formula (\ref{formulaconj}) has been conjectured by the author of the present manuscript and incorporated in a previous version of it. The work \cite{HUN&NGO} has been anticipated by the verification of the cases $m=1,2,3,4$ by Ngo Dac in \cite{DAC}. Although Conjecture \ref{generalisation-thakur} predicts that such formulas can all be derived from 
the harmonic product the method of Hung Le and Ngo Dac does not use it, and introduces new tools which do not reduce it to a mere computational verification, the latter being most likely out of reach.

\subsubsection{More about Theorem \ref{conj2}} 
 It is not hard to show that Hung Le and Ngo Dac's Theorem is equivalent to the following corollaries that were stated as Conjectures in the earlier versions of the present manuscript:

\begin{Corollary}\label{conj3}
Assuming that $m\geq 2$ and that $q>m$, we have the formula
$$\zeta_A(1,\sigma_\Sigma)=\sum_{r=0}^{m-2}\sum_{\begin{smallmatrix}U\sqcup V\sqcup {\Sigma'}=\Sigma\\ |V|=q-r-1\\
|U|=rq\end{smallmatrix}}
\tau^{-1}(\zeta_A(1,\sigma_{\Sigma'}))\mathcal{L}^{(-2)}_U\mathcal{L}^{(-1)}_V.$$
\end{Corollary}
The interest of Corollary \ref{conj3} is that it can be considered in parallel with similar (but not analogue) classical formulas by Euler. We recall that the well-known Riccati's differential equation
$f'=-1-f^2$ satisfied by the cotangent function $f(x)=\cot(x)$ implies, via the formula
$-\frac{\pi x}{2}\cot(\pi x)=\sum_{i\geq 0}\zeta(2i)x^{2i}$:
$$\left(n+\frac{1}{2}\right)\zeta(2n)=\sum_{i=1}^{n-1}\zeta(2i)\zeta(2n-2i),\quad n>1.$$ Note that the coefficients in the quadratic expression on the right-hand side are all equal to $1$. 

Theorem \ref{conj2} implies nice formulas for the polynomials $\mathbb{B}_\Sigma\in A[\underline{t}_\Sigma]$  (when $|\Sigma|>q$.
Indeed, observe that for all $m\geq 1$, 
\begin{equation}\label{tobeused}
\tau^{-m}((t-\theta)\omega)^{-1}=\Big(t-\theta^{\frac{1}{q^{m-1}}}\Big)\cdots
\Big(t-\theta^{\frac{1}{q}}\Big)\omega^{-1}.\end{equation}
Hence, $$\tau^{-m}\Big(\zeta_A(1,\chi_t)\Big)=-\frac{\widetilde{\pi}^{\frac{1}{q^m}}\Big(t-\theta^{\frac{1}{q^{m-1}}}\Big)\cdots
\Big(t-\theta^{\frac{1}{q}}\Big)}{\omega},\quad m\geq1.$$ Setting $b_m^*:=\Big(t-\theta^{\frac{1}{q^{m-1}}}\Big)\cdots
\Big(t-\theta^{\frac{1}{q}}\Big)$ (again for $m\geq 1$) and $B_m^*(\underline{t}_\Sigma)=\prod_{i\in\Sigma}b^*_m(t_i)$, we thus have:
\begin{Corollary}
The following formula holds, when $q>m$.
$$\mathbb{B}_\Sigma=(-1)^{m-1}\sum_{\begin{smallmatrix}U_1\sqcup\cdots\sqcup U_m=\Sigma\\
q^{-1}|U_1|+\cdots+q^{-m}|U_m|=1\end{smallmatrix}}B_1^*(\underline{t}_{U_1})\cdots B_m^*(\underline{t}_{U_m}).$$
\end{Corollary}
Similarly, Corollary \ref{conj3} is equivalent to:
\begin{Corollary}\label{lastconjecture}
The following formula holds, for $|\Sigma|=m(q-1)+1$ with $q>m\geq 2$.
$$\mathbb{B}_\Sigma=\sum_{r=0}^{m-2}(-1)^{r+1}\sum_{\begin{smallmatrix}U\sqcup V\sqcup \Sigma'=\Sigma\\ |U|=qr\\
|V|=q-r-1\\
|\Sigma'|=(m-r-1)(q-1)+1\end{smallmatrix}}\tau^{-1}(\mathbb{B}_{\Sigma'})\prod_{i\in U\sqcup\Sigma'}\Big(t_i-\theta^{\frac{1}{q}}\Big).$$
\end{Corollary}

\subsection{A modular analogue}

We end this work with a conjectural formula which can be derived from Theorem \ref{conj2}.
We set, with $U\subset\NN^*$ a finite subset and $j\in\ZZ$:
$$\mathcal{E}_U^{(j)}:=\tau^j\left(\bigotimes_{i\in U}\mathcal{E}(1;\rho_{t_i}^*)\right),$$ for $U\subset\Sigma$.
Note that this needs not to represent an analytic function $\Omega\rightarrow\LL_\Sigma^{N\times 1}$ for $N\geq 1$ if $j<0$.
\begin{Conjecture}\label{conjE}
For all $m\geq 0$, $|\Sigma|=s=m(q-1)+1$ and for all $q>m$, the following formula holds:
\begin{equation}\label{formulaconjE}
\mathcal{E}(1;\rho_\Sigma^*)=\sum_{\begin{smallmatrix}U_1\sqcup\cdots\sqcup U_m=\Sigma\\
q^{-1}|U_1|+\cdots+q^{-m}|U_m|=1\end{smallmatrix}}\mathcal{E}_{U_1}^{(-1)}\otimes\cdots\otimes\mathcal{E}_{U_m}^{(-m)}.\end{equation}
\end{Conjecture}
We note that (\ref{formulaconjE}) expresses the analytic function $\mathcal{E}(1;\rho_\Sigma^*)$ as a combination 
of non-analytic functions if $s\geq 2q-1$.
Clearly, Theorem \ref{conj2} and Conjecture \ref{conj-zetavarphi}, or Conjecture \ref{conj-mult-eisentein}
imply Conjecture \ref{conjE} (and the latter implies Theorem \ref{conj2}). The cases $s=1,q$ are obviously verified, see (\ref{eisenstein-small-sigma}). The case
$s=2q-1$ is at the moment unsolved. The author was only able to see that the $u$-expansions of the $\emptyset$-coordinates of both sides in (\ref{formulaconjE}) agree up to a certain order but this is not enough to conclude.

\section*{Addendum}

Between the second and the third version of the present manuscript the author wrote \cite{GEZ&PEL}, in collaboration with O. Gezmi\c s. The reader may notice that some results in that preprint partially depend on results written here (for instance on the content of our \S \ref{shufflerelations}) and that, in this third version, references to \cite{GEZ&PEL} have been introduced. However, there is no loop in the chains of deductions and the references to \cite{GEZ&PEL} have been introduced for the sake of completeness.

\end{document}